\documentclass[a4paper]{amsart}
\usepackage{amsmath,amssymb,amsthm}
\usepackage{hyperref}
\usepackage{xcolor}
\usepackage[all]{xy}
\usepackage{tikz-cd}
\usepackage{mathrsfs}
\usetikzlibrary{decorations,decorations.pathreplacing}
\usepackage{nameref}
\usepackage{zref-xr}
\zxrsetup{toltxlabel}
\zexternaldocument*[gl3:]{gl3}

\numberwithin{equation}{section}

\newtheorem{thm}{Theorem}[section]
\newtheorem{lem}[thm]{Lemma}
\newtheorem{prop}[thm]{Proposition}
\newtheorem{cor}[thm]{Corollary}
\newtheorem{conj}[thm]{Conjecture}

\newtheorem{assump}[thm]{Assumption}
\theoremstyle{remark}
\newtheorem{remark}[thm]{Remark}
\theoremstyle{definition}
\newtheorem{defn}[thm]{Definition}

\DeclareMathOperator{\Ind}{Ind}
\DeclareMathOperator{\nInd}{n-Ind}
\DeclareMathOperator{\ncInd}{nc-Ind}
\DeclareMathOperator{\cInd}{c-Ind}
\DeclareMathOperator{\AnInd}{Ind}
\DeclareMathOperator{\diag}{diag}

\DeclareMathOperator{\soc}{soc}
\DeclareMathOperator{\Hom}{Hom}
\DeclareMathOperator{\End}{End}
\DeclareMathOperator{\Ad}{Ad}
\DeclareMathOperator{\ad}{ad}
\DeclareMathOperator{\im}{im}
\DeclareMathOperator{\Lie}{Lie}
\DeclareMathOperator{\Rep}{Rep}
\DeclareMathOperator{\Res}{Res}
\DeclareMathOperator{\Gal}{Gal}

\DeclareMathOperator{\pr}{pr}
\DeclareMathOperator{\coker}{coker}
\DeclareMathOperator{\Ann}{Ann}
\DeclareMathOperator{\Aut}{Aut}
\newcommand{\C}{\mathbb{C}}
\newcommand{\Z}{\mathbb{Z}}
\newcommand{\Q}{\mathbb{Q}}
\newcommand{\R}{\mathbb{R}}
\DeclareMathOperator{\SL}{SL}
\DeclareMathOperator{\SU}{SU}
\newcommand{\tSU}{\widetilde{\SU}}
\DeclareMathOperator{\U}{U}
\DeclareMathOperator{\GU}{GU}
\DeclareMathOperator{\Sp}{Sp}
\DeclareMathOperator{\GSp}{GSp}
\DeclareMathOperator{\SO}{SO}
\DeclareMathOperator{\GSO}{GSO}
\DeclareMathOperator{\Spin}{Spin}
\DeclareMathOperator{\GSpin}{GSpin}
\DeclareMathOperator{\GL}{GL}
\DeclareMathOperator{\PGL}{PGL}
\DeclareMathOperator{\Sym}{Sym}
\DeclareMathOperator{\BCH}{BCH}

\newcommand{\sm}{\mathrm{sm}}
\newcommand{\adm}{\mathrm{adm}}
\DeclareMathOperator{\supp}{supp}
\newcommand{\alggrp}[1]{\underline{#1}}
\newcommand{\an}{\mathrm{an}}

\newcommand{\opp}{\mathrm{op}}
\newcommand{\der}{\mathrm{der}}
\newcommand{\abs}{\mathrm{abs}}
\newcommand{\cts}{\mathrm{cts}}

\newcommand{\alg}{\mathrm{alg}}
\newcommand{\unram}{\mathrm{ur}}
\newcommand{\red}{\mathrm{red}}
\newcommand{\is}{\mathrm{is}}
\renewcommand{\sc}{\mathrm{sc}}
\newcommand{\glob}[1]{\boldsymbol{#1}}
\newcommand{\xnr}{X_{\mathrm{nr}}}
\newcommand{\Nrd}{\mathrm{Nrd}}

\newcommand{\ang}[1]{\langle #1\rangle}
\newcommand{\congto}{\xrightarrow{\,\sim\,}}
\newcommand{\into}{\hookrightarrow}
\newcommand{\onto}{\twoheadrightarrow}

\renewcommand{\simeq}{\cong}

\newcommand{\lowprime}[1][\prime]{
  \mathchoice
    {\raisebox{-0.7ex}{$\scriptstyle #1$}}
    {\raisebox{-0.7ex}{$\scriptstyle #1$}}
    {\raisebox{-0.7ex}{$\scriptstyle #1$}}
    {\raisebox{-0.3ex}{$\scriptscriptstyle #1$}}
}

\renewcommand{\o}[1]{\overline{#1}}
\newcommand{\un}[1]{\underline{#1}}
\newcommand{\wt}[1]{\widetilde{#1}}
\newcommand{\wh}[1]{\widehat{#1}}

\title[On the irreducibility of $p$-adic Banach principal series]{On the irreducibility of $p$-adic Banach principal series of $p$-adic reductive groups}
\author{Noriyuki Abe}
\address[N. Abe]{Graduate School of Mathematical Sciences, the University of Tokyo, 3-8-1 Komaba, Meguro-ku, Tokyo 153-8914, Japan.}
\thanks{The first-named author was supported by JSPS KAKENHI Grant Number 18H01107.}
\email{abenori@ms.u-tokyo.ac.jp}
\author{Florian Herzig}
\address[F. Herzig]{Department of Mathematics, University of Toronto, 40 St.\ George Street, Toronto, ON M5S 2E4, Canada.}
\thanks{The second-named author was partially supported by an NSERC grant.}
\email{herzig@math.toronto.edu}

\begin{document}

\begin{abstract}
  Suppose that $G$ is the group of $F$-points of a connected reductive group over $F$, where $F/\Q_p$ is a finite extension.
  We study the (topological) irreducibility of principal series of $G$ on $p$-adic Banach spaces.
  For unitary inducing representations we obtain an optimal irreducibility criterion, and for $G = \GL_n(F)$ (as well as for arbitrary split groups under slightly stronger conditions) we obtain a variant of Schneider's conjecture \cite[Conjecture 2.5]{MR2275644}.
  In general we reduce the irreducibility problem to smooth inducing representations and almost simple simply-connected $G$.
  Our methods include locally analytic representation theory, the bifunctor of Orlik--Strauch, translation functors, as well as new results on reducibility points of smooth parabolic inductions.
\end{abstract}

\maketitle

\tableofcontents

\section{Introduction}
\label{sec:introduction}

Suppose that $F/\Q_p$ is a finite extension and $G = \alggrp G(F)$ the group of $F$-points of a connected reductive group $\alggrp G$ over $F$.
This paper concerns the continuous representations of $G$ on $p$-adic Banach spaces over a coefficient field $C$ that is a (sufficiently large) finite extension of $\Q_p$. Such Banach representations were introduced in the work of Schneider--Teitelbaum \cite{MR1900706} and play a fundamental role in the $p$-adic Langlands program (see for example \cite{MR2097893}, \cite{MR2359853}, \cite{Colmez}, \cite{emerton-local-global}, \cite{MR3150248}, \cite{MR3529394}), particularly the unitary such representations, i.e.\ those that admit a $G$-invariant defining norm.
Little has been known about Banach representations outside the group $\GL_2(\Q_p)$ so far.
The main goal of this paper is to greatly extend our knowledge about the (topological) irreducibility of Banach principal series representations.
In particular, we prove a variant of Schneider's conjecture \cite[Conjecture 2.5]{MR2275644}.

Fix a minimal parabolic subgroup $\alggrp B = \alggrp Z \alggrp U$ with Levi subgroup $\alggrp Z$ and unipotent radical $\alggrp U$.
Let $\alggrp P = \alggrp L \alggrp N$ be a parabolic subgroup containing $\alggrp B$ with Levi subgroup $\alggrp L$ containing $\alggrp Z$ and unipotent radical $\alggrp N$.
We write $B := \alggrp B(F)$, etc.
If $\sigma$ is a continuous representation of $L$ on a finite-dimensional $C$-vector space, then we inflate $\sigma$ to $P$ and form the parabolic induction
\begin{equation*}
  (\Ind_P^G \sigma)^\cts := \{ \text{$f\colon G\to \sigma$ cts.} : \text{$f(p'g) = \sigma(p')f(g)$ for any $g\in G$, $p'\in P$} \},
\end{equation*}
which carries a natural Banach topology making it into an (admissible) Banach representation of $G$ under right translation that we call a \emph{Banach principal series}.
We show moreover in Proposition~\ref{prop:adm-fin-length} that it is topologically of finite length, confirming an expectation of \cite[\S2]{MR2275644}.
(For technical reasons we will assume for the remainder of the introduction that $p > 2$, resp.\ $p > 3$, if the absolute root system of $\alggrp{G}$ has irreducible components of type $B$, $C$ or $F_{4}$, resp.\ $G_{2}$.)
From now on we will assume that $\sigma$ is absolutely irreducible.

\subsection{Unitary case}
\label{sec:unitary-case-intro}

Our first theorem gives an irreducibility criterion for general $\alggrp G$ in the case where $\sigma$ is unitary.
Let $\alggrp S$ be the maximal split subtorus of the center of $\alggrp Z$.
Let $\Delta$ (resp.\ $\Delta_L$) be the set of simple roots of $\alggrp S$ in $\alggrp B$ (resp.\ in $\alggrp B \cap \alggrp L$).
For any $\alpha \in \Delta$ let $\alggrp L_\alpha$ denote the Levi subgroup containing $\alggrp Z$ with $\Delta_{L_\alpha} = \{\alpha\}$.
Moreover, let $L_\alpha'$ denote the smallest normal subgroup of $L_\alpha$ generated by $U\cap L_{\alpha}$.
(This is not an algebraic subgroup in general.)

\begin{thm}[Theorem~\ref{thm:unitary case}]\label{thm:unitary case-parab-intro}
Let $\sigma$ be a finite-dimensional absolutely irreducible unitary Banach representation of $L$. Then the following are equivalent:
\begin{enumerate}
\item $(\Ind_{P}^{G}\sigma)^{\cts}$ is reducible;
\item there exists $\alpha \in \Delta \setminus \Delta_L$ such that $\sigma|_{Z \cap L'_{\alpha}}$ is trivial;
\item the representation $\sigma$ extends to a continuous representation of a Levi subgroup strictly containing $L$.
\end{enumerate}
\end{thm}

We briefly explain the easy implications (ii)$\Rightarrow$(iii)$\Rightarrow$(i).
If (ii) holds, let $\alggrp L_1 \supset \alggrp Z$ be the Levi subgroup with $\Delta_{L_1} = \Delta_L \sqcup \{\alpha\}$.
If $\sigma$ is trivial on $Z \cap L'_{\alpha}$, then $\sigma$ has a unique (continuous) extension to $L_1$ which is trivial on $N\cap L_1$~\cite[II.7 Proposition]{AHHV}, and we get (iii).
If (iii) holds, $\sigma$ has a continuous extension to a larger Levi $L_1 \supsetneq L$, which we still denote by $\sigma$.
Let $\alggrp P_1$ be the parabolic subgroup containing $\alggrp P$ with Levi subgroup $\alggrp L_1$.
Then $(\Ind_{P_1}^{G}\sigma)^{\cts}\hookrightarrow (\Ind_{P}^{G}\sigma)^{\cts}$ is a proper closed subrepresentation and so (i) holds.

In fact, we obtain the equivalence in Theorem~\ref{thm:unitary case-parab-intro} under the weaker assumption that $\sigma$ lies in a certain closed cone, defined as follows.
Let $\omega_{\sigma_{U \cap L}} \colon S \to C^\times$ be the central character of the coinvariant representation $\sigma_{U \cap L}$ (which is absolutely irreducible).
Then $\lvert\omega_{\sigma_{U \cap L}}\rvert_C \colon S \to \R^\times_{>0}$ is an unramified character of $S$ (where $\lvert\cdot\rvert_{C}$ denotes the absolute value of $C$ with $\lvert p\rvert_C = p^{-1}$) and we let $e(\omega_{\sigma_{U \cap L}}) \in X^*(\alggrp S) \otimes \R$ denote its Harish-Chandra parameter.
Then it suffices for our proof to demand that $e(\omega_{\sigma_{U \cap L}})$ be dominant (or even that it is almost dominant in a precise and explicit sense, see Remark~\ref{rk:unitary-open-cone}).
We also have a version of this theorem where we only demand a certain dominance condition on the central character $\omega_\sigma$ of $\sigma$.

\subsection{Schneider's conjecture}
\label{sec:schn-conj}

Dropping the assumption that $\sigma$ is unitary, we prove a variant of Schneider's conjecture for the group $\GL_n(F)$ and, under slightly weaker conditions, for general split (or even quasisplit) groups $G$.

To state our results, let $\delta_B$ denote the modulus character of $B$.
We say that a character $\lambda \colon F^\times \to C^\times$ is \emph{non-positive algebraic} if it is of the form $\lambda(t) = \prod_{\kappa\colon F\to C} \kappa(t)^{a_\kappa}$
for some $(a_\kappa) \in \Z_{\le 0}^{\Hom(F,C)}$.
For $1 \le i < j \le n$ let $\alpha_{i,j}^\vee \colon \mathbb G_m \to \SL_n$ denote the coroot sending $t$ to $\diag(1,\dots,t,\dots,t^{-1},\dots,1)$ with $t$ (resp.\ $t^{-1}$) appearing in the $i$-th (resp.\ $j$-th) entry.
Let $\lvert \cdot\rvert_F$ denote the normalized absolute value of $F$.

\begin{thm}[Theorem~\ref{thm:GLn-irred}]\label{thm:GLn-irred-intro}
  Let $G = \GL_{n}(F)$ or $\SL_{n}(F)$, $B$ the upper-triangular Borel subgroup, and $Z$ the diagonal maximal torus.
  Let $\chi\colon Z \to C^{\times}$ be a continuous character.
  If $(\Ind_{B}^{G}\chi)^{\cts}$ is reducible, then there exist $1\le i < j \le n$ and non-positive algebraic characters $\lambda_k$ for $i \le k < j$ such that 
  \begin{itemize}
  \item $\chi\delta_B^{-1/2} \circ \alpha_{k,k+1}^\vee \equiv \lambda_k$ near the identity for any $i \le k < j$, and
  \item $\chi\delta_B^{-1/2} \circ \alpha_{i,j}^\vee = \lvert \cdot\rvert_F^{-1}\prod_{k=i}^{j-1} \lambda_k$.
  \end{itemize}
\end{thm}

In particular, there exists a positive root $\alpha_{i,j}$ and a non-positive algebraic character $\lambda$ such that $\chi\delta_B^{-1/2} \circ \alpha_{i,j}^\vee = \lvert \cdot\rvert_F^{-1} \lambda$.
As in Theorem~\ref{thm:unitary case-parab-intro}(ii)$\Rightarrow$(i), if the condition of the previous sentence holds with $j = i+1$, then $(\Ind_{B}^{G}\chi)^{\cts}$ is reducible.
When $n = 3$ we show that these are in fact the only possible reducibilities, see our companion paper \cite{gl3}.

We have a slightly weaker theorem valid for any split connected reductive group (adding a Weyl group regularity condition).

\begin{thm}[Theorem~\ref{thm:split-groups}]\label{thm:split-groups-intro}
  Suppose that $\alggrp G$ is split.   Let $\chi \colon Z \to C^{\times}$ be a continuous character.
  Let $\alggrp P = \alggrp L \alggrp N$ be the largest standard parabolic such that for all positive roots $\alpha \colon \alggrp Z \to \mathbb G_m$ occurring in $\alggrp L$ we have
  $\chi\delta_B^{-1/2} \circ \alpha^\vee \equiv \lambda_\alpha$ near the identity, for some non-positive algebraic character $\lambda_\alpha$.
  Assume that for all $w \in N_L(Z) \setminus Z$ there exists a root $\alpha$ of $\alggrp L$ such that
  \begin{itemize}
  \item $(\chi \delta_B^{-1/2} \circ w^{-1}\alpha^\vee)\cdot (\chi \delta_B^{-1/2} \circ \alpha^\vee)^{-1} \ne \lambda_{w^{-1}\alpha} \lambda_\alpha^{-1}$.
  \end{itemize}
  If $(\Ind_{B}^{G}\chi)^{\cts}$ is reducible, then there exists a positive root $\alpha$ occurring in $\alggrp L$ such that
  \begin{itemize}
  \item $\chi\delta_B^{-1/2} \circ \alpha^\vee = \lvert \cdot\rvert_F^{-1} \lambda_\alpha$.
  \end{itemize}
\end{thm}

We now recall Schneider's conjecture and discuss its relationship with Theorems~\ref{thm:GLn-irred-intro}, \ref{thm:split-groups-intro}.
For this, we assume in addition that $\alggrp G$ is semisimple simply connected and split (as is assumed there).
Let $\delta^{1/2} \colon \alggrp Z \to \mathbb G_m$ denote the half-sum of positive roots, which is integral by our assumption on $\alggrp G$, and note that $\lvert \delta^{1/2}\rvert_F = \delta_B^{1/2}$.

\begin{conj}[{\cite[Conj.\ 2.5]{MR2275644}}]\label{conj:schneider}
  If $(\Ind_{B}^{G}\chi)^{\cts}$ is reducible, then there exists a positive root $\alpha$ of $\alggrp G$ such that
\begin{itemize}
\item $\chi\delta^{-1/2} \circ \alpha^\vee = (\cdot)^{m_\alpha-1}$ for some non-positive integer $m_\alpha$.
\end{itemize}
\end{conj}

This formulation is not quite correct when $F \ne \Q_p$:
let $\alggrp G = \SL_2$ and choose $\chi$ such that $\chi\delta^{-1/2} \circ \alpha^\vee = (\cdot)^{-1} \lambda$ where $\alpha$ denotes the unique positive root and $\lambda$ is non-positive algebraic but not an integer power.
(The last condition requires $F \ne \Q_p$.)
Equivalently $\chi = \lambda$, by using the identification $\alpha^\vee : \mathbb G_m \congto \alggrp Z \subset \SL_2$.
Therefore $(\Ind_{B}^{G}\chi)^{\cts}$ contains the algebraic induction of $\chi$ (the irreducible algebraic representation of lowest weight $\lambda$) as nonzero finite-dimensional, hence closed, subrepresentation.
However this issue can be avoided by modifying the condition in Conjecture~\ref{conj:schneider} to
\begin{equation}\label{eq:schneider-modif}
\chi\delta^{-1/2} \circ \alpha^\vee = (\cdot)^{-1} \lambda_\alpha\ \text{for some non-positive algebraic character $\lambda_\alpha$},
\end{equation}
and we do not know any counterexample to this modified conjecture.
(If $\alpha$ is a simple root, condition~\eqref{eq:schneider-modif} is equivalent to $\chi\delta_B^{-1/2} \circ \alpha^\vee = \lvert \cdot\rvert_F^{-1} \lambda_\alpha$, exactly as in Theorem~\ref{thm:split-groups-intro}.
For general $\alpha$ we remark that the conditions/conclusion in Theorems~\ref{thm:GLn-irred-intro}, \ref{thm:split-groups-intro} are compatible with the equivalence in Theorem~\ref{thm:criterion-intro} below, whereas condition~\eqref{eq:schneider-modif} is not.)

We obtain more refined versions of Theorem~\ref{thm:split-groups-intro} for any classical quasisplit group in \S\ref{subsec:classical quasisplit groups}.
(These results are stated for smooth $\sigma$, for simplicity.)

For general reductive $\alggrp G$ one may wonder whether $(\Ind_{B}^{G}\sigma)^{\cts}$ is reducible if and only if $(\Ind_{B\cap L_\alpha}^{L_\alpha}\sigma)^{\cts}$ is reducible for some \emph{simple} root $\alpha$.
Our evidence is rather limited, consisting mostly of Theorem~\ref{thm:unitary case-parab-intro} and our result for $\GL_3(F)$ \cite{gl3}.
The group $\alggrp L_\alpha$ is of semisimple rank 1, and we obtain an optimal irreducibility criterion for $(\Ind_{B\cap L_\alpha}^{L_\alpha}\sigma)^{\cts}$,
at least when $\sigma$ is simple as $\mathfrak z_C$-module (which holds when $\dim_C \sigma = 1$) or when the unipotent radical $\alggrp U$ is abelian.
See Theorem~\ref{thm:rank-1} and also Corollary~\ref{cor:rank-1-split}.
In particular, if the answer to our question is indeed yes, Corollary~\ref{cor:rank-1-split} would prove the modified version of Schneider's conjecture above (using condition~\eqref{eq:schneider-modif} for only simple roots $\alpha$).

\subsection{A criterion and some reductions}
\label{sec:criterion}

By replacing $\alggrp G$ by $\Res_{F/\Q_p} \alggrp G$ we may and will assume that $F = \Q_p$.
Then the action of $L$ on $\sigma$ becomes locally analytic~\cite[\S V.9]{MR1176100} and we can study the Banach representation $(\Ind_{P}^{G}\sigma)^{\cts}$ by means of its dense subspace of locally analytic vectors
\begin{equation*}
  (\Ind_P^G \sigma)^\an := \{ \text{$f\colon G\to \sigma$ loc.\ an.} : \text{$f(p'g) = \sigma(p')f(g)$ for any $g\in G$, $p'\in P$} \},
\end{equation*}
see \cite{MR1900706}, \cite{locallyanalytic-memoir}.
Note that $(\Ind_{P}^{G}\sigma)^{\an}$ naturally carries a compact type topology and becomes an admissible locally analytic representation of $G$ in the sense of Schneider--Teitelbaum \cite{ST}.
If $\sigma$ is smooth and we replace locally analytic functions by locally constant functions, then we obtain a smooth subrepresentation $(\Ind_P^G \sigma)^\sm$ of $(\Ind_P^G \sigma)^\an$. 
Orlik and Strauch \cite{OS2}, \cite{OS3} introduced a beautiful theory to understand the structure of the locally analytic principal series $(\Ind_{P}^{G}\sigma)^{\an}$.
(In their papers they restrict to split groups $\alggrp G$. 
See Appendix~\ref{app:orlik-strauch} for the general case, and note that another treatment will appear in the revised version of \cite{OS3}.)
We just recall some basic definitions for now.
Let $\mathfrak g$ denote the Lie algebra of $G$, $\mathfrak g_C := \mathfrak g \otimes_F C$, and likewise for other subgroups of $G$.
For any parabolic subgroup $\alggrp P$, Orlik--Strauch define noetherian and artinian abelian categories $\mathcal O^{\mathfrak p}$, an adaptation of a parabolic BGG category $\mathcal O$ over $C$ whose objects consist of certain finitely generated $U(\mathfrak g_C)$-modules, and $\mathcal O^P$ whose objects consist of pairs $(M,\tau)$ with $M \in \mathcal O^{\mathfrak p}$ and $\tau$ a locally finite locally analytic action of $P$ on $M$ satisfying $\tau(p') \circ X \circ \tau(p')^{-1} = \Ad(p')(X)$ for $p' \in P$, $X \in \mathfrak g_C$ which lifts the given action of $\mathfrak p_C$ (see~\S\ref{sec:funct-orlik-stra} for the precise definition).
If $\alggrp Q \supset \alggrp P$ is another parabolic subgroup, then $\mathcal O^{\mathfrak q} \subset \mathcal O^{\mathfrak p}$ and $\mathcal O^Q \subset \mathcal O^P$ are naturally full subcategories, and we say that $M \in \mathcal O^P$ is \emph{equimaximal} if for any parabolic subgroup $\alggrp Q$ containing $\alggrp P$ we have $M \in \mathcal O^Q$ if and only if $M|_{\mathfrak g_C} \in \mathcal O^{\mathfrak q}$.
For any simple object $W \in \mathcal O^L$ there is a unique simple object $\underline L(W) \in \mathcal O^P$ such that $\underline L(W)^N \cong W$ (Lemma~\ref{lm:N-invts}).
(Here, $\mathcal O^L$ is the category $\mathcal O^G$ with $G = L$. Its objects are automatically finite-dimensional.)

Recall that $\sigma$ is an absolutely irreducible finite-dimensional continuous representation of $L$.
Then $\sigma$ lies in $\mathcal O^L$ (recall that $F = \Q_p$ now). Assume now that $\sigma \cong \sigma_0 \otimes \tau$ for some $\sigma_0 \in \mathcal O^L$ whose underlying $\mathfrak l_C$-module is simple and some smooth $L$-representation $\tau$ such that moreover $\underline L(\sigma_0')$ is equimaximal, where $\sigma_0'$ denotes the dual of $\sigma_0$. Such a decomposition of $\sigma$ always exists when the derived subgroup $\alggrp G^\der$ is simply connected by Lemma~\ref{lem:sigma0-tau-decomp} and is unique up to smooth characters of $L_Q$ in general, where $\alggrp Q = \alggrp L_{\alggrp Q}\alggrp N_{\alggrp Q}$ is the largest parabolic subgroup containing $\alggrp P$ such that $\underline L(\sigma_0') \in \mathcal O^Q$.

We then have the following theorem, which is fundamental to our work (in particular, to the proof of all the above theorems).

\begin{thm}[Theorem~\ref{thm:equivalence-irred}, Corollary~\ref{cor:irreduciblity}]\label{thm:criterion-intro}
  The following are equivalent:
  \begin{enumerate}
  \item $(\Ind_P^G \sigma)^\cts$ is irreducible;
  \item $(\Ind_{P\cap L_Q}^{L_Q} \sigma)^\cts$ is irreducible;
  \item $(\Ind_{P\cap L_Q}^{L_Q} \tau)^\cts$ is irreducible;
  \item any irreducible subrepresentation of $(\Ind_{P\cap L_Q}^{L_Q} \tau)^\sm$ is dense in $(\Ind_{P\cap L_Q}^{L_Q} \tau)^\cts$.
  \end{enumerate}
\end{thm}

In particular, $(\Ind_P^G \sigma)^\cts$ is irreducible if $(\Ind_{P\cap L_Q}^{L_Q} \tau)^\sm$ is irreducible, by part (iv).
We obtain the following corollary.

\begin{cor}[Corollary~\ref{cor:criterion-easy-cases}]\label{cor:criterion-easy-cases-intro}
  Each of the following conditions implies the next:
  \begin{enumerate}
  \item $U(\mathfrak g_C) \otimes_{U(\mathfrak p_C)} \sigma_0'$ is irreducible as $U(\mathfrak g_C)$-module;
  \item $Q = P$;
  \item $(\Ind_P^G \sigma)^\cts$ is irreducible.
  \end{enumerate}
\end{cor}

By Theorem~\ref{thm:criterion-intro} we can reduce the irreducibility problem of $(\Ind_P^G \sigma)^\cts$ to the case where $\sigma$ is smooth.
The following general result on intertwiners and some Clifford theory furthermore allows us to reduce to the case where $\alggrp G$ is almost simple and simply connected (and isotropic, meaning it contains a non-trivial split torus),
see Propositions~\ref{prop:isogenies} and \ref{prop:irreducibility for product group}.

\begin{prop}[Proposition~\ref{prop:intertwiners}]\label{prop:intertwiners_intro}
  Suppose that $\alggrp P = \alggrp L\alggrp N$ is a parabolic subgroup and $\sigma$, $\tau$ are Banach representations of $L$.
  Then the natural map $$\Hom_{L}^\cts(\sigma,\tau) \to \Hom_G^\cts((\Ind_{P}^G\sigma)^{\cts},(\Ind_{P}^G\tau)^{\cts})$$ is an isomorphism.
\end{prop}

\begin{remark}
  Based on the equivalence in Theorem~\ref{thm:criterion-intro} we would expect that it is extremely difficult to find an optimal irreducibility criterion of Banach principal series, even for $G = \GL_n(F)$.
The problem is that the submodule structure of the smooth principal series $(\Ind_{P\cap L_Q}^{L_Q} \tau)^\sm$ is badly understood in general, and even when it is understood it is challenging to prove the $p$-adic density of a proper smooth subrepresentation.
\end{remark}

\subsection{Sketch of proof of Theorem \ref{thm:unitary case-parab-intro}}
\label{sec:proof-intro2}

We have already seen the implications (ii)$\Rightarrow$(iii)$\Rightarrow$(i) right after the statement of Theorem~\ref{thm:unitary case-parab-intro}, so we only need to show that (i) implies (ii).
For simplicity we will assume that $P = B$.
We first make a reduction to $F = \Q_p$ and $\alggrp G^\der$ simply-connected, so in particular we have a decomposition $\sigma = \sigma_0 \otimes \tau$ as in \S\ref{sec:criterion}. Then Theorem~\ref{thm:criterion-intro} implies that if $(\Ind_B^G \sigma)^\cts$ is reducible, then $(\Ind_{B\cap L_Q}^{L_Q} \tau)^\sm$ is reducible, where $Q$ is defined as in \S\ref{sec:criterion}.
By relabeling we may assume that $Q = G$, i.e.\ $\underline L(\sigma_0') \in \mathcal O^G$.
This implies that, up to twisting $\sigma$ by a locally analytic character of $G$, we may assume that $\sigma_0$ is algebraic with antidominant central character.
Write $(\Ind_{B}^{G} \tau)^\sm$ as normalized induction $(\nInd_{B}^{G} \tau\delta_B^{-1/2})^\sm$.
Using that $\sigma$ is unitary we deduce that $\tau\delta_B^{-1/2}$ is Weyl group regular.
We now fix an isomorphism $\overline C \cong \C$ and work over the complex numbers.
Then a result of Harish-Chandra implies that there exists a reduced positive root $\alpha$ such that $(\nInd_{B \cap L_\alpha}^{L_\alpha} \tau\delta_B^{-1/2})^\sm$ is reducible, where $\alggrp L_\alpha$ is the Levi subgroup of semisimple rank 1 associated to $\alpha$.
We can write $\tau\delta_B^{-1/2} \cong \tau_u \delta_{B \cap L_\alpha}^s \eta$ with $\tau_u$ unitary (in the complex sense!), $s \in \R$, and $\eta$ a positive-real unramified character of $L_\alpha$.
A result of Silberger shows that $-\frac 12 \le s \le \frac 12$.
There is a tension between $\tau$ being $p$-adic unitary and the Silberger bound that allows us to deduce that $s = -\frac 12$, $\alpha$ is simple, and $\sigma_0$ is trivial on $Z \cap L_\alpha'$.
Crucially, the extreme bounds $s = \pm \frac{1}{2}$
imply that $\tau_u$ is trivial on $Z\cap L_\alpha'$.
This was our guess, based on the available literature, but Jean-Loup Waldspurger kindly provided a beautiful argument in general, see Proposition~\ref{prop:Waldspurger}.
We then deduce that $\sigma$ is trivial on $Z\cap L_\alpha'$.

However, the above sketch glosses over one very important point.
In comparing real and $p$-adic absolute values it is essential to know that the real number $s$ is in fact \emph{rational}.
For this we need to prove Corollary~\ref{cor:rationality, rank one-intro} below.

\subsection{Harish-Chandra's \texorpdfstring{$\mu$}{mu}-function and rationality of poles}
\label{sec:harish-chandra-mu}

The following results in smooth representation theory over $\C$ may be of independent interest.
Harish-Chandra's $\mu$-function (or Plancherel measure) $\mu^G$ controls the reducibility points of smooth parabolic inductions, as recalled in subsection~\ref{sec:reduc-points-parab}.

Suppose that $\alggrp{G}$ is any connected reductive group over $F$, $\alggrp{P} = \alggrp{L}\alggrp{N}$ a maximal parabolic subgroup such that $\alggrp{L}$ is an inner form of a group $\alggrp{L}'$ satisfying 
\[(\widetilde{\alggrp{L}}^{\lowprime})^{\der} \subset \alggrp{L}' \subset \widetilde{\alggrp{L}}^{\lowprime},\]
where 
\[\widetilde{\alggrp{L}}^{\lowprime} := \left(\prod_{i=1}^r \Res_{E_i/F}\GL_{n_i}\right)/\alggrp{H}\]
for some finite extensions $E_i/F$, integers $n_i \ge 1$ and a central, induced subtorus $\alggrp{H}$. 
\begin{thm}[Theorem~\ref{thm:rationality, inner}]\label{thm:rationality, inner-intro}
Let $\sigma$ be a discrete series representation of $L$.
If $\mu^G(\sigma\delta_{P}^{s})$ has a pole at $s = s_{0}\in \R$, then $s_{0}\in \Q$.
\end{thm}

Our proof uses global intertwining operators and the global Jacquet--Langlands correspondence \cite{MR2390289} to reduce to the case where $\alggrp G$ is quasisplit,
in which case the rationality follows from results of Shahidi \cite{MR1070599} (see Proposition~\ref{prop:rationality-discrete-series}), using that the supercuspidal support of $\sigma$ is automatically generic.
(This generalizes an argument of Mui\'c--Savin \cite{MR1749954}, who used this method to compare $\mu$-functions of $G$ and its quasisplit inner form when $\alggrp G$ is a hermitian quaternionic group of maximal Witt rank and $\alggrp P$ is the Siegel parabolic.
See also \cite{MR3194161} for further work in this direction.)

By the classification of almost simple rank one groups we obtain the following corollary.

\begin{cor}[Corollary~\ref{cor:rationality, rank one}]\label{cor:rationality, rank one-intro}
  Suppose that the adjoint group $\alggrp{G}^{\ad}$ is almost simple of rank one over $F$.
  Let $\sigma$ be a unitary supercuspidal representation of $Z$.
  If $\mu^G(\sigma\delta_{B}^{s})$ has a pole at $s = s_{0}\in \R$, then $s_{0}\in \Q$.
\end{cor}

In fact, using \cite[Theorem 8.1]{MR1070599} we can also bound the denominator and obtain an explicit finite list of possible poles $s_0$, see Remark~\ref{rk:rank-1-poles}.
(This finiteness property is in fact related to a conjecture of Howe, see \cite{MR806844}.)
We remark that there are 7 families of almost simple rank one groups, and Corollary~\ref{cor:rationality, rank one-intro} was only known in 4 cases previously, see Remark~\ref{rk:rank-1-known-cases}.

We also obtain the following corollary from our argument.
Here, $\widetilde{\alggrp{L}}$ is an inner form of $\widetilde{\alggrp{L}}^{\lowprime}$, naturally obtained from the inner form $\alggrp L$ of $\alggrp{L}'$, such that $\widetilde{\alggrp{L}}^{\lowprime[\der]} \subset \alggrp{L} \subset \widetilde{\alggrp{L}}$.

\begin{cor}[Corollary~\ref{cor:rationality, inner}]\label{cor:rationality, inner-intro}
  Suppose that $\sigma_1$, $\sigma_2$ are discrete series representations of $L$ that are conjugate under the action of $\widetilde{L}$.
  Then $\mu^G(\sigma_1\delta_{P}^{s}) = \mu^G(\sigma_2\delta_{P}^{s})$.
\end{cor}

This verifies \cite[Working Hypothesis 1.1]{MR3194161} in our more general setup.

\subsection{Sketch of proof of Theorems \ref{thm:GLn-irred-intro} and \ref{thm:split-groups-intro}}
\label{sec:proof-intro3}

Consider $\alggrp G$ a split group over $F$ and $\chi \colon Z \to C^\times$ a continuous character.
Let us assume for simplicity that $\alggrp G^\der$ is simply connected.
By thinking of $G$ as the $\Q_p$-points of $\Res_{F/\Q_p} \alggrp G$ we can work over $\Q_p$ when needed.
As in \S\ref{sec:criterion} we write $\chi = \sigma_0\tau$ with $\underline L(\sigma_0') \in \mathcal O^Q$ (with $Q$ maximal) and $\tau$ smooth.
It is not hard to see that $Q = P$ in the notation of Theorem~\ref{thm:split-groups-intro}.
We now prove the contrapositive of Theorems \ref{thm:GLn-irred-intro} and \ref{thm:split-groups-intro}.
By Theorem~\ref{thm:criterion-intro} we may reduce to the case where $Q = G$, and it suffices to show that any irreducible subrepresentation of $(\Ind_{B}^{G} \tau)^\sm$ is dense in $(\Ind_{B}^{G} \tau)^\cts$.
As $\tau$ is smooth, the first bullet in Theorem~\ref{thm:split-groups-intro} implies that $\tau\delta_B^{-1/2}$ is Weyl group regular and the second bullet in Theorem~\ref{thm:split-groups-intro} becomes that $\tau\delta_B^{-1/2} \circ \alpha^\vee \ne \lvert \cdot\rvert_F^{-1}$ for all positive roots $\alpha$.
Work of Rodier \cite{MR644842} then shows that $(\Ind_{B}^{G} \tau)^\sm = (\nInd_{B}^{G} \tau\delta_B^{-1/2})^\sm$ has an irreducible socle that is moreover generic.
(When $G = \GL_n(F)$ we do not need the regularity condition by Bernstein--Zelevinsky \cite{MR0579172}.)

We then conclude by the following result.

\begin{prop}\label{prop:generic-intro}
  Any generic subrepresentation of $(\Ind_{B}^{G} \tau)^\sm$ is dense in $(\Ind_{B}^{G} \tau)^\cts$.
\end{prop}

To prove the proposition, we first show by the geometric lemma (cf.\ Proposition~\ref{prop:whittaker criterion}) that any generic subrepresentation $\pi$ of $(\Ind_{B}^{G} \tau)^\sm$ contains an element $f$ that is supported on the big cell $B\backslash B w_0 U$, where $w_0$ is the longest Weyl group element.
Then we deduce by a $p$-adic approximation argument (Lemma~\ref{lem:supported on big cell}) that $\pi$ is dense in $(\Ind_{B}^{G} \tau)^\cts$.

In fact, Proposition~\ref{prop:generic-intro} generalizes to arbitrary $\alggrp G$, where we say that a smooth representation $\pi$ is \emph{generic} if the twisted coinvariants $(\pi \otimes \o C)_{U,\theta}$ are nonzero for some character $\theta \colon U \to \o C^\times$ that is non-trivial on each simple root subgroup.
(Here $\o C$ is an algebraic closure of $C$.)
We then generalize Theorem~\ref{thm:GLn-irred-intro} to $\GL_n(D)$, where $D$ is a finite-dimensional division algebra, see Theorem~\ref{thm:GL(n,D),dim>1}.
(Note that by Theorem~\ref{thm:criterion-intro} we may assume that the inducing representation $\sigma$ is smooth.) For this we generalize Bernstein--Zelevinsky's theory of derivatives~\cite{MR0579172}, cf.\ subsection~\ref{sec:GL(n,D)}.

\subsection{Sketch of proof of Theorem~\ref{thm:criterion-intro}}
\label{sec:proof-intro1}

We recall that Orlik--Strauch \cite{OS2}, \cite{OS3} define a functor $\mathcal{F}_{P}^{G}$ from $(\mathcal{O}^{P})^{\opp}\times \Rep^{\adm}(L)$ to locally analytic representations of $G$, where $\Rep^{\adm}(L)$ denotes the category of admissible smooth $L$-representations.
This functor satisfies the following properties (see \cite{OS3} and subsection~\ref{sec:funct-orlik-stra}):
\begin{itemize}
\item The functor $\mathcal{F}_{P}^{G}$ is exact in both arguments.
\item Let $P' = L'N'\supset P$ be another parabolic subgroup.
If $M\in \mathcal{O}^{P'}$ and $\pi\in\Rep^{\adm}(L)$, then $\mathcal{F}_{P}^{G}(M,\pi)\simeq \mathcal{F}_{P'}^{G}(M,(\Ind_{P \cap L'}^{L'}\pi)^{\sm})$.
\item Assume that $M\in \mathcal{O}^{P}$ is equimaximal with maximal parabolic $P$ and $\pi \in \Rep^{\adm}(L)$.
Assume that $M|_{\mathfrak{g}_{C}} \in \mathcal O^{\mathfrak p}$ is simple and $\pi$ is irreducible.
Then $\mathcal{F}_{P}^{G}(M,\pi)$ is irreducible.
\end{itemize}
For $W\in \mathcal{O}^{L}$ the generalized Verma module $\underline{M}(W)$ is defined to be $U(\mathfrak{g}_{C})\otimes_{U(\mathfrak{p}_{C})}W$, where $P$ acts by $p'(X\otimes w) = \Ad(p')X\otimes p'w$ for $p'\in P$, $w\in W$ and $X\in \mathfrak{g}_{C}$.
Then $\underline{M}(W)\in \mathcal{O}^{P}$, and if moreover $W\in \mathcal O^L$ is simple, then $\underline{M}(W)$ has $\underline{L}(W)$ as unique simple quotient. Then the following property of $\mathcal{F}_{P}^{G}$ holds by construction:
\begin{itemize}
\item Suppose that $W \in \mathcal O^L$ and $\pi \in \Rep^{\adm}(L)$.
  Then $\mathcal{F}_{P}^{G}(\underline M(W),\pi) \cong (\AnInd_{P}^{G}W'\otimes\pi)^{\an}$.
\end{itemize}

We now discuss the proof of Theorem~\ref{thm:criterion-intro}.
The implication (i)$\Rightarrow$(ii) is clear by exactness of parabolic induction, and (iii)$\Rightarrow$(iv) is obvious.

To explain why (iv)$\Rightarrow$(i), we assume for simplicity that $Q = G$ (the proof is a bit more involved in general).
By the density of locally analytic vectors in Banach representations \cite{MR1900706}, it suffices to show that any irreducible closed subrepresentation of $(\Ind_{P}^{G} \sigma)^{\an}$ is dense in $(\Ind_{P}^{G} \sigma)^{\cts}$.
Note that $(\Ind_{P}^{G} \sigma)^{\an} \cong \mathcal{F}_{P}^{G} (\underline{M}(\sigma_0'),\tau)$.
Let $V := \underline{L}(\sigma_{0}')'$, which is by assumption in $\mathcal O^G$ and in particular is a finite-dimensional locally analytic representation of $G$.
The canonical $P$-linear surjection $V \onto \sigma_0$ gives rise to a commutative diagram
\begin{equation*}
  \xymatrix{
    V\otimes (\Ind_{P}^{G}\tau)^{\sm}\ar@{^{(}->}[r]\ar@{^{(}->}[rd] & V\otimes (\Ind_{P}^{G}\tau)^{\an}\ar[r]^\sim\ar@{^{(}->}[d] & (\Ind_{P}^{G}V \otimes \tau)^{\an}\ar@{->>}[r]\ar@{^{(}->}[d] & (\Ind_{P}^{G} \sigma)^{\an}\ar@{^{(}->}[d] \\ 
    & V\otimes (\Ind_{P}^{G}\tau)^{\cts}\ar[r]^\sim & (\Ind_{P}^{G}V\otimes \tau)^{\cts}\ar@{->>}[r] & (\Ind_{P}^{G} \sigma)^{\cts}
  }
\end{equation*}
It is not difficult to see (cf.\ Lemma~\ref{lem:reduction to smooth}) that the composition of the top row is injective with image $\mathcal{F}_{P}^{G} (\underline{L}(\sigma_0'),\tau)$.
Crucially, our generalization of a result of Breuil \cite{socle1} (based on \cite{orlik-schraen}) on locally analytic socles allows us to deduce that any irreducible closed subrepresentation of $(\Ind_{P}^{G} \sigma)^{\an}$ is contained in $\mathcal{F}_{P}^{G} (\underline{L}(\sigma_0'),\tau) \cong V\otimes (\Ind_{P}^{G}\tau)^{\sm}$.
It is then of the form $V \otimes \rho$ for some irreducible subrepresentation $\rho \subset (\Ind_{P}^{G}\tau)^{\sm}$.
By (iv) we see that $V \otimes \rho$ is dense in $V\otimes (\Ind_{P}^{G}\tau)^{\cts}$ and hence by the diagram $V \otimes \rho$ is dense in $(\Ind_{P}^{G} \sigma)^{\cts}$, as desired.

It remains to explain (ii)$\Rightarrow$(iii).
We relabel $Q$ as $G$ and let again $V := \underline L(\sigma_0')' \in \mathcal O^G$.
Let $\pi$ be a nonzero closed subrepresentation of $(\Ind_P^G \tau)^\cts$.
We consider a natural sequence
\begin{equation*}
  V \otimes \pi^\an \into V \otimes (\Ind_P^G \tau)^\an \congto (\Ind_P^G V \otimes \tau)^\an \onto (\Ind_P^G \sigma)^\an
\end{equation*}
of locally analytic representations and first use (ii) (comparing with a corresponding sequence of Banach representations) and locally analytic socles to show that the composition is surjective.
Let $\chi \colon Z(\mathfrak g_C) \to C$ denote the infinitesimal character of $(\Ind_P^G \sigma)^\an$.
Projecting onto generalized $\chi$-eigenspaces, by the definition of the translation functors, we obtain a sequence
\begin{equation*}
  T(\pi^\an) \into T((\Ind_P^G \tau)^\an) \onto (\Ind_P^G \sigma)^\an,
\end{equation*}
whose composition is surjective. Here $T$ is a suitable translation functor in the sense of \cite{MR581584} (it is an equivalence by \cite{MR581584}, see subsection \ref{sec:translation-functors} for more details).
By showing that the second map is an isomorphism (Proposition~\ref{prop:translation functors(parabolic induction)}) and applying a quasi-inverse of $T$ we obtain that $\pi^\an = (\Ind_P^G \tau)^\an$, which implies (iii), by density of locally analytic vectors.

\subsection{Previous work}
Theorems~\ref{thm:unitary case-parab-intro} and \ref{thm:GLn-irred-intro} were known for $\GL_2(\Q_p)$ by Schneider (and Teitelbaum) \cite[Proposition 2.6(i)]{MR2275644}, which was also based on locally analytic techniques.
The infinitesimal irreducibility criterion Corollary \ref{cor:criterion-easy-cases-intro}(i)$\Rightarrow$(iii) was known for split groups over $\Q_p$ due to the work of Frommer \cite{frommer}, \cite[Proposition 2.6(ii)]{MR2275644} and in general by Orlik--Strauch \cite{OS} (when $\dim_C \tau = 1$).

In a different direction, Ban--Hundley \cite{MR3537231} argue on the dual side like \cite{MR1900706} to prove the irreducibility of $(\Ind_B^G \chi)^\cts$ for split $\alggrp G$ and $\lvert\chi\rvert_C$ lying in a certain open cone, namely the cone where $e(\chi)$ (defined in \S\ref{sec:unitary-case-intro}) is \emph{strictly} dominant. 
(This region excludes the unitary locus if $\alggrp G$ is not a torus.)
Ban--Strauch \cite{arxiv.1912.11125} characterize the irreducibility of principal series of $\SL_n(F)$ in terms of the irreducibility of principal series of $\GL_n(F)$.

Finally, in the unitary case weak results can be obtained from smooth mod $p$ representation theory, by using that an admissible unitary Banach space representation is irreducible provided its reduction modulo a $G$-stable open and bounded lattice is irreducible as smooth representation.
The main results of our earlier work \cite{AHHV} then show that if $\sigma$ is unitary and the reduction $\o\sigma|_{Z \cap L'_{\alpha}}$ of $\sigma|_{Z \cap L'_{\alpha}}$ is non-trivial for all simple roots $\alpha$, then $(\Ind_{B}^{G}\sigma)^{\cts}$ is irreducible.

\subsection{Notation}
\label{sec:notation}
Let $C$ be a finite extension of $\Q_{p}$ with uniformizer $\varpi_C$. In this paper, unless otherwise stated, the coefficient field of any representation is $C$.

Let $F$ be a finite extension of $\Q_{p}$ contained in $C$ and $\alggrp{G}$ a connected reductive group over $F$, $\alggrp{Z}_{\alggrp{G}}$ the center of $\alggrp{G}$, $G = \alggrp{G}(F)$ is the group of rational points, $\mathfrak{g}$ the Lie algebra of $\alggrp{G}$, $\mathfrak{g}_{C} := \mathfrak{g}\otimes_{F}C$ and $U(\mathfrak{g}_{C})$ the enveloping algebra of $\mathfrak{g}_{C}$.
We use the same notation for other groups.
For simplicity we assume that $\alggrp{G}$ splits over $C$.
Moreover, starting in subsection~\ref{sec:funct-orlik-stra} we will assume that $C$ be sufficiently large, depending only on $\alggrp G$.
Let $\alggrp{S}\subset \alggrp{G}$ be a maximal split torus and $\alggrp{Z}$ the centralizer of $\alggrp{S}$ in $\alggrp{G}$.
This is a Levi subgroup of a minimal parabolic subgroup $\alggrp{B}$, which we fix from now on, and we let $\alggrp{U}$ denote its unipotent radical.

We fix a special point in the apartment of $\alggrp{S}$ and let $K$ be the corresponding special parahoric subgroup.

Suppose $\alggrp{P} = \alggrp{L}\alggrp{N}$ is a parabolic subgroup of $\alggrp{G}$ with Levi part $\alggrp{L}$ and unipotent radical $\alggrp{N}$.
For a smooth representation $\sigma$ of $L = \alggrp{L}(F)$, let $(\Ind_{P}^{G}\sigma)^{\sm}$ be the smooth parabolically induced representation, namely the space of locally constant functions $f\colon G\to \sigma$ such that $f(\ell ng) = \sigma(\ell )f(g)$ for $\ell \in L$, $n\in N$ and $g\in G$.
We use this notation even when $\sigma$ is defined over a field different from $C$ (for example, over the residue field of $C$ or the field of complex numbers).
Let $\delta_{P}$ be the modulus function of $P$ defined by $\int_{P}f(xg)dx = \delta_{P}(g)\int_{P}f(x)dx$ for any integrable function $f$ on $P$, where we use a left-invariant Haar measure.
Let $(\nInd_{P}^{G}\sigma)^{\sm} := (\Ind_{P}^{G}\sigma\delta_{P}^{1/2})^{\sm}$ be the normalized induced representation.

We say that a continuous representation of a topological group is \emph{irreducible} if it is topologically irreducible.

\subsection{Acknowledgements}
\label{acknowledgements}

We are very grateful to Jean-Loup Waldspurger for sending us a proof of Proposition~\ref{prop:Waldspurger} and for allowing us to reproduce it here,
and to Alberto M\'inguez for his help with Proposition~\ref{prop:qsplit-irred-gen-socle} as well as for useful discussions.
We thank Tasho Kaletha for help with the proof of Proposition~\ref{prop:special-point-base-change}.
We are thankful to Sascha Orlik and Matthias Strauch for helpful conversations about their work and for allowing us to include our appendix.
The debt this paper owes to the work of Orlik--Strauch should be evident to the reader.
We thank Hiraku Atobe, Laurent Clozel, Wee Teck Gan, Sug Woo Shin, and the referee for helpful comments, and Dubravka Ban for noting that we need to work over an algebraic closure in section~\ref{sec:generic}.
Part of this work was done during a pleasant stay of the first-named author at University of Toronto.

\section{Irreducibility criterion}
\label{sec:irred-crit}
\subsection{Banach representations}
\label{sec:banach-repr}
A representation $\pi$ of $G$ is called a \emph{Banach representation} if $\pi$ is a Banach space and the action map $G\times \pi\to \pi$ is continuous.
The notion of \emph{admissible} Banach representations was introduced by Schneider--Teitelbaum \cite{MR1900706}.
Any morphism between admissible Banach representations is strict and, in particular, has a closed image.
By \cite[Theorem 3.5]{MR1900706} admissible Banach representations of $G$ form an artinian abelian category, anti-equivalent (for any fixed compact open subgroup $H$ of $G$) to the category of $\mathcal O_C[[H]][\frac 1p]$-modules that have a compatible action of $G$.
Since it is an abelian category, we have notions such as finite length objects (traditionally called admissible Banach representations that are topologically of finite length).
The subobjects (resp.\ quotient objects) are precisely the closed $G$-subrepresentations (resp.\ Hausdorff quotient representations), using for example \cite[Proposition 1.3]{MR1900706}.
Admissible representations satisfy ``Schur's lemma'', namely if $\pi$ is an irreducible admissible Banach representation, then the $C$-algebra of continuous $G$-endomorphisms $\End_{G}^{\cts}(\pi)$ is finite-dimensional~\cite[Theorem~1.1]{MR3053464}.
In particular, if $\pi$ is absolutely irreducible (i.e.\ irreducible after any finite extension of scalars), then $\End_{G}^{\cts}(\pi)= C$ and $\pi$ therefore has a central character.

We say that a parabolic subgroup is \emph{semistandard} if it contains out fixed maximal split torus $\alggrp{S}$.
Let $\alggrp{P}$ be a semistandard parabolic subgroup, $\alggrp{N}$ the unipotent radical of $\alggrp{P}$, $\alggrp{L}$ a Levi part containing $\alggrp{S}$ (hence $P = LN$ and $P\cap K = (L\cap K)(N\cap K)$) and $\sigma$ a Banach representation of $L$. Then the representation $(\Ind_{P}^{G}\sigma)^{\cts}$ is defined as the space of continuous functions $f\colon G\to \sigma$ such that $f(\ell ng) = \sigma(\ell)f(g)$ for any $g\in G$, $\ell\in L$ and $n\in N$.
It is naturally a closed subspace of $\mathcal C^0(G,\sigma)$ equipped with the compact open topology, on which $G$ acts continuously \cite[Proposition 3.1.5]{locallyanalytic-memoir}.
The isomorphism $(\Ind_{P}^{G}\sigma)^{\cts} \cong (\Ind_{P \cap K}^{K}\sigma)^{\cts}$ shows that it is a Banach space.
It is admissible if $\sigma$ is admissible \cite[Lemma 3.3]{fu}.
By \cite[Lemma 6.5.5]{locallyanalytic-memoir} the topology of $\sigma$ can be defined by an $L\cap K$-invariant norm $\lvert\cdot\rvert$.
On the right-hand side of the isomorphism $(\Ind_{P}^{G}\sigma)^{\cts}\simeq(\Ind_{P\cap K}^{K}\sigma)^{\cts}$ we have a $K$-invariant norm $\lVert\cdot\rVert$ defined by $\lVert f\rVert = \sup_{x\in K/(P\cap K)}\lvert f(x)\rvert$, which defines the topology. Its unit ball is $((\Ind_{P}^{G}\sigma)^{\cts})^{0} = \{f\in (\Ind_{P\cap K}^{K}\sigma)^{\cts}\mid f(K)\subset \sigma^{0}\} = (\Ind_{P\cap K}^{K}\sigma^{0})^{\cts}$ and $((\Ind_{P}^{G}\sigma)^{\cts})^{0}/\varpi_{C}((\Ind_{P}^{G}\sigma)^{\cts})^{0}\simeq (\Ind_{P\cap K}^{K}(\sigma^{0}/\varpi_{C}\sigma^{0}))^{\sm}$ and this is a smooth representation of $K$.
If $\pi\subset(\Ind_{P}^{G}\sigma)^{\cts}$ is a closed subspace, then it is a Banach space with the induced norm, and
the natural map $\pi^0/\varpi_{C}\pi^0\to ((\Ind_{P}^{G}\sigma)^{\cts})^{0}/\varpi_{C}((\Ind_{P}^{G}\sigma)^{\cts})^{0}$ is injective.

\subsection{Locally analytic representations}
\label{sec:locally-analyt-repr}
Let $\pi$ be a Banach representation of $G$.
We say that $v\in \pi$ is an ($F$-)\emph{locally analytic vector} if the map $G\ni g\mapsto gv\in \pi$ is ($F$-)locally analytic.
The space of locally analytic vectors in $\pi$ is denoted by $\pi^{\an}$.
We impose a topology on $\pi^{\an}$ as in \cite[Definition~3.5.3]{locallyanalytic-memoir}.
Then $\pi^{\an}$ is a locally analytic representation in the sense that $\pi^{\an}$ is a barreled locally convex Hausdorff vector space and the map $G\ni g\mapsto gv\in \pi^{\an}$ is locally analytic for any $v\in \pi^{\an}$.
We have a notion of \emph{admissible} locally analytic representations introduced by Schneider--Teitelbaum \cite[page 176]{ST}, and the category of such representations is abelian~\cite[Proposition~6.4]{ST}.
Again, any morphism in this category is strict and, in particular, has closed image.
Likewise, the subobjects (resp.\ quotient objects) are precisely the closed $G$-subrepresentations (resp.\ Hausdorff quotient representations).
We will again talk about finite length objects in this abelian category.

\begin{thm}[{\cite[Theorem~7.1]{ST}, \cite[Proposition 6.2.4]{locallyanalytic-memoir}}]\label{thm:property of an}
Suppose that $F = \Q_{p}$.
Then the functor $(\cdot)^\an$ sends admissible Banach representations of $G$ to
admissible locally analytic representations of $G$ on compact type spaces.
It is moreover exact.
If $\pi$ is an admissible Banach representation, then $\pi^{\an}$ is dense in $\pi$.
\end{thm}

For proving density, the following criterion is sometimes useful.
A more general version can be found in \cite[\S\ref{gl3:sec:density}]{gl3}.
Let $\overline{\alggrp{N}}$ denote the unipotent radical of the opposite parabolic subgroup $\overline{\alggrp{P}} = \alggrp{L} \overline{\alggrp{N}}$.
\begin{lem}\label{lem:supported on big cell}
Let $\sigma$ be an irreducible Banach representation of $L$ with a central character and $0\ne f\in (\Ind_{P}^{G}\sigma)^{\cts}$.
Assume that $\supp(f)\subset P\backslash P\overline{N}$ and there exists $v\in \sigma$ such that $f(\overline{n})\in Cv$ for all $\overline{n}\in \overline{N}$.
Then $f$ is a topological generator of $(\Ind_{P}^{G}\sigma)^{\cts}$, i.e.\ the subspace spanned by $\{gf\mid g\in G\}$ is dense in $(\Ind_{P}^{G}\sigma)^{\cts}$.
\end{lem}
Here, $\supp(f)$ denotes the \emph{closure} of the set $P\backslash \{x \in G : f(x) \ne 0 \}$ in $P\backslash G$.
In particular, $\supp(f)$ is compact.
\begin{proof}
We choose an $L\cap K$-invariant norm $\lvert \cdot \rvert$ on $\sigma$ that defines its topology.
We may assume $\{\lvert v'\rvert\mid v'\in \sigma\} = \lvert C^{\times}\rvert\cup \{0\}$ and $\lvert v\rvert = 1$.

Let $\overline{N}_{0}$ be a compact open subgroup of $\overline{N}\cap K$.
Let $\pi$ be the closure of the subspace spanned by $\{gf\mid g\in G\}$ in $(\Ind_{P}^{G}\sigma)^{\cts}$.
By the action of the center of $L$ there exists $f'\in \pi$ such that $\supp(f')\subset P\overline{N}_{0}$ and $f'(\overline{n})\in Cv$ for any $\overline{n}\in \overline{N}$.
We may scale $f'$ such that $\lVert f'\rVert = 1$.
Hence there exists $f'\in \pi$ such that $\lVert f'\rVert = 1$, $\supp(f')\subset P\overline{N}_{0}$ and $f'(\overline{n})\in Cv$ for any $\overline{n}\in \overline{N}$.

We will now prove that $\pi = (\Ind_{P}^{G}\sigma)^{\cts}$.
Let $X_{0}$ be the closed subspace of $h\in (\Ind_{P}^{G}\sigma)^{\cts}$ such that $\supp(h)\subset P\overline{N}_{0}$ and $h(\overline{n})\in Cv$ for any $\overline{n}\in \overline{N}$.
Let $C^{0}(\overline{N}_{0},Cv)$ be the Banach space of continuous functions $\overline{N}_{0}\to Cv$ equipped with the supremum norm.
The restriction to $\overline{N}_{0}$ gives an $\overline{N}_{0}$-equivariant norm-preserving isomorphism $X_{0}\simeq C^{0}(\overline{N}_{0},Cv)$ of Banach spaces.

We have an embedding $(X_{0}\cap \pi)^{0}/\varpi_{C}(X_{0}\cap \pi)^{0}\hookrightarrow X_{0}^{0}/\varpi_{C}X_{0}^{0}$ and from the second paragraph the image is nonzero.
Since the image is a smooth $\overline{N}_{0}$-representation, it contains a nonzero $\overline{N}_{0}$-fixed vector.
Let $\overline{v}\in \sigma^{0}/\varpi_{C}\sigma^{0}$ be the image of $v$.
The space of $\overline{N}_{0}$-fixed vectors in $X_{0}^{0}/\varpi_{C}X_{0}^{0}\simeq C^{\infty}(\overline{N}_{0},(\mathcal{O}_{C}/\varpi_{C}\mathcal{O}_{C})\overline{v})$ is one-dimensional and spanned by the constant function $1_{\overline{N}_{0}}$.

Let $X$ be the space of $h\in (\Ind_{P}^{G}\sigma)^{\cts}$ such that $\supp(h)\subset P(\overline{N}\cap K)$ and $h(\overline{n})\in Cv$ for any $\overline{n}\in \overline{N}\cap K$.
Then $X^{0}/\varpi_{C}X^{0}\simeq C^{\infty}(\overline{N}\cap K,(\mathcal{O}_{C}/\varpi_{C}\mathcal{O}_{C})\overline{v})$ is spanned by $\{1_{\overline{N}_{0}}\mid \text{$\overline{N}_{0}\subset \overline{N}\cap K$ is an open subgroup}\}$ as $\overline{N}\cap K$-representation.
Since each $1_{\overline{N}_{0}}$ is in $(\pi \cap X)^{0}/\varpi_{C}(\pi \cap X)^{0}$ as we have proved, we have $(\pi \cap X)^{0}/\varpi_{C}(\pi \cap X)^{0} = X^{0}/\varpi_{C}X^{0}$.
In other words, for any $h\in X^{0}$, there exists $f_{0}\in (\pi\cap X)^{0}$ such that $h - f_{0}\in \varpi_{C}X^{0}$.
By iterating this argument, for each $k = 0,1,\ldots$ there exists $f_{k}\in \pi$ such that $h - f_{k}\in \varpi_{C}^{k}X^{0}$.
Then $h = \lim_{k\to\infty}f_{k}\in \pi$.
This shows that $X^{0}\subset \pi$, hence $X\subset \pi$.
By the action of the center of $L$, for any continuous function $\varphi\colon \overline{N}\to C$ with compact support, the function $\varphi\otimes v\colon \overline{N}\to Cv$ defined by $(\varphi\otimes v)(\overline{n}) = \varphi(\overline{n})v$ is in $\pi$.
Here we regard $\varphi\otimes v$ as element of $(\Ind_{P}^{G}\sigma)^{\cts}$ (supported on $P\overline{N}$).

Let $\ell\in L$ and define $\ell\varphi\colon \overline{N}\to C$ by $(\ell\varphi)(\overline{n}) = \varphi(\ell^{-1}\overline{n}\ell)$.
For each $\varphi\colon \overline{N}\cap K\to C$ and $\ell\in L$ there exists $z$ in the center of $L$ such that $z\ell^{-1}\varphi$ is supported on $\overline{N}\cap K$.
We have $\varphi\otimes \ell v = \omega_{\sigma}(z)\ell z^{-1}(z\ell^{-1}\varphi\otimes v)$, where $\omega_{\sigma}$ is the central character of $\sigma$.
Since $z\ell^{-1}\varphi\otimes v\in \pi$ by the previous paragraph, we have $\varphi\otimes \ell v\in \pi$.
Since $\sigma$ is irreducible, we have $\varphi\otimes v'\in \pi$ for any $\varphi \in C^0(\overline{N} \cap K,C)$ and $v'\in \sigma$.
In other words, $C^0(\overline{N} \cap K,C) \otimes \sigma \subset \pi$.

Let $h\colon \overline{N}\cap K\to \sigma^{0}$ be any continuous function and for each $k = 0,1,\ldots$ define $h_{k}$ by $\overline{N}\cap K\xrightarrow{h}\sigma^{0}\to \sigma^{0}/\varpi_{C}^{k}\sigma^{0}$.
Then $h_{k}$ is a locally constant function and therefore by compactness has finite image in $\sigma^{0}/\varpi_{C}^{k}\sigma^{0}$.
Hence there exists $h'_{k}\in C^{0}(\overline{N}\cap K,C)\otimes \sigma$ (which we can even take to be locally constant) such that $h - h'_{k}\in \varpi_{C}^{k}C^{0}(\overline{N}\cap K,\sigma)^{0}$.
Therefore $h = \lim_{k\to\infty} h'_{k}$.
Regarding $h$ as element of $(\Ind_{P}^{G}\sigma)^{\cts}$, the previous paragraph shows that $h\in \pi$.
Since such $h$ generate $(\Ind_{P}^{G}\sigma)^{\cts}$ as a $G$-representation, we get $\pi = (\Ind_{P}^{G}\sigma)^{\cts}$.
\end{proof}

If $\sigma$ is a locally analytic representation of $P$ on a locally convex topological vector space of compact type, we denote by $(\Ind_{P}^{G}\sigma)^{\an}$
the space of locally analytic functions $G\to \sigma$ such that $f(p'g) = \sigma(p')f(g)$ for all $g\in G$, $p'\in P$.
This is a locally analytic representation of $G$ of compact type, see e.g.\ \cite[\S2]{MR4190408}.
(Usually, $\sigma$ will arise by inflation from a locally analytic representation of $L$.)

\begin{lem}\label{lem:induction-tensor}\ 

  \begin{enumerate}
  \item If $\tau$ is a locally analytic representation of $P$ on a compact type space
    and $V$ a finite-dimensional locally analytic representation of $G$,
    then $V \otimes (\Ind_{P}^{G}\tau)^{\an} \cong (\Ind_{P}^{G}V \otimes \tau)^{\an}$ as locally analytic representations of $G$.
  \item If $\tau$ is a Banach representation of $P$
    and $V$ a finite-dimensional continuous representation of $G$,
    then $V \otimes (\Ind_{P}^{G}\tau)^{\cts} \cong (\Ind_{P}^{G}V \otimes \tau)^{\cts}$ as 
    Banach representations of $G$.
  \item If $\tau$ is a Banach representation of $P$,
    then $((\Ind_{P}^{G}\tau)^{\cts})^\an \cong (\Ind_{P}^{G}\tau^\an)^\an$
    as locally analytic representations of $G$.
  \item If $\pi$ is an Banach representation of $G$
    and $V$ a finite-dimensional locally analytic representation of $G$,
    then $(V \otimes \pi)^\an \cong V \otimes \pi^\an$ as locally analytic representations of $G$.
  \item If $\pi$ is an admissible Banach representation of $G$
    and $V$ a finite-dimensional continuous representation of $G$,
    then $V \otimes \pi$ is an admissible Banach representation of $G$.
  \end{enumerate}
\end{lem}

Note that the tensor products carry the projective, or equivalently inductive, topology.

\begin{proof}
  (i) We have a natural map of locally analytic representations
  \begin{align*}
    \Theta : V \otimes (\Ind_{P}^{G}\tau)^{\an} &\to (\Ind_{P}^{G}V \otimes \tau)^{\an}\\
    v \otimes f &\mapsto (g \mapsto gv \otimes f(g)).
  \end{align*}
  Let $(v_i)_{i=1}^n$ be a basis of $V$.
  Suppose that $h \in (\Ind_{P}^{G}V \otimes \tau)^{\an}$ and write $h(g) = \sum_i gv_i \otimes f_i(g)$
  for unique $P$-equivariant functions $f_i : G \to \tau$. 
  It suffices to show that the functions $f_i$ are locally analytic.
  Let $(v_i^*)_{i=1}^n$ denote the dual basis of $V$ and let $\lambda_j : G \to V^\vee$ denote the locally analytic function $\lambda_j(g) := gv_j^*$.
  Then $f_j$ is obtained from $h$ and $\lambda_j$ via the natural bilinear pairing
  $(V \otimes \tau) \times V^\vee \to \tau$, so it is locally analytic by \cite[Satz 2.4.3]{feaux-diss}.
  (The assumption BIL there is satisfied, as $V$ is finite-dimensional.)

  (ii) The proof is analogous but easier, using that the bilinear pairing above is continuous.

  (iii) is proved in \cite[Theorem 3.6]{fu}, (iv) is proved in \cite[Proposition 3.6.15]{locallyanalytic-memoir}
  (more generally), and (v) is proved in \cite[Proposition 6.2.6]{locallyanalytic-memoir}.
  \end{proof}

The following lemmas about extension of scalars and semisimplicity will be useful later.

\begin{lem}\label{lm:density-and-ext-of-scalars}
Let $V$ be a Banach space over $C$, $V_0$ a subspace of $V$ and $C'$ a finite extension of $C$.
Then $\overline{V_0\otimes_{C}C'} = \overline{V_0}\otimes_{C}C'$ inside $V \otimes_C C'$, where $\overline{V_0}$ is the closure of $V_0$ in $V$.
In particular, $V_0\otimes_{C}C'$ is dense in $V\otimes_{C}C'$ if and only if $V_0$ is dense in $V$.
\end{lem}
Here, $V \otimes_C C'$ carries the projective, or equivalently inductive, tensor product topology.
\begin{proof}
  Set $n = [C':C]$ and fix a basis of $\mathcal{O}_{C'}$ as a $\mathcal{O}_{C}$-module.
  Then $V\otimes_{C}C'\simeq V^{\oplus n}$ and $V_0\otimes_{C}C'\simeq (V_0)^{\oplus n}$ and $\o{V_0}\otimes_{C}C'\simeq (\o{V_0})^{\oplus n}$ as $C$-vector spaces.
  Once we show that the first isomorphism is a homeomorphism, the lemma is obvious.
  To see this, fix a bounded open $\mathcal{O}_C$-lattice $L$ of $V$.
  By the definition of the projective tensor product topology, the topology of $V \otimes_C C'$ is generated by the scalar multiples of the lattice $L \otimes_{\mathcal{O}_{C}} \mathcal{O}_{C'} \cong L^{\oplus n}$.
  Hence the first isomorphism is indeed a homeomorphism.
\end{proof}

\begin{lem}\label{lem:adm-ext-of-scalars}
  Suppose that $C'/C$ is a finite extension.
  \begin{enumerate}
  \item If $\pi$ is an admissible Banach (resp.\ locally analytic) representation of $G$ over $C$, then $\pi \otimes_C C'$ is admissible (over $C'$).
  \item If $\pi'$ is an admissible Banach (resp.\ locally analytic) representation of $G$ over $C'$, then $\pi'|_C$ is admissible (over $C$).
  \end{enumerate}
\end{lem}

\begin{proof}
First we prove the lemma when $C'/C$ is Galois.
  In (ii), note that $\pi'|_{C}\otimes_{C}C' = \bigoplus_{\gamma\in\Gal(C'/C)}\pi'\otimes_{C',\gamma}C'$ is admissible.
  Hence we are reduced to showing that if $\pi$ is a Banach (resp.\ locally analytic) representation over $C$, then $\pi$ is admissible (over $C$) if and only if $\pi \otimes_C C'$ is admissible (over $C'$).

  Suppose that $\pi$ is Banach.
  Fix a compact open subgroup $H \le G$ and let $\pi^0$ be an $H$-stable unit ball in $\pi$.
  Then $\pi^0 \otimes_{\mathcal O_C} \mathcal O_{C'}$ is an $H$-stable unit ball in $\pi \otimes_C C'$.
  Let $\o\pi := \pi^0 \otimes_{\mathcal O_C} k_C$, where $k_C$ (resp.\ $k_C'$) denotes the residue field of $C$ (resp.\ $C'$).
  The claim is then equivalent to showing that the smooth representation $\o\pi$ is admissible if and only if $\o\pi \otimes_{k_C} k_C'$ is admissible, which is clear.

  Suppose that $\pi$ is locally analytic.
  If $V$ is a locally convex space over $C$, we first show that $V$ is of compact type (over $C$) if and only if $V \otimes_C C'$ is of compact type (over $C'$).
  We observe that if $f : V' \to W'$ is a continuous map of Banach spaces over $C'$, then $f$ is compact if and only if $f|_C$ is compact, as follows from the definitions in \cite[\S1]{MR1887640}.
  If $V = \varinjlim_{n\ge 1} V_n$ is of compact type with $V_n$ Banach and compact, injective transition maps, then $V \otimes_C C' = \varinjlim_{n\ge 1} (V_n \otimes_C C')$ is of compact type over $C$ with compact, injective transition maps (over $C$) by \cite[Lemma 1.1.32]{locallyanalytic-memoir}, hence of compact type over $C'$ (by our observation).
  Conversely, if $V \otimes_C C'$ is of compact type, then $V \otimes_C C'$ is of compact type over $C$ (by our observation), hence so is its closed subspace $V$ (cf.\ \cite[Proposition 1.2]{MR1887640}).

  \renewcommand{\H}{\mathbb H}

  We now use Emerton's characterization of admissible locally analytic representations in \cite[Definition 6.1.1]{locallyanalytic-memoir} (and the lines that follow):
  $\pi$ is admissible if and only if it is of compact type and for a cofinal sequence of analytic (compact) open subgroups $H = \H(F)$ of $G$ the space $\pi_{\H\text{-}\an}$ admits an $H$-equivariant closed embedding in $\mathcal C^\an(\H,C)^{\oplus n}$ for some $n = n(\H) \ge 0$.
  (See \cite[Definition 2.1.9, Definition 3.3.13, \S3.5]{locallyanalytic-memoir} for the rigid analytic functions $\mathcal C^\an(\H,C)$, $\H$-analytic vectors $(-)_{\H\text{-}\an}$, and analytic open subgroups, respectively.)

  We claim that the functor $(-)_{\H\text{-}\an}$ commutes with scalar extension from $C$ to $C'$ (on locally analytic representations of compact type over $C$).
  Assuming this, if $\pi$ is admissible then for a cofinal sequence of analytic open subgroups $H$ of $G$ there exists a closed embedding $\pi_{\H\text{-}\an} \into \mathcal C^\an(\H,C)^{\oplus n}$ for some $n \ge 0$.
  Then we obtain a closed embedding $\pi_{\H\text{-}\an} \otimes_C C' \into \mathcal C^\an(\H,C)^{\oplus n} \otimes_C C'$ (\cite[Corollary 17.5]{MR1869547}), i.e.\ a closed embedding
  \begin{equation}
    \label{eq:pi-C'-emb}
    (\pi \otimes_C C')_{\H\text{-}\an} \into \mathcal C^\an(\H,C')^{\oplus n}.
  \end{equation}
  Conversely, if we have a closed embedding~\eqref{eq:pi-C'-emb} over $C'$, then by the claim we have closed embeddings over $C$:
  \begin{equation*}
    \pi_{\H\text{-}\an} \into (\pi \otimes_C C')_{\H\text{-}\an} \into \mathcal C^\an(\H,C')^{\oplus n} = \mathcal C^\an(\H,C)^{\oplus n[C':C]}.
  \end{equation*}
  We deduce that $\pi$ is admissible if and only if $\pi \otimes_C C'$ is admissible.

  To prove the claim we first note that the functor $(-)_{\H\text{-}\an}$ commutes with scalar extension from $C$ to $C'$ on Banach spaces over $C$ equipped with a topological action of $H$ (recall that
  $V_{\H\text{-}\an} = \mathcal C^\an(\H,V)^H$, and both functors $\mathcal C^\an(\H,-)$ and $(-)^H$ commute with scalar extensions by the definitions).
  Now if $V$ is of compact type over $C$, we present it as $V = \varinjlim_{n \ge 1} V_n$ as earlier in this proof, so $V \otimes_C C' = \varinjlim_{n \ge 1} (V_n \otimes_C C')$ presents $V \otimes_C C'$ as compact type space over $C'$.
  If $V$ is moreover equipped with a locally analytic action of $H$, then by compactness of $H$ we may assume that each $V_n$ is $H$-stable by \cite[Proposition 3.2.15]{locallyanalytic-memoir} and we obtain $V_{\H\text{-}\an} = \varinjlim_{n \ge 1} (V_n)_{\H\text{-}\an}$ by \cite[Proposition 3.3.27]{locallyanalytic-memoir}.
  Then
  \begin{equation*}
    (V \otimes_C C')_{\H\text{-}\an} = \varinjlim_{n \ge 1} (V_n \otimes_C C')_{\H\text{-}\an} = \varinjlim_{n \ge 1} (V_n)_{\H\text{-}\an} \otimes_C C' = V_{\H\text{-}\an} \otimes_C C'.
  \end{equation*}

In general, let $C''$ be the Galois closure of $C'/C$.
If $\pi$ is admissible, then, since $\pi\otimes_{C}C'\subset (\pi\otimes_{C}C'')|_{C'}$ is a closed subrepresentation and $(\pi\otimes_{C}C'')|_{C'}$ is admissible, $\pi\otimes_{C}C'$ is also admissible.
If $\pi'$ is admissible, then $\pi'|_{C}$ is admissible as $\pi'|_{C}\subset (\pi'\otimes_{C'}C'')|_{C}$ is a closed subrepresentation.
\end{proof}

We say that an admissible Banach representation (resp.\ locally analytic) is \emph{semisimple} if it is a direct sum of finitely many simple subobjects (i.e.\ irreducible closed subrepresentations).

\begin{lem}\label{lem:semisimple-banach}
Let $\pi$ be a finite-length admissible Banach (resp.\ locally analytic) representation.
Then the following are equivalent.
\begin{enumerate}
\item $\pi$ is semisimple.
\item $\pi$ is a finite sum of irreducible subrepresentations.
\item For any closed subrepresentation $\sigma\subset\pi$ there exists a closed subrepresentation $\tau$ such that $\pi = \sigma\oplus\tau$.
\end{enumerate}
In particular, any closed subrepresentation of a semisimple representation is semisimple.
\end{lem}

\begin{proof}
The proof works in any abelian category.
Assume that $\pi = \sum_{i = 1}^{n}\pi_{i}$ such that $\pi_{i}$ are simple subobjects and $n$ is minimal.
If the sum is not direct there exists $1 \le j \le n$ such that $\pi_j \cap \sum_{i \ne j} \pi_i \ne 0$, i.e.\ $\pi_j \subset \sum_{i \ne j} \pi_i$, contradicting the minimality of $n$.
Therefore (ii) implies (i), and the converse is obvious.

Assume $\pi$ is semisimple and write $\pi = \bigoplus_{i=1}^n \pi_i$ with $\pi_i$ simple.
Let $\sigma\subset\pi$ be a subobject.
Pick a sequence $1 \le k_1 < k_2 < \cdots < k_s \le n$ of minimal length such that $\sigma + \sum_{j=1}^s \pi_{k_j} = \pi$.
If the sum is not direct, then $\pi_{k_{j'}} \cap (\sigma + \sum_{j \ne j'} \pi_{k_j}) \ne 0$ for some $1 \le j' \le s$, i.e.\ $\pi_{k_{j'}} \subset \sigma + \sum_{j \ne j'} \pi_{k_j}$, which contradicts the minimality of $s$.
Hence we can take $\tau = \sum_{j=1}^s \pi_{k_j}$.
Therefore (i) implies (iii), and the converse follows by induction on the length of $\pi$.

Any subobject of a semisimple representation $\pi$ is also a quotient of $\pi$ by (iii) and hence satisfies (ii).
\end{proof}

If $\pi$ is an admissible Banach (resp.\ locally analytic) representation of finite length, we let its \emph{$G$-socle} $\soc_G(\pi)$ denote the largest semisimple subobject.

\begin{lem}\label{lem:soc,field extension}
Suppose that $C'/C$ is a finite extension.
\begin{enumerate}
\item If $\pi$ is a semisimple admissible Banach (resp.\ locally analytic) representation over $C$, then $\pi\otimes_{C}C'$ is admissible and semisimple.
\item If $\pi'$ is a semisimple admissible Banach (resp.\ locally analytic) representation over $C$, then $\pi'|_{C}$ is admissible and semisimple.
\item If $\pi$ is an admissible Banach (resp.\ locally analytic) representation of finite length over $C$, then $\soc_{G}(\pi)\otimes_{C}C' = \soc_{G}(\pi\otimes_{C}C')$.
\end{enumerate}
\end{lem}
For (i), see also \cite[Lemma 3.7]{MR3053464} in the case $C'/C$ Galois and $F = \Q_p$.
\begin{proof}
  As at the end of proof of Lemma~\ref{lem:adm-ext-of-scalars} we may assume that $C'/C$ is Galois.  
  The admissibility assertions follow from Lemma~\ref{lem:adm-ext-of-scalars}.

(i)
We may assume $\pi$ is irreducible.
The representation $\pi\otimes_{C}C'$ is of finite length, as its restriction to $C$ is ($C'/C$ being finite), and if $\pi'$ is any irreducible closed subrepresentation, then $\pi \otimes_C C' = \sum_{\gamma\in\Gal(C'/C)} \gamma(\pi')$ is semisimple.

(ii)
We may assume $\pi'$ is irreducible.
Note that $\pi'|_{C}\otimes_{C}C' = \bigoplus_{\gamma\in\Gal(C'/C)}\pi'\otimes_{C',\gamma}C'$ is of finite length.
Take $0 \ne \pi\subset\pi'|_{C}$ such that $\pi\otimes_{C}C'$ has minimal possible length.
Then $\pi$ is irreducible.
Take a $C$-basis $c_{1},\ldots,c_{d}$ of $C'$.
Then $\pi'|_{C} = \sum_{i}c_{i}\pi$ is semisimple.

(iii)
The inclusion $\soc_{G}(\pi)\otimes_{C}C' \subset \soc_{G}(\pi\otimes_{C}C')$ follows from (i).
By Galois descent, $\soc_{G}(\pi\otimes_{C}C')$ is defined over $C$.
Take a $C$-representation $\tau$ such that $\soc_{G}(\pi\otimes_{C}C') = \tau\otimes_{C}C'$.
Then $\tau\subset \soc_{G}(\pi\otimes_{C}C')|_{C}$ is semisimple by (ii).
We get $\soc_{G}(\pi) \supset \tau$.
\end{proof}

Let $G'$ denote the smallest normal subgroup of $G$ generated by $U$.
It is a closed locally analytic subgroup of $G$ \cite[\S6]{MR316587}.

\begin{lem}\label{lm:loc-an-char}\ 
  \begin{enumerate}
  \item 
    Any locally analytic character $\psi \colon G \to C^{\times}$ and any finite-dimensional smooth representation $\tau$ of $G$ is trivial on $G'$.
  \item 
    Suppose that $\alggrp{P} = \alggrp{L}\alggrp{N}$ is a parabolic subgroup and that $\tau$ is a finite-dimensional continuous representation of $P$.
    Then $\tau^N \ne 0$.
    In particular, if $\tau$ is irreducible, then $\tau$ is trivial on $N$.
  \end{enumerate}
\end{lem}

\begin{proof}
  (ii) We may assume that $F = \Q_p$, so $\tau$ is a locally analytic representation of $P$.
  By finite-dimensional representation theory of (a Borel subalgebra of) $\mathfrak{p}_C$, we have $\tau^{\mathfrak{n}_{C}}\ne 0$.
  This is a $P$-subrepresentation of $\tau$ on which $N$ acts smoothly.
  By using the $S$-action, we see that $\tau^{\mathfrak{n}_{C}} = \tau^N$.

  (i) The claim for $\psi$ follows from (ii), as $G'$ is the normal subgroup generated by $U$, and the argument for smooth $\tau$ is similar to (ii) (but easier and well-known).
\end{proof}

\subsection{The functor of Orlik--Strauch}
\label{sec:funct-orlik-stra}
In this section we discuss the categories $\mathcal{O}^{\mathfrak{p}}$, $\mathcal{O}^{P}$ and functors $\mathcal F_P^G$ of Orlik--Strauch for general $\alggrp G$.
We also take up and extend some general notions from \cite{orlik-functorial}.

Let $\alggrp{P} = \alggrp{L}\alggrp{N}$ be a standard parabolic subgroup.
Recall the abelian categories $\mathcal{O}^{\mathfrak{p}}$ and $\mathcal{O}^{P}$.
First, $\mathcal{O}^{\mathfrak{p}}$ is the full subcategory of $U(\mathfrak{g}_{C})$-modules $M$ such that
\begin{enumerate}
\item $M$ is a finitely generated $U(\mathfrak{g}_{C})$-module,
\item $M$ is a direct sum of absolutely simple finite-dimensional $\mathfrak{l}_{C}$-modules,
\item the action of $\mathfrak{p}_{C}$ on $M$ is locally finite.
\end{enumerate}

The objects of $\mathcal{O}^{P}$ consist of pairs $(M,\tau)$, where $M \in \mathcal{O}^{\mathfrak{p}}$ and $\tau$ is a locally finite-dimensional locally analytic action of $P$ on $M$
whose derivative equals the given action of $\mathfrak{p}_{C} \subset \mathfrak{g}_{C}$ and such that $g \circ X \circ g^{-1} = \Ad(g)(X)$ on $M$ for all $g \in P$, $X \in U(\mathfrak{g}_{C})$.
The morphisms consist of maps in $\mathcal O^{\mathfrak p}$ that are moreover $P$-linear.

In particular, $\mathcal{O}^{G}$ consists of all finite-dimensional locally analytic representations of $G$ on which the derived action of $\mathfrak{g}_{C}$ is a direct sum of absolutely simple $\mathfrak{g}_{C}$-modules.
This category is stable under duality.
Note that if $\tau$ is a finite-dimensional smooth representation of $P$, then $\tau \in \mathcal{O}^{P}$ (killed by $\mathfrak{g}_{C}$).
We also remark that any object in $\mathcal{O}^{\mathfrak{p}}$ or $\mathcal{O}^{P}$ is of finite length~\cite[1.11~Theorem]{humphreys-bgg}.

Let $\alggrp{T}'$ be a maximal split torus containing $\alggrp{S}_C$ in the split group $\alggrp{Z}_C$ and let $\mathfrak{t}'$ be its Lie algebra.
Let $\mathfrak{z}_{L}$ denote the Lie algebra of the center $\alggrp Z_{\alggrp L}$ of $\alggrp L$.

\begin{lem}\label{lm:cat-O-p-equivalence}
  Condition (ii) in the definition of category $\mathcal{O}^{\mathfrak p}$ may be replaced by either
  \begin{enumerate}
  \item[\upshape{(ii$'$)}] $M$ is a direct sum of 1-dimensional $\mathfrak{t}'$-modules; or
  \item[\upshape{(ii$''$)}] $M$ is a direct sum of 1-dimensional $\mathfrak{z}_{L,C}$-modules.
  \end{enumerate}
\end{lem}

\begin{proof}
  This is a question about $\mathfrak g_C$-modules, so we may work over $F = C$.
  In particular, $\alggrp G$ is split and $\alggrp{T}' = \alggrp{S}$ (and we can drop all
  extensions of scalars to $C$).
  Note that (ii$'$) implies (ii$''$), as $\mathfrak{z}_{L} \subset \mathfrak{t}'$.
  To see that (ii) implies (ii$'$), let $W$ be any absolutely simple finite-dimensional $\mathfrak{l}$-module.
  Then $W^{\mathfrak{u}}$ is 1-dimensional, so the surjection
  $U(\mathfrak{g}) \otimes_{U(\mathfrak{b})} W^{\mathfrak{u}} \onto W$ shows that $W$ is a direct sum of 1-dimensional $\mathfrak{t}'$-modules.

  Finally we show that (ii$''$) implies (ii).
  We know that $M$ is a sum of finite-dimensional $\mathfrak{l}$-submodules by (iii). 
  It suffices to show that any finite-dimensional $\mathfrak{l}$-module $W$ on which $\mathfrak{z}_{L}$ acts diagonalizably is a sum of absolutely simple $\mathfrak{l}$-modules.
  Write $\mathfrak{l} = \mathfrak{l}^{\der} \oplus \mathfrak{z}_{L}$.
  By the diagonalizability, $W$ is a direct sum of $\mathfrak{z}_{L}$-isotypic components, and this decomposition is preserved by $\mathfrak{l}^{\der}$.
  As $\mathfrak{l}^{\der}$ is a semisimple Lie algebra we may assume that $W$ is simple.
  By highest weight theory, any simple module of $\mathfrak{l}^{\der} \otimes \overline C$ may be defined over $C$, hence $W$ is absolutely simple.
\end{proof}

\begin{remark}\label{rk:tensor-O-P}
  Note that for any $M$, $M' \in \mathcal{O}^{P}$ with $M'$ finite-dimensional the tensor product $M \otimes M'$ lies in $\mathcal{O}^{P}$ as well. (Use condition (ii$'$).)
\end{remark}

\begin{remark}
  Suppose that $V$ is a finite-dimensional locally analytic representation of a parabolic subgroup $P = LN$.
  From the locally analytic homomorphism $P \to \GL(V)$ and the functoriality of the logarithm map \cite[III.7.6]{MR0573068} we see that
  \begin{equation}
    \label{eq:log-u}
    nv = \sum_{i=0}^\infty \frac{(\log n)^i}{i!} v
  \end{equation}
  for all $n \in N$ near 1 and all $v \in V$.
  By using the action of $L$ we see that in fact formula~\eqref{eq:log-u} is valid for all $n \in N$.
  We will just say that we ``integrate'' over $N$.
\end{remark}

\begin{lem}\label{lem:cop}
  Suppose $P \subset Q$.
  \begin{enumerate}
  \item The forgetful functor $\mathcal{O}^{Q}\to \mathcal{O}^{P}$ is fully faithful.
  \item If $M \in \mathcal{O}^{Q}$ and $M' \in \mathcal{O}^{P}$ is a subquotient, then $M' \in \mathcal{O}^{Q}$.
  \item If $M \in \mathcal{O}^{Q}$ is simple, it is simple in $\mathcal{O}^{P}$ as well.
  \end{enumerate}
\end{lem}

\begin{proof}
(i)
Let $M_{1},M_{2}\in \mathcal{O}^{Q}$ and we prove that any morphism $\varphi\colon M_{1}\to M_{2}$ in $\mathcal{O}^{P}$ is $Q$-equivariant.
Since $\varphi$ is $\overline{\mathfrak{u}}_{C}\cap \mathfrak{q}_{C}$-equivariant and this subalgebra acts locally nilpotently on $M_{1},M_{2}$, by integration, $\varphi$ is $\overline{U}\cap Q$-equivariant.
Since $Q$ is generated by $P$ and $\overline{U}\cap Q$, $\varphi$ is $Q$-equivariant.

(ii)
We may assume $M'$ is a subobject of $M$ in $\mathcal{O}^{P}$.
As in (i), by integration, $M'$ is $\overline{U}\cap Q$-stable and hence $Q$-stable.
This implies (iii).
\end{proof}

\begin{remark}\label{rk:cat-O-isogeny}
  If $\varphi\colon \alggrp G_{1}\to \alggrp G$ is a morphism such that $\varphi(\alggrp{G}_{1}^{\der}) = \alggrp{G}^{\der}$, $\ker\varphi\subset \alggrp{Z}_{\alggrp G_{1}}$, and the parabolic $\alggrp P_1 = \varphi^{-1}(\alggrp P)$ is obtained as pre-image of $\alggrp P$, then we obtain functors $\mathcal O^{\mathfrak p} \to \mathcal O^{\mathfrak p_1}$ and $\mathcal O^P \to \mathcal O^{P_1}$ by inflation, which will occasionally be useful.
\end{remark}

\begin{lem}\label{lem:abs-simple-g-module}
  Any simple object of $\mathcal{O}^{\mathfrak{p}}$ is absolutely simple.
  In particular, any $\mathfrak{g}_{C}$-simple object of $\mathcal{O}^{P}$ is absolutely simple.
\end{lem}

\begin{proof}
  This is a question about $\mathfrak{g}_{C}$-modules, so we may work over $F = C$.
  In particular, $\alggrp G$ is split (and we can drop all extensions of scalars to $C$). 
  Let $W \in \mathcal{O}^{\mathfrak{p}}$ simple.
  As in the proof of Lemma~\ref{lm:cat-O-p-equivalence} we can take $\lambda \into W^{\mathfrak u}$ a 1-dimensional $\mathfrak t'$-submodule, 
  so $U(\mathfrak{g}) \otimes_{U(\mathfrak b)} \lambda \onto W$. Therefore the weight space $W_\lambda$ is 1-dimensional
  and generates $W$, hence $\End_{\mathfrak{g}}(W) = C$, i.e.\ $W$ is absolutely simple.
\end{proof}

Let $W\in \mathcal{O}^{L}$.
Then on the generalized Verma module $\underline{M}(W) = \underline{M}_{G}(W) := U(\mathfrak{g}_{C})\otimes_{U(\mathfrak{p}_{C})}W$, $P$ acts by $p'(X\otimes w) = \Ad(p')X\otimes p'w$ for $p'\in P$, $w\in W$ and $X\in \mathfrak{g}_{C}$.
We have $\underline{M}(W)\in \mathcal{O}^{P}$.
(More generally, this construction works if $W$ is a finite-dimensional locally analytic representations of $P$ on which $\mathfrak{l}_{C}$ is a direct sum of absolutely simple $\mathfrak{l}_{C}$-modules, but this is not traditionally called a generalized Verma module.)
If $W\in \mathcal O^L$ is simple, then $\underline{M}(W)$ has a unique simple quotient in $\mathcal{O}^{P}$ (by the lemma that follows), which we denote by $\underline{L}(W) = \underline{L}_{G}(W)$.

Occasionally we will also use the generalized Verma module $M(W) := U(\mathfrak{g}_{C})\otimes_{U(\mathfrak{p}_{C})} W$ for $W \in \mathcal O^{\mathfrak l}$ and its simple quotient $L(W)$ (if $W$ is simple).

\begin{lem}\label{lm:N-invts}
  If $M \in \mathcal{O}^{P}$ is simple (resp.\ absolutely simple), then $M^N \in \mathcal{O}^{L}$ is simple (resp.\ absolutely simple). 
  Moreover, $\underline M(M^N)$ has $M$ as unique simple quotient in $\mathcal{O}^{P}$. 

  Conversely, if $M' \in \mathcal{O}^{L}$ is simple (resp.\ absolutely simple), then $\underline M(M')$ has a unique simple (resp.\ absolutely simple)
  quotient $Q$ in $\mathcal{O}^{P}$, and $Q^N \cong M'$. Moreover, $Q$ is the largest semisimple quotient of $\underline M(M')$ in $\mathcal{O}^{\mathfrak{p}}$.
\end{lem}

\begin{proof}
Pick any nonzero submodule $M' \subset M^N = M^{\mathfrak{n}_{C}}$ in $\mathcal{O}^{L}$, so $\underline M(M') = U(\mathfrak{g}_{C}) \otimes_{U(\mathfrak{p}_{C})} M' \twoheadrightarrow M$.
As $M$ is simple, $M|_{\mathfrak{g}_{C}}$ is semisimple, so $M|_{\mathfrak{g}_{C}} \simeq\bigoplus_{V}V\otimes M_{V}$, where $V$ runs through isomorphism classes of simple modules in $\mathcal{O}^{\mathfrak{p}}$ and $M_{V}$ is the multiplicity space.
Note that $M^{\mathfrak{n}_{C}}\simeq \bigoplus_{V}V^{\mathfrak{n}_{C}}\otimes M_{V}$ with $V^{\mathfrak{n}_{C}}\in \mathcal{O}^{\mathfrak{l}}$ (absolutely) simple $\mathfrak{l}_{C}$-modules.
Moreover by highest weight theory we see that $V_{1} \cong V_{2}$ if and only if $V_{1}^{\mathfrak{n}_{C}} \cong V_{2}^{\mathfrak{n}_{C}}$ for simple modules $V_{1},V_{2}\in \mathcal{O}^{\mathfrak{p}}$, i.e.\ $M^{\mathfrak{n}_{C}}\simeq \bigoplus_{V}V^{\mathfrak{n}_{C}}\otimes M_{V}$ is the decomposition into isotypic components.
Since $M'$ is an $\mathfrak{l}_{C}$-submodule of $M^{\mathfrak{n}_{C}}$, $M'\simeq \bigoplus_{V}V^{\mathfrak{n}_{C}}\otimes M'_{V}$ for some subspaces $M'_{V}\subset M_{V}$.
Therefore $M' = W^{\mathfrak{n}_{C}}$, where $W := \bigoplus_{V}V\otimes M'_{V}$ (considered as $\mathfrak g_C$-submodule of $M$).
Now the image of $U(\mathfrak{g}_{C}) \otimes_{U(\mathfrak{p}_{C})} M'\twoheadrightarrow M$ is contained in $W$, so $W = M$.
  Hence $M' = M^N$, i.e.\ $M^N$ is simple. The absolutely simple case follows.
  
  If $M' \in \mathcal{O}^{L}$ is simple, then $M'|_{\mathfrak{l}_{C}} \cong \bigoplus_{i=1}^r W'_i$ with $W'_i$ absolutely simple $\mathfrak{l}_{C}$-modules.
  Then $\underline M(M') = U(\mathfrak{g}_{C}) \otimes_{U(\mathfrak{p}_{C})} M'$ has a largest semisimple quotient $Q$ as $\mathfrak{g}_{C}$-module, and this quotient is in $\mathcal{O}^{P}$,
  as the $\mathfrak{g}_{C}$-radical is $P$-stable.
    Moreover the composition $M' \hookrightarrow \underline M(M') \twoheadrightarrow Q$ has image $Q^{\mathfrak{n}_{C}}$, 
  so $M' \cong Q^N$ as $L$-representations. Let $Q' \subset Q$ denote a simple subobject. Then $0 \ne {Q'}^N \subset Q^N$,
  so ${Q'}^N = Q^N$. By construction, $Q$ is a direct sum of absolutely simple $\mathfrak{g}_{C}$-modules, hence
  no proper $\mathfrak{g}_{C}$-submodule of $Q$ can have the same $\mathfrak{n}_{C}$-invariants as $Q$, so $Q' = Q$, i.e.\ $Q$ is simple in $\mathcal{O}^{P}$.
  The absolutely simple case follows, as our construction of $Q$ commutes with scalar extensions.

  Going back to the first part, $\underline M(M')$ surjects onto $M$, hence by the previous paragraph $M$ is its unique simple quotient.
\end{proof}

\begin{cor}\label{cor:N-invts}
  Suppose that $W \in \mathcal{O}^{L}$ is simple and that $Q = L_{Q}N_{Q}$ is a parabolic subgroup containing $P$.
  Then $\underline{L}(W)^{N_Q} \cong \underline{L}_{L_Q}(W)$ in $\mathcal{O}^{P \cap L_Q}$.
  Moreover, we have $\underline{L}(W) \in \mathcal{O}^{Q}$ if and only if $\underline{L}_{L_Q}(W) \in \mathcal{O}^{L_Q}$, and in this case $\underline L(W) \cong \underline L(\underline{L}_{L_Q}(W))$ in $\mathcal{O}^{Q}$.
\end{cor}

\begin{proof}
  We first note that if $M \in \mathcal O^{\mathfrak p}$, then $M^{\mathfrak n_{Q,C}} \in \mathcal O^{\mathfrak p \cap \mathfrak l_Q}$.   Hence if $M \in \mathcal O^P$, then $M^{N_Q} = M^{\mathfrak n_{Q,C}} \in \mathcal O^{P \cap L_Q}$.
  We claim that if $M \in \mathcal O^P$ is simple, then $M^{N_Q}$ is simple in $\mathcal O^{P \cap L_Q}$.   Suppose by contradiction that $0 \to M' \to M^{N_Q} \to M'' \to 0$ in $\mathcal O^{P \cap L_Q}$ with $M'$, $M''$ nonzero.
  Then the sequence splits as $\mathfrak l_{Q,C}$-modules, so it remains exact on $\mathfrak n_C \cap \mathfrak l_{Q,C}$-invariants.
  In other words, $0 \to (M')^{N \cap L_Q} \to M^N \to (M'')^{N \cap L_Q} \to 0$ is exact in $\mathcal O^L$.
  On the other hand, $(M')^{N \cap L_Q} = (M')^{\mathfrak n_C \cap \mathfrak l_{Q,C}} \ne 0$ and likewise for $(M'')^{N \cap L_Q}$, contradicting the simplicity of $M^N$ in $\mathcal O^L$ (Lemma~\ref{lm:N-invts}).
  This proves the claim.
  The isomorphism $\underline{L}(W)^{N_Q} \cong \underline{L}_{L_Q}(W)$ in $\mathcal{O}^{P \cap L_Q}$ follows by taking $M := \underline L(W)$ and applying Lemma~\ref{lm:N-invts}.

  In particular, if $\underline{L}(W) \in \mathcal{O}^{Q}$, then $\underline{L}_{L_Q}(W) \in \mathcal{O}^{L_Q}$.   Conversely, if $\underline{L}_{L_Q}(W) \in \mathcal{O}^{L_Q}$ then $\underline L(\underline{L}_{L_Q}(W)) \in \mathcal{O}^{Q}$ and by taking $N$-invariants in stages we can identify $\underline L(\underline{L}_{L_Q}(W))$ with $\underline L(W)$ in $\mathcal{O}^{P}$, and hence in $\mathcal{O}^{Q}$, by Lemma~\ref{lm:N-invts}.
\end{proof}

\begin{lem}\label{lem:verma-tensor}
  Suppose that $W \in \mathcal{O}^{L}$ and $X \in \mathcal{O}^{P}$ with $X = X^N$.
  Then $W \otimes X \in \mathcal{O}^{L}$ and $\underline{M}(W \otimes X) \cong \underline{M}(W) \otimes X$ in $\mathcal{O}^{P}$.
  If moreover $W \otimes X$ is simple in $\mathcal O^L$ and $\underline L(W)\otimes X$ is simple in $\mathcal O^P$, then $\underline{L}(W \otimes X) \cong \underline{L}(W) \otimes X$ in $\mathcal{O}^{P}$.
\end{lem}

In particular, $X$ could be a locally analytic character of $G$ (cf.\ Lemma~\ref{lm:loc-an-char}).

\begin{proof}
  We have $X = X^N \in \mathcal{O}^{L}$, so $W \otimes X \in \mathcal{O}^{L}$ by Remark~\ref{rk:tensor-O-P}.
  The natural isomorphism $U(\mathfrak{g}_{C}) \otimes_{U(\mathfrak{p}_{C})} (W \otimes X) \to (U(\mathfrak{g}_{C}) \otimes_{U(\mathfrak{p}_{C})} W) \otimes X$ of $\mathfrak g_C$-modules sending $1 \otimes (w \otimes x)$ to $(1 \otimes w) \otimes x$ (cf.\ \cite[\S3.6]{humphreys-bgg}) is also $P$-equivariant, which completes the proof of the first claim.
  The second claim follows by Lemma~\ref{lm:N-invts} (noting that $W$ is simple, as $W \otimes X$ is simple).
\end{proof}

Let $\Rep^{\adm}(L)$ be the category of admissible smooth $L$-representations.
When $\alggrp{G}$ is split, Orlik--Strauch \cite{OS3} define a functor $\mathcal{F}_{P}^{G}$ from $(\mathcal{O}^{P})^{\opp}\times \Rep^{\adm}(L)$ to the category of admissible locally analytic representations of $G$.
Their functor can be generalized to any $\alggrp{G}$, see Appendix~\ref{app:orlik-strauch}.

For the theory of Orlik--Strauch, we need a small assumption on $p$ as follows \cite[Assumption 4.1]{OS3}.
\begin{assump}\label{assump:on p}
If the absolute root system of $\alggrp{G}$ has irreducible components of type $B$, $C$ or $F_{4}$, we assume $p > 2$.
If the absolute root system of $\alggrp{G}$ has irreducible components of type $G_{2}$, we assume $p > 3$.
\end{assump}

\begin{defn}
  We say that $M \in \mathcal O^P$ is \emph{equimaximal} if for any parabolic subgroup $Q$ containing $P$ we have $M \in \mathcal O^Q$ if and only if $M|_{\mathfrak g_C} \in \mathcal O^{\mathfrak q}$.
  In this case we say that $M$ has \emph{maximal parabolic $Q$} if $Q$ is largest among all parabolic subgroups containing $P$ such that $M \in \mathcal O^{Q}$.
\end{defn}

Recall that we assumed in \S\ref{sec:notation} that $C$ be sufficiently large.
The reason is part (iii) of the following theorem (see Appendix~\ref{app:orlik-strauch}).

\begin{thm}[Theorem~\ref{thm:OS-main}]\label{thm:OS-main-in-main}
Assume Assumption~\ref{assump:on p}.
\begin{enumerate}
\item The functor $\mathcal{F}_{P}^{G}$ is exact in both arguments.
\item Let $Q = L_{Q}N_{Q}\supset P$ be another parabolic subgroup.
If $M\in \mathcal{O}^{Q}$ and $\pi\in\Rep^{\adm}(L)$, then $\mathcal{F}_{P}^{G}(M,\pi)\simeq \mathcal{F}_{Q}^{G}(M,(\Ind_{P \cap L_{Q}}^{L_{Q}}\pi)^{\sm})$.
\item Assume that $M\in \mathcal{O}^{P}$ is equimaximal with maximal parabolic $P$ and $\pi \in \Rep^{\adm}(L)$.
Assume that $M|_{\mathfrak{g}_{C}} \in \mathcal O^{\mathfrak p}$ is simple and $\pi$ is irreducible.
Then $\mathcal{F}_{P}^{G}(M,\pi)$ is irreducible.
\end{enumerate}
\end{thm}

We also note two more basic properties that follow by construction.
For a locally convex space $V$, let $V'$ be the strong dual space of $V$, namely $V'$ is the space of continuous linear maps $V\to C$ with the topology of uniform convergence on bounded subsets of $V$.

\begin{prop}\label{prop:OS-ind}\ 
  \begin{enumerate}
  \item Suppose that $W$ is a finite-dimensional locally analytic representation of $P$ on which $\mathfrak{l}_{C}$ acts as a direct sum of absolutely simple $\mathfrak{l}_{C}$-modules and $\pi \in \Rep^{\adm}(L)$.
    Then
    $\mathcal{F}_{P}^{G}(\underline M(W),\pi) \cong (\AnInd_{P}^{G}W'\otimes\pi)^{\an}$.
  \item If $\tau$ is any finite-dimensional smooth representation of $L$, $M \in \mathcal O^P$, $\pi \in \Rep^{\adm}(L)$, then
    $\mathcal F_P^G(M \otimes \tau,\pi) \cong \mathcal F_P^G(M,\pi \otimes \tau')$.
  \end{enumerate}
\end{prop}

The following corollary will be proved following Proposition~\ref{prop:twist-to-O-P} (which is needed as input).

\begin{cor}\label{cor:O-S-fin-length}
Assume Assumption~\ref{assump:on p}.
Let $M\in \mathcal{O}^{P}$ and $\pi\in \Rep^{\adm}(L)$ such that $\pi$ is of finite length.
Then $\mathcal{F}_{P}^{G}(M,\pi)$ is strongly admissible and topologically of finite length.
\end{cor}

We say that a finite-dimensional representation of $G$ (or more generally $P$) is \emph{algebraic} if it is obtained by restriction
from an algebraic representation of the split group $\alggrp{G}_{C} := \alggrp{G}\times_{F}C$ (or $\alggrp{P}_{C}$).

We let $\mathcal O^P_\alg \subset \mathcal O^P$ be the full subcategory consisting of those objects where the action of $P$ is locally finite-dimensional algebraic.
It is closed under subquotients.
The composition $\mathcal O^P_\alg \to \mathcal O^P \to \mathcal O^{\mathfrak p}$ is fully faithful with essential image $\mathcal O_\alg^{\mathfrak p}$, consisting of those objects where $\mathfrak t'$ acts by elements of $X^*(\alggrp T') \subset \mathfrak t'^*$ \cite[\S2, Lemma 3.2]{OS2}.
If $W \in \mathcal O^L_\alg$, then $\underline M(W) \in \mathcal O^P_\alg$ and hence $\underline L(W) \in \mathcal O^P_\alg$ (if $W$ is simple).
We state a useful consequence.

\begin{lem}\label{lem:alg-equimax}
  Any $M \in \mathcal O^P_\alg$ is equimaximal.
\end{lem}

\subsection{Some decompositions in \texorpdfstring{$\mathcal{O}$}{O}}
\label{sec:some-decomp}
Recall that $\alggrp{T}'$ is a maximal split torus containing $\alggrp{S}_C$ in the split group $\alggrp{Z}_C$ with Lie algebra $\mathfrak{t}'$.
Let $\mathfrak{b}' = \mathfrak{t}'\oplus\mathfrak{u}'$ be a Borel subalgebra of $\mathfrak{g}_{C}$.
We write $\widetilde{\Phi}$ for the set of $\mathfrak{t}'$-roots in $\mathfrak{g}_{C}$ and $\widetilde{\Phi}^{+}$ for the set of $\mathfrak{t}'$-roots in $\mathfrak{u}'$.
For $\lambda\in (\mathfrak{t}')^{*}$, the Verma module $U(\mathfrak{g}_{C})\otimes_{U(\mathfrak{b}')}\lambda$ has a unique irreducible quotient $L(\lambda)$ in $\mathcal O^{\mathfrak{b}'}$.

\begin{lem}\label{lem:extend-char}
  If $H$ is a compact abelian locally $F$-analytic group, then any $F$-linear map $\lambda \colon  \Lie H \to C$ lifts
  to a locally analytic homomorphism $H \to C^{\times}$ (after perhaps replacing $C$ by a finite extension).
  If $\alggrp{H}$ is a torus over $F$, then any $F$-linear map $\lambda \colon \Lie \alggrp{H} \to C$ lifts to a locally analytic
  homomorphism $\alggrp{H}(F) \to C^{\times}$ (after perhaps replacing $C$ by a finite extension).
\end{lem}

\begin{remark}\label{rk:base-ext}
  The finite extension of $C$ may depend on $\lambda$, as one can already see in case $H = \Z_{p}$ over $F = \Q_{p}$.
  (Note that if $\chi \colon \Z_{p} \to C^{\times}$, then $d\chi(\Z_{p}) \subset \log(\mathcal{O}^{\times}_{C})$ is bounded.)
\end{remark}

\begin{proof}
  The given $\lambda$ lifts to a locally analytic homomorphism $f \colon H_0 \to C^{\times}$ for some open subgroup $H_0$ of $H$ by \cite[Proposition~18.17]{MR2810332},
  so $H_0$ is of finite index. Using for example that $\mathcal{O}_{\overline{C}}^{\times}$ (where $\overline{C}$ is an algebraic closure of $C$) is an injective abelian group we deduce that we can extend
  $f$ to a homomorphism $H \to {C'}^{\times}$ for some finite extension $C'$ of $C$.  

  For the second part, note that $\alggrp{H}(F) \cong \alggrp{H}(F)^1 \times \Z^{d}$ for some $d \ge 0$, where $\alggrp{H}(F)^1$ is the maximal compact subgroup \cite[\S2.5(b), (c)]{KP}.
  In particular, $\alggrp{H}(F)^1$ is a compact abelian locally $F$-analytic group and we can apply the first part.
\end{proof}

\begin{lem}\label{lem:semisimple-lift}
  Suppose that $\alggrp{G}^{\der}$ is simply connected.   \begin{enumerate}
  \item If $\alggrp{G}$ is semisimple (simply-connected), then any simple object of $\mathcal{O}^{\mathfrak{g}}$ lifts to $\mathcal{O}^{G}$ (even $\mathcal{O}_\alg^{G}$). \label{lem:semisimple-lift-1}
    If moreover all simple factors of $\alggrp{G}$ are isotropic, then $\mathcal{O}^{\mathfrak{g}} = \mathcal{O}_\alg^{G} = \mathcal{O}^{G}$.
  \item In general, any simple object of $\mathcal{O}^{\mathfrak{p}}$ lifts to $\mathcal{O}^{P}$, after perhaps replacing $C$ by a finite extension. \label{lem:semisimple-lift-2}
    Any two lifts differ by a smooth character of $L$.
          \item If $M$ is a $\mathfrak{g}_{C}$-simple object in $\mathcal{O}^{G}$, then $M \cong M_\alg \otimes \psi$, where $M_\alg$ is an algebraic representation of $G$ and 
    $\psi$ a locally analytic character of $G$ (after perhaps replacing $C$ by a finite extension). \label{lem:semisimple-lift-3}
  \end{enumerate}
\end{lem}

\begin{remark}\label{rk:need-simply-conn}
  The second claim in \ref{lem:semisimple-lift-1} fails without the condition on the simple factors. For example, let $G = \SL_1(D)$, where $D$ is a finite-dimensional non-commutative
  division algebra over $F$. Then $G$ admits a (finite-dimensional) irreducible smooth representation of dimension greater than 1, which becomes reducible in $\mathcal{O}^{\mathfrak{g}}$.

  Part \ref{lem:semisimple-lift-2} may require an extension of scalars even when $\alggrp{G}$ is a torus, cf.\ Remark~\ref{rk:base-ext}.

  We cannot drop the condition that $\alggrp{G}^{\der}$ is simply connected. 
  For parts \ref{lem:semisimple-lift-1} and \ref{lem:semisimple-lift-2}, suppose $\alggrp{G} = \PGL_3$, $F = \Q_{p}$,
  $p \equiv 1 \pmod 3$.  Then we have $L(2/3, -1/3, -1/3) \in \mathcal{O}^{\mathfrak{g}}$, and we claim that it does not lift to $\mathcal{O}^{G}$.
  If there was a lift in $\mathcal{O}^{G}$, then we can inflate it to $\mathcal{O}^{\GL_3}$, so by part \ref{lem:semisimple-lift-3} it is of the form 
  $L(1,0,0) \otimes (\psi \circ \det)$, where $\psi$ is a locally analytic character of $\Q_{p}^{\times}$. 
  But by considering the 3-torsion subgroup of $\Q_{p}^\times$ it is easy to see that such a lift does not exist,
  as we cannot solve $\psi(x)^3 = x^{-1}$ ($x \in \Q_{p}^{\times}$).
  A similar example with $\alggrp{G} = \PGL_2$ can be found in \cite[Example 2.4]{OS3}.
  For part \ref{lem:semisimple-lift-3}, suppose $\alggrp{G} = \PGL_3$, $F = \Q_{p}$, $p \equiv 2 \pmod 3$, and $p^{1/3} \in C$.
  Then there exists a continuous (hence locally $\Q_{p}$-analytic) character $\psi \colon \Q_{p}^{\times} \to C^{\times}$ such that $\psi(x)^3 = x^{-1}$. Let $V$ be the standard representation
  of $\GL_3(\Q_{p})$. Then $V \otimes (\psi\circ\det) \in \mathcal{O}^{G}$ but is not an algebraic representation up to twist.
  (The last example is related to \cite[\S3 Example]{MR1835001}.)
\end{remark}

\begin{proof}
  \ref{lem:semisimple-lift-1} Take any $M \in \mathcal{O}^{\mathfrak{g}}$ simple.
  Then $M \cong L(\lambda)$ for some $\lambda \in (\mathfrak{t}')^*$ with $\langle\lambda,\alpha^{\vee}\rangle \in \Z_{\ge 0}$ for all $\alpha \in \widetilde{\Phi}^{+}$.
  As $\alggrp G$ is semisimple simply-connected, $\lambda \in X^*(\alggrp{T}')$ is dominant and so $M$ lifts to the algebraic representation $L(\lambda)$
  of $\alggrp G_{C}$; by restriction to $G$ we get our desired object of $\mathcal{O}_\alg^{G}$.

  If moreover all simple factors of $\alggrp{G}$ are isotropic, then by Kneser--Tits we know that $G$ is generated by the unipotent radicals
  $U$ and $\overline{U}$. If $M_1, M_2 \in \mathcal{O}^{G}$ and $f \colon M_1 \to M_2$ is $U(\mathfrak{g}_{C})$-linear, then by integrating the action over unipotents
  we deduce that $f$ is $U$-linear and $\overline{U}$-linear, hence $G$-linear. Therefore the forgetful functor $\mathcal{O}^{G} \to \mathcal{O}^{\mathfrak{g}}$ is fully faithful.
    By the previous paragraph the composite $\mathcal{O}_\alg^{G} \to \mathcal{O}^{G} \to \mathcal{O}^{\mathfrak{g}}$ of fully faithful embeddings is an essential surjection (as $\mathcal{O}^{\mathfrak{g}}$ is semisimple), hence we get equivalences $\mathcal{O}_\alg^{G} = \mathcal{O}^{G} = \mathcal{O}^{\mathfrak{g}}$.

  \ref{lem:semisimple-lift-2} First we consider the case $P = G$ and take $W \in \mathcal{O}^{\mathfrak{g}}$ (absolutely) simple.
  By the above we can lift $W|_{\mathfrak{g}_{C}^{\der}}$ to an algebraic representation of $G_{C}^{\der}$,
  which extends to an algebraic representation of $G_{C}$ (as $\alggrp{T}'^{\der}$ is a direct factor of $\alggrp{T}'$). So there exists an algebraic
  $G$-representation $M$ that agrees with $W$ on $\mathfrak{g}_{C}^{\der}$.
  The space $\lambda := \Hom_{\mathfrak{g}_C^{\der}}(M,W)$ has a natural action of $\mathfrak{g}_C/\mathfrak{g}_C^{\der}$ and, since $M,W$ are absolutely simple $\mathfrak{g}_C^{\der}$-modules, $\lambda$ is $1$-dimensional.
  Hence $M\otimes \lambda\to W$ is an isomorphism in $\mathcal{O}^{\mathfrak{g}}$.
  By Lemma~\ref{lem:extend-char} (extending scalars if necessary) we can lift $\lambda \colon \mathfrak{g}_{C}/\mathfrak{g}_{C}^{\der} \to C$ to a locally analytic character $G/G^{\der} \to C^{\times}$, 
  and this implies the claim.
  
  Now suppose $P$ is arbitrary standard parabolic subgroup and $W \in \mathcal{O}^{\mathfrak{p}}$ simple. Then $W^{\mathfrak{n}_C} \in \mathcal{O}^{\mathfrak{l}}$ is simple and we can lift it to an object $M'$ of $\mathcal{O}^{L}$ by the previous paragraph
  (after a scalar extension).
  Then the unique simple quotient $\underline{L}(M')$ of $\underline{M}(M')$ in $\mathcal{O}^{P}$ lifts $W$ by Lemma~\ref{lm:N-invts}.
  
  If $M_1$, $M_2 \in \mathcal{O}^{P}$ are lifts of the same simple object of $\mathcal{O}^{\mathfrak{p}}$, then $M_1^N|_{\mathfrak{l}_{C}} \cong M_2^N|_{\mathfrak{l}_{C}}$ and we let
  $\eta := \Hom_{\mathfrak{l}_{C}}(M_1^N,M_2^N)$, a smooth character of $L$. Then $M_1^N \otimes \eta \cong M_2^N$ in $\mathcal{O}^{L}$ and hence
  $M_1 \otimes \eta \cong M_2$ in $\mathcal{O}^{P}$ by Lemma~\ref{lm:N-invts}.

  \ref{lem:semisimple-lift-3} In the proof of part \ref{lem:semisimple-lift-2} we saw that $M|_{\mathfrak{g}_{C}}$ admits a lift of the form $M' = M_\alg \otimes \psi'$ in $\mathcal{O}^{G}$, where $M_\alg$ is algebraic
  and $\psi'$ a character of $G/G^{\der}$. Then $\eta := \Hom_{\mathfrak{g}_{C}}(M',M)$ is a 1-dimensional smooth representation of $G$ and
  $M \cong M' \otimes \eta \cong M_\alg \otimes \psi'\eta$.
\end{proof}

\begin{prop}\label{prop:simple-O-P}\ 
  \begin{enumerate}
  \item If $M_0 \in \mathcal{O}^{P}$ is $\mathfrak{g}_{C}$-simple and $\tau$ is finite-dimensional irreducible (resp.\ absolutely irreducible) smooth $L$-representation,
    then $M_0 \otimes \tau$ is a simple (resp.\ absolutely simple) object of $\mathcal{O}^{P}$.
  \item The decomposition $M_0 \otimes \tau$ in (i) is unique up to a smooth character of $L$.
  \item If $\alggrp{G}^{\der}$ is simply connected, then any absolutely simple object of $\mathcal{O}^{P}$ can be decomposed as in (i), after perhaps
    replacing $C$ by a finite extension.
  \end{enumerate}
\end{prop}

\begin{proof}\
  (i) Consider $M = M_0 \otimes \tau$, where $M_0 \in \mathcal{O}^{P}$ is $\mathfrak{g}_{C}$-simple and $\tau$ is a finite-dimensional irreducible smooth $L$-representation.
  Then $M$ lies in $\mathcal{O}^{P}$ clearly. By restricting to $\mathfrak{g}_{C}$ we see that any subobject is of the form $M_0 \otimes V'$ for some $L$-stable subspace 
  $V' \subset \tau$, hence $M$ is simple. If moreover $\tau$ is absolutely irreducible, then $M$ is absolutely simple.

  (ii) Suppose $M_1 \otimes \tau_1 \cong M_2 \otimes \tau_2$ in $\mathcal{O}^{P}$ with $M_i, \tau_i$ as in the statement of part (i). 
  By taking $N$-invariants we get $M_1^N \otimes \tau_1 \cong M_2^N \otimes \tau_2$ in $\mathcal{O}^{L}$, where $M_i^N \in \mathcal{O}^{L}$ are $\mathfrak{l}_{C}$-simple.

  Restricting to $\mathfrak{l}_{C}$ we deduce $M_1^N|_{\mathfrak{l}_{C}} \cong M_2^N|_{\mathfrak{l}_{C}}$. 
  Then $\eta := \Hom_{\mathfrak{l}_{C}}(M_1^N,M_2^N)$ is a smooth 1-dimensional representation of $L$ and $M_1^N \otimes \eta \cong M_2^N$.
    By replacing $(M_1,\tau_1)$ by $(M_1 \otimes \eta,\tau_1 \otimes \eta^{-1})$
  we may assume $\eta = 1$. From $\Hom_{\mathfrak{l}_{C}}(M_1^N, M_1^N \otimes \tau_i) \cong \tau_i$ we then deduce $\tau_1 \cong \tau_2$.
  Finally we deduce $M_1 \cong M_2$ from 
    $\underline{M}(M_{1}^{N})\simeq \underline{M}(M_{2}^{N})$ and Lemma~\ref{lm:N-invts}.

  (iii) We first treat the case where $P = G$, so suppose $M \in \mathcal{O}^{G}$ absolutely simple. 
      Pick $W_0$ a simple subobject of $M$ in $\mathcal{O}^{\mathfrak{g}}$.
  Extending scalars if necessary, we can find $M_0 \in \mathcal{O}^{G}$ lifting $W_0$ by Lemma~\ref{lem:semisimple-lift} (using that $\alggrp{G}^{\der}$ is simply connected).
  We now have absolutely irreducible $G$-representations $M_0, M$ such that
  \begin{equation*}
    \tau := \Hom_{\mathfrak{g}_{C}}(M_0, M) \ne 0.
  \end{equation*}
  Note that $G$ acts smoothly on $\tau$ and that we get a nonzero $G$-linear map $M_0 \otimes \tau \to M$. This is an isomorphism,
  as the both sides are simple. It follows that $\tau$ is absolutely irreducible.

  Now in general, suppose that $M \in \mathcal{O}^{P}$.
  Note that $M^N \in \mathcal{O}^{L}$ is absolutely simple by Lemma~\ref{lm:N-invts}. Write $M^{N} = M_0 \otimes \tau$ by the previous paragraph (after perhaps extending scalars), 
  where $M_0 \in \mathcal{O}^{L}$ is $\mathfrak{l}_{C}$-simple and 
  $\tau$ is an absolutely irreducible smooth $L$-representation.
  Then $M \simeq \underline{L}(M^{N})\simeq \underline{L}(M_{0})\otimes \tau$ by Lemma~\ref{lem:verma-tensor} and $\underline{L}(M_{0})\in \mathcal{O}^{P}$ is $\mathfrak{g}_{C}$-simple.
\end{proof}

\begin{cor}\label{cor:tau-trivial}
  Suppose that $P = LN$ and that all simple factors of the adjoint group $\alggrp{L}^{\ad} = \alggrp{L}/\alggrp{Z}_{\alggrp{L}}$ are isotropic.   Then every absolutely simple object of $\mathcal O^P$ is $\mathfrak{g}_{C}$-simple, i.e.\ a decomposition as in Proposition~\ref{prop:simple-O-P}(i) exists (with $\tau = 1$).
\end{cor}

\begin{proof}
  Suppose that $M \in \mathcal O^P$ is absolutely simple.
  If $\alggrp{G}^{\der}$ is simply connected, then we can write $M \cong M_0 \otimes \tau$ as in Proposition~\ref{prop:simple-O-P}.
  As $\alggrp{L}^{\der}$ is simply connected and all simple factors of $\alggrp{L}^{\der}$ are isotropic, by Kneser--Tits we know that $L' = L^\der$.   Hence the finite-dimensional smooth representation $\tau$ is trivial on $L^\der$ by Lemma~\ref{lm:loc-an-char}.
  Since $L/L^\der$ is abelian, we deduce that $\tau$ is 1-dimensional and hence $M$ is $\mathfrak{g}_C$-simple.

  For general $\alggrp{G}$, let $\wt{\alggrp G} \onto \alggrp G$ be a $z$-extension, so $\wt {\alggrp G}^{\lowprime[\der]}$ is simply connected, $\wt G \onto \wt G$ on $F$-points, and $\wt{\mathfrak g}_C \onto \mathfrak g_C$ for the Lie algebras.
  Then the inflation $\wt M$ of $M$ becomes an absolutely simple object of $\mathcal O^{\wt P}$, where $\wt P$ is the pre-image of $P$ in $\wt G$.
  By the previous paragraph, $\wt M$ is $\wt{\mathfrak g}_C$-simple, so $M$ is $\mathfrak{g}_C$-simple.
\end{proof}

\begin{prop}\label{prop:twist-to-O-P}  
  Suppose that $\alggrp{G}^{\der}$ is simply connected.
  \begin{enumerate}
  \item If $M$ is a $\mathfrak{g}_{C}$-simple object in $\mathcal{O}^{B} \cap \mathcal{O}^{\mathfrak{p}}$, then, perhaps after replacing $C$ by a finite extension, there
    is a smooth character $\eta$ of $Z$ such that $M \otimes \eta$ lies in $\mathcal{O}^{P}$.  Moreover, $\eta$ is unique up to a smooth
    character of $L$.
  \item 
    If $M$ is a $\mathfrak{g}_{C}$-simple object of $\mathcal{O}^{P}$, then, perhaps after replacing $C$ by a finite extension, there is a smooth character $\eta$ of $L$ such that $M \otimes \eta$ is
    equimaximal.  Moreover, $\eta$ is unique up to a smooth character of $L_Q$, where $Q$ is the maximal parabolic such that
    $M \in \mathcal{O}^{\mathfrak{q}}$.
  \end{enumerate}
\end{prop}

\begin{remark}\label{rk:need-simply-conn2}
  We cannot drop the condition that $\alggrp{G}^{\der}$ is simply connected in part (i) (see the example with $G = \PGL_3(F)$, $p \equiv 1\pmod 3$ in Remark~\ref{rk:need-simply-conn}).
\end{remark}

\begin{proof}
  (i) Take $M$ a $\mathfrak{g}_{C}$-simple object in $\mathcal{O}^{B} \cap \mathcal{O}^{\mathfrak{p}}$.
  By Lemma~\ref{lem:semisimple-lift}(ii) there exists $W\in \mathcal{O}^{P}$ such that $W|_{\mathfrak{g}_{C}}\simeq M|_{\mathfrak{g}_{C}}$ and moreover there exists a smooth character $\eta$ of $Z$ such that $W\simeq M\otimes \eta$ in $\mathcal{O}^{B}$.
  The uniqueness follows from the uniqueness of Lemma~\ref{lem:semisimple-lift}(ii).  
  
  (ii) By part (i) there exists a smooth character $\eta$ of $Z$ such that $M \otimes \eta \in \mathcal{O}^{Q}$, i.e. $M \otimes \eta$ is equimaximal.
  By the uniqueness part of (i) applied to $M, M \otimes \eta \in \mathcal{O}^{P}$, we see that $\eta$ is in fact a smooth character of $L$.
  The uniqueness assertion follows by the same reasoning.
\end{proof}

We can now prove Corollary~\ref{cor:O-S-fin-length}.

\begin{proof}[Proof of Corollary~\ref{cor:O-S-fin-length}]
  The strong admissibility follows from the proof of \cite[Lemma 2.4(ii)]{OS2}.
  (By Lemma~\ref{lm:loc-an-char} any irreducible finite-dimensional locally analytic representation of $P$ is trivial on $N$.)
  
  For the finite length claim, we first make a reduction to the case where $\alggrp{G}^{\der}$ is simply connected.
  We take a $z$-extension $\wt {\alggrp G} \onto \alggrp G$, so $\wt {\alggrp G}^{\lowprime[\der]}$ is simply connected and $\wt G \onto G$ on $F$-points.
  By pullback to $\wt G$ we obtain $\wt P = \wt L \wt N$, and inflation gives $\wt M \in \mathcal O^{\wt P}$ and $\wt\pi$ finite-length smooth of $\wt L$.
  The construction of Orlik--Strauch is compatible with pullback, i.e.\ the inflation of $\mathcal F_P^G(M,\pi)$ to $\wt G$ is naturally isomorphic to $\mathcal F_{\wt P}^{\wt G}(\wt M,\wt \pi)$.
  Thus if $\mathcal F_{\wt P}^{\wt G}(\wt M,\wt \pi)$ is topologically of finite length, so is $\mathcal F_P^G(M,\pi)$.

  Suppose $\alggrp{G}^{\der}$ is simply connected.
  By exactness of $\mathcal F_P^G$, we may assume that $M$ is simple.
  It suffices to prove the result after a finite scalar extension, and we allow such extensions in the proof without further comment.
  In particular, we may assume that $M$ is absolutely simple.
  By Proposition~\ref{prop:simple-O-P}(iii) we can write $M = M_0 \otimes \tau$, where $M_0$ is $\mathfrak g_C$-simple and $\tau$ is a finite-dimensional smooth $L$-representation.
  Using Proposition~\ref{prop:twist-to-O-P}(ii) we may moreover twist by a smooth character of $L$ and assume that $M_0$ is equimaximal.
  By Proposition~\ref{prop:OS-ind}(ii), $\mathcal F_P^G(M,\pi) \cong \mathcal F_P^G(M_0,\pi \otimes \tau')$.
  Moreover, $\pi \otimes \tau'$ is of finite length, since it is admissible and finitely generated.
  Hence we may assume that $M$ is $\mathfrak g_C$-simple and equimaximal.
  If $Q$ denotes the maximal parabolic of $M$, then $\mathcal F_P^G(M,\pi) \cong \mathcal F_Q^G(M,(\Ind_{P \cap L_{Q}}^{L_{Q}}\pi)^{\sm})$, and $(\Ind_{P \cap L_{Q}}^{L_{Q}}\pi)^{\sm}$ is still of finite length.
  Finally, $\mathcal F_Q^G(M,(\Ind_{P \cap L_{Q}}^{L_{Q}}\pi)^{\sm})$ is of finite length by exactness of $\mathcal F_Q^G$ in the second argument and Theorem~\ref{thm:OS-main-in-main}(iii).
\end{proof}

\begin{lem}\label{lem:subobject-verma}
      Suppose that $W \in \mathcal{O}^{L}$ is $\mathfrak{l}_{C}$-simple. Then every subobject of $\underline{M}(W)$ in $\mathcal{O}^{\mathfrak{p}}$ is $P$-stable, i.e.\ lies in $\mathcal{O}^{P}$.
  In particular, every Jordan--H\"older factor of $\underline{M}(W)$ in $\mathcal{O}^{P}$ is $\mathfrak{g}_{C}$-simple.
\end{lem}

\begin{proof}
  Suppose first that $\alggrp{G}^{\der}$ is simply connected.
  It suffices to check this after a finite scalar extension. So by Lemma~\ref{lem:semisimple-lift} we may write $W = W_\alg \otimes \psi$
  with $W_\alg$ algebraic and $\psi$ a locally analytic character of $L$. By integration, every subobject of $\underline{M}(W)$ in $\mathcal{O}^{\mathfrak{p}}$ is $N$-stable.
  As an $L$-representation we have $\underline{M}(W) = U(\mathfrak{g}_{C}) \otimes_{U(\mathfrak{p}_{C})} W \cong U(\mathfrak{u}_{P,C}^{-}) \otimes W_\alg \otimes \psi$, and by twisting it suffices
  to check $L$-stability when $\psi = 1$. As $U(\mathfrak{u}_{P,C}^{-}) \otimes W_\alg$ is a (locally finite) algebraic representation of $L$,
  it is clear that any $\mathfrak{l}_{C}$-submodule is $L$-stable.
  For general $\alggrp G$, take a $z$-extension $\varphi : \wt {\alggrp G} \onto \alggrp G$, so $\wt {\alggrp G}^{\lowprime[\der]}$ is simply connected and $\wt G \onto G$ on $F$-points.
  Let $\wt{\alggrp P} := \varphi^{-1}(\alggrp P)$, $\wt{\alggrp L} := \varphi^{-1}(\alggrp L)$ and let $\wt W \in \mathcal O^{\wt L}$ be obtained by inflation (cf.\ Remark~\ref{rk:cat-O-isogeny}).
  Then $\underline M_{\wt G}(\wt W) \in \mathcal O^{\wt P}$ is obtained from $\underline M_{G}(W) \in \mathcal O^{P}$ by inflation, and the claim follows from the previous case.
\end{proof}

\begin{lem}\label{lem:simple-in-O-Q}
  Suppose that $W \in \mathcal{O}^{L}$ is absolutely simple and that $Q = L_Q N_Q$ is any parabolic subgroup containing $P$.
  \begin{enumerate}
  \item 
    Suppose that $W \cong W_0 \otimes \tau$ in $\mathcal O^L$, where $W_0$ is $\mathfrak l_C$-simple and $\tau$ is smooth.     Then $\underline L(W) \in {\mathcal O}^Q$ if and only if $\underline L(W_0) \in {\mathcal O}^{Q}$ and $\tau$ extends to a smooth representation of $L_Q$.
  \item   
    Suppose that $\alggrp{G}^{\der}$ is simply connected. 
    Then $\underline L(W) \in {\mathcal O}^Q$ if and only if (after perhaps replacing $C$ by a finite extension) $W \cong W_\alg \otimes \psi|_L \otimes \wt\tau|_L$ for some algebraic representation $W_\alg$ of $L$, a locally analytic character $\psi$ of $L_Q$, and a smooth representation $\wt\tau$ of $L_Q$ such that moreover $L(W_\alg) \in {\mathcal O}^{\mathfrak{q}}$.
      \end{enumerate}
\end{lem}

In part (ii) we could alternatively demand that $W \cong W_\alg \otimes \psi|_L \otimes \tau$, where $\tau$ is a smooth representation $L$ that is trivial on $L \cap L_Q'$ (after perhaps replacing $C$ by a finite extension).
(If $\wt\tau$ is as above, then it is trivial on $L_Q'$ by Lemma~\ref{lm:loc-an-char}.
Conversely, $\tau$ is a representation of $L/(L \cap L_Q') \cong L_Q/L_Q'$, so extends to $L_Q$.)
Also note that the property $L(W_\alg) \in {\mathcal O}^{\mathfrak{q}}$ is equivalent to the lowest weight of $W_\alg$ in $X^*(\alggrp T')$ being antidominant relative to the Levi $\alggrp L_{\alggrp Q}$.

\begin{proof}
  (i) Suppose that $\underline L(W) \in {\mathcal O}^Q$ and $W \cong W_0 \otimes \tau$ as in the statement of the lemma.
  Then $\underline L(W) \cong \underline L(W_0) \otimes \tau$ in $\mathcal O^P$ by Lemma~\ref{lem:verma-tensor}.
  Moreover $\underline L(W)^{N_Q} \in \mathcal O^{L_Q}$ is finite-dimensional and by Lemma~\ref{lm:loc-an-char} we have $\underline L(W)^{N_Q} \cong \underline L(W_0)^{N_Q} \otimes \tau$.
  Hence $\underline L(W_0)^{N_Q}$ is finite-dimensional, hence lies in $\mathcal O^{L_Q}$, so $\underline L(W_0) \in \mathcal O^Q$ as it is a quotient of $\underline M(\underline L(W_0)^{N_Q}) \in \mathcal O^Q$.
  Let $\wt\tau := \Hom_{\mathfrak g_C}(\underline L(W_0),\underline L(W))$.
  Then $\wt\tau$ has a smooth action of $L_Q$ and $\wt\tau|_L \cong \tau$ by the isomorphism above.

  Conversely, let $\wt\tau$ be the (unique) smooth extension of $\tau$ to $L_Q$.
  Then $\underline L(W) \cong \underline L(W_0) \otimes \wt\tau \in \mathcal O^Q$ because they are isomorphic in $\mathcal O^P$.

  (ii) If $\alggrp{G}^{\der}$ is simply connected, then we can always decompose $W \cong W_0 \otimes \tau$ as in (i) by Proposition~\ref{prop:simple-O-P} (after perhaps extending scalars).
  If $\underline L(W) \in {\mathcal O}^Q$, then by (i) we deduce that $\tau$ extends to a smooth representation $\wt\tau$ of $L_Q$ and that $\underline L(W_0) \in {\mathcal O}^{Q}$.
  By Lemma~\ref{lem:semisimple-lift}\ref{lem:semisimple-lift-3} we can write $\underline L(W_0)^{N_Q} \cong M_\alg \otimes \psi$ in $\mathcal O^{L_Q}$ with $M_\alg$ algebraic and $\psi$ a locally analytic character of $L_Q$ (after perhaps extending scalars).
  Taking $N \cap L_Q$-invariants and using Lemma~\ref{lm:loc-an-char} we get that $W_0 \cong M_\alg^{N \cap L_Q} \otimes \psi|_L$ with $M_\alg^{N \cap L_Q}$ algebraic.

  Conversely, if $W \cong W_\alg \otimes \psi|_L \otimes \wt\tau|_L$ as in the statement of the lemma, then $\underline L(W_\alg)$ is equimaximal as $W_\alg$ is algebraic (Lemma~\ref{lem:alg-equimax}), so $\underline L(W_\alg) \in \mathcal O^Q$.
  Let $W_1 := W_\alg \otimes \psi|_L$, which is an $\mathfrak l_C$-simple object of $\mathcal O^L$, and $\underline L(W_1) \cong \underline L(W_\alg) \otimes \psi|_L$ in $\mathcal O^P$ by Lemmas~\ref{lem:verma-tensor}, \ref{lm:loc-an-char} and shows that $\underline L(W_1) \cong \underline L(W_\alg) \otimes \psi\in \mathcal O^Q$.
  Then $W \cong W_1 \otimes \wt\tau|_L$ lies in $\mathcal O^Q$ by (i).
\end{proof}

\subsection{The socle of locally analytic parabolic induction}
\label{sec:socle-locally-analyt}

The following proposition generalizes \cite[Theorem 3.5]{orlik-schraen} and \cite[Proposition 2.4]{socle1}, which assumed $M \in \mathcal{O}_{\alg}^P$ and $\alggrp{G}$ split.

\begin{prop}\label{prop:N-invariants}
  Suppose that $P = LN$ and that $M \in \mathcal{O}^{P}$ is $\mathfrak{g}_{C}$-simple and equimaximal with maximal parabolic $P$.
  Let $\pi$ be an admissible smooth representation of $L$. Then
  \begin{equation}\label{eq:10}
    H^0(N,\mathcal{F}_P^G(M,\pi)') \cong M^N \otimes \pi'
  \end{equation}
  as representations of $L$ (on nuclear Fr\'echet spaces).
\end{prop}

\begin{proof}
  The proof follows the same lines as \cite[Theorem~3.5]{orlik-schraen} and \cite[Proposition~2.4]{socle1}.
  We let $\o P = L\o N$ denote the opposite parabolic subgroup and $\o{\mathfrak n}$ the Lie algebra of $\o N$.

    We fix some notation.
  Recall that we fixed a special parahoric subgroup $K\subset G$.
  Let $P_0 := P \cap K$.
      For $X$ a locally $F$-analytic manifold, let $D(X)$ be the locally convex space of locally $F$-analytic distributions on $X$ with coefficients in $C$.
  If $H$ is a locally $F$-analytic group, then $D(H)$ is a (separately continuous) locally convex algebra.
  For $h\in H$, we have $\delta_{h}\in D(H)$ defined by $f\mapsto f(h)$.
    If $H$ is a closed subgroup of $G$,   let $D(\mathfrak{g},H)$ be the subalgebra of $D(G)$ generated by $U(\mathfrak{g}_{C})$ and $D(H)$.
  Let $\Phi$ be the root system for $(\alggrp G,\alggrp S)$, $\Phi^{+}$ the set of roots in $\alggrp B$, $\Delta$ the set of simple roots and $W$ the Weyl group.
  
  First assume $\pi = 1$.
  Assume $r^{p^m} \in p^{\Q} \cap (p^{-1},p^{-1/\kappa(p-1)})$ for some $m \ge 0$ and $r$ sufficiently close to $1$
  (where $\kappa \in \{1,2\}$ is as in \S\ref{app:orlik-strauch}).
  Let $I \subset \Delta$ be the set of simple roots corresponding to $P$ and $W_{I}$ the subset of $W$ generated by reflections for elements in $I$.
  For each $w\in W$ we fix a representative $\dot{w}\in K$.
  Let $\mathcal{I}$ be an Iwahori subgroup fixing a facet of the apartment of $\alggrp S$ having vertex $x_0$.
  Then
  \begin{equation*}
    \mathcal{F}_P^G(M)' \cong \bigoplus_{w \in W^I} \delta_{\dot w} D(\dot{w}^{-1} \mathcal{I} \dot{w} \cdot P_0) \otimes_{D(\mathfrak{g},P_0)} M,
  \end{equation*}
  where $W^I$ denotes the Kostant representatives of $W/W_I$.
  Fix $w$ and let $\hat H := \dot{w}^{-1} \mathcal{I} \dot{w}$, $\mathcal{M} := D(\hat HP_0) \otimes_{D(\mathfrak{g},P_0)} M$ (coadmissible).
  We now use the notation and definitions of Appendix~\ref{app:orlik-strauch}, in particular defining $G_0$ and an open normal $L$-uniform subgroup $H \lhd G_0$ that has an Iwahori decomposition $H = H^- H^+$ with respect to $\o N \times P$. 
      We may assume that $H$ is contained in $\hat H$.
  The group $H$ is used to define norms $q_r$ (cf.\ \cite[2.2.6]{OS}) on $D(\widetilde{H})$ for any compact subgroup $\widetilde{H}$ of $G$ that contains $H$
  (and likewise $H^{\pm}$ is used to define norms on $D(\widetilde{H}^\pm)$ for any compact subgroup $\widetilde{H}^\pm$ of $\o N$, resp.\ $P$, that contains $H^\pm$).
  We also have an Iwahori decomposition $\hat H = \hat H^- \hat H^+$ with respect to $\o N \times P$.
  Let $D_{r}(H)$ be the completion of $D(H)$ with respect to $q_{r}$.
  Let $U_{r}(\mathfrak{g})$ denote the closure of $U(\mathfrak{g}_{C})$ in $D_r(H)$ (or equivalently in $D_r(\hat H)$) and $D_r(\mathfrak{g},P_0)$ the subring of $D_r(K)$ generated by $U_{r}(\mathfrak{g})$ and $D_r(P_0)$.
  Let $\mathcal{M}_{r} := D_r(\hat HP_0) \otimes_{D(\hat HP_0)} \mathcal{M} \cong D_r(\hat H) \otimes_{D(\mathfrak{g},\hat H \cap P_0)} M$ (so $\mathcal M \cong \varprojlim_r \mathcal M_r$ by coadmissibility) and $\mathfrak{m}_{r} := U_{r}(\mathfrak{g}) M \subset \mathcal{M}_{r}$.
  As in the proof of \cite[Theorem 4.5]{OS3} the module $\mathfrak{m}_{r}$ is $D_r(\mathfrak{g},P_0)$-stable, $D_r(\hat HP_0) = \bigoplus_{\hat HP_0/H^mP_0} \delta_g D_r(\mathfrak{g},P_0)$,
  $\mathfrak{m}_{r} \cong D_r(\mathfrak{g},P_0) \otimes_{D(\mathfrak{g},P_0)} M$, and $\mathcal{M}_{r} \cong \bigoplus_{\hat H^-/H^{-,m}} \delta_u \mathfrak{m}_{r}$,
  where $H^m$ (resp.\ $H^{-,m}$) is the $(m+1)$-st term in the lower $p$-series of $H$ (resp.\ $H^{-}$).
  By Lemma~\ref{lm:m_r} we get $\mathfrak{m}_{r} \cong U_{r}(\mathfrak{g})\otimes_{U(\mathfrak{g}_{C})} M$ and Lemma~\ref{lm:feaux} applies to $M \subset \mathfrak{m}_{r}$
  (by the beginning of the proof of Theorem~\ref{thm:OS-4.7}).

  If $w \in W^I \setminus \{1\}$, then there exists a reduced root $\beta \in \Phi^+ \setminus \Phi^+_I$ such that $w^{-1}\beta \in \Phi^- \setminus \Phi^-_I$.
  (For any $w \not\in W_I$ we can write $w = w_1 w' w_2$ with $w_i \in W_I$ and $w' \ne 1$ the Kostant representative for the double coset. Then there exists a reduced $\beta > 0$ such that $w'^{-1}\beta < 0$.) 
  If moreover $H^0(\mathfrak{n}_C, \delta_{\dot w}\mathcal{M}_{r}) \ne 0$, then $H^0(\mathfrak{n}_C, \delta_{\dot w}\delta_u \mathfrak{m}_{r}) \ne 0$ for some $u \in \hat H^-$,
  so $\Ad(u^{-1})y$ fails to act injectively on $\mathfrak{m}_{r}$ for any $y \in \mathfrak{g}_{(w^{-1}\beta),C} \subset \Ad(w^{-1})\mathfrak{n}_C \cap \o{\mathfrak{n}}_C$.
  Arguing as in the proof of Theorem~\ref{thm:OS-4.7}, by equimaximality, we deduce that $-w^{-1}\beta \in \Phi^+_I$, contradiction.

  Now suppose $w = 1$ and $H^0(\mathfrak{n}_C, \delta_u \mathfrak{m}_{r}) \ne 0$ for some $u \in \hat H^-$. 
  Note that $\mathfrak{m}_{r}$ is simple as $U_{r}(\mathfrak{g})$-module by Theorem~\ref{thm:OS-4.7}.
  As in Step 2 of the proof of \cite[Theorem 4.7]{OS3} we can embed $\mathfrak{m}_{r}$ into the formal completion $\widehat{M} = \prod_{\lambda \in \mathfrak{a}_P^*} M_\lambda$,
  where $\mathfrak{a}_P$ denotes the Lie algebra of the maximal split torus of the center of $\alggrp L$, where each $\lambda$-weight space $M_\lambda$ is finite-dimensional.
  The action of $\o{\mathfrak{n}}$ exponentiates to a locally analytic action of $\o N$ and we have $\o n \circ X \circ {\o n}^{-1} = \Ad(\o n)(X)$ on $\widehat M$
  for all $\o n \in \o N$, $X \in \mathfrak{g}_C$. 
  (Use, for example, that $\log(\Ad(\o n)) = \ad(\log(\o n)) \in \GL(\mathfrak{g}_C)$ for all $\o n \in \o N$.)
    By assumption, $0 \ne H^0(\Ad(u^{-1})\mathfrak{n}_C, \mathfrak{m}_{r}) = \mathfrak{m}_{r} \cap u^{-1} \widehat{M}^{\mathfrak{n}_C} = \mathfrak{m}_{r} \cap u^{-1} M^{\mathfrak{n}_C}$ inside $\widehat{M}$, so $M^{\mathfrak{n}_C} \into u\mathfrak{m}_{r}$ as $\mathfrak{p}_C$-modules,
  hence we get a surjective map $U_{r}(\mathfrak{g}) \otimes_{U(\mathfrak{p}_{C})} M^{\mathfrak{n}_C} \onto \delta_u \star\mathfrak{m}_{r}$ of $U_{r}(\mathfrak{g})$-modules.
  (Here, $\delta_u \star\mathfrak{m}_{r}$ denotes the space $\mathfrak{m}_{r}$ equipped with the action of $U_{r}(\mathfrak{g})$ twisted by $\delta_{u}$.)
  By Lemma~\ref{lm:feaux} the left-hand side has $\mathfrak{m}_{r}$ as unique simple quotient, so we get $\mathfrak{m}_{r} \congto \delta_u \star\mathfrak{m}_{r}$,
  hence $u \in H^{-,m}$ by Theorem~\ref{thm:OS-4.7}.

  Therefore, $H^0(\mathfrak{n}_C, \mathcal{M}_{r}) = H^0(\mathfrak{n}_C, \mathfrak{m}_{r}) = M^{\mathfrak{n}_C}$, giving an isomorphism
  $M^{\mathfrak{n}_C} \cong H^0(\mathfrak{n}_C, \mathcal{F}_P^G(M)')$. More precisely this is the image of the map $i : M^N = M^{\mathfrak{n}_C} \into M \to D(G) \otimes_{D(\mathfrak{g},P)} M \cong \mathcal{F}_P^G(M)'$
  sending $x$ to $\delta_1 \otimes x$. Using a choice of locally analytic section $s : G/P \to G$ of the projection $G \onto G/P$
  we obtain $\mathfrak{g}_C$-linear isomorphisms $\mathcal{F}_P^G(M,\pi) \cong \mathcal{F}_P^G(M) \widehat{\otimes} \pi$ and $\mathcal{F}_P^G(M,\pi)' \cong \mathcal{F}_P^G(M)' \widehat{\otimes} \pi'$.
  We may take $s$ such that $s(\overline{1}) = 1$.
  As in \cite[Theorem 3.5]{orlik-schraen} we find that $H^0(\mathfrak{n}_C, \mathcal{F}_P^G(M,\pi)')$ is the image of the map
  \begin{equation}\label{eq:9}
    i \otimes 1 \colon M^N \otimes \pi' \into \mathcal{F}_P^G(M)' \widehat{\otimes} \pi' \cong \mathcal{F}_P^G(M,\pi)'
  \end{equation}
  whose strong dual is computed to be $\mathcal{F}_P^G(M,\pi) \into (\Ind_P^G (M^N)' \otimes \pi)^{\an} \onto (M^N)' \otimes \pi$,
  where the second map is $f \mapsto f(1)$
  (using $s(\overline{1}) = 1$), which is $P$-linear (this is not clear a
  priori). Taking $N$-invariants in~\eqref{eq:9} we finally obtain the $L$-linear isomorphism~\eqref{eq:10}.
\end{proof}

The following corollary generalizes \cite[Corollary~2.5]{socle1}, which assumed $M \in \mathcal{O}_{\alg}^{P}$ and $\alggrp{G}$ split.

\begin{cor}\label{cor:socle-orlik-strauch}
  Suppose that $P = LN$ and that $M \in \mathcal{O}^{P}$ is $\mathfrak{g}_{C}$-simple and equimaximal with maximal parabolic $Q = L_Q N_Q$. 
      Let $\pi$ be an (admissible) smooth representation of $L$ of finite length.  Then
  \begin{align*}
    \soc_G \mathcal{F}_P^G(M,\pi) &= \mathcal{F}_Q^G(M, \soc_{L_Q} (\Ind_{P\cap L_Q}^{L_Q} \pi)^{\sm}) \\
    &= \soc_G ((\Ind_P^G (M^N)' \otimes \pi)^{\an}).
  \end{align*}
\end{cor}

\begin{remark}
We note that any finite-length smooth $C$-representation of $G$ is admissible.
For the proof, we can use a classical argument as follows.
We may assume that the representation is irreducible.
Let $\pi$ be an irreducible smooth representation.
Then there exists a parabolic subgroup $P = LU$ and an irreducible cuspidal representation $\sigma$ such that $\pi\hookrightarrow (\Ind_{P}^{G}\sigma)^\sm$~\cite[II.2.4]{MR1395151}.
Since $(\Ind_{P}^{G} -)^\sm$ preserves admissible representations, it is sufficient to prove that $\sigma$ is admissible.
Hence we may assume that $\pi$ is cuspidal.

Since $G$ is $\sigma$-compact, $\dim_C \pi$ is countable.
Hence $\dim_C \End_{G}(\pi)$ is countable.
Since $C$ is uncountable, the division algebra $\End_{G}(\pi)$ cannot contain a field of rational functions.
Therefore for any $z\in Z_{G}$, the image of $z$ in $\End_{G}(\pi)$ is algebraic over $C$.
Since $Z_{G}$ is topologically finitely generated, after tensoring with a finite extension of $C$ and taking an irreducible subquotient, we may assume that $\pi$ has a central character.
Here note that a representation $\pi$ is admissible (resp.\ finite length) if and only if $\pi\otimes_{C}C'$ is admissible (resp.\ finite length) for a finite extension $C'$.
By \cite[II.2.7]{MR1395151}, $\pi$ is $Z_{G}$-compact.
Let $K'$ be a compact open subgroup of $G$ and let $e_{K'}\colon V\to V^{K'}$ be the $K'$-equivariant projection.
Fix a nonzero vector $v\in \pi$ and set $D := \{g\in G\mid e_{K'}gv\ne 0\}$.
Then $D/Z_{G}$ is compact by \cite[I.7.3]{MR1395151}.
Therefore $K'\backslash D/Z_{G}$ is finite.
Since $\pi$ is irreducible, $\pi = \sum_{g\in G}Cgv$.
Hence $\pi^{K'} = \sum_{g\in G}Ce_{K'}gv = \sum_{g\in K'\backslash D/Z_{G}}\sum_{z\in Z_{G}}Ce_{K'}gzv$.
The representation $\pi$ has a central character, hence $Ce_{K'}gzv = Ce_{K'}gv$.
Therefore $\pi^{K'} = \sum_{g\in K'\backslash D/Z_{G}}Ce_{K'}gv$ is finite-dimensional.
\end{remark}

\begin{proof}[Proof of Corollary~\ref{cor:socle-orlik-strauch}]
  The proof proceeds as in \cite{socle1}, using Proposition \ref{prop:N-invariants} instead of \cite[Proposition~2.4]{socle1}.
  We first make a reduction to the case where $\alggrp{G}^{\der}$ is simply connected.
  We take a $z$-extension $1 \to \alggrp T \to \wt {\alggrp G} \to \alggrp G \to 1$, where ${\wt{\alggrp{G}}}^{\lowprime[\der]}$ is simply connected and $\alggrp T$ is a central induced torus, so $1 \to T \to \wt G \to G \to 1$.
  By pullback to $\wt G$ we obtain $\wt P = \wt L \wt N$ and $\wt Q = \wt L_Q \wt N_Q$, and by inflation we obtain $\wt M \in \mathcal O^{\wt P}$ and $\wt\pi$ finite-length smooth of $\wt L$.
  Note that $\wt M$ is $\wt{\mathfrak{g}}_{C}$-simple and equimaximal with maximal parabolic $\wt Q$.
  Moreover, the construction of Orlik--Strauch is compatible with pullback, i.e.\ the inflation of $\mathcal F_P^G(M,\pi)$ to $\wt G$ is naturally isomorphic to $\mathcal F_{\wt P}^{\wt G}(\wt M,\wt \pi)$ and likewise for $\mathcal F_P^G(\underline M(M^N),\pi) = (\Ind_P^G (M^N)' \otimes \pi)^{\an}$.
  This completes the reduction.

  It is clear that the functor $\mathcal{F}_P^G$ commutes with finite extensions of scalars.
  Also, the functor $\soc_G$ (resp.\ $\soc_{L_Q}$) commutes with finite extensions of scalars on the category of finite-length admissible locally analytic representations of $G$ by Lemma~\ref{lem:soc,field extension}.
             Therefore it is enough to prove the result after a finite extension of scalars.
  So by Proposition~\ref{prop:twist-to-O-P} we may assume that each simple constituent of $\underline M(M^N)$ in $\mathcal{O}^{P}$ is equimaximal
  up to twist by a smooth character of $L$. (By Lemma~\ref{lem:subobject-verma} all such constituents are $\mathfrak{g}_{C}$-simple.)

  By the equimaximality assumption the irreducible constituents of $\mathcal{F}_P^G(M,\pi)$ are precisely
  the $\mathcal{F}_Q^G (M,\pi_Q)$, where $\pi_Q$ is an irreducible constituent of $(\Ind_{P\cap L_{Q}}^{L_{Q}}\pi)^{\sm}$.
  If $\mathcal{F}_Q^G(M,\pi_Q)$ injects into $\mathcal{F}_P^G(M,\pi) \cong \mathcal{F}_Q^G(M,I)$ where $I := (\Ind_{P\cap L_Q}^{L_Q} \pi)^{\sm}$ and
  $\pi_Q$ an irreducible constituent of $I$, then from Proposition \ref{prop:N-invariants} we get an $L_Q$-linear map
  $M^{N_Q} \otimes I' \to M^{N_Q} \otimes \pi_Q'$ which has to be nonzero by \cite[Lemme 2.2]{socle1} (whose proof remains unchanged
    if $W$ is any finite-dimensional locally analytic representation of $P$).  As $M^{N_Q}$ is
  $\mathfrak{l}_Q$-simple we can dualize and take $\Hom_{\mathfrak{l}_Q}((M^{N_Q})',-)$ to obtain an injection $\pi_Q \into I$ of smooth $L_Q$-representations.
  In fact, this gives an injection (hence isomorphism) $\Hom_G(\mathcal{F}_Q^G(M,\pi_Q),\mathcal{F}_P^G(M,\pi)) \into \Hom_{L_Q}(\pi_Q,I)$, which justifies the first equality.

  Suppose that $\sigma$ is an irreducible subrepresentation of $(\Ind_P^G (M^N)' \otimes \pi)^{\an} = \mathcal{F}_P^G(\underline M(M^N),\pi)$. 
  Then $\sigma$ is a constituent of $\mathcal{F}_P^G(\widetilde{M},\pi)$ for some simple constituent $\widetilde{M}$ of $\underline M(M^N)$ in $\mathcal{O}^{P}$.
  By the beginning of the proof there exists a smooth character $\eta$ of $L$ such that $\widetilde{M}\eta$ is equimaximal,
  and we denote its maximal parabolic by $\widetilde{Q} = L_{\widetilde{Q}} N_{\widetilde{Q}}$.
  We deduce that $\sigma \cong \mathcal{F}_{\widetilde{Q}}^G(\widetilde{M} \eta,\pi_{\widetilde{Q}})$ for some irreducible smooth representation $\pi_{\widetilde{Q}}$ of $L_{\widetilde{Q}}$.
  Then as in \cite{socle1} we obtain from Proposition \ref{prop:N-invariants} a nonzero $L$-equivariant map
  \begin{align*}
    M^N \otimes (\pi')^N \xrightarrow{\eqref{eq:9}} H^0(N, \mathcal{F}_{\widetilde{Q}}^G (\widetilde{M} \eta,\pi_{\widetilde{Q}})') & = H^0(N \cap L_{\widetilde{Q}}, (\widetilde{M} \eta)^{N_{\widetilde{Q}}} \otimes \pi_{\widetilde{Q}}')\\
    &= (\widetilde{M}\eta)^N \otimes (\pi_{\widetilde{Q}}')^{N \cap L_{\widetilde{Q}}}
    = \widetilde{M}^N \otimes \eta(\pi_{\widetilde{Q}}')^{N \cap L_{\widetilde{Q}}}.
  \end{align*}
  In particular, $M^N \cong \widetilde{M}^N$ as $\mathfrak{l}_{C}$-modules, but this implies $\widetilde{M} \cong M$ in $\mathcal{O}^{P}$ (only $M^N$ contains the highest weight for the action of $\mathfrak{t}'$).
  Therefore, $\sigma$ has to lie in the image of the map $\mathcal{F}_P^G(M,\pi) \into \mathcal{F}_P^G(\underline M(M^N),\pi)$.
\end{proof}

The following corollary generalizes \cite[Corollary 2.7]{socle1}.

\begin{cor}\label{cor:F-P-G-intertwiner}
  Suppose $M_1, M_2 \in \mathcal{O}^{B}$ are $\mathfrak{g}_{C}$-simple and equimaximal. Let $P_i = L_i N_i$ denote the maximal parabolic for $M_i$ and suppose that
  $\pi_i$ is a smooth representation of $L_i$ of finite length. Then we have $\mathcal{F}_{P_1}^G(M_1,\pi_1) \cong \mathcal{F}_{P_2}^G(M_2,\pi_2)$
  if and only if $P_1 = P_2$ and there is a smooth character $\eta$ of $L_1 = L_2$ such that $M_1 \otimes \eta \cong M_2$ and $\pi_1 \otimes \eta \cong \pi_2$.
\end{cor}

\begin{proof}
  By Proposition~\ref{prop:N-invariants} we have
  \begin{equation*}
    H^0(U,\mathcal{F}_{P_i}^G(M_i,\pi_i)) = H^0(U \cap L_i, M_i^{N_i} \otimes \pi_i') \cong M_i^U \otimes (\pi_i')^{U \cap L_i}.
  \end{equation*}
  We deduce that $M_1^U|_{\mathfrak{z}_C} \cong M_2^U|_{\mathfrak{z}_C}$, hence $M_1$, $M_2$ have the same highest weight in $(\mathfrak{t}')^*$.
  In particular, $P_1 = P_2$, and we henceforth denote it by $P = LN$.
  The isomorphism of $L$-representations $M_1^{N} \otimes \pi_1'\cong M_2^{N} \otimes \pi_2'$ shows that
  $M_1^N|_{\mathfrak{l}_{C}} \cong M_2^N|_{\mathfrak{l}_{C}}$, so as in the proof of Lemma~\ref{lem:semisimple-lift}(ii) we deduce that $M_1 \otimes \eta \cong M_2$ 
  for some smooth character $\eta$ of $L$.
  Applying $\Hom_{\mathfrak{l}_{C}}((M_1^{N})',-)$ to the dual isomorphism $(M_1^{N})' \otimes \pi_1\cong (M_1^{N} \otimes \eta)' \otimes \pi_2$ we get $\pi_1 \cong \pi_2 \otimes \eta^{-1}$.
  The converse is clear.
\end{proof}

We also have a weak version when we drop the equimaximality condition.

\begin{cor}\label{cor:F-P-G-intertwiner2}
  Suppose $M_i \in \mathcal{O}^{P_i}$ is $\mathfrak{g}_{C}$-simple and $\pi_i$ is a smooth representation of $L_i$ of finite length ($i = 1,2$).
  Let $Q_i \supset P_i$ be maximal such that $M_i \in \mathcal{O}^{\mathfrak q_i}$.
  Suppose that $\mathcal{F}_{P_1}^G(M_1,\pi_1)$, $\mathcal{F}_{P_2}^G(M_2,\pi_2)$ share at least one irreducible constituent.
  Then $Q_1 = Q_2$ and there exists a smooth character $\eta \colon L_1 \cap L_2 \to C^\times$ such that $M_1 \otimes \eta \cong M_2$ in $\mathcal O^{P_1 \cap P_2}$.
  In particular, $M_1 \cong M_2$ in $\mathcal O^{\mathfrak b}$.
\end{cor}

\begin{proof}
  We take a $z$-extension $1 \to \alggrp T \to \wt {\alggrp G} \to \alggrp G \to 1$, as in the proof of Corollary~\ref{cor:socle-orlik-strauch}.
  We keep the notation of that proof.
  By Proposition~\ref{prop:twist-to-O-P} for each $i$ there exists a smooth character $\wt\eta_i$ of $\wt L_i$ such that $\wt M_i \wt\eta_i$ is equimaximal, i.e.\ lies in $\mathcal O^{\wt Q_i}$.
  The inflation of $\mathcal{F}_{P_i}^G(M_i,\pi_i)$ to $\wt G$ becomes
  \begin{equation*}
    \mathcal{F}_{\wt P_i}^{\wt G}(\wt M_i,\wt\pi_i) \cong \mathcal{F}_{\wt Q_i}^{\wt G}(\wt M_i\wt \eta_i,(\Ind_{\wt P_i \cap \wt L_{Q_i}}^{\wt L_{Q_i}} \wt\pi_i\wt\eta_i)^\sm).
  \end{equation*}
  By assumption and Theorem~\ref{thm:OS-main-in-main} we deduce from Corollary~\ref{cor:F-P-G-intertwiner} that $\wt Q_1 = \wt Q_2$ and $\wt M_1 \wt\eta_1 \wt\eta_3 \cong \wt M_2 \wt\eta_2$ in $\mathcal O^{\wt Q_1} = \mathcal O^{\wt Q_2}$ for some smooth character $\wt\eta_3$ of $\wt L_{Q_1} = \wt L_{Q_2}$.
  In particular, $\wt\eta := \wt\eta_1 \wt\eta_3 \wt\eta_2^{-1}$ is trivial on $T$, so descends to a smooth character $\eta$ of $L_1 \cap L_2$.
  It follows that $M_1 \otimes \eta \cong M_2$ in $\mathcal O^{P_1 \cap P_2}$.
\end{proof}

Recall that if $W$ is a finite-dimensional locally analytic
$P$-representation on which $\mathfrak{t}'$ acts diagonalizably and $\pi$ an admissible smooth representation of $L$, then we have a pairing
\begin{equation*}
  \ang{\cdot,\cdot}_{C^{\an}(G,\pi)} \colon (U(\mathfrak{g}_{C}) \otimes_{U(\mathfrak{p}_{C})} W) \times (\Ind_P^G W' \otimes \pi)^{\an} \to C^{\an}(G,\pi),
\end{equation*}
which is used to define $\mathcal{F}_P^G(M,\pi)$ for $M \in \mathcal{O}^{P}$ \cite[\S 3.8]{OS3} (recalled in Appendix~\ref{app:orlik-strauch}).
Recall that $\mathcal{F}_P^G(U(\mathfrak{g}_{C}) \otimes_{U(\mathfrak{p}_{C})} W,\pi) = (\Ind_P^G W' \otimes \pi)^{\an}$.

The following lemma generalizes \cite[Lemma 3.1]{socle1}. 

\begin{lem}\label{lm:breuil-3-1}
  Suppose that $P = LN$, $M \in \mathcal{O}^{P}$, and $\pi$ an admissible smooth $L$-representation.
  Suppose we are given morphisms $W_1 \to M^N \leftarrow W_2$ in $\mathcal{O}^{L}$ so that we have corresponding
  diagrams in $\mathcal{O}^{P}$, respectively locally analytic representations of $G$:
  \begin{equation*}
  \xymatrix@R-15pt{
      U(\mathfrak{g}_{C}) \otimes_{U(\mathfrak{p}_{C})} W_1 \ar^(0.6){\phi_1}[rd] && (\Ind_P^G W_1' \otimes \pi)^{\an}  \\
      & M && \mathcal{F}_P^G(M,\pi) \ar[lu] \ar[ld] \\
      U(\mathfrak{g}_{C}) \otimes_{U(\mathfrak{p}_{C})} W_2 \ar_(0.6){\phi_2}[ru] && (\Ind_P^G W_2' \otimes \pi)^{\an} \\
    }
  \end{equation*}  
  Let $f \in \mathcal{F}_P^G(M,\pi)$ with images $h_i \in (\Ind_P^G W_i' \otimes \pi)^{\an}$ and $x_i \in U(\mathfrak{g}_{C}) \otimes_{U(\mathfrak{p}_{C})} W_i$
  such that $\phi_1(x_1) = \phi_2(x_2)$ in $M$. Then $\ang{x_1,h_1}_{C^{\an}(G,\pi)} = \ang{x_2,h_2}_{C^{\an}(G,\pi)}$ in $C^{\an}(G,\pi)$.
\end{lem}

\begin{proof}
  By considering $W_1\oplus W_2 \to M^N$ we may reduce to the case where there exists a map $\theta \colon W_1 \to W_2$ inducing
  commutative diagrams
  \begin{equation*}
  \xymatrix@R-15pt{
      W_1 \ar[dd]_\theta\ar[rd] && U(\mathfrak{g}_{C}) \otimes_{U(\mathfrak{p}_{C})} W_1 \ar[dd]_{1\otimes\theta}\ar[rd] && (\Ind_P^G W_1' \otimes \pi)^{\an} \\
      &M^N && M && \mathcal{F}_P^G(M,\pi) \ar[lu] \ar[ld]\\
      W_2 \ar[ru] && U(\mathfrak{g}_{C}) \otimes_{U(\mathfrak{p}_{C})} W_2 \ar[ru] && (\Ind_P^G W_2' \otimes \pi)^{\an} \ar[uu]_{\theta^*}\\
    }
  \end{equation*}
  Next note that we have
  \begin{equation*}
    \mathcal{F}_P^G(M,\pi) \onto \mathcal{F}_P^G(\im(\phi_i),\pi) = (\Ind_P^G W_i' \otimes \pi)^{\ker(\phi_i)}
  \end{equation*}
  by the definition of $\mathcal{F}_P^G$, i.e.\ $h_i$ is killed by $\ker(\phi_i)$ with respect to the pairing $\ang{\cdot,\cdot}_{C^{\an}(G,\pi)}$.
  By assumption, $x_2-(1\otimes\theta)(x_1) \in \ker(\phi_{2})$, so we may replace $x_2$ by $(1\otimes\theta)(x_1)$. But then
  the lemma comes down to the claim that $\ang{x_1,\theta^*(h_2)}_{C^{\an}(G,\pi)} = \ang{(1\otimes\theta)(x_1),h_2}_{C^{\an}(G,\pi)}$,
  which is obvious from the definitions.
\end{proof}

The following proposition generalizes \cite[Proposition 3.2]{socle1}. 

\begin{prop}\label{prop:loc-const-fns}
  Suppose that $P = LN$, $M \in \mathcal{O}^{P}$ $\mathfrak{g}_{C}$-simple, and $\pi$ an (admissible) smooth $L$-representation of finite length.
  Let $f \in (\Ind_P^G (M^N)' \otimes \pi)^{\an}$ such that the restriction $f|_{\overline{N}}$ is locally constant.
  Then $f \in (\Ind_P^G (M^N)' \otimes \pi)^{\ker(\phi)} = \mathcal{F}_P^G(M,\pi)$, where $\phi$ is the natural map $U(\mathfrak{g}_{C}) \otimes_{U(\mathfrak{p}_{C})} M^N \onto M$.
\end{prop}

\begin{proof}
  The proof proceeds exactly as in \cite{socle1}, noting that by Lemma~\ref{lm:N-invts} we know that $W := M^N = M^{\mathfrak{n}_C}$ is in $\mathcal{O}^{L}$ and $\mathfrak{l}_{C}$-simple.
  We may assume $f\ne 0$.
  Note that $\Delta_{f} = \langle\cdot,f\rangle_{C^{\an}(G,\pi)} \colon U(\mathfrak{g}_{C}) \otimes_{U(\mathfrak{p}_{C})} W \to C^{\an}(G,\pi)$ is also $P$-linear, by letting $P$ act on $C^{\an}(G,\pi)$ by $(p'f)(g) := p'f(p'^{-1}g)$.
  Now $\ker(\Delta_f)$ is $(\mathfrak{g}_C,P)$-stable, and hence a subobject in $\mathcal{O}^{P}$, so also $\im(\Delta_f)$ lies in $\mathcal{O}^{P}$.
  As $f \in \mathcal{F}_P^G(\im(\Delta_f),\pi)$ by construction, it suffices to show that $M' := \ker(\im(\Delta_f) \onto M) \in \mathcal{O}^{P}$ is zero.
    If not, then $V := M'^N = M'^{\mathfrak{n}_C}$ is nonzero in $\mathcal{O}^{L}$ (it need not be simple, but it does not matter for the argument).
      Then let $\psi \colon U(\mathfrak{g}_{C}) \otimes_{U(\mathfrak{p}_{C})} V \to \im(\Delta_f)$ be the induced map in $\mathcal{O}^{P}$.
  Following the proof \cite[Proposition 3.2]{socle1}, we use Lemma~\ref{lm:breuil-3-1} instead of \cite[Lemma 3.1]{socle1}.
  For the argument with weights near the end, we may use weights of $\mathfrak{t}'$.
\end{proof}

\subsection{Intertwiners}
\label{sec:intertwiners}

Suppose that $\alggrp{P} = \alggrp{L}\alggrp{N}$ be a parabolic subgroup and $\sigma$, $\tau$ are continuous representations of $L$ on Banach spaces.

\begin{prop}\label{prop:intertwiners}
  The natural map $\Hom_{L}^\cts(\sigma,\tau) \to \Hom_G^\cts((\Ind_{P}^G\sigma)^{\cts},(\Ind_{P}^G\tau)^{\cts})$ is an isomorphism.
\end{prop}

When $\alggrp{G}$ is split, $\alggrp{P}$ is minimal, and $\dim_C \sigma = \dim_C \tau = 1$, then it follows from the main theorem of \cite{MR4346019}.

\begin{proof}
  By Frobenius reciprocity \[\Hom_G^\cts((\Ind_{P}^G\sigma)^{\cts},(\Ind_{P}^G\tau)^{\cts}) \cong \Hom^\cts_{P}((\Ind_{P}^G\sigma)^{\cts},\tau).\]
  Fix a nonzero element $\mu \in \Hom^\cts_{P}((\Ind_{P}^G\sigma)^{\cts},\tau)$.

  Recall that we fixed a special point $x_0$ in the apartment of $\alggrp S$ corresponding to $K$.
  Let $I$ be an Iwahori subgroup fixing a facet of the apartment of $\alggrp S$ having vertex $x_0$.
  Then $G = \coprod_{w \in W_{L}\backslash W} PwI$ where $W_{L}$ is the Weyl group $N_{\alggrp L}(\alggrp S)/\alggrp Z$.
  Then $(\Ind_{P}^G\sigma)^{\cts} \cong \bigoplus_{w} (\Ind_{P}^{P w I}\sigma)^{\cts}$
  and $(\Ind_{P}^{P w I}\sigma)^{\cts} \cong \mathcal C^0(I \cap w^{-1}\o Nw, \sigma)$.
  We first claim that the restriction of $\mu$ to $(\Ind_{P}^{P w I}\sigma)^{\cts}$ is zero for all $w \in W\setminus W_{L}$.
    If $w \notin W_{L}$ there exists a reduced root $\alpha > 0$ that appears in $\Lie(\alggrp{N})$ such that $w\alpha$ appears in $\Lie(\overline{\alggrp{N}})$, so $U_\alpha \subset N \cap w^{-1}\overline{N}w$.
  (The argument is exactly as in the proof of Proposition \ref{prop:N-invariants}.)
      Note that $I \cap w^{-1}\overline{N}w = \prod_{\beta} (I \cap U_{w^{-1}\beta})$ in any fixed order, with $\beta$ running through the roots of $\o{\alggrp N}$, and we list $\beta = w\alpha$ first.
  Then Lemma~\ref{lem:no-haar} shows that any continuous $I \cap U_\alpha$-linear map $\mathcal C^0(I \cap w^{-1}\overline{N}w,\sigma)\to\tau$ vanishes.
  This implies the claim, as $I \cap U_\alpha \subset I \cap N \cap w^{-1}\overline{N}w$.

  By the claim, the given map $\mu$ factors as
  \begin{equation*}
    (\Ind_{P}^G\sigma)^{\cts} \onto (\Ind_{P}^{P I}\sigma)^{\cts} \cong \mathcal C^0(I \cap \overline{N}, \sigma) \to \tau,
  \end{equation*}
  where the first map is given by restriction.
  By the action of $L$, $\mu$ also factors through $(\Ind_{P}^{G}\sigma)^{\cts}\to (\Ind_{P}^{PI\ell}\sigma)^{\cts}\simeq \mathcal{C}^{0}(\ell^{-1}I\ell\cap \overline{N},\sigma)$ for any $\ell\in L$.
  In other words, $\mu$ factors through $\mathcal C^0(\overline{N}_0, \sigma)$ for any compact open subgroup $\overline{N}_0$ of $\overline{N}$.
  Fix now some compact open subgroup $\overline{N}_0$ of $\overline{N}$ and $v \in \sigma$.
  Write $1_{\overline{N}_0,v} \in (\Ind_B^{B \overline{N}_0}\sigma)^{\cts}$ for the element taking constant value $v$ on $\overline{N}_0$.
  Then $\mu(1_{\ell \overline{N}_0 \ell^{-1},\ell v}) = \mu(\ell 1_{\overline{N}_0,v}) = \ell\mu(1_{\overline{N}_0,v})$ for $\ell\in L$ by $L$-linearity.
  But we also have $\mu(1_{\overline{N}_0,\ell v}) = \mu(1_{\ell \overline{N}_0 \ell^{-1},\ell v})$, as $\mu$ only depends on the restriction to a small neighborhood of $1 \in \overline{N}$.
  Defining $\phi_\mu : \sigma \to \tau$ by $\phi_\mu(v) := \mu(1_{\overline{N}_0,v})$, we deduce that $f$ is $L$-linear.
  It is clearly continuous.
  Moreover, $\mu(f) = \phi_\mu(f(1))$ for all $f \in \mathcal C^\infty(\overline{N}_0, \sigma)$ and hence by continuity for all $f \in \mathcal C^0(\overline{N}_0, \sigma)$.
  It is clear that $\mu \mapsto \phi_\mu$ is inverse to the given map.
\end{proof}

\begin{lem}\label{lem:no-haar}
  Suppose that $H$ is a compact locally analytic group with closed subgroups $H_1$, $H_2$ such that multiplication induces a topological isomorphism $H_1 \times H_2 \congto H$.
  Suppose there exist a basis of open neighborhoods of 1 consisting of subgroups of the form $H_1'H_2'$ with $H_i'$ an open subgroup of $H_i$.
  Suppose that $H_1$ is infinite and that $V_1$, $V_2$ are Banach spaces.
  Then any left $H_1$-invariant continuous map $\mu\colon \mathcal C^0(H,V_1) \to V_2$ vanishes.
\end{lem}

\begin{proof}
  Suppose that $\mu$ is nonzero.
  By density of smooth functions and by assumption, $\mu$ has to be nonzero on a function of the form $1_{Uh,v}$ for some open subgroup $U$, some $h \in H$, and $v \in V_1$.
  Without loss of generality, $U = H_1'H_2'$ with $H_i'$ an open subgroup of $H_i$.
  Then for any open subgroup $H_1''$ of $H_1'$ we have $\mu(1_{H_1''H_2'g,v}) = (H_1':H_1'')^{-1}\mu(1_{H_1'H_2'g,v})$ by left $H_1$-invariance.
  By assumption on $H_1$ we deduce that the $p$-adic absolute value of $(H_1':H_1'')^{-1}$ is unbounded.
  (Note that $H_1'$ is compact locally analytic, hence profinite and contains an open normal pro-$p$ subgroup, which is infinite by assumption.)
  This contradicts the continuity of $\mu$, as the set $\{1_{H_1''H_2'g,v}\}$ is bounded.
\end{proof}

\begin{cor}\label{cor:indecomposable}
If $\sigma$ is indecomposable then $(\Ind_{P}^{G}\sigma)^{\cts}$ is indecomposable.
\end{cor}

Here, we say that a Banach representation $\pi$ is \emph{indecomposable} if it cannot be written as a direct sum of two closed subrepresentations.
(Equivalently, the ring $\End_G^\cts(\pi)$ does not contain any non-trivial idempotents.)

\begin{lem}\label{lem:roche}
  Suppose that $G$ is a locally analytic group with open normal subgroup $N$ such that $G/N$ is finite abelian.
  If $V$ is an absolutely irreducible admissible Banach representation of $G$ and $C$ is sufficiently large, 
  then there exists an absolutely irreducible closed subrepresentation $W$ of $V|_N$ such that, if we let $H := \{g\in G\mid gW = W\}$ (stabilizer of the subspace), we have
   $W \circ \Ad(g) \cong W$ as $H$-representations for $g \in G$ implies $g \in H$.
  In particular, the natural map $\Ind_H^G W \to V$ is an isomorphism.
\end{lem}

\begin{proof}
  (See \cite[1.6.3]{MR2508719} in the context of smooth representations.)
  By Clifford theory \cite[Proposition~2.1.1]{arxiv.1912.11125}, $V|_N$ is a direct sum of finitely many irreducible subrepresentations.
  So as $C$ is sufficiently large, using Schur's lemma and Lemma~\ref{lem:soc,field extension}, we may assume that any irreducible closed subrepresentation of $V|_N$ is absolutely irreducible.
  Choose now $W$ among all irreducible closed subrepresentations of $V|_N$ such that $H := \{g\in G\mid gW = W\}$ is maximal.
        Note that $H$ is contained in $\wt H := \{g\in G\mid \text{$W\circ\Ad(g)\simeq W$ as $H$-representations}\}$, and if $H = \wt H$, then we are done.
  If not, pick $H \subsetneq H_1 \subset \wt H$ such that $H_1/H$ is cyclic.   As $C$ is sufficiently large, $W$ extends to an irreducible representation $\wt W$ of $H_1$.
  From $\Hom_H(W,V|_H) \cong \Hom_{H_1}(\Ind_H^{H_1} W,V|_{H_1})$ and $\Ind_H^{H_1} W \cong \bigoplus_{\eta:H_1/H \to C^\times} \wt W \otimes \eta$
  we deduce that some extension $\wt W \otimes \eta$ of $W$ occurs as closed subrepresentation of $V|_{H_1}$.
  (Note that the image of any morphism of admissible Banach representations is closed.)
  This contradicts the maximality of $H$.
\end{proof}

\begin{prop}\label{prop:isogenies}
Let $\varphi\colon \alggrp G_{1}\to \alggrp G$ be a morphism such that $\varphi(\alggrp{G}_{1}^{\der}) = \alggrp{G}^{\der}$ and $\ker\varphi\subset \alggrp{Z}_{\alggrp G_{1}}$.
Suppose that $\alggrp P = \alggrp L \alggrp N$ is a parabolic subgroup of $\alggrp G$ and let $\alggrp{P}_{1} := \varphi^{-1}(\alggrp{P})$, $\alggrp{L}_{1} := \varphi^{-1}(\alggrp{L})$, $\alggrp{N}_{1} := \varphi^{-1}(\alggrp{N})$.
Suppose that $\sigma$ is an admissible Banach representation of $L$, and
write $\sigma_{1}$ for the composition of $\sigma$ with $L_{1}\to L$.
Then we have the following.
\begin{enumerate}
\item If $\sigma$ is irreducible or it has a central character, then $\sigma_1$ is admissible.
\item If $(\Ind_{P_{1}}^{G_{1}}\sigma_{1})^{\cts}$ is irreducible, then $(\Ind_{P}^{G}\sigma)^{\cts}$ and $\sigma_{1}$ are irreducible.
\item If $(\Ind_{P}^{G}\sigma)^{\cts}$ is irreducible, then $(\Ind_{P_{1}}^{G_{1}}\tau_{1})^{\cts}$ is irreducible for any irreducible closed subrepresentation $\tau_{1}$ of $\sigma_{1}$.
\item Assume that $\sigma$ is absolutely irreducible.
If $(\Ind_{P_{1}}^{G_{1}}\tau_{1})^{\cts}$ is absolutely irreducible for one (equivalently any) absolutely irreducible closed subrepresentation $\tau_{1}$ of $\sigma_{1}$, then $(\Ind_{P}^{G}\sigma)^{\cts}$ is absolutely irreducible.
\end{enumerate}
\end{prop}

\begin{remark}\label{rk:isogenies}
In particular, if $\dim_C \sigma = 1$, then $(\Ind_{P}^{G}\sigma)^{\cts}$ is irreducible if and only if $(\Ind_{P_{1}}^{G_{1}}\sigma_{1})^{\cts}$ is irreducible.
A special case of this statement can be found in \cite[Proposition~2.3.2]{arxiv.1912.11125}.
\end{remark}

\begin{proof}
Set $G'_{1} := \varphi(G_{1})$, a closed locally analytic subgroup of $G$.
(It is closed by \cite[3.19 Proposition]{MR316587}.)
We observe that $Z_{G}G'_{1}$ is open and normal in $G$ with finite abelian quotient group.
(This follows from \cite[3.20 Corollary]{MR316587} applied to the surjective homomorphism $\alggrp{G}_1 \times \alggrp{Z}_{\alggrp{G}}^\circ \onto \alggrp{G}$.)

We justify that $\sigma_1$ is admissible.
We put $L'_{1} := \varphi(L_{1})$.
Note that $\varphi(\alggrp{L}_{1}^{\der}) = \alggrp{L}^{\der}$, so we have a surjective homomorphism $\alggrp{L}_{1} \times \alggrp Z_{\alggrp L}^\circ \onto \alggrp L$, which is therefore surjective on Lie algebras and hence induces an open homomorphism on $F$-points.
Pick any compact open subgroups $H$ of $L_1$ and $H_Z$ of $Z_L^\circ$.
Then $\varphi(H)H_Z$ is a compact open subgroup of $L$, so there exists a $\varphi(H)H_Z$-stable unit ball $\sigma^0$ in $\sigma$ and its reduction is admissible smooth as $\varphi(H)H_Z$-representation (as $\sigma$ is admissible).
If $\sigma$ has a central character, then, taking $H_Z$ small enough so that it acts trivially on $\sigma$, the reduction of $\sigma^0$ is admissible as $\varphi(H)$-representation, i.e.\ $\sigma_1$ admissible \cite[Proposition 6.5.7]{locallyanalytic-memoir}.
If $\sigma$ is irreducible, then the endomorphism algebra $\End_{L}^{\cts}(\sigma)$ is a finite-dimensional division algebra~\cite{MR3053464}.
Let $C'$ be the subalgebra of $\End_{L}^{\cts}(\sigma)$ generated by the image of $Z_{L}$, which is a finite extension of $C$.
Then $\sigma$ is admissible as Banach representation of $L$ over $C'$ (Lemma~\ref{lem:adm-ext-of-scalars}) and has a central character, so $\sigma_1$ is admissible over $C'$, so $\sigma_1$ is admissible over $C$ (Lemma~\ref{lem:adm-ext-of-scalars}).

The group $G_1$ acts on $(\Ind_{P_{1}}^{G_{1}}\sigma_{1})^{\cts}$ through $G_{1}\onto G'_{1}$ and as $G'_{1}$-representations we have $(\Ind_{P_{1}}^{G_{1}}\sigma_{1})^{\cts}\simeq (\Ind_{P}^{G}\sigma)^{\cts}|_{G'_{1}}$.
(To see this, we note that $\varphi$ induces an isomorphism $P_1\backslash G_1 \congto P\backslash G$ because it induces an isomorphism $\alggrp P_1\backslash \alggrp G_1 \congto \alggrp P\backslash \alggrp G$ and we obtain the former isomorphism by passing to $F$-points.)
Hence if $(\Ind_{P_{1}}^{G_{1}}\sigma_{1})^{\cts}$ is irreducible then $(\Ind_{P}^{G}\sigma)^{\cts}$ is also irreducible.
Moreover $\sigma_1$ is clearly irreducible.

Conversely, assume that $(\Ind_{P}^{G}\sigma)^{\cts}$ is irreducible.
Then in particular $\sigma$ is irreducible and again $\End_{L}^{\cts}(\sigma)$ is a finite-dimensional division algebra.
Let $C'$ be the subfield of $\End_{L}^{\cts}(\sigma)$ generated by the image of $Z_{G}$.
Considering $\sigma$ as irreducible admissible representation of $L$ over $C'$, by Clifford theory \cite[Proposition~2.1.1]{arxiv.1912.11125} $(\Ind_{P}^{G}\sigma)^{\cts}|_{Z_{G}G'_{1}}$ is semisimple (as defined in \S\ref{sec:locally-analyt-repr}) over $C'$.
As $\sigma$ has a central character over $C'$, $(\Ind_{P}^{G}\sigma)^{\cts}|_{G'_1}$ is semisimple over $C'$ and hence semisimple over $C$ (Lemma~\ref{lem:soc,field extension}).
In particular, by Lemma~\ref{lem:semisimple-banach} the direct summand $(\Ind_{P_{1}}^{G_{1}}\tau_{1})^{\cts}$ of $(\Ind_{P}^{G}\sigma)^{\cts}|_{G'_1}$ is semisimple over $C$.
But it is also indecomposable by Corollary~\ref{cor:indecomposable}, hence irreducible.

Finally assume that $\sigma$ is absolutely irreducible.
As we have seen before (for $\alggrp{G} = \alggrp{L}$), $Z_L L'_1$ is open and normal in $L$ with finite abelian quotient group.
By Lemma~\ref{lem:roche}, applied with $N = Z_{L}L'_{1}$ and extending scalars if necessary, we can take an absolutely irreducible closed subrepresentation $\tau$ of $\sigma|_{Z_{L}L'_{1}}$ such that $H := \{g\in L\mid g\tau = \tau\}$ (stabilizer of the subspace) is equal to $\{g\in L\mid \text{$\tau\circ\Ad(g)\simeq \tau$ as $H$-representations}\}$, so $\sigma \simeq \Ind_{H}^{L}\tau$.
Since $\sigma$ has a central character, $\tau$ is also absolutely irreducible as $L'_{1}$-representation.
Let $\tau_{1}$ be the composition of $\tau$ with $L_1 \onto L_1' \subset H$, which is again absolutely irreducible.
Then the subspace $W := \{f\in (\Ind_{P}^{G}\sigma)^{\cts}\mid f(Z_G G'_{1})\subset \tau\}$ is $G'_{1}$-stable and isomorphic to $(\Ind_{P_{1}}^{G_{1}}\tau_{1})^{\cts}$.
Note that if $\ell \in L$ and $f\in (\Ind_{P}^{G}\sigma)^{\cts}$ such that $f(Z_GG'_{1})\subset \tau$, then for $g\in Z_GG'_{1}$ we have $(\ell f)(g) = f(g\ell) = f(\ell (\ell^{-1}g\ell)) = \ell f(\ell^{-1}g\ell) \in \ell \tau_1$.
In particular, $W$ is also $H$-stable.
Moreover, together with $\sigma = \bigoplus_{\ell \in L/H} \ell \tau$ we deduce that $(\Ind_{P}^{G}\sigma)^{\cts} = \bigoplus_{\ell \in L/H} \ell W$.
As $H \backslash L \congto HG'_1 \backslash G$, it follows that the natural continuous homomorphism $\Ind_{HG'_{1}}^{G} W \to (\Ind_{P}^{G}\sigma)^{\cts}$ is an isomorphism.

We assume $(\Ind_{P_{1}}^{G_{1}}\tau_{1})^{\cts} (\cong W)$ is absolutely irreducible.
To prove $(\Ind_{P}^{G}\sigma)^{\cts}$ is absolutely irreducible, by the previous paragraph it is sufficient to prove that for $g\in G$, if $(\Ind_{P_{1}}^{G_{1}}\tau_{1})^{\cts}\circ\Ad(g)\simeq (\Ind_{P_{1}}^{G_{1}}\tau_{1})^{\cts}$ as $HG'_{1}$-representations, then $g\in HG'_{1}$.
As $G = LG'_{1}$, we may assume $g\in L$.
The map $f\mapsto (x\mapsto f(gxg^{-1}))$ gives an $HG_1'$-linear isomorphism $(\Ind_{P_{1}}^{G_{1}}\tau_{1})^{\cts}\circ \Ad(g)\simeq (\Ind_{P_{1}}^{G_{1}}\tau_{1}\circ\Ad(g))^{\cts}$.
By Proposition~\ref{prop:intertwiners} any $G_1$-linear isomorphism $\Psi \colon (\Ind_{P_{1}}^{G_{1}}\tau_{1}\circ\Ad(g))^{\cts} \congto (\Ind_{P_{1}}^{G_{1}}\tau_{1})^{\cts}$
arises from an $L_1$-linear isomorphism $\psi \colon \tau_{1}\circ\Ad(g)\congto \tau_{1}$.
From the relations $\Psi(f)(1) = \psi(f(1))$ and $(hf)(1) = h\cdot f(1)$ for $f \in (\Ind_{P_{1}}^{G_{1}}\tau_{1})^{\cts}$ and $h\in H$ we deduce that $\psi$ is even $H$-linear.
Therefore, $g \in H$.
\end{proof}

\begin{remark}\label{rk:reduce-simply-conn}
  In this way the problem of understanding when a parabolic induction $(\Ind_{P}^{G}\sigma)^{\cts}$ (with $\sigma$ admissible) is absolutely irreducible can be reduced to the case where $\alggrp G$ is a simply-connected group.
  If $\sigma$ is finite-dimensional we may reduce further, see Remark~\ref{rk:reduce-abs-simple-simply-conn} below.
\end{remark}

\subsection{Translation functors}
\label{sec:translation-functors}
In this subsection we recall translation functors and extend some properties to locally analytic representations.
This topic is also studied in \cite{arxiv.2107.08493} for split reductive groups, from the point of view of distribution algebras.
(In particular, in \cite[Theorem 4.2.12]{arxiv.2107.08493} they compute the effect of translation functors on locally analytic principal series.)
We do not refer to this paper and give proofs for the sake of completeness.

We first recall some facts about translation functors on $\mathfrak{g}_{C}$-modules from \cite{MR581584}.
Let $Z(\mathfrak{g}_{C})$ be the center of $U(\mathfrak{g}_{C})$.
We say that a $\mathfrak{g}_{C}$-module $M$ is \emph{$Z(\mathfrak{g}_{C})$-finite} if the ideal $\Ann_{Z(\mathfrak{g}_{C})}M\subset Z(\mathfrak{g}_{C})$ has finite codimension.
Let $\mathcal{M}_{\mathrm{Zf}}$ be the category of $Z(\mathfrak{g}_{C})$-finite $\mathfrak{g}_{C}$-modules.
By the Harish-Chandra isomorphism we have an embedding $Z(\mathfrak{g}_{C})\hookrightarrow \Sym(\mathfrak{t}')$ whose image is the set of vectors fixed by the absolute Weyl group.
For $\lambda\in (\mathfrak{t}')^{*}$, let $\chi_{\lambda}\colon Z(\mathfrak{g}_{C})\to C$ be the composition of the Harish-Chandra isomorphism and the evaluation at $\lambda$.
We say that a $\mathfrak{g}_C$-module $M$ \emph{admits generalized infinitesimal character $\lambda$} if $(\ker\chi_{\lambda})^{n}M = 0$ for some $n\in\Z_{\ge0}$ and let $\mathcal{M}_{\lambda}$ be the full subcategory of $\mathcal{M}_{\mathrm{Zf}}$ consisting of all $M$ which admit a generalized infinitesimal character $\lambda$.
(We note that our definitions are slightly more restrictive than the ones used in \cite{MR581584}.)

In this subsection, we say that $\lambda\in (\mathfrak{t}')^{*}$ is \emph{integral} if its value at any absolute coroot is an integer.
Note that this is weaker than $\lambda\in X^{*}(\alggrp T')$.
Let $\lambda,\mu\in (\mathfrak{t}')^{*}$ and assume that $\mu - \lambda$ is integral.
We define the translation functor $T_{\lambda}^{\mu}\colon \mathcal{M}_{\lambda}\to \mathcal{M}_{\mu}$ as follows.
Let $V$ be the absolutely simple finite-dimensional $\mathfrak{g}_{C}$-module having $\mu - \lambda$ as an extremal weight.
By a theorem of Kostant, we have $V\otimes \mathcal{M}_{\mathrm{Zf}}\subset \mathcal{M}_{\mathrm{Zf}}$~\cite[2.6 Corollary]{MR581584} (and its proof).
Then we define $T_{\lambda}^{\mu}(X) := \pr_{\mu}(V\otimes X)$, where $\pr_{\mu}\colon \mathcal{M}_{\mathrm{Zf}}\to \mathcal{M}_{\mu}$ is the projection.
The properties of translation functors are summarized as follows.
\begin{prop}\label{prop:translation functors(algebraic)}
Let $\lambda,\mu\in (\mathfrak{t}')^{*}$ such that $\mu - \lambda$ is integral.
\begin{enumerate}
\item The functor $T_{\lambda}^{\mu}$ is exact.
\item The pair $(T_{\lambda}^{\mu},T_{\mu}^{\lambda})$ is an adjoint pair.
\item Assume that for any absolute root $\alpha$ we have $\langle \lambda,\alpha^{\vee}\rangle \in\Z_{>0}$ if and only if $\langle\mu,\alpha^{\vee}\rangle \in\Z_{>0}$. Then $T_{\lambda}^{\mu}$ gives an equivalence of categories.
\end{enumerate}
\end{prop}
\begin{proof}
(i) and (ii) are easy and (iii) is \cite[4.1~Theorem]{MR581584}.
\end{proof}

We upgrade the above constructions to locally analytic representations.
Let $\mathcal{M}^{\an}_{\lambda} = \mathcal{M}_{\lambda}^{\an}(G)$ be the category of locally analytic representations $\pi$ of $G$ such that $\pi|_{\mathfrak{g}_{C}}\in \mathcal{M}_{\lambda}$.
Assume that there exists a finite-dimensional locally analytic representation $V$ of $G$ such that $V|_{\mathfrak{g}_{C}}$ is absolutely simple with extremal weight $\mu - \lambda$ and we fix such $V$.
(It exists by Lemma~\ref{lem:semisimple-lift}(ii) if $\alggrp{G}^{\der}$ is simply connected.)
Then for $X\in \mathcal{M}^{\an}_{\lambda}$, $V\otimes X$ is also a locally analytic representation of $G$.
Since $\pr_{\mu}(V\otimes X)$ is the kernel of certain $z\in Z(\mathfrak{g}_{C})$, it is a closed subspace.
Moreover it is $G$-stable since the action of $G$ commutes with that of $Z(\mathfrak{g}_C)$, as $Z(\mathfrak{g}_{C}) = U(\mathfrak{g}_{C})^{\ad(\mathfrak{g}_{C})} = U(\mathfrak{g}_{C})^{\Ad(G_{C})}$. (Note that the action on $U(\mathfrak{g}_{C})$ is locally finite-dimensional algebraic.)
Therefore we can define the functor $T_{\lambda}^{\mu}(G,V)\colon \mathcal{M}_{\lambda}^{\an}\to \mathcal{M}_{\mu}^{\an}$ by $T_{\lambda}^{\mu}(G,V)(X) := \pr_{\mu}(V\otimes X)$.

\begin{lem}\label{lem:V-tensor-decomposition}
  Suppose that $\pi$ is a locally analytic representation such that   $\pi|_{\mathfrak{g}_{C}}$ is $Z(\mathfrak{g}_{C})$-finite.
  For any $\mu\in (\mathfrak{t}')^{*}$ there exists a functorial decomposition $\pi \cong \pr_{\mu}(\pi) \oplus \pr_\mu'(\pi)$ as locally analytic representations of $G$.
\end{lem}

\begin{proof}
  By assumption there exists an ideal $I$ of $Z(\mathfrak{g}_{C})$ of finite codimension such that $I\pi = 0$.
  Then $Z(\mathfrak{g}_{C})/I$ is an artinian ring and we have a (finite) decomposition $\pi = \bigoplus_{\mathfrak m} \pi_{\mathfrak m}$, where $\mathfrak m$ runs through the maximal ideals of $Z(\mathfrak{g}_{C})$.
  We define $\pr_\mu'(\pi) := \bigoplus_{\mathfrak m \ne \ker \chi_\mu} \pi_{\mathfrak m}$.
  As above, each $\pi_{\mathfrak m}$ is a closed $G$-equivariant subspace of $\pi$ and the projection $\pi \onto \pi_{\mathfrak m}$ is continuous.
  We obtain continuous bijections $\pr_{\mu}(\pi) \oplus \pr_\mu'(\pi) \to \pi \to \pr_{\mu}(\pi) \times \pr_\mu'(\pi)$, which are topological isomorphisms since the composition is one.
\end{proof}

\begin{prop}\label{prop:translation functors(analytic)}
Let $\lambda,\mu\in (\mathfrak{t}')^{*}$ such that $\mu-\lambda$ is integral.
\begin{enumerate}
\item The functor $T_{\lambda}^{\mu}(G,V)$ preserves strict exact sequences.
\item The pair $(T_{\lambda}^{\mu}(G,V),T_{\mu}^{\lambda}(G,V'))$ is an adjoint pair.
\item Assume that for any absolute root $\alpha$ we have $\langle \lambda,\alpha^{\vee}\rangle \in\Z_{>0}$ if and only if $\langle\mu,\alpha^{\vee}\rangle \in\Z_{>0}$. Then $T_{\lambda}^{\mu}(G,V)\colon \mathcal{M}_{\lambda}^{\an}\to \mathcal{M}_{\mu}^{\an}$ gives an equivalence of categories.
\end{enumerate}
\end{prop}
\begin{proof}
For (i), by Proposition~\ref{prop:translation functors(algebraic)}(i) it suffices to show that $T_{\lambda}^{\mu}$ sends strict morphisms to strict morphisms.
In fact, the same is true for $X \mapsto V \otimes X$ and $X \mapsto \pr_{\mu}(X)$.
This follows from the basic fact that if $f_i \colon \pi_i \to \pi'_i$ ($i = 1,\dots,n$) in the category of locally convex spaces, then
$f_i$ is strict for all $i$ if and only if $\oplus f_i \colon \bigoplus_{i=1}^n \pi_i \to \bigoplus_{i=1}^n \pi'_i$ is strict.
For (ii), recall that units and counits of $(T_{\lambda}^{\mu}(G,V),T_{\mu}^{\lambda}(G,V'))$ on $\mathcal{M}_{\lambda},\mathcal{M}_{\mu}$ are induced by $C\to V'\otimes V$ and $V'\otimes V\to C$.
Both are $G$-equivariant linear morphisms, so the unit and the counit are morphisms of locally analytic representations.
Part (iii) follows from (ii) and Proposition~\ref{prop:translation functors(algebraic)} (iii).
\end{proof}

For the following result, recall that in \S\ref{sec:some-decomp} we fixed a system of positive roots for $(\mathfrak{g}_{C},\mathfrak{t}')$.
Let $\rho\in (\mathfrak{t}')^{*}$ (resp.\ $\rho_{L}\in (\mathfrak{t}')^{*}$) be the half sum of positive roots in $\mathfrak{g}_{C}$ (resp.\ in $\mathfrak{l}_{C}$).

\begin{prop}\label{prop:translation functors(parabolic induction)}
Let $\alggrp{P} = \alggrp{L}\alggrp{N}$ be a parabolic subgroup of $\alggrp{G}$.
Suppose $V \in \mathcal O^G$ is $\mathfrak{g}_{C}$-simple with lowest weight $\lambda$.
Let $V_N$ denote the $N$-coinvariants in $V$, which is a locally analytic representation of $L$ (in $\mathcal O^L$).
\begin{enumerate}
\item If $\sigma\in \mathcal{M}_{\mu}^{\an}(L)$ for some $\mu\in (\mathfrak{t}')^{*}$, then $(\Ind_{P}^{G}\sigma)^{\an}\in \mathcal{M}_{\mu - \rho + \rho_{L}}^{\an}(G)$.
\item For any smooth representation $\tau$ of $L$, we have
\[
T_{-\rho}^{\lambda - \rho}(G,V)((\Ind_{P}^{G}\tau)^{\an})\simeq (\Ind_{P}^{G}V_{N}\otimes \tau)^{\an}.
\]
\end{enumerate}
\end{prop}

Note that $\lambda$ is automatically integral.
Moreover, the proof shows that the isomorphism in (ii) is obtained by applying $\pr_{\lambda-\rho}$ to the natural surjection $(\Ind_{P}^{G}V \otimes \tau)^{\an}\onto (\Ind_{P}^{G}V_{N}\otimes \tau)^{\an}$.

\begin{proof}
Let $\gamma\colon Z(\mathfrak{g}_C)\into \Sym\mathfrak{t}'$ and $\gamma_{L}\colon Z(\mathfrak{l}_C)\into \Sym\mathfrak{t}'$ be the Harish-Chandra isomorphisms.
We also define $\gamma'\colon Z(\mathfrak{g}_{C})\to Z(\mathfrak{l}_{C})$ as follows.
By considering the adjoint action of the center of $\mathfrak{l}_{C}$, we have $Z(\mathfrak{g}_{C})\subset U(\mathfrak{l}_{C})\oplus \mathfrak{n}_{C}U(\mathfrak{g}_{C})\overline{\mathfrak{n}}_{C}$ and let $\gamma'$ be the first projection along this decomposition.
It is easy to see that $\gamma'$ is an $\mathfrak{l}_{C}$-bimodule homomorphism, hence $\gamma'(Z(\mathfrak{g}_{C}))\subset Z(\mathfrak{l}_{C})$.
For each $\lambda\in (\mathfrak{t}')^{*}$, define $t_{\lambda}\colon \Sym\mathfrak{t}'\to \Sym\mathfrak{t}'$ by $t_{\lambda}(H) = H + \lambda(H)$ for $H\in \mathfrak{t}'$.
Then we have $\gamma = t_{\rho - \rho_{L}}\circ\gamma_{L}\circ\gamma'$ by the definitions.
(Note that we work here with opposite Borels because the usual Harish-Chandra map is obtained by projecting from $Z(\mathfrak{g}_{C})\subset U(\mathfrak{t}')\oplus \overline{\mathfrak{u}}'U(\mathfrak{g}_{C})\mathfrak{u}'$, where we recall that $\mathfrak{u}'$ is the unipotent radical of the Borel subalgebra $\mathfrak{b}'$, and we let $\overline{\mathfrak{u}}'$ denote its opposite.)

We prove (i).
For each $u\in U(\mathfrak{g}_C)$, let $R_{u}$ (resp.\ $L_{u}$) be the right translation (resp.\ left translation) of $u$ on locally analytic functions on $G$.
Let $f\in (\Ind_{P}^{G}\sigma)^{\an}$ and $z\in Z(\mathfrak{g}_{C})$.
Then for $g\in G$, we have $(zf)(g) = (R_{z}f)(g) = (L_{\Ad(g)z}f)(g)$ and this is equal to $(L_{z}f)(g)$ since $z\in Z(\mathfrak{g}_{C}) = U(\mathfrak{g}_{C})^{\Ad(G_{C})}$.
As $f$ is left $N$-invariant, $L_{u}f = 0$ for $u\in \mathfrak{n}_{C}U(\mathfrak{g}_{C})$.
Hence $zf = L_{z}f = L_{\gamma'(z)}f$.
We also have $(L_{u}f)(g) = \sigma(u)f(g)$ for $u\in U(\mathfrak{l}_{C})$.
Therefore $(L_{\gamma'(z)}f)(g) = \gamma'(z)f(g)$.
As $\gamma = t_{\rho - \rho_{L}}\circ\gamma_{L}\circ\gamma'$, if $\sigma$ is killed by $(\ker\chi_{\mu})^{n}$ for $n\in\Z_{>0}$, then $(\Ind_{P}^{G}\sigma)^{\an}$ is killed by $(t_{\rho - \rho_{L}}(\ker\chi_{\mu}))^{n} = (\ker\chi_{\mu - \rho + \rho_{L}})^{n}$.
Part (i) follows.

By Lemma~\ref{lem:induction-tensor}, we have $V\otimes (\Ind_{P}^{G}\tau)^{\an}\simeq (\Ind_{P}^{G} V\otimes \tau)^{\an}$.
The $P$-representation $V|_P$ has a filtration such that the successive quotients are $\mathfrak l_C$-simple and $N$ acts trivially (apply Lemma~\ref{lem:subobject-verma} with $W = V^N$).
Let $V_0$ be such a subquotient and let $\nu\in (\mathfrak{t}')^{*}$ be the lowest weight of $V_0$.
As $V_N$ has the same lowest weight as $V$, it is sufficient to prove that if $\pr_{\lambda-\rho}((\Ind_{P}^{G} V_0\otimes \tau)^{\an})\ne 0$ then $\nu = \lambda$.

The representation $V_0$ has infinitesimal character $\nu - \rho_{L}$.
By (i), $(\Ind_{P}^{G} V_0\otimes \tau)^{\an}$ has infinitesimal character $\nu - \rho$.
Therefore if it has infinitesimal character $\lambda - \rho$ then $\lambda - \rho = w(\nu - \rho)$ for an element $w$ in the absolute Weyl group.
Then we have $\lambda - w(\nu) = \rho - w(\rho)$.
The weight $w(\nu)$ is a weight in $V$ and since $\lambda$ is lowest weight of $V$, $\lambda - w(\nu)$ is a non-positive linear combination of positive roots.
On the other hand $\rho - w(\rho)$ is a non-negative linear combination of positive roots.
Hence both sides are zero.
From $\rho = w(\rho)$ we have $w = 1$ and we get $\lambda = w(\nu) = \nu$, as desired.
\end{proof}

\subsection{A criterion}
\label{subsec:A criterion}
\emph{In this section, we assume $F = \Q_{p}$.}
Let $\alggrp{P} = \alggrp{L}\alggrp{N}$ be a parabolic subgroup of $\alggrp{G}$ and $\sigma$ a finite-dimensional absolutely irreducible continuous representation of $L$.

\begin{lem}\label{lem:cts-rep}
We have $\sigma \in \mathcal{O}^{L}$, after perhaps replacing $C$ by a finite extension.
\end{lem}
\begin{proof}
  By \cite[\S V.9]{MR1176100} 
  (and as $F = \Q_{p}$)
      $\sigma$ is in fact a locally analytic representation. If $W \subset \sigma|_{\mathfrak{l}_{C}}$ is a simple submodule, then
  by irreducibility, $\sigma = \sum_{x \in L} xW$ and each $xW$ is $\mathfrak{l}_{C}$-stable, so $\sigma|_{\mathfrak{l}_{C}}$ is semisimple. After
  a finite scalar extension, $\sigma|_{\mathfrak{l}_{C}}$ is a direct sum of absolutely simple $\mathfrak{l}_{C}$-modules, i.e.\ $\sigma \in \mathcal{O}^{L}$.
\end{proof}

Therefore we may assume that $\sigma$ is an absolutely simple object of $\mathcal O^L$.

We now confirm an expectation of \cite[\S2]{MR2275644} for Banach representations.

\begin{prop}\label{prop:adm-fin-length}
  Assume Assumption~\ref{assump:on p}.
  Suppose that $\sigma$ is a finite-dimensional continuous representation of $L$.
  Then the Banach representation $(\Ind_P^G \sigma)^{\cts}$ is admissible and topologically of finite length.  
\end{prop}

\begin{proof}
  The admissibility follows as in \cite[Proposition~2.4]{MR2275644} (where $\dim_C \sigma = 1$).
  For the finite length statement,   by \cite[Theorem~7.1]{ST} (as $F = \Q_{p}$) it suffices to show that the admissible locally analytic representation $(\Ind_P^G \sigma)^{\an} \cong \mathcal{F}_P^G(\underline{M}(\sigma'),1)$ is topologically of finite length, but this is a consequence of Corollary~\ref{cor:O-S-fin-length}.
\end{proof}

\begin{lem}\label{lem:sigma0-tau-decomp}
  Suppose that $\alggrp G^\der$ is simply connected.
  Then, after perhaps replacing $C$ by a finite extension, we can write $\sigma \cong \sigma_0 \otimes \tau$, where $\sigma_0 \in \mathcal{O}^{L}$ is $\mathfrak{l}_{C}$-simple
  and $\tau$ is an (absolutely irreducible) smooth $L$-representation such that moreover $\underline{L}(\sigma_0') \in \mathcal{O}^{P}$ is equimaximal.
\end{lem}

\begin{proof}
  By Proposition~\ref{prop:simple-O-P}(iii) we can write $\sigma \cong \sigma_0 \otimes \tau$, where $\sigma_0 \in \mathcal{O}^{L}$ is $\mathfrak{l}_{C}$-simple
  and $\tau$ is an absolutely irreducible smooth $L$-representation.
  By Proposition~\ref{prop:twist-to-O-P} there exists a smooth character $\eta$ of $L$ such that $\underline{L}(\sigma_0') \otimes \eta \in \mathcal{O}^{P}$ is equimaximal.
  (Both steps may require a finite scalar extension.)
  By replacing $(\sigma_0,\tau)$ by $(\sigma_0\eta,\tau\eta^{-1})$ we may assume that $\eta = 1$, i.e.\ $\underline{L}(\sigma_0')$ is equimaximal.
\end{proof}

\emph{For the remainder of this subsection we will assume that $\sigma \cong \sigma_0 \otimes \tau$, where $\sigma_0 \in \mathcal{O}^{L}$ is $\mathfrak{l}_{C}$-simple
  and $\tau$ is an absolutely irreducible smooth $L$-representation such that moreover $\underline{L}(\sigma_0') \in \mathcal{O}^{P}$ is equimaximal with maximal parabolic $Q = L_QN_Q$ (containing $P$).}
This decomposition of $\sigma$ always exists if $\alggrp G^\der$ is simply connected, by Lemma~\ref{lem:sigma0-tau-decomp}, but not in general.
By Lemma~\ref{lem:semisimple-lift}\ref{lem:semisimple-lift-2} the decomposition is unique up to a smooth character of $L_Q$.

Note that
\begin{equation*}
  \mathcal{F}_{P \cap L_Q}^{L_Q} (\underline{L}_{L_Q}(\sigma_0'),\tau) \subset \mathcal{F}_{P \cap L_Q}^{L_Q} (\underline{M}_{L_Q}(\sigma_0'),\tau) = (\Ind_{P \cap L_Q}^{L_Q} \sigma)^{\an} \subset (\Ind_{P \cap L_Q}^{L_Q} \sigma)^{\cts}.
\end{equation*}
We say that $\sigma$ \emph{satisfies condition $(\ast)$} if 
\begin{equation*}
\fbox{\text{any irreducible subrepresentation of $\mathcal{F}_{P \cap L_Q}^{L_Q} (\underline{L}_{L_Q}(\sigma_0'),\tau)$ is dense in $(\Ind_{P \cap L_Q}^{L_Q} \sigma)^{\cts}$.}}
\phantomsection
\label{eq:star} \end{equation*}
Note that this condition does not depend on the choice of factorization $\sigma \cong \sigma_0\otimes \tau$.
Note also that $\underline{L}_{L_Q}(\sigma_0')$ lies in $\mathcal{O}^{L_Q}$ by Corollary~\ref{cor:N-invts} and that
tensoring with $\underline{L}_{L_Q}(\sigma_0')'$ gives a correspondence between (closed) subrepresentations of $(\Ind_{P \cap L_Q}^{L_Q} \tau)^{\sm}$
and (closed) subrepresentations of 
\begin{equation*}
\mathcal{F}_{P \cap L_Q}^{L_Q} (\underline{L}_{L_Q}(\sigma_0'),\tau) = \mathcal{F}_{L_Q}^{L_Q} (\underline{L}_{L_Q}(\sigma_0'),(\Ind_{P \cap L_Q}^{L_Q} \tau)^{\sm}) = \underline{L}_{L_Q}(\sigma_0')' \otimes (\Ind_{P \cap L_Q}^{L_Q} \tau)^{\sm}.
\end{equation*}
(The point is that $\underline{L}_{L_Q}(\sigma_0')'$ is absolutely simple as $\mathfrak{l}_{Q,C}$-module and that all representations here carry the finest locally convex topology.)

\begin{lem}\label{lem:reduction to smooth}
Assume Assumption~\ref{assump:on p}.
If any irreducible subrepresentation of $\mathcal{F}_{P\cap L_{Q}}^{L_{Q}}(1,\tau) = (\Ind_{P\cap L_{Q}}^{L_{Q}}\tau)^{\sm}$ is dense in $(\Ind_{P\cap L_{Q}}^{L_{Q}}\tau)^{\cts}$, then \upshape{(\hyperref[eq:star]{$*$})} holds.
\end{lem}
\begin{proof}
We may simplify notation by relabeling $Q$ as $G$, i.e.\ assume that $\underline{L}(\sigma_0') \in \mathcal{O}^{G}$.

Let $V := \underline{L}(\sigma_{0}')'$, which is by assumption a finite-dimensional locally analytic representation of $G$.
Note that $\sigma_0' \into V'$ as $P$-representations, i.e.\ $V \onto \sigma_0$, giving by Lemma~\ref{lem:induction-tensor} a commutative diagram (where all maps are continuous):
\begin{equation*}
  \xymatrix{
    V\otimes (\Ind_{P}^{G}\tau)^{\sm}\ar@{^{(}->}[r]\ar@{^{(}->}[rd] & V\otimes (\Ind_{P}^{G}\tau)^{\an}\ar[r]^\sim\ar@{^{(}->}[d] & (\Ind_{P}^{G}V \otimes \tau)^{\an}\ar@{->>}[r]\ar@{^{(}->}[d] & (\Ind_{P}^{G} \sigma)^{\an}\ar@{^{(}->}[d] \\ 
    & V\otimes (\Ind_{P}^{G}\tau)^{\cts}\ar[r]^\sim & (\Ind_{P}^{G}V\otimes \tau)^{\cts}\ar@{->>}[r] & (\Ind_{P}^{G} \sigma)^{\cts}
  }
\end{equation*}

We first show that the composition of the top horizontal arrows is injective with image $\mathcal{F}_{P}^{G} (\underline{L}(\sigma_0'),\tau)$.
By the discussion before this lemma we know that $V\otimes (\Ind_{P}^{G}\tau)^{\sm} \cong \mathcal{F}_{P}^{G} (\underline{L}(\sigma_0'),\tau)$ and that its irreducible constituents are of the form $\mathcal{F}_{G}^{G} (\underline{L}(\sigma_0'),\pi')$ for some irreducible smooth representations $\pi'$ of $G$.
It thus suffices to show that every irreducible constituent of $\ker((\Ind_{P}^{G}V \otimes \tau)^{\an}\onto (\Ind_{P}^{G} \sigma)^{\an})$ and of $\coker(\mathcal{F}_{P}^{G} (\underline{L}(\sigma_0'),\tau)\into (\Ind_{P}^{G} \sigma)^{\an})$ is not of this form.
This is clear for the cokernel (by Corollary~\ref{cor:F-P-G-intertwiner2}) because $L(\sigma_0')$ occurs with multiplicity one in $M(\sigma_0')$.
For the kernel, note first that $V' \in \mathcal{O}^{G}$ implies that it is a finite-dimensional locally analytic representation of $P$ on which $\mathfrak{l}_{C}$ acts as a direct sum of absolutely simple $\mathfrak{l}_{C}$-modules.
Hence $(\Ind_{P}^{G}V \otimes \tau)^{\an} \cong \mathcal{F}_{P}^{G} (\underline M(V'|_P),\tau)$ by Proposition~\ref{prop:OS-ind}, so by Corollary~\ref{cor:F-P-G-intertwiner2} it suffices to show that 
\begin{equation}\label{eq:mult}
  [\underline M(V'/\sigma_0') : L(\sigma_0')]_{\mathcal{O}^{\mathfrak p}} = 0.
\end{equation}
Write $\sigma_0' \cong L_L(\lambda)$ as $\mathfrak l_C$-module for some $\lambda \in (\mathfrak t')^*$, so $V' \cong L(\lambda)$ as $\mathfrak g_C$-module.
Thus only weights $< \lambda$ occur in $V'/\sigma_0'$, which implies~\eqref{eq:mult}.

Let $\pi\subset \mathcal{F}_{P}^{G}(\underline{L}(\sigma_{0}'),\tau) \subset (\Ind_{P}^{G} \sigma)^{\an}$
be an irreducible subrepresentation and take an irreducible subrepresentation $\pi_{0}$ of $(\Ind_{P}^{G} \tau)^{\sm}$ such that $\pi$ is the image of $V\otimes\pi_{0} \subset V\otimes (\Ind_{P}^{G}\tau)^{\sm}$.
Then $\pi_{0}$ is dense in $(\Ind_{P}^{G}\tau)^{\cts}$ by assumption, so also $V \otimes \pi_0$ is dense in $V \otimes (\Ind_{P}^{G}\tau)^{\cts}$.
The above diagram then shows that $\pi$ is dense in $(\Ind_{P}^{G} \sigma)^{\cts}$.
\end{proof}

\begin{thm}\label{thm:main2}
Assume Assumption~\ref{assump:on p}.
    The Banach space representation $(\Ind_P^G \sigma)^{\cts}$ is irreducible if and only if condition~\upshape{(\hyperref[eq:star]{$*$})} above holds.
\end{thm}

\begin{proof}
  Suppose that \upshape{(\hyperref[eq:star]{$*$})} holds and that $\pi \subset (\Ind_{P}^G\sigma)^{\cts}$ is a nonzero closed subrepresentation.
  Then $\pi^{\an} \ne 0$ by Theorem~\ref{thm:property of an}.
  We have $(\Ind_{P}^G\sigma)^{\an} \cong \mathcal{F}_P^G(\underline{M}(\sigma_0'),\tau)$.
  By Corollary \ref{cor:socle-orlik-strauch} we have
  \begin{equation*}
    \soc_G ((\Ind_{P}^G\sigma)^{\an}) \cong \mathcal{F}_Q^G(\underline{L}(\sigma_0'),\soc_{L_Q}(\Ind_{P \cap L_Q}^{L_Q} \tau)^{\sm}),
  \end{equation*}
  and hence $\pi^{\an}$ contains $\mathcal{F}_Q^G(\underline{L}(\sigma_0'),\rho)$ for some irreducible subrepresentation $\rho$ of $(\Ind_{P \cap L_Q}^{L_Q} \tau)^{\sm}$.
  Let $\widetilde{\rho} := \underline{L}_{L_Q}(\sigma_0')' \otimes \rho$ denote the corresponding (topologically) irreducible subrepresentation of $\mathcal{F}_{P \cap L_Q}^{L_Q} (\underline{L}_{L_Q}(\sigma_0'),\tau)$.
  Note that   $\underline{L}(\underline{L}_{L_Q}(\sigma_0')) \cong \underline{L}(\sigma_0')$ by Corollary~\ref{cor:N-invts}.
  Therefore $\mathcal{F}_Q^G(\underline{L}(\sigma_0'),\rho)$ is contained in
  \begin{align}
    \mathcal{F}_Q^G(\underline M(\underline{L}_{L_Q}(\sigma_0')),\rho) &\cong (\Ind_Q^G \underline{L}_{L_Q}(\sigma_0')' \otimes \rho)^{\an}\label{eq:4b} 
    = (\Ind_Q^G \widetilde{\rho})^{\an} \\
    &\subset (\Ind_Q^G (\Ind_{P\cap L_Q}^{L_Q} \sigma)^{\an})^{\an} \cong (\Ind_P^G \sigma)^{\an}\notag.
  \end{align}
  By Proposition~\ref{prop:loc-const-fns} (and Lemma~\ref{lm:N-invts}) we deduce that $\pi^{\an}$ contains all functions in the right-hand side of
  \eqref{eq:4b} that are supported on $\overline{N}_{Q,0}Q/Q$ and are locally constant on $\overline{N}_{Q,0}$ (where $\overline{N}_{Q,0}$ is a 
  fixed compact open subgroup of $\overline{N}_Q$).  This space is isomorphic to
  $C^\infty(\overline{N}_{Q,0}, \widetilde{\rho})$. By condition~\upshape{(\hyperref[eq:star]{$*$})} we know that $\widetilde{\rho}$ is dense in $(\Ind_{P \cap L_Q}^{L_Q} \sigma)^{\cts}$, and hence
  $C^\infty(\overline{N}_{Q,0}, \widetilde{\rho})$ is dense in $C^{0}(\overline{N}_{Q,0}, (\Ind_{P \cap L_Q}^{L_Q}\sigma)^{\cts})$.  But $C^{0}(\overline{N}_{Q,0}, (\Ind_{P \cap L_Q}^{L_Q}\sigma)^{\cts})$
  generates $(\Ind_P^G\sigma)^{\cts} \cong (\Ind_Q^G(\Ind_{P \cap L_Q}^{L_Q} \sigma)^{\cts})^{\cts}$ as $G$-representation.

  Conversely, if $(\Ind_P^G \sigma)^{\cts}$ is irreducible, then $(\Ind_{P \cap L_Q}^{L_Q} \sigma)^{\cts}$ is irreducible (as follows from the exactness of
  $\Ind_{Q}^{G}(-)^{\cts} \cong C^0(G/Q,-)$, cf.\ \cite[(2.1.3)]{locallyanalytic-memoir} and \cite[Corollary~2.2]{MR4190408}), hence condition~\upshape{(\hyperref[eq:star]{$*$})} holds.
    \end{proof}

\begin{cor}\label{cor:irreduciblity}
Assume Assumption~\ref{assump:on p}.
If every irreducible subrepresentation of $(\Ind_{P\cap L_{Q}}^{L_{Q}}\tau)^{\sm}$ is dense in $(\Ind_{P\cap L_{Q}}^{L_{Q}}\tau)^{\cts}$, then $(\Ind_{P}^{G}\sigma)^{\cts}$ is irreducible.
In particular, if $(\Ind_{P\cap L_{Q}}^{L_{Q}}\tau)^{\sm}$ is irreducible, then $(\Ind_{P}^{G}\sigma)^{\cts}$ is irreducible.
\end{cor}
\begin{proof}
Note that $\mathcal{F}_{P\cap L_{Q}}^{L_{Q}}(1,\tau) = (\Ind_{P\cap L_{Q}}^{L_{Q}}\tau)^{\sm}$.
Hence the corollary follows from Lemma~\ref{lem:reduction to smooth} and Theorem~\ref{thm:main2}.
The last statement follows from the fact that $(\Ind_{P\cap L_{Q}}^{L_{Q}}\tau)^{\sm}$ is dense in $(\Ind_{P\cap L_{Q}}^{L_{Q}}\tau)^{\cts}$.
This follows, for example, from Lemma~\ref{lem:supported on big cell}.
\end{proof}

\begin{thm}\label{thm:equivalence-irred}
Assume Assumption~\ref{assump:on p}.
  The following are equivalent:
  \begin{enumerate}
  \item $(\Ind_P^G \sigma)^\cts$ is irreducible;
  \item $(\Ind_{P\cap L_Q}^{L_Q} \sigma)^\cts$ is irreducible;
  \item $(\Ind_{P\cap L_Q}^{L_Q} \tau)^\cts$ is irreducible.
  \end{enumerate}
\end{thm}

In particular, to understand when $(\Ind_P^G \sigma)^\cts$ is irreducible
it suffices to restrict to the case where $\sigma$ is smooth!

\begin{proof}
  Recall that $\underline{L}_{L_Q}(W) \in \mathcal{O}^{L_Q}$ (i.e.\ is equimaximal) by Corollary~\ref{cor:N-invts}, so parts (i) and (ii) are equivalent because condition~\upshape{(\hyperref[eq:star]{$*$})} is literally the same in both cases.
  We know that (iii) implies (i) by Lemma~\ref{lem:reduction to smooth} and Theorem~\ref{thm:main2}.
  To prove that (ii) implies (iii), we relabel $L_Q$ as $G$ and may therefore assume that 
  $\underline{L}(\sigma'_0) \in \mathcal O^G$ is a finite-dimensional locally analytic representation of $G$.
  Let $V := \underline{L}(\sigma'_0)'$.
  Take any nonzero closed subrepresentation $\pi \subset (\Ind_P^G \tau)^\cts$.
  We first claim that the composition (already considered in the proof of Lemma~\ref{lem:reduction to smooth})
  \begin{equation}\label{eq:map}
    V \otimes \pi^\an \into V \otimes (\Ind_P^G \tau)^\an \congto (\Ind_P^G V \otimes \tau)^\an \onto (\Ind_P^G \sigma)^\an
  \end{equation}
  is surjective.

  To prove the claim, by Lemma~\ref{lem:induction-tensor} the maps~\eqref{eq:map} are obtained by applying the functor of locally $\Q_p$-analytic vectors to the maps
  \begin{equation}\label{eq:map2}
    V \otimes \pi \into V \otimes (\Ind_P^G \tau)^\cts \congto (\Ind_P^G V \otimes \tau)^\cts \onto (\Ind_P^G \sigma)^\cts
  \end{equation}
  of admissible Banach space representations of $G$.
  The surjectivity of the composition~\eqref{eq:map} is equivalent to the surjectivity of the composition~\eqref{eq:map2}. 
  (Note that the image of the composition~\eqref{eq:map2} is closed by admissibility and that the functor $(\cdot)^\an$ is exact.)
  By assumption, $(\Ind_P^G \sigma)^\cts$ is irreducible, so it suffices to show that the composition~\eqref{eq:map2} is nonzero or equivalently that the composition~\eqref{eq:map} is nonzero.
  By Corollary~\ref{cor:socle-orlik-strauch} applied with $M = 1$ (the trivial representation), $Q = G$ and $\pi = \tau$ we see that any nonzero closed subrepresentation of $(\Ind_P^G \tau)^\an$ intersects $(\Ind_P^G \tau)^\sm$ non-trivially.
  But the composition
  \begin{equation*}
    V \otimes (\Ind_P^G \tau)^\sm \into V \otimes (\Ind_P^G \tau)^\an \congto (\Ind_P^G V \otimes \tau)^\an \onto (\Ind_P^G \sigma)^\an
  \end{equation*}
  is injective by the proof of Lemma~\ref{lem:reduction to smooth}, hence indeed the claim holds.
  
  Recall that we have fixed a system of positive roots for $(\mathfrak{g}_{C},\mathfrak{t}')$, and let $\lambda$ be the lowest weight of $V$.
  Applying the projection functor $\pr_{\lambda - \rho}$ to \eqref{eq:map}, we get
  \begin{equation}\label{eq:map3}
    \pr_{\lambda - \rho}(V \otimes \pi^\an) \into \pr_{\lambda - \rho}(V \otimes (\Ind_P^G \tau)^\an) \congto \pr_{\lambda - \rho}((\Ind_P^G V \otimes \tau)^\an) \onto (\Ind_P^G \sigma)^\an.
  \end{equation}
  By the definition of translation functors we have $\pr_{\lambda - \rho}(V \otimes \pi^\an) = T_{-\rho}^{\lambda - \rho}(G,V)(\pi^{\an})$ and $\pr_{\lambda - \rho}(V \otimes (\Ind_P^G \tau)^\an) = T_{-\rho}^{\lambda - \rho}(G,V)((\Ind_P^G \tau)^{\an})$.
  By Proposition~\ref{prop:translation functors(parabolic induction)}(ii), the last map in \eqref{eq:map3} is an isomorphism.
  Hence the map
  \[
    T_{-\rho}^{\lambda - \rho}(G,V)(\pi^{\an})\to T_{-\rho}^{\lambda - \rho}((\Ind_P^G \tau)^{\an})
  \]
  is surjective.
  Apply $T_{\lambda - \rho}^{-\rho}(G,V')$.
  By Proposition~\ref{prop:translation functors(analytic)}, the inclusion $\pi^{\an}\to (\Ind_P^G \tau)^{\an}$ is surjective.
  Therefore $\pi = (\Ind_{P}^{G}\tau)^{\cts}$ (after taking closure), hence we get (iii).
\end{proof}

\begin{cor}\label{cor:criterion-easy-cases}
  Assume Assumption~\ref{assump:on p}.
  We continue to write $\sigma = \sigma_0 \otimes \tau$ as above, with $Q \supset P$ denoting the maximal parabolic of $\un L(\sigma_0')$.
  Each of the following conditions implies the next:
  \begin{enumerate}
  \item $M(\sigma_0') = U(\mathfrak g_C) \otimes_{U(\mathfrak p_C)} \sigma_0'$ is irreducible as $U(\mathfrak g_C)$-module;
  \item $Q = P$;
  \item $(\Ind_P^G \sigma)^\cts$ is irreducible.
  \end{enumerate}
\end{cor}

\begin{proof}
  If (i) holds, and $Q$ is strictly bigger than $P$, then $L(\sigma_0')^{\mathfrak n_{Q,C}} \in \mathcal O^{\mathfrak l_{Q}}$ is finite-dimensional, so the natural map
  $U(\mathfrak l_{Q,C}) \otimes_{U(\mathfrak p_C \cap \mathfrak l_{Q,C})} \sigma_0' \onto L(\sigma_0')^{\mathfrak n_{Q,C}}$
  is not an isomorphism, so after extending scalars to $U(\mathfrak g_C)$ we obtain
  \[M(\sigma_0') \onto U(\mathfrak g_C) \otimes_{U(\mathfrak q_C)} L(\sigma_0')^{\mathfrak n_{Q,C}} \onto L(\sigma_0'),\]
  where the first map is not an isomorphism, contradicting (i).

  If (ii) holds, then we deduce (iii) from Corollary~\ref{cor:irreduciblity}.
\end{proof}

\begin{remark}
  Note that part (i) (when $\dim_C \tau = 1$) is Orlik--Strauch's sufficient condition for the irreducibility of the locally analytic representation $(\Ind_P^G \sigma)^\an$ \cite{OS}, which
  in turn implies (iii).
    This criterion also follows from Theorem~\ref{thm:OS-main-in-main}, see Proposition~\ref{prop:irred-loc-an-criterion} below.
  The application of this criterion to the irreducibility of Banach representation was noted in \cite[Proposition 2.6(ii)]{MR2810332}.
\end{remark}

We note in passing that the converse of Orlik--Strauch's criterion is true as well, since we do not know a reference in the literature.

\begin{prop}\label{prop:irred-loc-an-criterion}
  Assume Assumption~\ref{assump:on p}.
  We continue to write $\sigma = \sigma_0 \otimes \tau$ as above.
  The following are equivalent:
  \begin{enumerate}
  \item $M(\sigma_0') = U(\mathfrak g_C) \otimes_{U(\mathfrak p_C)} \sigma_0'$ is irreducible as $U(\mathfrak g_C)$-module;
  \item $(\Ind_P^G \sigma)^\an$ is irreducible.  
  \end{enumerate}
\end{prop}

\begin{proof}
  Let $Q \supset P$ denote the maximal parabolic of $\un L(\sigma_0')$.

  If (i) holds, then $Q = P$ by Corollary~\ref{cor:criterion-easy-cases}.
  The assumption implies that $\un M(\sigma_0')$ is simple in $\mathcal{O}^{P}$, so $\un M(\sigma_0') = \un L(\sigma_0')$.
  Then $(\Ind_P^G \sigma)^\an = \mathcal{F}_{P}^{G}(\underline M(\sigma_0'),\tau) = \mathcal{F}_{P}^{G}(\underline L(\sigma_0'),\tau)$, which is irreducible by Theorem~\ref{thm:OS-main-in-main}(iii).

  Conversely, if (ii) holds, then $\mathcal{F}_{P}^{G}(\underline M(\sigma_0'),\tau)$ is irreducible, so by Lemma~\ref{lem:OS-functor-nonzero} we deduce that $\underline M(\sigma_0')$ is simple in $\mathcal{O}^{P}$, so irreducible as $U(\mathfrak g_C)$-module (as $\sigma_0'$ is $\mathfrak{l}_{C}$-simple).
\end{proof}

\begin{lem}\label{lem:OS-functor-nonzero}
  If $M \in \mathcal{O}^{P}$, $\pi \in \Rep^{\adm}(L)$ are both nonzero, then $\mathcal{F}_{P}^{G}(M,\pi) \ne 0$.
\end{lem}

\begin{proof}
  For example, this follows from Proposition~\ref{prop:N-invariants}, as $M \ne 0$ implies $M^N \ne 0$.
\end{proof}

\begin{prop}\label{prop:irreducibility for product group}
Assume Assumption~\ref{assump:on p}.
Suppose that $\alggrp{G}_{1},\alggrp{G}_{2}$ are connected reductive groups, $\alggrp{P}_{1} = \alggrp{L}_{1}\alggrp{N}_{1}\subset \alggrp{G}_{1}$, $\alggrp{P}_{2} = \alggrp{L}_{2}\alggrp{N}_{2}\subset \alggrp{G}_{2}$ parabolic subgroups and $\sigma_{1}$ (resp.\ $\sigma_{2}$) a finite-dimensional absolutely irreducible continuous representation of $L_{1}$ (resp.\ $L_{2}$).
Then $(\Ind_{P_{1}\times P_{2}}^{G_{1}\times G_{2}}\sigma_{1}\boxtimes\sigma_{2})^{\cts}$ is absolutely irreducible if and only if $(\Ind_{P_{1}}^{G_{1}}\sigma_{1})^{\cts}$ and $(\Ind_{P_{2}}^{G_{2}}\sigma_{2})^{\cts}$ are both absolutely irreducible.
\end{prop}
\begin{proof}
We first note that $(\Ind_{P_{1}}^{G_{1}}\sigma_{1})^{\cts}\widehat{\boxtimes}(\Ind_{P_{2}}^{G_{2}}\sigma_{2})^{\cts}\simeq (\Ind_{P_{1}\times P_{2}}^{G_{1}\times G_{2}}\sigma_{1}\boxtimes\sigma_{2})^{\cts}$ as admissible Banach representations of $G_1 \times G_2$.
(This follows exactly as in the proof of \cite[Lemma 2.8]{MR4190408}, using the isomorphism $C^0(P_1\backslash G_1,\sigma_1) \wh\otimes C^0(P_2\backslash G_2,\sigma_2) \cong C^0(P_1\backslash G_1 \times P_2\backslash G_2,\sigma_1 \otimes \sigma_2)$, cf.\ the end of \cite[\S17]{MR1869547}, instead of \cite[(2.6)]{MR4190408}.)
The ``only if'' direction follows (using, for example, \cite[Lemma 2.1(ii)]{MR4190408}).

For the ``if'' direction, by a $z$-extension we may assume $\alggrp{G}_{1}^{\der},\alggrp{G}_{2}^{\der}$ are simply connected.
Then by Lemma~\ref{lem:sigma0-tau-decomp}, after perhaps replacing $C$ by a finite extension, we can take a decomposition $\sigma_{i} = \sigma_{i,0}\otimes \tau_{i}$ as in this subsection, for $i = 1,2$.
Then $\sigma_1 \boxtimes \sigma_2 \cong (\sigma_{1,0} \boxtimes \sigma_{2,0}) \otimes (\tau_1 \boxtimes \tau_2)$, where $\sigma_{1,0} \boxtimes \sigma_{2,0}$ is $\mathfrak l_{1,C} \times \mathfrak l_{2,C}$-simple and $\underline L(\sigma_{1,0}' \boxtimes \sigma_{2,0}') \cong \underline L(\sigma_{1,0}') \boxtimes \underline L(\sigma_{2,0}')$ equimaximal with maximal parabolic $Q_1 \times Q_2$.
Let $P(Q_i) := P_i \cap L_{Q_i}$ for $i = 1,2$.
By Corollary~\ref{cor:irreduciblity}, it suffices to show that every irreducible subrepresentation of $(\Ind_{P(Q_1) \times P(Q_2)}^{L_{Q_1}\times L_{Q_2}}\tau_{1}\boxtimes\tau_{2})^{\sm}$ is dense in $(\Ind_{P(Q_1) \times P(Q_2)}^{L_{Q_1}\times L_{Q_2}}\tau_{1}\boxtimes\tau_{2})^{\cts}$.
An irreducible subrepresentation of $(\Ind_{P(Q_1) \times P(Q_2)}^{L_{Q_1}\times L_{Q_2}}\tau_{1}\boxtimes\tau_{2})^{\sm}\simeq (\Ind_{P(Q_1)}^{L_{Q_1}}\tau_{1})^{\sm}\boxtimes (\Ind_{P(Q_2)}^{L_{Q_2}}\tau_{2})^{\sm}$ is of the form $\pi_{1}\boxtimes \pi_{2}$ where $\pi_{i}\subset (\Ind_{P(Q_i)}^{L_{Q_i}}\tau_{i})^{\sm}$ is an irreducible subrepresentation for $i = 1,2$.
By assumption and Theorem~\ref{thm:equivalence-irred}, $\pi_{i}$ is dense in $(\Ind_{P(Q_i)}^{L_{Q_i}}\tau_{i})^{\cts}$ for $i = 1,2$.
Hence $\pi_{1}\boxtimes\pi_{2}\subset (\Ind_{P(Q_1)}^{L_{Q_1}}\tau_{1})^{\cts}\widehat{\boxtimes}(\Ind_{P(Q_2)}^{L_{Q_2}}\tau_{2})^{\cts}\simeq (\Ind_{P(Q_1) \times P(Q_2)}^{L_{Q_1}\times L_{Q_2}}\tau_{1}\boxtimes\tau_{2})^{\cts}$ is dense, as required.
\end{proof}

\begin{remark}\label{rk:reduce-abs-simple-simply-conn}
  The problem of understanding when a parabolic induction $(\Ind_{P}^{G}\sigma)^{\cts}$ (with $\sigma$ finite-dimensional) is absolutely irreducible can be reduced to the case where $\alggrp G$ is an absolutely almost simple simply-connected group.
  (First apply Remark~\ref{rk:reduce-simply-conn} to reduce to the simply-connected cover of $\alggrp G^\der$.
  Then use Proposition~\ref{prop:irreducibility for product group} to reduce to an almost simple simply-connected group.
  Then observe that $\alggrp G = \Res_{E/F} \alggrp H$ with $\alggrp H$ absolutely almost simple simply-connected.)
  We may moreover assume that $\alggrp G$ is isotropic, as otherwise $\alggrp G$ is the only parabolic subgroup.
\end{remark}

\subsection{Genericity}
\label{sec:generic}
Recall that $\alggrp{B} = \alggrp{Z}\alggrp{U}$ is a minimal parabolic subgroup.
We assume that we are in the setting of \S\ref{subsec:A criterion} with $\alggrp P = \alggrp B$, i.e.\ $F = \Q_p$ and $\sigma = \sigma_0 \otimes \tau$, where $\sigma_0 \in \mathcal{O}^{Z}$ is $\mathfrak{z}_{C}$-simple and $\tau$ is an absolutely irreducible smooth $Z$-representation such that moreover $\underline{L}(\sigma_0') \in \mathcal{O}^{B}$ is equimaximal with maximal parabolic $Q = L_{Q}N_{Q}$.

Fix an algebraic closure $\overline{C}$ of $C$.
A smooth character $\theta\colon U\to \overline{C}^{\times}$ is called \emph{non-degenerate} if the restriction of $\theta$ to each simple root subgroup is non-trivial.
For a smooth representation $\pi$ of $U$, let $\pi_{\overline{C},U,\theta} := (\pi\otimes_{C}\overline{C})/\langle \pi(u)v - \theta(u)v\mid u\in U\rangle$ be the space of twisted coinvariants.
We say that a representation $\pi$ of $G$ is \emph{generic} if $\pi_{\overline{C},U,\theta}\ne 0$ for some non-degenerate $\theta$.
By a slight abuse we also say that a smooth representation $\pi'$ of $U$ over $\o C$ is \emph{generic} if $\pi'_{U,\theta}\ne 0$ for some non-degenerate $\theta$.
When $\alggrp{G}$ is quasisplit, this is the familiar notion.

\begin{prop}\label{prop:whittaker criterion}
Assume Assumption~\ref{assump:on p}.
If any irreducible subrepresentation of $(\Ind_{B\cap L_{Q}}^{L_{Q}}\tau)^{\sm}$ is generic, then $(\Ind_{B}^{G}\sigma)^{\cts}$ is irreducible.
\end{prop}
\begin{proof}
To simplify notation, we relabel $L_Q$ as $G$. Let $W$ be the Weyl group of $G$ and $w_{0}\in W$ the longest element.
Then we have the Bruhat decomposition $G = \coprod_{w\in W}BwB$ and $(\Ind_{B}^{G}\tau)^{\sm}$ has a $B$-stable filtration $F_{w}$ with graded pieces $(\cInd_{B}^{B w B} \tau)^{\sm}$ for $w \in W$, and $F_{w_{1}}\subset F_{w_{2}}$ if $w_{1}\ge w_{2}$ with respect to the Bruhat order.

Let $\pi\subset (\Ind_{B}^{G}\tau)^{\sm}$ be an irreducible subrepresentation and take a non-degenerate character $\theta\colon U\to \overline{C}^{\times}$ such that $\pi_{\overline{C},U,\theta}\ne 0$.
From an argument in \cite[p.~211]{MR581582} or the proof of the geometric lemma~\cite[5.2 Theorem]{MR0579172} we have $((\cInd_{B}^{BwB} \tau)^{\sm})_{\overline{C},U,\theta} = 0$ if $w \ne w_0$.
  
  By above, and as $(-)_{\o C,U,\theta}$ is an exact functor on the category of smooth $U$-representations, we have an exact sequence of smooth $B$-representations
  \[ 0 \to (\cInd_{B}^{B w_0 B} \tau)^{\sm} \to (\Ind_{B}^{G}\tau)^{\sm} \to \pi_{0} \to 0 \]
  with $(\pi_{0})_{\o C,U,\theta} = 0$, which induces an exact sequence
  \begin{equation}\label{eq:3}
    0 \to \pi \cap (\cInd_{B}^{Bw_0B} \tau)^{\sm} \to \pi \to \pi_{1} \to 0
  \end{equation}
  with $(\pi_1)_{\overline{C},U,\theta} = 0$.
    By assumption we have $\pi_{\overline{C},U,\theta} \ne 0$ and hence by \eqref{eq:3} we have $(\pi \cap (\cInd_{B}^{Bw_0B} \tau)^{\sm})_{\overline{C},U,\theta} \ne 0$.
  Note that the map $(\cInd_{B}^{Bw_0B} \tau)^{\sm}\otimes_{C}\overline{C} \onto \tau\otimes_{C}\overline{C}$, $f \mapsto \int_U f(w_0 u) \theta^{-1}(u) du$ identifies $\tau\otimes_{C}\overline{C}$ with the twisted coinvariants of $(\cInd_{B}^{Bw_0B} \tau)^{\sm}\otimes_{C}\overline{C}$.
  We can thus take $f \in (\pi \cap (\cInd_{B}^{Bw_0B} \tau)^{\sm})\otimes_{C}\overline{C}$ such that $v := \int_U f(w_0 u) \theta^{-1}(u) du \ne 0$.
  Suppose that $\supp(f) \subset B\backslash Bw_0 U_0$ for some compact open subgroup $U_0$ of $U$.
  Then $f' := \int_{U_0} (u_0f) \cdot \theta^{-1}(u_0)du_{0} \in (\pi \cap (\cInd_{B}^{Bw_0B} \tau)^{\sm})\otimes_{C}\overline{C}$ is supported on $B\backslash Bw_0 U_0$ and $f'(w_0 u_0) = \theta(u_0)v \in \overline{C}v$.
  Take a finite extension $C'/C$ such that $f'\in (\pi \cap (\cInd_{B}^{Bw_0B} \tau)^{\sm})\otimes_{C}C'$.
  By Lemma~\ref{lem:supported on big cell} we see that $\pi\otimes_{C}C'$ is dense in $(\Ind_{B}^{G}\tau\otimes_{C}C')^{\cts} = (\Ind_{B}^{G}\tau)^{\cts}\otimes_{C}C'$ and therefore, from Lemma~\ref{lm:density-and-ext-of-scalars}, $\pi$ is dense in $(\Ind_{B}^{G}\tau)^{\cts}$.
  Now we get the proposition by Corollary~\ref{cor:irreduciblity}.
\end{proof}

\begin{cor}\label{cor:whittaker criterion}
  Assume Assumption~\ref{assump:on p}.
  If any irreducible subrepresentation of $(\Ind_{B\cap L_{Q}}^{L_{Q}}\tau)^{\sm}\otimes_{C}\overline{C}$ is generic, then $(\Ind_{B}^{G}\sigma)^{\cts}$ is irreducible.  
\end{cor}

\begin{proof}
  For any nonzero subrepresentation $\pi$ of $(\Ind_{B\cap L_{Q}}^{L_{Q}}\tau)^{\sm}$, the representation $\pi\otimes_{C}\overline{C}$ is of finite length, hence contains an irreducible subrepresentation $\pi'$, which is generic by assumption.
  Hence $\pi$ is generic, and we conclude by Proposition~\ref{prop:whittaker criterion}.
\end{proof}

\section{Applications}
\label{sec:applications}
In this section, we give applications of our irreducibility criterion.

Recall that we have fixed a maximal split torus $\alggrp{S}$ of $\alggrp{G}$ and a minimal parabolic subgroup $\alggrp{B} = \alggrp{Z}\alggrp{U}$ such that $\alggrp{S}\subset \alggrp{Z}$.
Let $\alggrp{P} = \alggrp{L}\alggrp{N}$ be a parabolic subgroup, $\overline{\alggrp{P}} = \alggrp{L}\overline{\alggrp{N}}$ the opposite parabolic subgroup.
Let $\alggrp{A}_{\alggrp{L}}\subset \alggrp{L}$ be the maximal split torus in the center of $\alggrp L$ and $\Phi(\alggrp{G},\alggrp{A}_{\alggrp{L}})$ the set of roots of $\alggrp A_{\alggrp L}$.
We have $\alggrp{A}_{\alggrp{L}}\subset \alggrp{A}_{\alggrp{Z}}$ and $\Phi(\alggrp{G},\alggrp{A}_{\alggrp{L}}) = \{\alpha|_{\alggrp A_{\alggrp L}}\mid \alpha\in \Phi(\alggrp{G},\alggrp{A}_{\alggrp{Z}})\}\setminus\{0\}$.
Let $\Phi(\alggrp{P},\alggrp{A}_{\alggrp{L}})$ be the set of $\alpha\in\Phi(\alggrp{G},\alggrp{A}_{\alggrp{L}})$ that appear in $\Lie(\alggrp N)$.
For $\alpha\in\Phi(\alggrp{G},\alggrp{A}_{\alggrp{L}})$, let $\alggrp{L}_{\alpha}$ be the centralizer of the connected component of $\ker\alpha\subset \alggrp{A}_{\alggrp{L}}$ in $\alggrp{G}$.
It is a Levi subgroup containing $\alggrp{L}$ and $\alggrp{P}\cap \alggrp{L}_{\alpha}$ is a maximal parabolic subgroup of $\alggrp{L}_{\alpha}$.
Let $\Phi_{\red}(\alggrp{P},\alggrp{A}_{\alggrp{L}})$ (resp.\ $\Phi_{\red}(\alggrp{G},\alggrp{A}_{\alggrp{L}})$) be the set of reduced elements in $\Phi(\alggrp{P},\alggrp{A}_{\alggrp{L}})$ (resp.\ $\Phi(\alggrp{G},\alggrp{A}_{\alggrp{L}})$) and $\Delta(\alggrp{P},\alggrp{A}_{\alggrp{L}})$ the set of simple roots in $\Phi(\alggrp{P},\alggrp{A}_{\alggrp{L}})$.
When $\alggrp{P}$ is a minimal parabolic subgroup $\alggrp{B} = \alggrp{Z}\alggrp{U}$, we put $\Phi := \Phi(\alggrp{G},\alggrp{A}_{\alggrp{Z}})$, $\Phi_{\red} := \Phi_{\red}(\alggrp{G},\alggrp{A}_{\alggrp{Z}})$, $\Phi^+_{\red} := \Phi_{\red}(\alggrp{B},\alggrp{A}_{\alggrp{Z}})$, $\Delta := \Delta(\alggrp{B},\alggrp{A}_{\alggrp{Z}})$.
If $\alggrp{P}$ is standard, then $\Delta(\alggrp{P},\alggrp{A}_{\alggrp{L}}) = \{\alpha|_{\alggrp{A}_{\alggrp{L}}}\mid \alpha\in \Delta\}\setminus\{0\}$.
In general, let $X^{*}(\alggrp{H})$ (resp.\ $X_{*}(\alggrp{H})$) be the rational character (resp.\ cocharacter) group of an algebraic group $\alggrp{H}$.
For $K \in \{\Q,\R,\C\}$ we let $\mathfrak{a}_{L,K}^* := X^*(\alggrp{A}_{\alggrp{L}}) \otimes K = X^*(\alggrp{L}) \otimes K$ and denote by $\mathfrak{a}_{L,K}$ its dual vector space.

If $\alggrp{L}_1 \subset \alggrp{L}_2$ are semistandard Levi subgroups we have $\alggrp{A}_{\alggrp{L}_2} \subset \alggrp{A}_{\alggrp{L}_1}$ and hence get canonical maps
$\mathfrak{a}_{L_1,\Q}^* \twoheadrightarrow \mathfrak{a}_{L_2,\Q}^*$ and $\mathfrak{a}_{L_2,\Q}^* \hookrightarrow \mathfrak{a}_{L_1,\Q}^*$ giving a splitting $\mathfrak{a}_{L_1,\Q}^* = \mathfrak{a}_{L_2,\Q}^* \oplus (\mathfrak{a}^{L_2}_{L_1,\Q})^*$
(and similarly over $\R$ and $\C$).
We note that if $\alggrp{P} = \alggrp{L}\alggrp{N}$ is semistandard, then $\Delta(\alggrp{P},\alggrp{A}_{\alggrp{L}})$ is a $\Q$-basis of $(\mathfrak{a}_{L,\Q}^G)^*$.

Given $\alpha \in \Phi_{\mathrm{red}}(\alggrp G, \alggrp A_{\alggrp L})$ we define the coroot $\alpha^{\vee} \in \mathfrak{a}_{L,\Q}$ as follows.
The subspace $(\mathfrak{a}_{L,\Q}^{L_\alpha})^*$ is one-dimensional with basis $\alpha$ and we let $\alpha^\vee$
be the unique element of $\mathfrak{a}_{L,\Q}^{L_\alpha}$ such that $\ang{\alpha,\alpha^{\vee}} = 2$.
Note that if $L = Z$, then this coincides with the usual notion of relative coroot ($\alpha^{\vee} \in X_*(\alggrp A_{\alggrp Z}) \subset \mathfrak{a}_{Z,\Q}$).

It is convenient to fix a $W$-invariant positive definite inner product on $\mathfrak{a}_{Z,\R}^*$. Then any decomposition $\mathfrak{a}_{L,\R}^* = \mathfrak{a}_{G,\R}^* \oplus (\mathfrak{a}_{L,\R}^G)^*$ is orthogonal (checking first for the pairs $(Z,G)$ and $(Z,L)$). Hence the isomorphism $\mathfrak{a}_{Z,\R} \congto \mathfrak{a}_{Z,\R}^*$ defined by the inner product identifies $\mathfrak{a}_{L,\R}$ with $\mathfrak{a}_{L,\R}^*$ and $\mathfrak{a}_{L,\R}^G$ with $(\mathfrak{a}_{L,\R}^G)^*$. In particular, taking $G = L_\alpha$ we see that $\alpha^{\vee} \in \mathfrak{a}_{L,\R}$ is identified with $2\alpha/(\alpha,\alpha) \in \mathfrak{a}_{L,\R}^*$.

\begin{lem}\label{lm:coroots}
  Suppose $\alggrp L_{1} \subset \alggrp L$ are semistandard Levi subgroups. 

  \begin{enumerate}
  \item Take $\alpha \in \Phi_{\mathrm{red}}(\alggrp G, \alggrp A_{\alggrp L_{1}})$ such that
    $\alpha_{L} = \alpha|_{\alggrp A_{\alggrp L}} \ne 0$.  Then the image of $\alpha^{\vee}$ under $\mathfrak{a}_{L_{1},\Q} \twoheadrightarrow \mathfrak{a}_{L,\Q}$ lies in
    $\Q_{>0}\alpha_{L}^{\vee}$.
  \item If $\alpha_1,\dots,\alpha_\ell \in \Phi_{\mathrm{red}}(\alggrp G, \alggrp A_{\alggrp L})$ are linearly independent in $\mathfrak{a}_{L_{1},\R}^*$, then $\alpha_1^{\vee},\dots,\alpha_\ell^{\vee}$
    are linearly independent in $\mathfrak{a}_{L_{1},\R}$.
  \end{enumerate}  
\end{lem}

\begin{proof}
  (i) The chosen inner product identifies the projection $\mathfrak{a}_{L_{1},\R} \twoheadrightarrow \mathfrak{a}_{L,\R}$ with the projection $\mathfrak{a}_{L_{1},\R}^* \twoheadrightarrow \mathfrak{a}_{L,\R}^*$.
  Hence the projection of $\alpha^\vee$ equals $\frac{(\alpha_{L},\alpha_{L})}{(\alpha,\alpha)}\alpha_{L}^\vee$.

  (ii) This is obvious by identifying $\mathfrak{a}_{L_{1},\R}$ and $\mathfrak{a}_{L_{1},\R}^*$ via the chosen inner product.
\end{proof}

\begin{lem}\label{lem:dominant-proj}
  Let $\alggrp{P} = \alggrp{L}\alggrp{N}$ be a standard parabolic subgroup.
  Write $\Delta_L := \Delta(\alggrp B \cap \alggrp L,\alggrp A_{\alggrp Z})$.
  \begin{enumerate}
  \item 
    If $\alpha\in \Delta\setminus\Delta_{L}$, then $\alpha^{\vee}\in \Q_{>0}\alpha_L^\vee \oplus \Q_{\le 0}\Delta_{L}^{\vee}$, where $\alpha_L := \alpha|_{\alggrp A_{\alggrp L}} \ne 0$.
  \item 
    Suppose that $x \in \mathfrak{a}_{Z,\R}^*$. 
    Write $x = x_L + x_L'$ with $x_L \in \mathfrak{a}_{L,\R}^*$ and $x_L' \in (\mathfrak{a}^L_{Z,\R})^*$.
    If $x$ is dominant, then $x_L$ is dominant.
  \end{enumerate}
\end{lem}

Here, $x_L \in \mathfrak{a}_{L,\R}^*$ is \emph{dominant} if $\ang{x_L,\alpha_L^\vee} \ge 0$ for all $\alpha_L \in \Phi(\alggrp P,\alggrp{A}_{\alggrp{L}})$, or equivalently for all $\alpha_L \in \Delta(\alggrp P,\alggrp{A}_{\alggrp{L}})$.
By Lemma~\ref{lm:coroots} this is equivalent to $x_{L}$ being dominant as an element in $\mathfrak{a}_{Z,\R}^{*}$. 

\begin{proof}
  (i) Write $\alpha^\vee = x+y$ with $x\in \mathfrak{a}_{L,\Q}$ and $y = \sum_{\beta \in \Delta_L} \lambda_\beta \beta^\vee \in \mathfrak{a}_{Z,\Q}^{L}$ with $\lambda_\beta \in \Q$.
  By Lemma~\ref{lm:coroots}(i) we have $x = \lambda \alpha_L^\vee$ with $\lambda \in \Q_{>0}$.
  On the other hand, for all $\gamma \in \Delta_L$ we get $0 \ge \ang{\gamma,\alpha^\vee} = \sum_{\beta \in \Delta_L} \lambda_\beta \ang{\gamma,\beta^\vee}$.
  The fundamental coweight $\varpi_\delta$ of $\delta \in \Delta_L$ in $(\mathfrak{a}_{Z,\R}^{L})^*$ is contained in $\sum_{\gamma \in \Delta_L} \R_{\ge 0} \gamma$, whence $0 \ge \lambda_\delta$ for all $\delta \in \Delta_L$.

  (ii) If $\alpha \in \Delta_L$, then $\alpha^\vee \in \Delta_L^\vee \subset \mathfrak{a}^L_{Z,\R}$, so $\ang{x_L',\alpha^\vee} = \ang{x,\alpha^\vee} \ge 0$.
  If $\alpha \in \Delta\setminus \Delta_L$, write $\alpha^\vee = \lambda \alpha_L^\vee + \sum_{\beta \in \Delta_L} \lambda_\beta \beta^\vee$ by (i), with $\lambda > 0$ and $\lambda_\beta \le 0$.
  Then $\lambda\ang{x_L,\alpha_L^\vee} = \ang{x,\alpha^\vee} - \sum_{\beta \in \Delta_L} \lambda_\beta \ang{x_L',\beta^\vee} \ge 0$, so $\ang{x_L,\alpha_L^\vee} \ge 0$.
\end{proof}

\subsection{On reducibility points of parabolic induction}
\label{sec:reduc-points-parab}
Let $\alggrp{P} = \alggrp{L}\alggrp{N}$ be a parabolic subgroup.
By Corollary~\ref{cor:irreduciblity}, to prove $(\Ind_{P}^{G}\sigma)^{\cts}$ is irreducible, it is sufficient to prove that $(\Ind_{P \cap L_Q}^{L_Q}\tau)^{\sm}$ is irreducible for a certain parabolic subgroup $\alggrp Q$ containing $\alggrp P$, 
at least whenever $\sigma$ can be decomposed as tensor product $\sigma_{0}\otimes\tau$ as in subsection~\ref{subsec:A criterion} (for example, when $\alggrp G^\der$ is simply connected).
We collect some known facts about the reducibility of smooth parabolic inductions over $\C$.

Let $\sigma_0$ be an irreducible smooth \emph{complex} representation of $L$ and $\mathcal{O}_{\C}$ be the set of isomorphism classes of $\sigma_0\otimes\chi$, where $\chi\colon L\to \C^{\times}$ is an unramified character.
We have Harish-Chandra's homomorphism $H_{L} \colon L \to \mathfrak{a}_{L,\R}$ normalized by $q^{\ang{\chi,H_{L}(\ell)}} = \lvert\chi(\ell)\rvert_F$ for all $\chi \in X^*(\alggrp L)$ and $\ell \in L$.
Then for $\nu \in \mathfrak{a}_{L,\C}^*$ we define the unramified character $\chi_\nu \colon L \to \C^{\times}$ by $\chi_\nu(\ell) := q^{\ang{\nu,H_{L}(\ell)}}$.
Then the map $\nu \mapsto \chi_\nu$ identifies the group of unramified characters $\xnr(L)$ with the quotient of $\mathfrak{a}_{L,\C}^{*}$ by a lattice in $i\mathfrak{a}_{L,\R}^{*}$
(a complex torus with character group $L/\ker H_L$).
In this way $\mathcal{O}_{\C}$ has the structure of an algebraic variety over $\C$ (a homogeneous space for $\xnr(L)$ with finite stabilizer subgroups).

Let $\alggrp{Q} = \alggrp{L}\alggrp{N}_{\alggrp Q}$ be a semistandard parabolic subgroup which has the same Levi part as $P$.
For $\sigma\in \mathcal{O}_{\C}$, we have an intertwining operator $J_{Q|P}(\sigma)\colon (\nInd_{P}^{G}\sigma)^{\sm}\to (\nInd_{Q}^{G}\sigma)^{\sm}$ defined by
\[
(J_{Q|P}(\sigma)f)(g) = \int_{(N \cap N_{Q})\backslash N_{Q}}f(ng)dn.
\]
It converges if (the unramified part of) $\sigma$ is sufficiently dominant, and has meromorphic continuation to $\mathcal{O}_{\C}$.
In fact it is a rational function on $\mathcal{O}_{\C}$, see~\cite[Th\'eor\`eme~IV.1.1]{MR1989693}.

The definition of $J_{Q|P}$ depends on a choice of Haar measure.
Here we fix a measure as follows: we have a bijective map $\prod_{\alpha\in \Phi_{\red}(\alggrp{Q},\alggrp{A}_{\alggrp{L}}) \setminus \Phi_{\red}(\alggrp{P},\alggrp{A}_{\alggrp{L}})}(N_Q \cap L_{\alpha})\congto (N \cap N_{Q})\backslash N_{Q}$, where we fix an order of $\Phi_{\red}(\alggrp{Q},\alggrp{A}_{\alggrp{L}}) \setminus \Phi_{\red}(\alggrp{P},\alggrp{A}_{\alggrp{L}})$.
For each $\alpha\in\Phi_{\red}(\alggrp{Q},\alggrp{A}_{\alggrp{L}}) \setminus \Phi_{\red}(\alggrp{P},\alggrp{A}_{\alggrp{L}})$ we fix a Haar measure on $N_Q \cap L_{\alpha}$ and take the product measure on $(N \cap N_{Q})\backslash N_{Q}$.
We will fix more specific measures on $N_{Q}\cap L_{\alpha}$ later.

The most important case is when $\sigma_0$ is a \emph{discrete series} and $\alggrp{Q} = \overline{\alggrp{P}}$.
The set of $\sigma\in \mathcal{O}_{\C}$ such that $(\nInd_{P}^{G}\sigma)^{\sm}$ is irreducible is open and non-empty~\cite[Proposition~IV.2.2]{MR1989693}.
Hence there exists a rational function $j(\sigma)$ such that $J_{P|\overline{P}}(\sigma)J_{\overline{P}|P}(\sigma) = j(\sigma)$.
Note that $j(\sigma)$ does not depend on $P$~\cite[IV.3(1)]{MR1989693}.
We define Harish-Chandra's rational function $\mu^{G}(\sigma)$ by the same formula as in \cite[V.2]{MR1989693}. (In \cite[V.2]{MR1989693}, $\sigma$ is assumed to be unitary, however the definition works for any $\sigma$ and gives a rational function on $\mathcal{O}_{\C}$.)
We have $\mu^G(\sigma) \in \R_{>0}^\times \cdot j(\sigma)^{-1}$, where the implied constant only depends on $(\alggrp{G},\alggrp{L})$.
We also have $\mu^G(\sigma) \ge 0$ for all unitary $\sigma \in \mathcal{O}_{\C}$ by \cite[Lemme V.2.1]{MR1989693}.
Finally, it is clear that $\mu^G$ is a rational function on the quotient $\xnr(G)\backslash \mathcal{O}_{\C}$.
The function $j(\sigma)$ depends on the choice of measure, but $\mu^G(\sigma)$ does not depend on it.
The function $\mu^G$ gives very precise information about the reducibility points of $(\nInd_{P}^{G}\sigma)^{\sm}$ for $\sigma$ supercuspidal.

\begin{remark}
We normalize $\mu^G$ differently compared to \cite{MR544991}. However, the normalizations agree up to a factor in $\R_{>0}^\times$. This follows from the comparison of \cite[Theorem~5.2.4.4]{MR544991} (noting that $\mu(\omega) = \mu(\omega:0)$ in that reference) and \cite[Lemme~V.2.2]{MR1989693}.
\end{remark}

\begin{prop}[Harish-Chandra's product formula]\label{prop:product formula}
If $\sigma$ is a discrete series we have \[\mu^{G}(\sigma) = \prod_{\alpha\in\Phi_{\red}(\alggrp{P},\alggrp{A}_{\alggrp{L}})}\mu^{L_{\alpha}}(\sigma).\]
\end{prop}
\begin{proof}
This is true for $\sigma\in \mathcal{O}_{\C}$ unitary~\cite[Lemme~V.2.1]{MR1989693} and hence for all $\sigma\in \mathcal{O}_{\C}$ since both sides are rational functions.
\end{proof}

Suppose that $P$ is maximal and $\sigma$ is a unitary supercuspidal representation of $L$.
Note that the group $N_{G}(L)/L$ has at most two elements, and let $W_{G}(\sigma) := \{ g \in N_{G}(L)/L : \sigma \circ \Ad(g)\simeq \sigma \}$.
Also note that $\xnr(G)\backslash \mathcal{O}_{\C}$ is a torus of rank 1 (without fixed base point).
Define $2\rho_P \in X^*(\alggrp A_{\alggrp L})$ (as sum of the roots in $\Phi(\alggrp{P},\alggrp{A}_{\alggrp{L}})$, with multiplicities) such that $\delta_P = \chi_{2\rho_P}$.

\begin{prop}\label{prop:mu-max-parab}
  Keep the above notation.

  \begin{enumerate}
  \item If $W_{G}(\sigma) = 1$, then $(\nInd_{P}^{G}\sigma\chi)^{\sm}$ is irreducible for all unramified $\chi\colon L\to \R_{>0}^\times$.
      \item Otherwise, there exists a unique $0 \le s_0 \le 1/2$ such that $(\nInd_{P}^{G}\sigma\delta_P^{s})^{\sm}$ ($s \in \R$) is reducible if and only if $s \in \{\pm s_0\}$.
  \item 
    In case (i) and case (ii) when $s_0 = 0$, the function $\mu^{G}$ is holomorphic and non-vanishing at all $\sigma\chi$ with $\chi\colon L\to \R_{>0}^\times$ unramified.
  \item 
    In case (ii) when $s_0 > 0$, the function $\mu^{G}(\sigma\delta_P^s)$ has a double zero at $s = 0$, simple poles at $s = \pm s_0$, and is holomorphic and non-vanishing at all other $s \in \R$.
          \end{enumerate}
\end{prop}

\begin{proof}
  First suppose that $W_{G}(\sigma) = 1$. Then (iii) follows from \cite[Corollary 5.4.2.2]{MR544991} and \cite[Lemma 1.3]{MR577138}.
  Then (i) follows from \cite[Lemma 5.4.2.4]{MR544991} for $\chi$ that do not extend to $G$.
  If $\chi$ extends to $G$, by twisting we may suppose $\chi = 1$ and we may suppose that $N_{G}(L)/L$ has order 2 by \cite[Theorem 5.4.4.1]{MR544991}.
  Then (i) follows by combining Lemmas 5.4.5.2, 5.4.1.5 of \cite{MR544991}.
  (A different proof of (i) can be found in \cite[Theorem 28]{bernstein}.)

  Now suppose that $W_{G}(\sigma)$ has order 2.
  If $\mu^G(\sigma) > 0$, then $\mu^{G}$ is holomorphic and non-vanishing at all $\sigma\chi$ with $\chi\colon L\to \R_{>0}^\times$ unramified
  by \cite[Lemma 1.3]{MR577138}.
  Otherwise, $\mu^G(\sigma) = 0$.     We have a commutative square
  \begin{equation*}
    \xymatrix{
      (\mathfrak{a}_{G,\C})^* \ar@{^{(}->}[r]\ar@{->>}[d] & (\mathfrak{a}_{L,\C})^* \ar@{->>}[d] \\ 
      \xnr(G) \ar@{^{(}->}[r] & \xnr(L)
    }
  \end{equation*}
  where the vertical maps are given by $\nu \mapsto \chi_\nu$.
    Hence the kernel of the map $(\mathfrak{a}_{L,\C}^G)^* \onto \xnr(G)\backslash \mathcal{O}_{\C}$, $\nu \mapsto \sigma\chi_\nu$ is a lattice $L^*(\sigma)$ (of rank 1) in $\sqrt{-1}(\mathfrak{a}_{L,\R}^G)^*$.   There is a unique element $\alpha(\sigma) \in \mathfrak{a}_{L,\R}^G$ such that $q^{\ang{\nu,\alpha(\sigma)}} = 1$ if and only if $\nu \in L^*(\sigma)$
  and $\ang{\rho_P,\alpha(\sigma)} > 0$.
  We let $z := q^{\ang{\nu,\alpha(\sigma)}}$ (a generator of the character lattice of the torus $\xnr(G)\backslash \mathcal{O}_{\C}$).
  Then $\alpha(\sigma)$ and $z$ agree with the ones defined in \cite{MR577138}, except that our $\nu$ becomes $\sqrt{-1}\nu$ in \cite{MR577138}.
  From \cite[Theorem 1.6]{MR577138} and $\delta_P^s = \chi_{2s\rho_P}$ we deduce that $\mu^G(\sigma) = 0$ implies that $\mu^G(\sigma\delta_P^s)$ is as described in part (iv), for some $0 < s_0 \le 1/2$.

  Finally, part (ii) follows from parts (iii) and (iv) using \cite[Lemma 5.4.2.3]{MR544991} (when $s = 0$) and \cite[Lemma 5.4.2.4]{MR544991} (when $s \ne 0$).
\end{proof}

The following comparison of $\mu$-functions will often be useful, especially in combination with Proposition~\ref{prop:mu-max-parab}.

\begin{prop}[{\cite[Proposition 2.2]{solleveld}}]\label{prop:solleveld}
  Let $\varphi\colon \alggrp G_{1}\to \alggrp G$ be a morphism such that $\varphi(\alggrp{G}_{1}^{\der}) = \alggrp{G}^{\der}$ and $\ker\varphi\subset \alggrp{Z}_{\alggrp G_{1}}$.
  Let $\sigma$ be a unitary supercuspidal representation of $L$ and $\sigma_1$ an irreducible constituent of the inflation $\varphi^*(\sigma)$, a unitary supercuspidal representation of $L_1 := \varphi^{-1}(L)$.
  Then $\mu^L(\sigma \chi)$ and $\mu^{L_1}(\sigma_1 \varphi^*(\chi))$ agree up to nonzero constant as rational functions of $\chi \in \xnr(L)$.
\end{prop}

The following result is crucial to us. Its proof was provided to us by J.-L.\ Waldspurger.

\begin{prop}[Waldspurger]\label{prop:Waldspurger}
Assume that $\alggrp{P}$ is a maximal parabolic subgroup, $\overline{\alggrp{P}} = \alggrp{L}\overline{\alggrp{N}}$ the parabolic subgroup opposite to $\alggrp{P}$ and $\langle N,\overline{N}\rangle$ the group generated by $N$ and $\overline{N}$.
Let $\sigma$ be a unitary supercuspidal representation of $L$ and assume that $(\nInd_{P}^{G}\sigma \delta_{P}^{1/2})^{\sm}$ is reducible.
Then $\sigma$ is trivial on $L\cap \langle N,\overline{N}\rangle$.
\end{prop}
The final statement implies that $\sigma$ extends to a smooth representation of $G$ that is trivial on $N$ by~\cite[II.7 Proposition]{AHHV}.
(To see this, choose a minimal parabolic $\alggrp{B} = \alggrp{Z}\alggrp{U} \subset \alggrp{P}$ and note then that $Z \cap L_\beta'$ is contained in $L \cap \langle N,\overline{N}\rangle$
for all $\beta \in \Delta$ that do not occur in $L$.)
In that case $(\nInd_{P}^{G}\sigma \delta_{P}^{-1/2})^{\sm} = (\Ind_{P}^{G}\sigma)^{\sm}$ is obviously reducible.
So by Proposition~\ref{prop:mu-max-parab} the converse of the proposition is true as well.
\begin{proof}
By Proposition~\ref{prop:mu-max-parab} there exists $s\in N_{G}(L)\setminus L$ and $\sigma\circ\Ad(s)\simeq \sigma$.
Let $\alpha$ be the unique element of $\Delta(\alggrp{P},\alggrp{A}_{\alggrp{L}})$.
Let $r_{P}(\pi) := \pi_{N}\delta_{P}^{-1/2}$ be the normalized Jacquet module.
Then by the geometric lemma, $r_{P}((\nInd_{P}^{G}\sigma\delta_{P}^{1/2})^{\sm})\simeq \sigma\delta_{P}^{-1/2}\oplus \sigma\delta_{P}^{1/2}$.
Since the supercuspidal support of any subquotient of $(\nInd_{P}^{G}\sigma\delta_{P}^{1/2})^{\sm}$ is $L$, any subquotient has a nonzero Jacquet module.
Hence $(\nInd_{P}^{G}\sigma\delta_{P}^{1/2})^{\sm}$ has length two and the normalized Jacquet modules of the irreducible subquotients are $\sigma\delta_{P}^{-1/2}$ and $\sigma\delta_{P}^{1/2}$, respectively.
By Casselman's criterion of square-integrability~\cite[Theorem~4.4.6]{Casselman-note}, one of them is square-integrable and let $\pi$ be the other irreducible subquotient.
Then $r_{P}(\pi)$ is isomorphic to $\sigma\delta_{P}^{-1/2}$.
Let $\theta\colon \pi\to \pi_{N}\simeq \sigma$ be the natural projection.

Let $K_{0}\subset G$ be a compact open subgroup which has Iwahori decomposition $K_{0} = (K_{0}\cap \overline{N})(K_{0}\cap L)(K_{0}\cap N)$ and $\sigma^{K_{0}\cap L} \ne 0$. Normalize Haar measure on $K_{0}$ such that the volume of $K_{0}$ is $1$ and set $e_{K_{0}} := \int_{K_{0}}\pi(k)dk$.
For $c > 0$ put $A_{L}^{c,-} := \{a\in A_{L}\mid \lvert \alpha(a)\rvert_F \le c\}$.
Then by a result of Casselman there exists $c\le 1$ such that 
\begin{enumerate}
\item for any $a\in A_{L}^{c,-}$ the space $e_{K_{0}}(\pi(a)\pi^{K_{0}})$ does not depend on $a$~\cite[Proposition~4.1.6]{Casselman-note}; we denote it by $\pi^{K_{0}}_{A_{L}^{c,-}}$;
\item the map $\theta$ gives an isomorphism $\pi^{K_{0}}_{A_{L}^{c,-}}\xrightarrow{\sim}\sigma^{K_{0}\cap L}$~\cite[Proposition~4.1.4]{Casselman-note};
\item for $a\in A_{L}^{c,-}$, $e_{K_{0}}\circ\pi(a)$ preserves $\pi_{A_{L}^{c,-}}^{K_{0}}$ and for $v\in \pi_{A_{L}^{c,-}}^{K_{0}}$, we have $\theta(e_{K_0}\circ \pi(a)v) = \sigma(a)\theta(v)$~\cite[Lemma~4.1.1]{Casselman-note}.
\end{enumerate}
Let $\omega$ be the central character of $\sigma$.
Then $\omega$ is unitary and the conclusion in (iii) can also be written as $\theta(e_{K_{0}}\circ \pi(a)v) = \omega(a)\theta(v)$ for any $a\in A_{L}^{c,-}$.

By a result of Tadi\'c~\cite{MR976070}, $\pi$ is unitary.
Fix a nonzero $G$-invariant inner product $(\cdot,\cdot)$ on $\pi$.
Let $v\in \pi_{A_{L}^{c,-}}^{K_{0}}$ be a nonzero element and $a\in A_{L}^{c,-}$.
Then by $(e_{K_{0}}\circ\pi(a)v,v) = \omega(a)(v,v)$, we have
\[
\int_{K_{0}}(\pi(ka)v,v)\omega(a)^{-1}(v,v)^{-1}dk = 1.
\]
By the Cauchy--Schwarz inequality and since $\pi,\omega$ are unitary, we have \[\lvert (\pi(ka)v,v)\omega(a)^{-1}(v,v)^{-1}\rvert \le 1.\]
Hence $(\pi(ka)v,v)\omega(a)^{-1}(v,v)^{-1} = 1$ for any $k\in K_{0}$.
In particular $(\pi(a)v,v) = \omega(a)(v,v)$.
Again by the Cauchy--Schwarz inequality, we get $\pi(a)v = \omega(a)v$ for any $a\in A_{L}^{c,-}$.
The subset $A_{L}^{c,-}$ generates $A_{L}$ as a group.
Hence $\pi(a)v = \omega(a)v$ for any $a\in A_{L}$.

Let $n \in N$.
Then there exists $a\in A_{L}$ such that $ana^{-1}$ fixes $v$.
As $\pi(a)v = \omega(a)v$, we have $\pi(n)v = v$.
By the same argument $\pi(\overline{n})v = v$ for any $\overline{n}\in\overline{N}$.
Therefore $\pi^{\langle N,\overline{N}\rangle}\ne 0$.
Since $\langle N,\overline{N}\rangle$ is normalized by $L$, it is also normalized by $G = \ang{N,L,\overline{N}}$.
Hence $\pi$ is trivial on $\langle N,\overline{N}\rangle$ by irreducibility of $\pi$.
Since $\sigma\simeq \pi_{N}$, $\sigma$ is trivial on $L\cap \langle N,\overline{N}\rangle$.
\end{proof}

Finally, the following criterion will be useful.
We say that $\sigma$ is \emph{$G$-regular} if for $g\in N_{G}(L)\setminus L$, we have $\sigma\circ\Ad(g)\not\simeq \sigma$.
\begin{prop}[{\cite[Theorem~5.4.3.7]{MR544991}}]\label{prop:G-regular}
If $\sigma$ is supercuspidal and $G$-regular, then $(\nInd_{P}^{G}\sigma)^{\sm}$ is reducible if and only if $\mu^G$ has a pole at $\sigma$.
\end{prop}

\subsection{Split groups}
\label{sec:split-groups}

We now prove Theorems~\ref{thm:GLn-irred-intro} and \ref{thm:split-groups-intro}.

\begin{thm}\label{thm:GLn-irred}
  Let $G = \GL_{n}(F)$, $B$ the upper-triangular Borel subgroup, and $Z$ the diagonal maximal torus.
  Let $\chi = \chi_{1}\otimes\cdots\otimes\chi_{n}\colon Z = (F^\times)^n \to C^{\times}$ be a continuous (hence locally $\Q_{p}$-analytic) character.
  We have $d\chi\in \Hom_{\Q_{p}}(\mathfrak{z},C)\simeq \bigoplus_{\kappa\colon F\to C}\Hom_{C}(\mathfrak{z} \otimes_{F,\kappa} C,C)$ and let $\lambda_{\kappa} = (\lambda_{\kappa,1},\ldots,\lambda_{\kappa,n})$ be the $\kappa$-component of $d\chi$, where $\lambda_{\kappa,k}\in \Hom_{C}(C,C)\simeq C$.
  Assume that there exists no $1\le i < j \le n$ such that 
  \begin{itemize}
  \item $\lambda_{\kappa,k} - \lambda_{\kappa,k + 1}\in \Z_{\le 0}$ for any $k = i,\ldots,j - 1$ and $\kappa\colon F\to C$;
  \item $\chi_{i}\chi_{j}^{-1}(t) = \lvert t\rvert_F^{j - i - 1}\prod_{\kappa\colon F\to C}\kappa(t)^{\lambda_{\kappa,i} - \lambda_{\kappa,j}}$ for all $t\in F^{\times}$.
  \end{itemize}
  Then $(\Ind_{B}^{G}\chi)^{\cts}$ is absolutely irreducible.
\end{thm}
\begin{proof}
We have $\alggrp{G} = \Res_{F/\Q_{p}}\GL_{n}$ and the subgroup $\alggrp{S}$ is the diagonal split torus of rank $n$ over $\Q_p$. Let $e_{i}\colon \alggrp{S}\to \mathbb{G}_{m}$ be the character projecting to the $i$-th entry.
Then the set of simple roots is $\{e_{i} - e_{i + 1}\mid 1\le i\le n - 1\}$.
Set $I := \{e_{i} - e_{i + 1}\mid \text{$\lambda_{\kappa,i} - \lambda_{\kappa,i + 1}\in \Z_{\le 0}$ for all $\kappa\colon F\to C$}\}$ and $\alggrp Q$ the standard parabolic subgroup corresponding to $I$. Then $\alggrp Q$ is maximal subject to $L(-d\chi)\in \mathcal{O}^{\mathfrak{q}}$.
Take a locally $\Q_{p}$-analytic character $\sigma_{0} = \sigma_{0,1}\otimes\cdots\otimes \sigma_{0,n}\colon Z\to C^{\times}$ and a smooth character $\tau\colon Z\to C^{\times}$ such that $\chi = \sigma_{0}\tau$ and if $1 \le i\le n-1$ and $e_{i} - e_{i + 1} \in I$ then $(\sigma_{0,i}\sigma_{0,i + 1}^{-1})(t) = \prod_{\kappa\colon F\to C}\kappa(t)^{\lambda_{\kappa,i} - \lambda_{\kappa,i + 1}}$.
Then $\underline{L}(\sigma_{0}')\in \mathcal{O}^{Q}$ by using Lemma~\ref{lem:simple-in-O-Q}.
Namely, $\chi = \sigma_{0}\tau$ is the decomposition given in section~\ref{subsec:A criterion}.
By our hypothesis, $\tau_i\tau_j^{-1} \ne \lvert \cdot\rvert_F^{j - i - 1}$ whenever $i < j$ and $-(e_i - e_j)$ is a root of $Q$.

We now check the assumption of Corollary~\ref{cor:whittaker criterion}.
By writing $L_Q$ as a product of general linear groups, we may assume $I = \{e_{i} - e_{i + 1}\mid 1\le i \le n - 1\}$ (i.e.\ $Q = G$).
Then it follows from \cite[4.11 Theorem]{MR0579172} that every nonzero subrepresentation of $(\Ind_{B}^{G}\tau)^{\sm}\otimes_{C}\overline{C}$ is generic.
\end{proof}

\begin{thm}\label{thm:split-groups}
  Assume Assumption~\ref{assump:on p}.
  Let $\alggrp G$ be split.
  Let $\chi\colon Z \to C^{\times}$ be a continuous (hence locally $\Q_{p}$-analytic) character.
  We have $d\chi\in \Hom_{\Q_{p}}(\mathfrak{z},C)\simeq \bigoplus_{\kappa\colon F\to C}\Hom_{C}(\mathfrak{z} \otimes_{F,\kappa} C,C)$ and let $d\chi_{\kappa} \in X^*(\underline Z) \otimes C$  be the $\kappa$-component of $d\chi$.
  Let $\alggrp P = \alggrp L \alggrp N$ be the largest standard parabolic subgroup such that $\ang{d\chi_\kappa,\alpha^\vee} \in \Z_{\le 0}$ for all positive roots $\alpha$ of $\alggrp L$ and all $\kappa \colon F \to C$.
  Assume that for all $w \in N_L(Z) \setminus Z$ there exists a root $\alpha$ of $\alggrp L$ such that
  \begin{equation}\label{eq:split-condition1}
    \text{$(\chi \delta_B^{-1/2} \circ w^{-1}\alpha^\vee)\cdot (\chi \delta_B^{-1/2} \circ \alpha^\vee)^{-1}$ is non-algebraic}
  \end{equation}
  and that there exists no positive root $\alpha$ of $\alggrp L$ such that 
  \begin{equation}\label{eq:split-condition2}
    \chi \delta_B^{-1/2} \circ \alpha^\vee = \lvert \cdot\rvert_F^{-1}\prod_{\kappa\colon F\to C}\kappa(\cdot)^{\ang{d\chi_\kappa,\alpha^\vee}}.
  \end{equation}
  Then $(\Ind_{B}^{G}\chi)^{\cts}$ is absolutely irreducible.
\end{thm}

\begin{proof}
  Assume first that $\alggrp G^\der$ is simply connected.
  We will work over $\Q_p$ by letting $\wt{\alggrp G} := \Res_{F/\Q_p} \alggrp G$, and likewise for $\wt{\alggrp B}$, etc.
  Then $\wt{\alggrp P}$ is the maximal standard parabolic subgroup of $\wt{\alggrp G}$ such that $L(-d\chi) \in \mathcal O^{\wt{\mathfrak p}}$.
  By Lemma~\ref{lem:sigma0-tau-decomp} we write $\chi = \sigma_0 \tau$ with $\underline L(\sigma_0') \in \mathcal O^{\wt P}$ and $\tau$ smooth.
  By Lemma~\ref{lem:simple-in-O-Q} we have $\sigma_0 = \sigma_\alg (\psi|_Z)$ with $\sigma_\alg \colon Z \to C^\times$ algebraic and $\psi \colon L \to C^\times$ locally analytic.
  
  Fix $\alpha$ a positive root (of $\alggrp Z = \alggrp S$) in $\alggrp L$. We claim that equation~\eqref{eq:split-condition2} is equivalent to
  \begin{equation}\label{eq:split-condition2b}
    \tau \delta_{B \cap L}^{-1/2} \circ \alpha^\vee = \lvert \cdot\rvert_F^{-1}.
  \end{equation}
  Note that $\psi \circ \alpha^\vee = 1$ by Lemma \ref{lm:loc-an-char}.
  In particular, $\sigma_\alg \circ \alpha^\vee$ is an algebraic character of $F^\times$ with derivative $d\sigma_\alg \circ \alpha^\vee = d\chi \circ \alpha^\vee$, so 
  $\sigma_\alg \circ \alpha^\vee = \prod_{\kappa\colon F\to C}\kappa(\cdot)^{\ang{d\chi_\kappa,\alpha^\vee}}$.
  Finally note that $\delta_B = \delta_{B \cap L}\delta_N$ and $\delta_N \circ \alpha^\vee = 1$.
  We deduce the claim.
  Likewise fix $w \in N_L(Z)$. Then $w\psi = \psi$ and $w\delta_N = \delta_N$ for all $w \in N_L(Z)$, so for $w \notin Z$ equation~\eqref{eq:split-condition1} becomes
  \begin{equation}\label{eq:split-condition1b}
    \tau \delta_{B \cap L}^{-1/2}\circ w^{-1}\alpha^\vee \ne \tau \delta_{B \cap L}^{-1/2}\circ \alpha^\vee
  \end{equation}
  for some root $\alpha$ of $\alggrp L$, which implies that $w(\tau \delta_{B \cap L}^{-1/2}) \ne \tau \delta_{B \cap L}^{-1/2}$.

  Using equations~\eqref{eq:split-condition2b}, \eqref{eq:split-condition1b} we deduce that $(\Ind_{B \cap L}^L \tau)^\sm \otimes_C \o C \cong (\nInd_{B \cap L}^L \tau\delta_{B \cap L}^{-1/2})^\sm\otimes_C \o C$ has an irreducible socle that is generic by \cite[Proposition 1, Proposition 4]{MR644842}.
  (Alternatively we could argue using Propositions~\ref{prop:qsplit-irred-gen-socle}, \ref{prop:mu-max-parab}, and \ref{prop:G-regular}, using the known reducibility points for $\SL_2(F)$.)
  We conclude by Corollary~\ref{cor:whittaker criterion}.

  For general $\alggrp G$, let $\varphi \colon \alggrp G^\sc \to \alggrp G$ be the simply-connected cover of the derived subgroup.
  Let $\alggrp Z^\sc := \varphi^{-1}(\alggrp Z)$ and $\chi^\sc := \chi \circ \varphi$.
  As $\varphi$ is compatible with coroots of $\alggrp G^\sc$ and $\alggrp G$ we deduce that $d\chi^\sc$ determines the parabolic subgroup $\varphi^{-1}(\alggrp P)$ of $\alggrp G^\sc$ by the recipe in the statement of the theorem.
  Moreover, conditions~\eqref{eq:split-condition1} and \eqref{eq:split-condition2} are the same for $\alggrp G^\sc$ and $\alggrp G$.
  We conclude by Proposition~\ref{prop:isogenies} (see Remark~\ref{rk:isogenies}) and by what we already established for $\alggrp G^\sc$.
\end{proof}

\begin{remark}
  This theorem generalizes to all quasisplit groups, by using Propositions~\ref{prop:qsplit-irred-gen-socle}, \ref{prop:mu-max-parab}, \ref{prop:G-regular}, and the known reducibility points for $\SL_2(F)$ and $\SU_3(F)$.
  However, it can no longer be stated in terms of relative coroots.
  For general $\alggrp G$ one can at least formulate a weaker version in the same way, using Remark~\ref{rk:irred-smooth-princ-series} instead of Proposition~\ref{prop:qsplit-irred-gen-socle}, as well as Remark~\ref{rk:rank-1-poles} for the possible location of reducibility points.
\end{remark}

\subsection{Classical quasisplit groups}
\label{subsec:classical quasisplit groups}
We will now give some irreducibility theorems for all quasisplit classical groups.
For simplicity we will state them for a smooth inducing character only, which is enough in light of Theorem~\ref{thm:equivalence-irred} (at least for all but the orthogonal groups).
We use standard conventions for classical groups, as for example in \cite{MR1296726}, \cite{MR1224616}, \cite{MR3430367}.

\begin{thm}\label{thm:Sp2n-irred}
  Let $G = \Sp_{2n}(F)$ (split), $B$ the upper-triangular Borel, and $Z$ the diagonal maximal torus.
  Let $\chi = \chi_{1}\otimes\cdots\otimes\chi_{n}\colon Z\to C^{\times}$ be a smooth character.
  Assume $p > 2$ and the following:
  \begin{itemize}
  \item $\chi_i \chi_j^{\pm 1} \ne \lvert \cdot\rvert_F^{-1}$ for all $i < j$;
  \item $\chi_i \ne \lvert \cdot\rvert_F^{-1}$ for all $i$;
  \item the set $\{\chi_i : \text{$\chi_i^2$ has order two} \}$ is linearly independent over $\Z/2\Z$.   \end{itemize}
  Then $(\Ind_{B}^{G}\chi\delta_B^{1/2})^{\cts}$ is absolutely irreducible.
\end{thm}

Note that the last condition is equivalent to saying that there do not exist (at least two) \emph{distinct} $\chi_i$ of order 2 whose product is trivial.

\begin{thm}\label{thm:SO2n+1-irred}
  Let $G = \SO_{2n+1}(F)$ (split), $B$ the upper-triangular Borel, and $Z$ the diagonal maximal torus.
  Let $\chi = \chi_{1}\otimes\cdots\otimes\chi_{n}\colon Z\to C^{\times}$ be a smooth character.
  Assume $p > 2$ and the following:
  \begin{itemize}
  \item $\chi_i \chi_j^{\pm 1} \ne \lvert \cdot\rvert_F^{-1}$ for all $i \le j$.
  \end{itemize}
  Then $(\Ind_{B}^{G}\chi\delta_B^{1/2})^{\cts}$ is absolutely irreducible.
\end{thm}

\begin{thm}\label{thm:SO2n-irred}
  Let $G = \SO_{2n}(F)$ (split), $B$ the upper-triangular Borel, and $Z$ the diagonal maximal torus.
  Let $\chi = \chi_{1}\otimes\cdots\otimes\chi_{n}\colon Z\to C^{\times}$ be a smooth character.
  Assume the following:
  \begin{itemize}
  \item $\chi_i \chi_j^{\pm 1} \ne \lvert \cdot\rvert_F^{-1}$ for all $i < j$;
  \item there does not exist a subset of $\{\chi_i : \text{$\chi_i^2$ has order two} \}$ of even size whose product is 1.
  \end{itemize}
  Then $(\Ind_{B}^{G}\chi\delta_B^{1/2})^{\cts}$ is absolutely irreducible.
\end{thm}

Note that the last condition allows many $\chi_i^2$ to equal 1. (Similarly below.)

\begin{thm}\label{thm:SO2n-qsplit-irred}
  Let $G = \SO^*_{2n}(F)$ (non-split quasisplit) splitting over a quadratic extension $E/F$, $B$ the upper-triangular Borel, and $Z$ the diagonal maximal torus.
  Let $\chi = \chi_{1}\otimes\cdots\otimes\chi_{n}\colon Z = (F^\times)^{n-1} \times (E^\times)^{N_{E/F}=1}\to C^{\times}$ be a smooth character.
  Assume the following:
  \begin{itemize}
  \item $\chi_i \chi_j^{\pm 1} \ne \lvert \cdot\rvert_F^{-1}$ for all $i < j < n$;
  \item $(\chi_i \circ N_{E/F}) \chi_n^{c-1} \ne \lvert\cdot\rvert_E^{-1}$ for all $i < n$ if $\chi_n^2 = 1$;
      \item 
    there does not exist a subset $\Sigma$ of $\{\chi_i : \chi_i^2=1, (\chi_i \circ N_{E/F}) \chi_n^{c-1} \ne 1 \}$ such that
    $\prod_{\chi_i \in \Sigma} \chi_i \circ N_{E/F} = \chi_n^{1-c}$ (resp.\ 1) if $|\Sigma|$ is odd (resp.\ $|\Sigma|$ is even).
      \end{itemize}
  Then $(\Ind_{B}^{G}\chi\delta_B^{1/2})^{\cts}$ is absolutely irreducible.
\end{thm}

Here, $\chi^{c-1}(x) := \chi(\o x x^{-1})$ for $x \in E^\times$.

\begin{thm}\label{thm:U2n-irred}
  Let $G = \U_{2n}(F)$ (quasisplit) splitting over a quadratic extension $E/F$, $B$ the upper-triangular Borel subgroup, and $Z$ the diagonal maximal torus.
  Let $\chi = \chi_{1}\otimes\cdots\otimes\chi_{n}\colon Z = (E^\times)^n \to C^{\times}$ be a smooth character.
  Assume the following:
  \begin{itemize}
  \item $\chi_i \chi_j^{-1} \ne \lvert \cdot\rvert_E^{-1}$ for all $i < j$;
  \item $\chi_i \chi_j^{c} \ne \lvert \cdot\rvert_E^{-1}$ for all $i < j$;
  \item $\chi_i \ne \eta \lvert \cdot\rvert_E^{-1/2}$ with $\eta|_{F^\times} = 1$ for all $i$;
  \item the set $\{\chi_i : \chi_i|_{F^\times} = \omega_{E/F} \}$ has at most one element.
  \end{itemize}
  Then $(\Ind_{B}^{G}\chi\delta_B^{1/2})^{\cts}$ is absolutely irreducible.
\end{thm}

Here, $\chi^c(x) := \chi(\o x)$ for $x \in E^\times$, and $\omega_{E/F}$ denotes the non-trivial character of $F^\times/N_{E/F}(E^\times)$.

\begin{thm}\label{thm:U2n+1-irred}
  Let $G = \U_{2n+1}(F)$ (quasisplit) splitting over a quadratic extension $E/F$, $B$ the upper-triangular Borel subgroup, and $Z$ the diagonal maximal torus.
  Let $\chi = \chi_{1}\otimes\cdots\otimes\chi_{n+1}\colon Z = (E^\times)^n \times (E^\times)^{N_{E/F}=1}\to C^{\times}$ be a smooth character.
  Assume the following:
  \begin{itemize}
  \item $\chi_i \chi_j^{-1} \ne \lvert \cdot\rvert_E^{-1}$ for all $i < j \le n$;
  \item $\chi_i \chi_j^{c} \ne \lvert \cdot\rvert_E^{-1}$ for all $i < j \le n$;
  \item $\chi_i \chi_{n+1}^{c-1} \ne \lvert \cdot\rvert_E^{-1}$ for all $i \le n$;
  \item $\chi_i \chi_{n+1}^{c-1} \ne \eta \lvert \cdot\rvert_E^{-1/2}$ with $\eta|_{F^\times} = \omega_{E/F}$ for all $i \le n$;
  \item $\chi_i \chi_{n+1}^{c-1} = 1$ or $\chi_i \chi_{n+1}^{c-1}|_{F^\times} \ne 1$ for all $i \le n$.
  \end{itemize}
  Then $(\Ind_{B}^{G}\chi\delta_B^{1/2})^{\cts}$ is absolutely irreducible.
\end{thm}

To prepare for the proofs, we first prove some lemmas about even orthogonal groups.

\begin{lem}\label{lem:SO4}
  Suppose $G = \SO^*_4(F)$ with diagonal maximal torus $Z = \GL_1(F) \times \SO^*_2(F) \cong F^\times \times (E^\times)^{N_{E/F}=1}$ and $\chi= \chi_1 \otimes \chi_2: Z \to \C^\times$.
  Write $\chi_1 = \psi_1 \lvert\cdot\rvert_F^s$ with $\psi_1$ unitary and $s \in \R$.
  Then the principal series $(\nInd_{B}^{G}\chi)^{\sm}$ is reducible if and only if $\psi_1^2 = \chi_2^2 = 1$ and either
  \begin{itemize}
  \item $\psi_1 \circ N_{E/F} = \chi_2^{1-c}$ and $s = \pm 1$, or
  \item $\psi_1 \circ N_{E/F} \ne \chi_2^{1-c}$ and $s = 0$.
  \end{itemize}\end{lem}

\begin{proof}
  The simply-connected cover of $\alggrp G$ is $\wt{\alggrp{G}} = \Spin^*_4 \cong \Res_{E/F} \SL_2$.
  Note that the non-trivial Weyl group element fixes $\psi_1 \otimes \chi_2$ if and only if $\psi_1^2 = \chi_2^2 = 1$.
  Using a root datum calculation we verify that $\chi$ pulls back to the character $\diag(x,x^{-1}) \mapsto \chi_1(x \o x) \chi_2(x^{-1}\o x)$ of the diagonal maximal torus of $\SL_2(E)$.
  We conclude by comparison with $\SL_2(E)$ using Propositions~\ref{prop:solleveld}, \ref{prop:mu-max-parab}.
\end{proof}

\begin{lem}\label{lem:GSO4}
  Suppose $G = \GSO^*_4(F)$ with diagonal maximal torus $Z = \GL_1(F) \times \GSO^*_2(F) \cong F^\times \times E^\times$ and $\chi= \chi_1 \otimes \chi_2: Z \to \C^\times$.
  Write $\chi_1 = \psi_1 \lvert\cdot\rvert_F^{s_1}$, $\chi_2 = \psi_2 \lvert\cdot\rvert_E^{s_2}$ with $\psi_i$ unitary and $s_i \in \R$.
  Then the principal series $(\nInd_{B}^{G}\chi)^{\sm}$ is reducible if and only if $\psi_1^2 = 1$, $\psi_1 \circ N_{E/F} = \psi_2^{1-c}$, and $s_1 = \pm 1$.
  \end{lem}

\begin{proof}
  The non-trivial Weyl group element fixes $\psi_1 \otimes \psi_2$ if and only if $\psi_1^2 = 1$ and $\psi_1 \circ N_{E/F} = \psi_2^{1-c}$.
  In his case we determine the reducibility point by restriction to $\SO^*_4(F)$ (using Propositions~\ref{prop:solleveld}, \ref{prop:mu-max-parab}, and Lemma~\ref{lem:SO4}).
\end{proof}

\begin{lem}\label{lem:R-group-GSO}
  Suppose $G = \GSO_{2n}(F)$.
  Let $\chi = \chi_{1}\otimes\cdots\otimes\chi_{n} \otimes \rho\colon Z\to \C^{\times}$ be a unitary smooth character.
  Then the $R$-group $R(\chi)$ is isomorphic to $(\Z/2\Z)^r$, where $2^r$ is the number of subsets $\Sigma \subset \{ \chi_i : \chi_i^2 = 1 \}$ such that $|\Sigma|$ is even
  and $\prod_{\chi_i \in \Sigma} \chi_i = 1$.
\end{lem}
\begin{proof}
  Keys' argument still applies as in \cite[Lemma 6.7]{MR1296726} to show that any element of $R(\chi)$ is a product of (an even number of) sign changes.
  We may use the Weyl group to assume that there exist $r_1 < r_2 < \cdots < r_k$ such that $\chi_i^2 = 1$ if and only if $i \le r_k$,
  $\chi_i = \chi_{i+1}$ for all $i < r_k$ such that $i \notin \{r_1,\dots,r_k\}$, and $\chi_{r_j}$ ($1 \le j \le k$) are pairwise distinct.
  Let $c_{i}$ be the element in the Weyl group which changes the $i$-th sign.
  Then it is straightforward to show that $R(\chi)$ consists of all elements $\prod_{i \in \Sigma'} c_i$, where $\Sigma' \subset \{r_1,\dots,r_k\}$, $|\Sigma'|$ is even, and $\prod_{i \in \Sigma'} \chi_i = 1$.
  This implies the result.
\end{proof}

\begin{lem}\label{lem:R-group-GSO-star}
  Suppose $G = \GSO_{2n}^*(F)$.
  Let $\chi = \chi_{1}\otimes\cdots\otimes\chi_{n} \colon Z = (F^\times)^{n-1} \times E^\times \to \C^{\times}$ be a unitary smooth character.
  Then the $R$-group $R(\chi)$ is isomorphic to $(\Z/2\Z)^r$, where $2^r$ is the number of subsets $\Sigma$ of $\{\chi_i : \chi_i^2=1, (\chi_i \circ N_{E/F}) \chi_n^{c-1} \ne 1 \}$ such that
      $\prod_{\chi_i \in \Sigma} \chi_i \circ N_{E/F} = \chi_n^{1-c}$ (resp.\ $\prod_{\chi_i \in \Sigma} \chi_i \circ N_{E/F} = 1$) if $|\Sigma|$ is odd (resp.\ $|\Sigma|$ is even).
\end{lem}
\begin{proof}
  By Lemma~\ref{lem:GSO4} we see that the set $\Delta'$ in the definition of the $R$-group \cite[\S1]{MR1296726} consists of
  all positive roots $e_i-e_j$ if $\chi_i=\chi_j$ ($i < j < n$), 
  $e_i+e_j$ if $\chi_i=\chi_j^{-1}$ ($i < j < n$), 
  $e_i$ if $\chi_i^2 = 1$ and $(\chi_i \circ N_{E/F})\chi_n^{c-1} = 1$ ($i < n$).
    Keys' argument still applies as in \cite[Lemma A.2]{MR3430367} to show that any element of $R(\chi)$ is a product of sign changes.
  We may use the Weyl group to assume that there exist $r_1 < r_2 < \cdots < r_k$ such that $\chi_i^2 = 1$ and $(\chi_i \circ N_{E/F})\chi_n^{c-1} \ne 1$ if and only if $i \le r_k$,
  $\chi_i = \chi_{i+1}$ for all $i < r_k$ such that $i \notin \{r_1,\dots,r_k\}$, and $\chi_{r_j}$ ($1 \le j \le k$) are pairwise distinct.
  Then it is straightforward to show that $R(\chi)$ consists of all elements $\prod_{i \in \Sigma'} c_i$, where $\Sigma' \subset \{r_1,\dots,r_k\}$, and 
  $\prod_{i \in \Sigma'} \chi_i \circ N_{E/F} = 1$ if $|\Sigma'|$ is even (resp.\ $\prod_{i \in \Sigma'} \chi_i \circ N_{E/F} = \chi_n^{1-c}$ if $|\Sigma'|$ is odd).
  This implies the result.
\end{proof}

\begin{prop}\label{prop:qsplit-irred-gen-socle}
  Suppose $\alggrp G$ is quasisplit and $\chi \colon Z \to \C^{\times}$ a smooth character.
  Write $\chi = \psi \chi_\nu$ with $\psi$ unitary and $\nu \in \mathfrak{a}^{*}_{Z,\R}$.
  Let $L$ be the maximal semistandard Levi subgroup such that $\nu \in \mathfrak{a}^{*}_{L,\R}$.
  Then the following two conditions are equivalent:
  \begin{enumerate}
  \item
    the socle of $(\nInd_{B}^{G} \chi)^{\sm}$ is of the same length as (the semisimple representation) $(\nInd_{B \cap L}^{L} \chi)^{\sm}$,
    and every irreducible subrepresentation of $(\nInd_{B}^{G} \chi)^{\sm}$ is generic;
  \item 
    \begin{enumerate}
    \item for all $\alpha \in \Phi^+_{\red}$ such that $(\nInd_{B \cap L_\alpha}^{L_\alpha} \chi)^{\sm}$ is reducible we have
      $\ang{\nu,\alpha^\vee} \ge 0$;
    \item every irreducible subrepresentation of $(\nInd_{B \cap L}^{L} \chi)^{\sm}$ is generic.
    \end{enumerate}
  \end{enumerate}
\end{prop}

\begin{proof}
  (We thank Alberto M\'inguez for providing the key ideas for this argument.)
  First assume that (i) holds.
  Note first that by~\cite[Theorem~2]{MR0354942} in the split case and by the geometric lemma \cite[5.2 Theorem]{MR0579172} in general, for each non-degenerate character $\theta$ of $U$ there is a unique $\theta$-generic irreducible constituent of $(\nInd_{B}^{G}\chi)^{\sm}$.
  Suppose first that there is a parabolic subgroup $\alggrp{P}' = \alggrp{L}'\alggrp{N}'$ such that $0 \to \tau_1 \to (\nInd_{B \cap L'}^{L'}\chi)^{\sm} \to \tau_2 \to 0$ with $\tau_1 \ne 0$ non-generic.
  Then every irreducible submodule of $(\nInd_{P'}^{G}\tau_1)^{\sm}$ is non-generic (as the generic constituents have to lie in $(\nInd_{P'}^{G}\tau_2)^{\sm}$), contradiction.
  Therefore, every irreducible subrepresentation of $(\nInd_{B \cap L'}^{L'}\chi)^{\sm}$ is generic.
  In particular this applies to $L' = L$.
  Consider now all $L' = L_\alpha$ for $\alpha \in \Phi^+_{\red}$ such that $(\nInd_{B \cap L_\alpha}^{L_\alpha} \chi)^{\sm}$ is reducible and $\ang{\nu,\alpha^\vee} < 0$.
  Then $(\nInd_{B \cap L_\alpha}^{L_\alpha} \chi)^{\sm}$ is of length 2 and has an irreducible socle (by considering the Jacquet module).
  We claim that the socle is non-generic.
  We get the same length 2 reducibility after restriction to the simply-connected cover of the derived subgroup of $L_\alpha$, which is one of $\SL_2(E)$, $\SU_3(E)$ for some finite extension $E/F$.
  In case of $\SL_2(E)$ we are done because the subrepresentation is trivial by the condition $\ang{\nu,\alpha^\vee} < 0$.
  In the other case we lift to $\U_3(E)$ and still obtain reducibility (e.g.\ by Proposition~\ref{prop:solleveld}).
  Relabeling $E$ as $F$, we reduce to $G = \U_3(F)$ and $\ang{\nu,\alpha^\vee} < 0$ for the unique simple root $\alpha$.
  Then the socle is non-generic, e.g.\ by (the contragredient of) \cite[Theorem 1]{MR1634020}.   (Note that there is only one orbit of non-degenerate characters of $\U_3$.)
  Thus we deduce (ii).
 
  Conversely, suppose that (ii) holds.
  We first show that the socle of $(\nInd_{B}^{G} \chi)^{\sm}$ is of the same length as $(\nInd_{B \cap L}^{L} \chi)^{\sm}$.
  If $\nu$ is not dominant, then there is an $\alpha \in \Delta$ such that $\ang{\nu,\alpha^\vee} < 0$.
  By the first assumption, $(\nInd_{B \cap L_\alpha}^{L_\alpha} \chi)^{\sm}$ is irreducible, so we deduce that $(\nInd_{B}^{G} \chi)^{\sm} \cong (\nInd_{B}^{G} s_\alpha(\chi))^{\sm}$ by transitivity of parabolic induction.
  Since $s_\alpha$ permutes the set $\Phi^+_{\red}\setminus\{\alpha\}$, our condition on $\chi$ above also holds for $s_\alpha(\chi) = s_\alpha(\psi)\chi_{s_\alpha(\nu)}$.
  But in this way we reduce the number of $\alpha \in \Phi^+_{\red}$ such that $\ang{\nu,\alpha^\vee} < 0$ by one, and we can reduce to $\nu$ dominant in finitely many steps.

  Suppose now that $\nu$ is dominant, as we may.
  Then $\alggrp{L}$ is standard and the standard parabolic subgroup $\alggrp{P}$ containing $\alggrp{L}$ as a Levi part is associated to the subset $\{\alpha \in \Delta : \ang{\nu,\alpha^\vee} = 0\}$ of $\Delta$.
  As $\nu$ lies in $\mathfrak a^*_{L,\R}$, $\chi_{\nu}$ extends to an unramified character of $L$.
  Let $\sigma$ be any irreducible constituent of $(\nInd_{B \cap L}^{L} \chi)^{\sm} \cong (\nInd_{B \cap L}^{L} \psi)^{\sm} \otimes \chi_\nu$, which is a semisimple representation (as it is unitary up to twist).
  Then $\sigma$ is generic by our second assumption, and we now show that $(\nInd_{P}^{G} \sigma)^{\sm}$ has an irreducible socle.
  By the geometric lemma, $r_P((\nInd_{P}^{G} \sigma)^{\sm})$ has a filtration with graded pieces $\sigma_w := (\nInd_{L \cap w P w^{-1}}^{L} w r_{L \cap w^{-1} P w}\sigma)^{\sm}$, where $r_P$ denotes the normalized Jacquet module and $w$ runs through Kostant representatives of $W_L\backslash W/W_L$.
  Note that the supercuspidal support of $\sigma$ (resp.\ of any irreducible constituent of $\sigma_w$) is $(Z,\chi)$ (resp.\ $(Z,ww'\chi)$ for some $w' \in W_L$), up to $L$-conjugacy.
  But $w\chi = \chi$ implies that $w\chi_\nu = \chi_\nu$, hence $w\nu = \nu$, i.e.\ $w \in W_L$ by definition of $L$.
  We see that $\sigma$ occurs with multiplicity one in $r_P((\nInd_{P}^{G} \sigma)^{\sm})$ and hence by exactness of $r_P$ deduce that $(\nInd_{P}^{G} \sigma)^{\sm}$ has an irreducible socle.
  (If $\pi \subset (\nInd_{P}^{G} \sigma)^{\sm}$ is irreducible, then $r_P \pi \onto \sigma$.)

  It remains to show that every irreducible subrepresentation of $(\nInd_{B}^{G} \chi)^{\sm}$ is generic.   Continue to assume, as we may, that $\nu$ is dominant.
  Taking $\sigma$ to be any irreducible constituent of the semisimple representation $(\nInd_{B \cap L}^{L} \chi)^{\sm} \cong (\nInd_{B \cap L}^{L} \psi)^{\sm} \otimes \chi_\nu$ it suffices to show that the irreducible socle $\pi_\sigma$ of $(\nInd_{P}^{G} \sigma)^{\sm}$ is generic.
  Fix a non-degenerate character $\theta$ such that $(\nInd_{P}^{G} \sigma)^{\sm}$ contains a $\theta$-generic constituent $\pi_\theta$.
  We claim that there exists a (nonzero) intertwining morphism $T \colon (\nInd_{\o P}^{G} \sigma)^{\sm} \to (\nInd_{P}^{G} \sigma)^{\sm}$ whose kernel does not contain $\pi_\theta$ as constituent.
  Assuming the claim we are done: by second adjointness and the previous paragraph we know that $(\nInd_{\o P}^{G} \sigma)^{\sm}$ has irreducible cosocle $\pi_\sigma$.
  (If $(\nInd_{\o P}^{G} \sigma)^{\sm} \onto \pi'$ irreducible, then $\sigma \into r_P \pi'$, so $\pi'$ has to be the unique constituent with $\sigma$ contributing to $r_P \pi'$.
  In fact, we see that $\sigma$ is a direct summand of $r_P \pi_\sigma$.)
  Then the image of $T$ has $\pi_\sigma$ both as its socle and its cosocle, so the image of $T$ equals $\pi_\sigma$, since $\pi_\sigma$ occurs only once in $(\nInd_{P}^{G} \sigma)^{\sm}$.
  The claim implies that $\pi_\sigma \cong \pi_\theta$ is generic.

  To prove the final claim, it suffices to construct a morphism $T' \colon (\nInd_{\o B}^{G} \chi)^{\sm} \to (\nInd_{B}^{G} \chi)^{\sm}$ that does not kill any generic constituents, because we can decompose $(\nInd_{\o B}^{G} \chi)^{\sm} \cong \bigoplus_\sigma (\nInd_{\o P}^{G} \sigma)^{\sm}$ and $(\nInd_{B}^{G} \chi)^{\sm} \cong \bigoplus_\sigma (\nInd_{P}^{G} \sigma)^{\sm}$ (by semisimplicity) and consider where $\pi_\sigma$ is sent.
  By writing $w_0$ as a reduced product of simple reflections, it suffices to construct $(\nInd_{w^{-1} s_\beta^{-1} B s_\beta w}^{G} \chi)^{\sm} \to (\nInd_{w^{-1} B w}^{G} \chi)^{\sm}$ that does not kill any generic constituents, for any $w \in W$ and simple reflection $s_\beta$ with $\ell(s_\beta w) > \ell(w)$ (i.e.\ $w^{-1}(\beta) > 0$) where $\ell$ is the length function on $W$.
  By parabolic induction, it suffices to construct $(\nInd_{w^{-1} s_\beta^{-1} B s_\beta w}^{w^{-1} L_\beta w} \chi)^{\sm} \to (\nInd_{w^{-1} B w}^{w^{-1} L_\beta w} \chi)^{\sm}$ that does not kill any generic constituents.
  We can write this morphism as $T'_\alpha \colon (\nInd_{\o B \cap L_\alpha}^{L_\alpha} \chi)^{\sm} \to (\nInd_{B \cap L_\alpha}^{L_\alpha} \chi)^{\sm}$, where $\alpha := w^{-1}(\beta) \in \Phi^+_{\red}$.
  If $\ang{\nu,\alpha^\vee} = 0$ or if $(\nInd_{B \cap L_\alpha}^{L_\alpha} \chi)^{\sm}$ is irreducible, these representations are semisimple and we take $T'_\alpha$ to be any isomorphism.
  Otherwise, by (a) we know that $\ang{\nu,\alpha^\vee} > 0$.
  We also know that $(\nInd_{\o B \cap L_\alpha}^{L_\alpha} \chi)^{\sm}$ has an irreducible cosocle which is the same as the socle of $(\nInd_{B \cap L_\alpha}^{L_\alpha} \chi)^{\sm}$, and we take $T'_\alpha$ to be the unique (up to scalar) nonzero map possible.
  By the condition that $\ang{\nu,\alpha^\vee} > 0$ we know that the kernel of $T_\alpha$ is the unique non-generic constituent of $(\nInd_{\o B \cap L_\alpha}^{L_\alpha} \chi)^{\sm}$ (see the first paragraph of this proof).
  This completes the proof.
\end{proof}

\begin{remark}\label{rk:irred-smooth-princ-series}
  We have a similar criterion for irreducibility, for any connected reductive group $\alggrp{G}$.
  Suppose that $\sigma$ is a (finite-dimensional) irreducible smooth representation of the minimal Levi subgroup $\alggrp{Z}$.
  Write $\sigma \cong \sigma_u \chi_\nu$ with $\sigma_u$ unitary and $\nu \in \mathfrak a_{Z,\R}^*$, and let $\alggrp{L}$ be the maximal semistandard Levi subgroup such that $\nu \in \mathfrak{a}^{*}_{L,\R}$. 
  \emph{Then $(\nInd_{B}^{G} \sigma)^{\sm}$ is irreducible if and only if $(\nInd_{B \cap L_\alpha}^{L_\alpha} \sigma)^{\sm}$
  is irreducible for all $\alpha \in \Phi^+_{\red}$ and $(\nInd_{B \cap L}^{L} \sigma)^{\sm}$ is irreducible.}
  For the ``if'' direction, the first condition shows that $(\nInd_{B}^{G} \sigma)^{\sm} \cong (\nInd_{B}^{G} w(\sigma))^{\sm}$ for any $w \in W$ (by induction on $\ell(w)$).
  In particular, $(\nInd_{\o B}^{G} \sigma)^{\sm} \cong (\nInd_{B}^{G} \sigma)^{\sm}$ and we may also assume that $\nu$ is dominant. 
  Let $\tau := (\nInd_{B \cap L}^{L} \sigma)^{\sm}$.   Then $(\nInd_{\o P}^{G} \tau)^{\sm} \cong (\nInd_{P}^{G} \tau)^{\sm}$, where $\alggrp{P} = \alggrp{L}\alggrp{N}$ is the standard parabolic with Levi subgroup $\alggrp{L}$.
  By the same argument as in the proof of Proposition~\ref{prop:qsplit-irred-gen-socle} we see that $(\nInd_{P}^{G} \tau)^{\sm}$ has an irreducible socle which is also the cosocle of $(\nInd_{\o P}^{G} \tau)^{\sm}$ and occurs in these representations with multiplicity one.
  This implies irreducibility.  
  (If $\alggrp G$ is split, this can also be deduced from \cite[Proposition 4.2]{MR560847}, and in general from the main theorem of \cite{luo}.)
\end{remark}

\begin{proof}[Proof of Theorems~\ref{thm:Sp2n-irred}--\ref{thm:U2n+1-irred}]
  By Corollary~\ref{cor:whittaker criterion} we may extend scalars to $\o C \cong \C$, and it suffices to verify that both assumptions in Proposition~\ref{prop:qsplit-irred-gen-socle}(ii) hold.
  (To apply the corollary we work with $\alggrp G$ obtained by restriction of scalars from $F$ to $\Q_p$, and take $\sigma_0 = 1$, $\tau = \chi\delta_B^{1/2}$, $\alggrp{Q} = \alggrp{G}$.)
    Note that $(\Ind_{B}^{G}\chi\delta_B^{1/2})^{\sm} = (\nInd_{B}^{G}\chi)^{\sm}$.
  The first assumption is easy to verify by the irreducibility criteria for smooth principal series of $\Sp_2(F) = \SL_2(F)$, $\SO_3(F) \cong \PGL_2(F)$, $\GL_2(E)$, $\SO^*_4(F)$, $\U_2(F)$ and $\U_3(F)$,
  using all but the last hypothesis in each theorem.
  (In this proof we consider Theorem~\ref{thm:SO2n+1-irred} to have a vacuous last hypothesis.
  For the last three irreducibility criteria, see Lemma~\ref{lem:SO4}, respectively \cite[\S11.1, \S12.2]{MR1081540}.)
  
  We now verify the second assumption of Proposition~\ref{prop:qsplit-irred-gen-socle}(ii), using the last hypothesis in each theorem.
  Write $L \cong \prod_i \GL_{m_i}(E) \times G'$, where $E = F$ in all but the unitary cases and $G'$ is a classical group of the same type as $G$.
  (Note that $G'$ may be trivial, except when the type is $\SO^*$.)
  As unitary principal series of $\GL_m(E)$ are irreducible and generic, we may reduce to the case where $G' = G$, i.e.\ $L = G$.
  By a twist we may assume that $\chi$ is unitary, i.e.\ $\nu = 0$.
  In case $G = \SO_{2n+1}(F)$   (resp.\ $G = \U_{2n+1}(F)$) we apply \cite[Theorem 6.5(1)]{MR1296726}   (resp.\ \cite[Theorem 3.4]{MR1224616}) to see that $(\nInd_{B}^{G}\chi)^{\sm}$ is irreducible and hence generic.

  In case $G = \Sp_{2n}(F)$ we lift $\chi$ to the character $\wt\chi := \chi_1 \otimes \cdots \otimes \chi_n \otimes \rho$ of the diagonal maximal torus $\wt Z$ of $\GSp_{2n}(F)$, where $\rho : F^\times \to \C^\times$ is an arbitrary unitary character. 
  The unitary principal series of $\GSp_{2n}(F)$ obtained from $\wt\chi$ is irreducible and hence generic by \cite[Theorem 2.6]{MR1464132}.
  (Note that $d_1 = 0$ and $d_\chi \le 1$ for all $\chi \ne 1$ in the notation of that paper.)
  By restriction to $\Sp_{2n}(F)$ it follows that all irreducible constituents of $(\nInd_{B}^{G}\chi)^{\sm}$ are generic.

  In case $G = \U_{2n}(F)$ we lift $\chi$ to the character $\wt\chi := \chi_1 \otimes \cdots \otimes \chi_n \otimes \rho$ of the diagonal maximal torus $\wt Z$ of $\GU_{2n}(F)$, where $\rho : F^\times \to \C^\times$ is an arbitrary unitary character. 
  The unitary principal series of $\GSp_{2n}(F)$ obtained from $\wt\chi$ is irreducible and hence generic by \cite[Theorem 2.6]{MR1464132}.
  (Note that $d_1 = 0$ and $\Lambda(\sigma) = \Lambda(\sigma)' = \varnothing$ in the notation of that paper.
  Here we use that by comparison with $\SL_2(F)$, using Proposition~\ref{prop:solleveld}, all unitary principal series of $\GU_2(F)$ are irreducible.)
  By restriction to $\U_{2n}(F)$ it follows that all irreducible constituents of $(\nInd_{B}^{G}\chi)^{\sm}$ are generic.

  In case $G = \SO_{2n}(F)$ (resp.\ $G = \SO^*_{2n}(F)$) we likewise lift to a character of the diagonal maximal torus $\wt Z$ of $\GSO_{2n}(F)$ and apply Lemma~\ref{lem:R-group-GSO}
  (resp.\ Lemma~\ref{lem:R-group-GSO-star}).
\end{proof}

\begin{remark}\label{rk:optimal}
  The conditions in Theorems~\ref{thm:Sp2n-irred}--\ref{thm:U2n+1-irred} are optimal in the sense that if every irreducible subrepresentation
  of $(\Ind_{B}^{G}\chi\delta_B^{1/2})^\sm$ (over $\o C \cong \C$) is generic, then the conditions in the theorems hold.
  (The analogue of course holds for $\GL_n(F)$ as well, by the same argument.)

  To justify this, suppose that every irreducible subrepresentation of $(\nInd_{B}^{G}\chi)^\sm$ is generic.
  Then by Proposition~\ref{prop:qsplit-irred-gen-socle} (item (a)) we get all but the last condition in each theorem.
  
  It remains to discuss the symplectic, even orthogonal, and unitary groups.
  Observe that there is a unique orbit of non-degenerate characters of $U$ for the groups $\GSp_{2n}(F)$, $\GSO_{2n}(F)$, $\GSO_{2n}^*(F)$, $\GU_{2n}(F)$, and $\U_{2n+1}(F)$ under the action of the diagonal maximal torus $\alggrp{Z}$.
  If $G = \U_{2n+1}(F)$ we let $\alggrp{L}$ be the maximal semistandard Levi subgroup such that $\nu \in \mathfrak a^*_{L,\R}$.
  Then the semisimple representation $(\nInd_{B \cap L}^{L} \chi)^{\sm}$ has to be irreducible by (b) (as it contains a unique generic constituent), which implies the last condition in this case by the $R$-group result we already used.
  If $G$ is one of the groups $\Sp_{2n}(F)$, $\SO_{2n}(F)$, $\SO^*_{2n}(F)$, $\U_{2n}(F)$ we lift to the similitude group $\GSp_{2n}(F)$, $\GSO_{2n}(F)$, $\GSO^*_{2n}(F)$, $\GU_{2n}(F)$ as in the proof above and then apply the same reasoning.
  (Note that the lifted principal series still has the property that every irreducible subrepresentation is generic.)
\end{remark}

For completeness, we also state the irreducibility criteria we get from Remark~\ref{rk:irred-smooth-princ-series} for classical groups, using the
same notation as in Theorems~\ref{thm:Sp2n-irred}--\ref{thm:U2n+1-irred}.
For the group $\Sp_{2n}(F)$, see also \cite[Theorem 7.1]{MR1266251}.

\begin{thm}\label{thm:irred-smooth-classical-grps}\ 

  \begin{enumerate}
  \item If $G = \Sp_{2n}(F)$ (split), $\chi = \chi_{1}\otimes\cdots\otimes\chi_{n}\colon Z\to \C^{\times}$ smooth, then
    $(\Ind_{B}^{G}\chi\delta_B^{1/2})^{\sm}$ is irreducible if and only if
    \begin{itemize}
    \item $\chi_i \chi_j^{\pm 1} \ne \lvert \cdot\rvert_F^{\pm 1}$ for all $i < j$;
    \item $\chi_i \ne \lvert \cdot\rvert_F^{\pm 1}$ for all $i$;
    \item $\chi_i$ is not of order 2 for all $i$.     \end{itemize}
  \item If $G = \SO_{2n+1}(F)$ (split), $\chi = \chi_{1}\otimes\cdots\otimes\chi_{n}\colon Z\to \C^{\times}$ smooth, then
    $(\Ind_{B}^{G}\chi\delta_B^{1/2})^{\sm}$ is irreducible if and only if
    \begin{itemize}
    \item $\chi_i \chi_j^{\pm 1} \ne \lvert \cdot\rvert_F^{\pm 1}$ for all $i \le j$.
    \end{itemize}
  \item If $G = \SO_{2n}(F)$ (split), $\chi = \chi_{1}\otimes\cdots\otimes\chi_{n}\colon Z\to \C^{\times}$ smooth, then
    $(\Ind_{B}^{G}\chi\delta_B^{1/2})^{\sm}$ is irreducible if and only if
    \begin{itemize}
    \item $\chi_i \chi_j^{\pm 1} \ne \lvert \cdot\rvert_F^{\pm 1}$ for all $i < j$;
    \item the set $\{\chi_i : \chi_i^2 = 1\}$ has at most one element.      \end{itemize}
  \item If $G = \SO^*_{2n}(F)$ (non-split quasisplit) splitting over a quadratic extension $E/F$, $\chi = \chi_{1}\otimes\cdots\otimes\chi_{n}\colon Z = (F^\times)^{n-1} \times (E^\times)^{N_{E/F}=1}\to \C^{\times}$ smooth, then
    $(\Ind_{B}^{G}\chi\delta_B^{1/2})^{\sm}$ is irreducible if and only if
    \begin{itemize}
    \item $\chi_i \chi_j^{\pm 1} \ne \lvert \cdot\rvert_F^{\pm 1}$ for all $i < j < n$;
    \item $(\chi_i \circ N_{E/F}) \chi_n^{c-1} \ne \lvert\cdot\rvert_E^{\pm 1}$ for all $i < n$ if $\chi_n^2 = 1$;
    \item the set $\{\chi_i : \text{$i < n$ and $\chi_i^2 = 1$}\}$ has at most one element if $\chi_n^2 \ne 1$; \\
      ($\chi_i^2 \ne 1$ or $(\chi_i \circ N_{E/F})\chi_n^{c-1} = 1$) for all $i < n$ if $\chi_n^2 = 1$. 
    \end{itemize}
  \item If $G = \U_{2n}(F)$ (quasisplit) splitting over a quadratic extension $E/F$, $\chi = \chi_{1}\otimes\cdots\otimes\chi_{n}\colon Z = (E^\times)^n \to \C^{\times}$ smooth, then
    $(\Ind_{B}^{G}\chi\delta_B^{1/2})^{\sm}$ is irreducible if and only if
    \begin{itemize}
    \item $\chi_i \chi_j^{-1} \ne \lvert \cdot\rvert_E^{\pm 1}$ for all $i < j$;
    \item $\chi_i \chi_j^{c} \ne \lvert \cdot\rvert_E^{\pm 1}$ for all $i < j$;
    \item $\chi_i \ne \eta \lvert \cdot\rvert_E^{\pm 1/2}$ with $\eta|_{F^\times} = 1$ for all $i$;
    \item $\chi_i|_{F^\times} \ne \omega_{E/F}$ for all $i$.     \end{itemize}
  \item If $G = \U_{2n+1}(F)$ (quasisplit) splitting over a quadratic extension $E/F$, $\chi = \chi_{1}\otimes\cdots\otimes\chi_{n+1}\colon Z = (E^\times)^n \times (E^\times)^{N_{E/F}=1}\to \C^{\times}$ smooth, then
    $(\Ind_{B}^{G}\chi\delta_B^{1/2})^{\sm}$ is irreducible if and only if
    \begin{itemize}
    \item $\chi_i \chi_j^{-1} \ne \lvert \cdot\rvert_E^{\pm 1}$ for all $i < j \le n$;
    \item $\chi_i \chi_j^{c} \ne \lvert \cdot\rvert_E^{\pm 1}$ for all $i < j \le n$;
    \item $\chi_i \chi_{n+1}^{c-1} \ne \lvert \cdot\rvert_E^{\pm 1}$ for all $i \le n$;
    \item $\chi_i \chi_{n+1}^{c-1} \ne \eta \lvert \cdot\rvert_E^{\pm 1/2}$ with $\eta|_{F^\times} = \omega_{E/F}$ for all $i \le n$;
    \item $\chi_i \chi_{n+1}^{c-1} = 1$ or $\chi_i \chi_{n+1}^{c-1}|_{F^\times} \ne 1$ for all $i \le n$.
    \end{itemize}
  \end{enumerate}
\end{thm}

\begin{proof}
  This follows as in the proof of Theorems~\ref{thm:Sp2n-irred}--\ref{thm:U2n+1-irred}.
  If $G = \Sp_{2n}(F)$ (resp.\ $\SO_{2n}(F)$, resp.\ $\SO^*_{2n}(F)$, resp.\ $G = \U_{2n}(F)$), the relevant $R$-group result
  can be found in \cite[Theorem 6.4]{MR1296726} (resp.\ \cite[Theorem 6.8]{MR1296726}, resp.\ \cite[Theorem A.4]{MR3430367}, resp.\ \cite[Theorem 3.4]{MR1224616}).
\end{proof}

\subsection{The group \texorpdfstring{$\GL_{n}(D)$}{GL\_n(D)}}
\label{sec:GL(n,D)}

Let $D$ be a central division $F$-algebra of dimension $d^{2}$ and $G_{n} = \GL_{n}(D)$, $B = B_{n}$ the minimal parabolic subgroup of upper-triangular matrices, $U = U_{n}$ the unipotent radical of $B_{n}$, and $Z = Z_{n}$ the diagonal minimal Levi subgroup.
The mirabolic subgroup $P_{n}$ of $\GL_{n}(D)$ is defined by $P_{n} := \{(g_{ij})\in G\mid g_{nn} = 1, g_{ni} = 0\ (1\le i\le n  - 1)\}$.
We say that a representation $\pi$ of $P_{n}$ is \emph{generic} if $\pi_{\overline{C},U_{n},\theta} \ne 0$ for one (equivalently any) non-degenerate character $\theta\colon U_{n}\to \overline{C}^{\times}$.
Again, by a slight abuse we also say that a smooth representation $\pi'$ of $P_n$ over $\o C$ is \emph{generic} if $\pi'_{U_n,\theta}\ne 0$ for some non-degenerate $\theta$.
Let $\Nrd\colon D^{\times}\to F^{\times}$, $\det\colon G_n\to F^{\times}$ be the reduced norm.
For an absolutely irreducible smooth representation $\sigma$ of $D^{\times}$, let $\nu_{\sigma} = \lvert\Nrd\rvert_{F}^{s(\sigma)}$ be the character of $D^{\times}$ of \cite[Section 2]{MR1040995}.
Here, $s(\sigma)$ is a positive integer dividing $d$.
It is characterized by the fact that for $\sigma'$ another absolutely irreducible smooth representation of $D^\times$ the induction $(\nInd_{B_{2}}^{G_{2}}\sigma\boxtimes \sigma')^\sm$ is reducible if and only if $\sigma' \cong \sigma\nu_{\sigma}^{\pm 1}$ \cite[2.5 Lemma, 4.2 Lemma]{MR1040995}.
(Note that $\sigma'$ does not denote the strong dual of $\sigma$ in this subsection.)

\begin{thm}\label{thm:GL(n,D),dim>1}
Let $\sigma = \sigma_{1}\boxtimes\cdots\boxtimes\sigma_{n}$ be an absolutely irreducible smooth representation of $Z\simeq (D^{\times})^{n}$ over $C$.
Assume the following condition.
\begin{equation}\label{eq:seq}
  \begin{gathered}
    \text{There exists no sequence $1 \le i_0 < i_1 < \dots < i_e \le n$ such that}\\
    \text{$\sigma_{i_j} \cong \sigma_{i_0} \nu_{\sigma_{i_0}}^j$ for all $0 \le j \le e$ and
      $\nu_{\sigma_{i_0}}^e = \lvert\Nrd\rvert_{F}^{d}$}.
  \end{gathered}
\end{equation}
Then any nonzero $P_{n}$-subrepresentation of $(\nInd_{B}^{G}\sigma)^{\sm}$ is generic.
In particular, $(\Ind_{B}^{G}\sigma\delta_{B}^{1/2})^{\cts}$ is irreducible.
\end{thm}

The last part follows from Proposition~\ref{prop:whittaker criterion} for the group $\Res_{F/\Q_p} \GL_n(D)$ (and $\sigma_0 = 1$, $\tau = \sigma$, $Q = G$).
For the first part, exactly as in the proof of Corollary~\ref{cor:whittaker criterion} it suffices to prove the claim after extending scalars to $\o C \cong \C$.
\emph{Therefore, for the remainder of this subsection we may and will work over $\C$.}

We fix an irreducible smooth $Z$-representation $\sigma = \sigma_{1}\boxtimes\cdots\boxtimes\sigma_{n}$ satisfying~(\ref{eq:seq}).
We first claim that $(\nInd_{B}^{G}\sigma)^{\sm} \cong (\nInd_{B}^{G}\tau)^{\sm}$ for some $\tau = \tau_{1}\boxtimes\cdots\boxtimes\tau_{n}$
in the Weyl group orbit of $\sigma$ such that $\tau_{i}\not\simeq\tau_{j}\otimes\lvert\Nrd\rvert_{F}^{-d}$ for any $i < j$.
Letting $\sim$ be the equivalence relation on irreducible smooth representations of $D^\times$ induced by $\zeta \sim \zeta \nu_{\zeta}$.
As $(\nInd_{B_{2}}^{G_{2}}\sigma_i\boxtimes \sigma_{i+1})^\sm \cong (\nInd_{B_{2}}^{G_{2}}\sigma_{i+1}\boxtimes \sigma_i)^\sm$ whenever $\sigma_i \not\sim \sigma_{i+1}$,
by repeated transposition of consecutive, inequivalent representations we may assume that whenever $\sigma_i \sim \sigma_j$ for some $i < j$, then $\sigma_i \sim \sigma_{i+1} \sim \dots \sim \sigma_j$, while preserving condition \eqref{eq:seq}.
In this way we reduce to the case where all $\sigma_i$ ($1 \le i \le n$) lie in the same equivalence class, i.e.\ $\sigma_i \cong \zeta \nu_{\zeta}^{k(i)}$ for some $k(i) \in \Z$.
Let $e$ be the divisor of $d$ such that $\nu_\zeta^e = \lvert\Nrd\rvert_{F}^{d}$.
Condition \eqref{eq:seq} says that there exists no sequence $1 \le i_0 < i_1 < \dots < i_e \le n$ such that $k(i_j) - k(i_0) = j$ for all $0 \le j \le e$.
If there exists any $0 \le \ell < n$ such that $k(\ell+1)-k(\ell) \ge 2$, then we can transpose $\sigma_{\ell}$ and $\sigma_{\ell+1}$, while keeping condition \eqref{eq:seq} satisfied.
(Subsequences $i_0 < \cdots < i_e$ as in \eqref{eq:seq} containing at most one of $\ell$, $\ell+1$ are unaffected by the transposition. Subsequences containing both will no longer satisfy \eqref{eq:seq} after transposition because the function $k(i_j)$ has to be increasing.)
Each such transposition decreases the sum $\sum_{\ell=1}^n \ell k(\ell)$, so after finitely many steps we may assume that $k(\ell+1)-k(\ell) \le 1$ for all $0 \le \ell < n$ and that condition \eqref{eq:seq} still holds.
Suppose there exist $i < j$ such that $\sigma_i \cong \sigma_j \otimes\lvert\Nrd\rvert_{F}^{-d}$, i.e.\ $k(j)-k(i) = e$.
We have $e > 1$ by condition \eqref{eq:seq}.
More generally, suppose that for some $i < j$ we have $k(j)-k(i) = e' > 1$.
Pick $i \le i' < j$ maximal such that $k(i')-k(i) \le 1$, so $k(i'+1)-k(i) \ge 2$.
Then $1 \ge k(i'+1)-k(i') \ge 2-1 = 1$, whence equality holds and in particular $k(i')-k(i) = 1$.
Hence, taking again $e'=e$, then by induction there exists a sequence $i = i_0 < i_1 < \dots < i_e = j$ such that $k(i_j)-k(i_{j-1}) = 1$ for all $1 \le j \le e$, contradiction.

Therefore, we may assume that $\sigma_{i}\not\simeq\sigma_{j}\otimes\lvert\Nrd\rvert_{F}^{-d}$ for any $i < j$.
If $D = F$ then the theorem is proved in \cite[4.11 Theorem]{MR0579172} and in general we follow their argument.

For a locally profinite topological group $H$, let $\Rep^{\infty}(H)$ be the category of smooth $H$-representations over $\C$.
If $H'$ is a closed subgroup of $H$ and $\tau\in\Rep^{\infty}(H)$, we put $(\nInd_{H'}^{H}\tau)^{\sm} := (\Ind_{H'}^{H}\tau\delta_{H}^{-1/2}\delta_{H'}^{1/2})^{\sm}$.
Let $N_{n}$ be the unipotent radical of $P_{n}$.
We fix a non-degenerate character $\theta$ of $U_{n}$ and let $\eta_{n}$ be the restriction of $\theta$ to $N_{n}$.
We define exact functors as follows:
\begin{align*}
& \Phi^{-}\colon \Rep^{\infty}(P_{n})\to \Rep^{\infty}(P_{n - 1})\quad \pi\mapsto \lvert\det\rvert_{F}^{-d/2}\otimes \pi_{N_{n},\eta_{n}},\\
& \widehat{\Phi}^{+}\colon \Rep^{\infty}(P_{n - 1})\to \Rep^{\infty}(P_{n})\quad  \tau\mapsto (\nInd_{P_{n - 1}N_{n}}^{P_{n}}\tau)^\sm,\\
& \Phi^{+}\colon \Rep^{\infty}(P_{n - 1})\to \Rep^{\infty}(P_{n})\quad  \tau\mapsto (\ncInd_{P_{n - 1}N_{n}}^{P_{n}}\tau)^\sm,
\end{align*}
where $\ncInd$ denotes normalized compact induction.
(We keep the notation of \cite{MR0579172}, but caution that $\Phi^{+}$ should not be confused with the set of positive roots. 
Also note that \cite{MR0579172} uses a different definition of normalized induction.)
In the definition of $\widehat{\Phi}^{+}$ and $\Phi^+$ we extend $\pi$ to $P_{n - 1}N_{n}$ by letting $N_{n}$ act via $\eta_n$.
We remark that $\delta_{P_{n - 1}N_{n}} = \lvert\det\rvert_{F}^{2d}$ and $\delta_{P_{n}} = \lvert\det\rvert_{F}^{d}$, so $(\nInd_{P_{n - 1}N_{n}}^{P_{n}}\tau)^{\sm} = (\Ind_{P_{n - 1}N_{n}}^{P_{n}}\tau\lvert\det\rvert_{F}^{d/2})^{\sm}$.
In particular, the functors $\Phi^{-}$, $\widehat{\Phi}^{+}$, $\Phi^{+}$ coincide with the ones in \cite{MR0579172} when $D = F$.

\begin{lem}\label{lem:from BZ}
The functor $\Phi^{-}$ is right adjoint to $\Phi^{+}$ and left adjoint to $\widehat{\Phi}^{+}$.
We have $\Phi^{-}(\pi) = 0$ if and only if $N_{n}$ acts trivially on $\pi$.
\end{lem}
\begin{proof}
By Frobenius reciprocity, $\widehat{\Phi}^{+}$ is the right adjoint of $\Phi^{-}$.
For $\Phi^{+}$, the proof of \cite[3.2 Proposition (b)]{MR0579172} applies.
The last claim follows from \cite[3.2 Proposition (e)]{MR0579172} when $D = F$ and the same proof applies.
\end{proof}

For a smooth representation $\tau$, let $\tau^{\vee}$ be the smooth dual of $\tau$.

\begin{lem}\label{lem:dual and Phi^-}
We have $\Phi^{-}(\pi^{\vee})\simeq \Phi^{-}(\pi)^{\vee}$ for $\pi\in\Rep^{\infty}(P_{n})$.
\end{lem}
\begin{proof}
Let $\tau\in\Rep^{\infty}(P_{n - 1})$.
We have $\Phi^{+}(\tau)^{\vee}\simeq \widehat{\Phi}^{+}(\tau^{\vee})$, by our normalization of induction.
Then we have
\begin{align*}
\Hom_{P_{n - 1}}(\tau,\Phi^{-}(\pi^{\vee}))& \simeq \Hom_{P_{n}}(\Phi^{+}(\tau),\pi^{\vee})\\
& \simeq \Hom_{P_{n}}(\pi,\Phi^{+}(\tau)^{\vee})\\
& \simeq \Hom_{P_{n}}(\pi,\widehat{\Phi}^{+}(\tau^{\vee}))\\
& \simeq \Hom_{P_{n - 1}}(\Phi^{-}(\pi),\tau^{\vee})\\
& \simeq \Hom_{P_{n - 1}}(\tau,\Phi^{-}(\pi)^{\vee}).
\end{align*}
The lemma follows.
\end{proof}

Let $Q_{n}$ be the standard parabolic subgroup of $G_{n} = \GL_{n}(D)$ corresponding to the partition $n = 1 + (n - 1)$.
Then $Q_{n-1}N_n = (G_1 \times P_{n-1})N_n'$, where the normal subgroup $N_n'$ denotes the unipotent radical of $Q_n$.
To keep notation short we write $Q_{n-1}N_n$ below, even though we think of it as $(G_1 \times P_{n-1})N_n'$.

\begin{lem}\label{lem:embedding of derivative}
Let $\sigma_1$ be a smooth representation of $G_1$ and $\sigma'$ a smooth representation of $P_{n - 1}$.
Suppose that $\pi_{0}\subset (\nInd_{Q_{n - 1}N_{n}}^{P_{n}}\sigma_{1}\boxtimes\sigma')^{\sm}$ is a $P_{n}$-subrepresentation, where we let $N_n'$ act trivially on $\sigma_{1}\boxtimes\sigma'$.
\begin{enumerate}
\item If $n \ge 3$ we have an embedding $\Phi^{-}(\pi_{0})\hookrightarrow (\nInd_{Q_{n - 2}N_{n-1}}^{P_{n-1}}\sigma_{1}\boxtimes\Phi^{-}(\sigma'))^{\sm}$.
\item If $(\pi_{0})_{N_{n},\eta_{n}} = 0$ and $n \ge 2$, then $\pi_{0}|_{G_{n - 1}}$ is embedded into the following two representations:
\begin{gather*}
(\nInd_{Q_{n - 1}}^{G_{n - 1}}\sigma_{1}\boxtimes r(\sigma'))^{\sm}\otimes \lvert \det\rvert_{F}^{d/2},\\
(\nInd_{Q_{n - 1}}^{G_{n - 1}}\sigma_{1}\boxtimes r((\sigma')^{\vee})^{\vee})^{\sm}\otimes \lvert\det\rvert_{F}^{-d/2}.
\end{gather*}
Here $r(\sigma') := (\sigma')_{N_{n - 1}}\otimes\lvert \det\rvert^{-d/2}_{F}$ is the normalized Jacquet module of $\sigma'$ with respect to the parabolic subgroup corresponding to $(n - 1) = (n - 2) + 1$.
\end{enumerate}
\end{lem}
\begin{proof}
For simplicity we put $\Pi := (\nInd_{Q_{n - 1}N_{n}}^{P_{n}}\sigma_{1}\boxtimes\sigma')^{\sm}$.

For (i), we apply the geometric lemma to calculate the $P_{n - 1}$-representation $\Pi_{N_{n},\eta_{n}}$.
We have
\begin{align*}
Q_{n - 1}N_{n}\backslash P_{n}/P_{n - 1}N_{n} & \simeq Q_{n - 1}\backslash G_{n - 1}/P_{n - 1}\\
& \simeq \Aut(\{2,\ldots,n - 1\})\backslash S_{n - 1}/\Aut(\{1,\ldots,n - 2\})
\end{align*}
and it has two elements, represented by the identity element and $(1\; n - 1)$ (transposition).
The orbit corresponding to $(1\; n - 1)$ does not contribute.
Hence we have
\[
\Phi^{-}(\Pi)\simeq (\nInd_{Q_{n - 2}N_{n - 1}}^{P_{n - 1}}\sigma_{1}\otimes \Phi^{-}(\sigma'))^{\sm}.
\]
We get (i).

Assume that $(\pi_{0})_{N_{n},\eta_{n}}  = 0$.
Then $N_{n}$ acts trivially on $\pi_{0}$ by Lemma~\ref{lem:from BZ}.
Hence $\pi_{0}|_{G_{n - 1}}\hookrightarrow \Pi_{N_{n}}$.
Applying the geometric lemma, we have $\Pi_{N_{n}}\simeq (\nInd_{Q_{n - 1}}^{G_{n - 1}}\sigma_1\boxtimes r(\sigma'))^{\sm}\otimes \lvert \det\rvert_{F}^{d/2}$ as $G_{n-1}$-representations.
(The character $\lvert \det\rvert_{F}^{d/2}$ comes from the normalization of coinvariants.)

We also have $\pi_{0}|_{G_{n - 1}}\hookrightarrow ((\Pi^{\vee})^{\vee})^{N_{n}} \simeq ((\Pi^{\vee})_{N_{n}})^{\vee}$.
We calculate
\begin{align*}
((\Pi^{\vee})_{N_{n}})^{\vee} & \simeq (((\nInd_{Q_{n - 1}N_{n}}^{P_{n}}\sigma_{1}^{\vee}\boxtimes(\sigma')^{\vee})^{\sm})_{N_{n}})^{\vee}\\
& \simeq ((\nInd_{Q_{n - 1}}^{G_{n - 1}}\sigma_{1}^{\vee}\boxtimes r((\sigma')^{\vee}))^{\sm}\otimes \lvert\det\rvert_{F}^{d/2})^{\vee}\\
& \simeq (\nInd_{Q_{n - 1}}^{G_{n - 1}}\sigma_{1}\boxtimes r((\sigma')^{\vee})^{\vee})^{\sm}\otimes \lvert\det\rvert_{F}^{-d/2}.
\end{align*}
Part (ii) follows.
\end{proof}

For $\pi\in \Rep^{\infty}(P_{n})$, we define the $k$-th derivative $\pi^{(k)} \in \Rep^{\infty}(G_{n - k})$ of $\pi$ by $\pi^{(k)} := \lvert\det\rvert_{F}^{-d/2}\otimes (\Phi^{-})^{k - 1}(\pi)_{N_{n - k + 1}}$ for $k = 1,\ldots,n$.
We put $\pi^{(k)} := 0$ for $k > n$ and $\pi^{(0)} := \pi$.
For $n_{1},\ldots,n_{r}\ge 1$, $\pi_{n_{i}}\in \Rep^{\infty}(G_{n_{i}})$ ($i = 1,\ldots,r$), we put $\pi_{1}\times\cdots\times\pi_{r} := (\nInd_{P}^{G_{n_{1} + \cdots + n_{r}}}\pi_{1}\boxtimes\cdots\boxtimes\pi_{r})^{\sm}$, where $P$ is the standard parabolic subgroup corresponding to $n_{1} + \cdots + n_{r}$.
\begin{lem}\label{lem:derivative of product}
Let $\pi_{1}\in\Rep^{\infty}(G_{n_{1}})$ and $\pi_{2}\in \Rep^{\infty}(G_{n_{2}})$.
Then $(\pi_{1}\times\pi_{2})^{(k)}$ has a filtration whose successive quotients are $\pi_{1}^{(l)}\times\pi_{2}^{(k - l)}$ with $l = 0,\ldots,k$.
\end{lem}
\begin{proof}
The geometric lemma implies the lemma, see the proof of \cite[4.5 Lemma]{MR0579172}.
\end{proof}

\begin{lem}\label{lem:derivative}
Let $\pi_{0}\in \Rep^{\infty}(P_{n})$ and $k\in\Z_{\ge 0}$ the maximal integer such that $\pi_{0}^{(k)}\ne 0$.
Assume that $k < n$.
Then $(\Phi^{-})^{i}(\pi_{0}) = 0$ for $i\ge k$.
\end{lem}
\begin{proof}
We prove the lemma by backward induction on $i$.
We have $(\Phi^{-})^{n}(\pi_{0}) = 0$. Assume $k \le i < n$.
By the inductive hypothesis we have $\Phi^{-}((\Phi^{-})^{i}(\pi_{0})) = 0$ and by $\pi_{0}^{(i + 1)} = 0$ we have $(\Phi^{-})^{i}(\pi_{0})_{N_{n - i}} = 0$.
Hence $(\Phi^{-})^{i}(\pi_{0}) = 0$ by Lemma~\ref{lem:from BZ}.
\end{proof}

\begin{lem}\label{lem:genericity for GLnD, lemma}
Assume that $\sigma_{1}\not\simeq\sigma_{i}\otimes\lvert\Nrd\rvert^{-d}_{F}$ for $i = 2,\ldots,n$.
Then any nonzero $P_{n}$-subrepresentation $\pi_{0}$ of $(\nInd_{Q_{n - 1}N_{n}}^{P_{n}}\sigma_{1}\boxtimes(\sigma_{2}\times\cdots\times\sigma_{n}))^{\sm}$ is generic.
\end{lem}
Here, inside the induction, $\sigma_{2}\times\cdots\times\sigma_{n}$ is restricted from $G_{n-1}$ to $P_{n-1}$.
\begin{proof}
Set $\sigma' := \sigma_{2}\times\cdots\times\sigma_{n}$.
Assume that $\pi_{0}$ is not generic.
Take the maximal $k$ such that $\pi_{0}^{(k)}\ne 0$.
Since $\pi_{0}$ is not generic, we have $k < n$.
Since $(\Phi^{-})^{k}(\pi_{0}) = 0$ by Lemma~\ref{lem:derivative}, $\pi_{0}^{(k)}$ is a subrepresentation of $(\nInd_{Q_{n - k}}^{G_{n - k}}\sigma_{1}\boxtimes (\sigma')^{(k)})^{\sm}$ by Lemma~\ref{lem:embedding of derivative}.
Note that $(\sigma')^{(k)}$ is of finite length by Lemma~\ref{lem:derivative of product}.
Hence $\pi_{0}^{(k)}$ is also of finite length.
Take an irreducible subrepresentation $\omega$ of $\pi_0^{(k)}$.
By Lemma~\ref{lem:derivative of product}, the cuspidal support of $\omega$ is $(Z_{n-k},\sigma_{1}\boxtimes \sigma_{i_{1}}\boxtimes\cdots\boxtimes\sigma_{i_{n - k - 1}})$ for some $2\le i_{1} < \cdots < i_{n - k - 1}\le n$.

On the other hand, again by Lemma~\ref{lem:embedding of derivative}, $\pi_{0}^{(k)}$ is also embedded into 
\begin{multline*}
(\nInd_{Q_{n - k}}^{G_{n - k}}\sigma_{1}\boxtimes r(((\Phi^{-})^{(k - 1)}(\sigma'))^{\vee})^{\vee})^{\sm}\otimes \lvert\det\rvert_{F}^{-d}\\
\simeq
(\nInd_{Q_{n - k}}^{G_{n - k}}\sigma_{1}\boxtimes (((\sigma')^{\vee})^{(k)})^{\vee})^{\sm}\otimes \lvert\det\rvert_{F}^{-d},
\end{multline*}
where we use Lemma~\ref{lem:dual and Phi^-}.
Therefore the cuspidal support of $\omega$ is $(Z_{n-k},(\sigma_{1}\otimes\lvert\Nrd\rvert_{F}^{-d})\boxtimes(\sigma_{j_{1}}\otimes\lvert\Nrd\rvert_{F}^{-d})\boxtimes\cdots\boxtimes(\sigma_{j_{n - k - 1}}\otimes\lvert\Nrd\rvert_{F}^{-d}))$, where $2\le j_{1} < \cdots < j_{n - k - 1}\le n$.
Hence $\{\sigma_{1},\sigma_{i_{1}},\ldots,\sigma_{i_{n - k - 1}}\} = \{\sigma_{1}\otimes\lvert\Nrd\rvert_{F}^{-d},\sigma_{j_{1}}\otimes\lvert\Nrd\rvert_{F}^{-d},\ldots,\sigma_{j_{n - k - 1}}\otimes\lvert\Nrd\rvert_{F}^{-d}\}$.
By our assumption, $\sigma_{1}$ is not contained in the right-hand side.
Hence we have a contradiction.
\end{proof}

\begin{proof}[Proof of Theorem~\ref{thm:GL(n,D),dim>1}]
We prove the theorem by induction on $n$ (over $\C$, as noted above).
For $n = 1$ there is nothing to prove, so we assume that $n \ge 2$.
Set $\sigma' := \sigma_{2}\times\cdots\times\sigma_{n}$.
We apply the geometric lemma to $(\nInd_{Q_{n}}^{G_{n}}\sigma_{1}\boxtimes\sigma')^{\sm}|_{P_{n}}$.
We have $Q_{n}\backslash G_{n}/P_{n} = \Aut(\{2,\ldots,n\})\backslash S_{n}/\Aut(\{1,\ldots,n - 1\}) = \{1,(1\; 2\; \cdots\; n)\}$ and therefore we have an exact sequence
\[
0\to (\ncInd_{G_{n - 1}}^{P_{n}}\sigma'\boxtimes\sigma_{1})^{\sm}\to (\nInd_{Q_{n}}^{G_{n}}\sigma_1\boxtimes\sigma')^{\sm}|_{P_n}
\to (\nInd_{Q_{n - 1}N_{n}}^{P_{n}}\sigma_1\boxtimes\sigma')^{\sm}\to 0.
\]
We now prove that any nonzero subrepresentation $\pi_{0}\subset (\nInd_{Q_n}^{G_n}\sigma_1\boxtimes\sigma')^{\sm}|_{P_n}$ is generic.
Assume that $\pi_{0}\cap (\ncInd_{G_{n - 1}}^{P_{n}}\sigma'\boxtimes\sigma_{1})^{\sm}\ne\{0\}$.
By replacing $\pi_{0}$ with this intersection, we may assume $\pi_{0}\subset (\ncInd_{G_{n - 1}}^{P_{n}}\sigma'\boxtimes\sigma_{1})^{\sm}$.
As an $N_{n}$-representation, we have $(\ncInd_{G_{n - 1}}^{P_{n}}\sigma'\boxtimes\sigma_{1})^{\sm}\simeq C_{c}^{\infty}(N_{n})\otimes (\sigma'\otimes\sigma_{1})$.
Hence there are no $N_{n}$-invariants.
Therefore $\pi_{0}^{N_{n}} = \{0\}$ and this implies $(\pi_{0})_{N_{n},\eta_{n}}\ne 0$ by Lemma~\ref{lem:from BZ}.
We also have a $P_{n - 1}$-isomorphism $((\ncInd_{G_{n - 1}}^{P_{n}}\sigma'\boxtimes\sigma_{1})^{\sm})_{N_{n},\eta_{n}} \simeq (\sigma'\otimes\sigma_{1})\otimes\lvert\det\rvert^{d/2}_{F}$, where $P_{n-1}$ acts trivially on $\sigma_1$. Hence by the inductive hypothesis, $(\pi_{0})_{U_{n},\theta} = ((\pi_{0})_{N_{n},\eta_{n}})_{U_{n - 1},\theta|_{U_{n - 1}}}\ne 0$.
Therefore $\pi_{0}$ is generic.

Next assume that $\pi_{0}\cap (\ncInd_{G_{n - 1}}^{P_{n}}\sigma'\boxtimes\sigma_{1})^{\sm} = \{0\}$.
In this case we have $\pi_{0}\subset (\nInd_{Q_{n - 1}N_{n}}^{P_{n}}\sigma_{1}\otimes\sigma')^{\sm}$, and the theorem follows from Lemma~\ref{lem:genericity for GLnD, lemma}.
\end{proof}

\subsection{Unitary case}
\label{sec:unitary-case}
We now state one of our main results.
We need some preparations for the proof, which will be completed in subsection \ref{subsec:proof of theorem for unitary}.
For $\alpha\in\Phi_{\red}^{+}$, let $L'_{\alpha}$ be the (closed) subgroup generated by $U \cap L_{\alpha}$ and $\overline{U} \cap L_{\alpha}$, where $\overline{\alggrp{U}}$ is the unipotent radical of the parabolic subgroup opposite to $\alggrp{B}$.
It has a concrete description, cf.\ \cite[II.4]{AHHV}.

For a character $\omega\colon G\to C^{\times}$ we define $e(\omega) = e_{\alggrp{G}}(\omega)\in \mathfrak{a}_{G,\R}^{*}$ by $\lvert\omega\rvert_{C} = \chi_{e(\omega)}$, where the absolute value $\lvert\cdot\rvert_C$ on $C$ is normalized such that $\lvert p\rvert_C = p^{-1}$.
(Recall that the unramified character $\chi_\nu$ for $\nu \in \mathfrak{a}_{G,\C}^{*}$ was defined in \S\ref{sec:reduc-points-parab}.)
In particular, it follows that $e(\omega) = [F:\Q_p]^{-1} \omega$ if $\omega \in X^*(\alggrp G) \subset \mathfrak{a}_{G,\R}^{*}$, as $\lvert\cdot\rvert_C^{[F:\Q_p]}$ restricts to $\lvert\cdot\rvert_F$ on $F \subset C$.

\begin{thm}\label{thm:unitary case}
  Assume Assumption~\ref{assump:on p}.
  Let $\sigma$ be a finite-dimensional absolutely irreducible continuous representation of $L$ and $\omega_{\sigma}$ the central character of $\sigma$.
  \begin{enumerate}
  \item The action of $Z$ on the coinvariants $\sigma_{U\cap L}$ is absolutely irreducible, so it has a central character $\omega_{\sigma_{U\cap L}}$. 
  \item If $e(\omega_{\sigma}|_{A_{L}})$ is dominant, then $(\Ind_{P}^{G}\sigma)^{\cts}$ is reducible if and only if, after perhaps replacing $C$ by a finite extension, $\underline L(\sigma')\in \mathcal{O}^{\wt{P}_{1}}$ for a parabolic subgroup $\alggrp{P}_{1}\supsetneq \alggrp{P}$.
    Here, $\wt{\alggrp P_1} := \Res_{F/\Q_p} \alggrp P_1$.
    \item     If $e(\omega_{\sigma_{U \cap L}}|_{S})$ is dominant, then $(\Ind_{P}^{G}\sigma)^{\cts}$ is reducible if and only if there exists $\alpha \in \Delta \setminus \Delta_L$ such that $\sigma|_{Z \cap L'_{\alpha}}$ is trivial.
  \end{enumerate}
\end{thm}

Recall that $\sigma' \in \mathcal O^{\wt L}$ after perhaps replacing $C$ by a finite extension (Lemma~\ref{lem:cts-rep}), where $\wt{\alggrp L} := \Res_{F/\Q_p} \alggrp L$.
Note that $e(\omega_{\sigma}|_{A_{L}}) = e(\omega_{\sigma_{U\cap L}}|_{A_{L}})$ is dominant if $e(\omega_{\sigma_{U\cap L}}|_S)$ is dominant by Lemma~\ref{lem:dominant-proj}(ii).
By \cite[II.7 Proposition]{AHHV} and Lemma~\ref{lm:loc-an-char}, $\sigma|_{Z \cap L'_{\alpha}}$ is trivial for some $\alpha \in \Delta \setminus \Delta_L$ if and only if $\sigma$ extends to a continuous representation of a larger Levi subgroup.
The condition $\underline L(\sigma')\in \mathcal{O}^{\wt P_{1}}$ in (ii) is made more explicit in Lemma \ref{lem:simple-in-O-Q}.

\begin{remark}
We say that a Banach representation $\sigma$ of $L$ is \emph{unitarizable} if it admits an $L$-invariant defining norm or equivalently an $L$-invariant open and bounded $\mathcal{O}_C$-lattice.
Suppose now that $\sigma$ has a central character $\omega_{\sigma}$ (for example, if it is admissible and absolutely irreducible).
Then $\sigma$ unitarizable implies that $\omega_\sigma$ is unitary, or equivalently $e(\omega_{\sigma}|_{A_L}) = 0$. If $L = Z$, then the converse is true too.
(To see that if $\omega_\sigma$ is unitary then $\sigma$ is unitarizable, note that $Z_0 S$ is of finite index in $Z$, where $Z_0$ is the unique maximal compact subgroup of $Z$.
By \cite[Lemma 6.5.5]{locallyanalytic-memoir} there exists a bounded open $Z_0$-invariant lattice $\Lambda$ in $\sigma$, so $\sum_{z \in Z/Z_0S} z\Lambda$ is a bounded open $Z$-invariant lattice.)
If $\sigma$ is unitarizable then $\sigma_{U\cap L}$ is also unitarizable.
Hence the hypotheses on central characters in parts (ii) and (iii) of Theorem~\ref{thm:unitary case} hold for any unitarizable $\sigma$.
\end{remark}

\begin{remark}
  Schneider \cite[Conjecture~2.5]{MR2275644} stated a conjecture concerning the irreducibility of Banach principal series of a split simply-connected group
  over a $p$-adic field, under an antidominance condition on the inducing character, as discussed in the introduction.
  Theorem~\ref{thm:unitary case} resolves this completely for unitarizable inducing representations, for an essentially arbitrary connected reductive group.
\end{remark}

\begin{remark}
When $\alggrp{G}$ is split, $\alggrp{P} = \alggrp{B}$ and $e(\omega_\sigma|_S)$ is strictly dominant, i.e.\ $\langle e(\omega_\sigma|_S),\alpha^{\vee}\rangle > 0$ for any positive root $\alpha$, Theorem~\ref{thm:unitary case} was proved by Ban--Hundley~\cite{MR3537231} (at least when the derived subgroup is simply connected).
Note that in this case there are no reducibilities, as $\sigma|_{Z \cap L'_{\alpha}}$ trivial implies $\ang{e(\omega_\sigma|_S),\alpha^\vee} = 0$.
\end{remark}

\begin{remark}
  We note that irreducible admissible Banach representations of $Z$ do not have to be finite-dimensional, even though $\alggrp Z$ is anisotropic modulo its center.
  For example, when $Z = D^\times$ with $D$ a central division algebra, \cite{MR4404133} shows the existence of an irreducible admissible representation of $\delta$-dimension 1 (in particular, of infinite dimension).  
\end{remark}

\subsection{Rationality: the quasisplit case}
\label{sec:rati-quasisplit}
To prove Theorem~\ref{thm:unitary case}, we require the rationality of reducibility points of certain smooth parabolic inductions over $\C$.
The reason is that we are given $p$-adic properties about our representation, for example that it is unitarizable over $C$, but our information about reducibility of smooth representations uses classical methods over $\C$.
Concretely, in our proof we use that if $s \in \Q$ then we know how to compare $|p^s|_C$ with $|p^s|_{\C}$, where we fixed some isomorphism $C \cong \C$ (but we lose control if $s \notin \Q$).

We will finish the proof of the necessary property in subsection~\ref{sec:rank-one-groups}.
\emph{In subsections~\ref{sec:rati-quasisplit}--~\ref{sec:rank-one-groups} all representations are smooth and over $\C$.}

Suppose now that $\alggrp G$ is quasisplit, so that $\alggrp Z$ is a maximal torus.
For $\nu \in \mathfrak{a}_{Z,\C}^*$ and $\alpha \in \Phi(\alggrp G,\alggrp A_{\alggrp Z})$ we can define $\ang{\nu,\alpha^\vee}_{\abs} := \ang{i(\nu),\widetilde{\alpha}^\vee}$, where $i \colon \mathfrak{a}_{Z,\C}^* = X^*(\alggrp Z) \otimes \C \into X^*(\alggrp Z \times \overline{F}) \otimes \C$ is the natural map and $\widetilde{\alpha} \in X^*(\alggrp Z \times \overline{F})$ is any absolute root lifting $\alpha$.
Note that this is well defined because the lifts $\widetilde{\alpha}$ form a $\Gal(\overline{F}/F)$-orbit. 
Note that $\ang{\alpha,\alpha^\vee}_{\abs} \in \Q_{>0}$.
(In fact, if $d$ denotes the number of lifts $\widetilde{\alpha}$ of $\alpha$ we have $\ang{\nu,\alpha^\vee}_{\abs} = \frac 1{2d}\ang{\nu,\alpha^\vee}$ if $2\alpha$ is a root and $\ang{\nu,\alpha^\vee}_{\abs} = \frac 1{d}\ang{\nu,\alpha^\vee}$ otherwise, cf.\ \cite[Lemma A.9]{MR4397148}.)

\begin{lem}\label{lm:rationality-supercusp}
  Suppose that $\alggrp{G}$ is quasisplit, $\alggrp{P} = \alggrp{L}\alggrp{N}$ maximal and semistandard, $\sigma$ a generic supercuspidal representation of $L$.
  If $\mu^G(\sigma \chi_\nu)$ has a pole at $\nu = \nu_0 \in \mathfrak{a}_{L,\R}^*$, then $\nu_0 \in (\mathfrak{a}^G_{L,\Q})^* \oplus \mathfrak{a}_{G,\R}^*$.
  Equivalently, if $\mu^G(\sigma \delta_P^s)$ has a pole at $s = s_0 \in \R$, then $s_0 \in \Q$.
  More precisely, there exists $\gamma \in \Phi(\alggrp G,\alggrp A_{\alggrp Z})$ occurring in $\alggrp N$ such that $\ang{\nu,\gamma^\vee}_{\abs} \in \{\pm \frac 12,\pm 1\}$.
\end{lem}

\begin{proof}
  First note that $\chi_\nu$ extends to an unramified character of $G$ if and only if $\nu \in \mathfrak{a}_{G,\R}^*$, so we may assume $\nu \in (\mathfrak{a}^{G}_{L,\R})^*$.
  As $P$ is maximal we can write $\Delta(\alggrp{P},\alggrp{A}_{\alggrp{L}}) = \{\alpha\}$, so $2\rho_P \in c \alpha$ for some $c \in \Z_{>0}$, i.e.\ $\delta_P = \chi_{c\alpha}$.
  Moreover, $\alpha$ is a basis of $(\mathfrak{a}^G_{L,\Q})^*$, which explains the equivalence of the two statements.
  Choose a Borel subgroup $\alggrp B$ contained in $\alggrp P$.
  Then there is a unique simple root $\gamma \in \Delta(\alggrp B,\alggrp A_{\alggrp Z})$ that occurs in $\alggrp N$,
  and \cite[Theorem~8.1]{MR1070599} implies that $\ang{\nu,\gamma^\vee}_{\abs} \in \{\pm \frac 12,\pm 1\}$.
  As $\ang{\alpha,\alpha^\vee}_{\abs} \in \Q_{>0}$, we deduce the second statement.
\end{proof}

\begin{prop}\label{prop:rationality-supercusp}
  Suppose that $\alggrp{G}$ is quasisplit, $\tau$ a discrete series representation of $G$ whose supercuspidal support is generic.
  Then there exists a standard parabolic subgroup $\alggrp{P} = \alggrp{L}\alggrp{N}$, a unitary (generic) supercuspidal representation $\sigma$ of $L$
  and $\nu \in (\mathfrak{a}_{L,\Q}^G)^* \subset \mathfrak{a}_{L,\R}^*$ such that $\tau$ is a quotient of $(\nInd_P^G \sigma \chi_\nu)^\sm$.
  More precisely, $\nu = \sum_{\alpha\in\Delta(\alggrp{P},\alggrp{A}_{\alggrp{L}})} c_\alpha \alpha$, where $c_\alpha \in \Q_{<0}$ for all $\alpha$.
\end{prop}

\begin{proof}
  All except the rationality of $\nu$, or equivalently of the $c_\alpha$, follows from \cite[p.\ 582]{MR577138} or \cite[Proposition III.1.1]{MR1989693}.
  (The parabolic subgroup there may be conjugated to be standard.)
  In particular, $\nu \in (\mathfrak{a}_{L,\R}^G)^{*}$.
  By \cite[p.\ 582]{MR577138} there exists a linearly independent set of reduced roots $\alpha_1,\dots,\alpha_\ell \in \Phi_{\mathrm{red}}(\alggrp{P},\alggrp{A}_{\alggrp{L}})$ with $\ell = |\Delta(\alggrp{P},\alggrp{A}_{\alggrp{L}})|$
  such that $\mu^{L_{\alpha_i}}(\sigma \chi_{\nu'})$ has a pole at $\nu' = \nu$ for $1 \le i \le \ell$.
  By Lemma~\ref{lm:rationality-supercusp} we deduce that the projection of $\nu$ under the projection $(\mathfrak{a}^G_{L,\R})^* \to (\mathfrak{a}^{L_{\alpha_i}}_{L,\R})^*$
  lies in $(\mathfrak{a}^{L_{\alpha_i}}_{L,\Q})^*$ for all $i$.
  It therefore suffices to show that the natural map $(\mathfrak{a}^G_{L,\R})^* \to \bigoplus_{i=1}^\ell (\mathfrak{a}^{L_{\alpha_i}}_{L,\R})^*$ is an isomorphism,
  or equivalently injective, or equivalently that $\bigcap_i \mathfrak{a}_{L_{\alpha_i},\R}^* = \mathfrak{a}_{G,\R}^*$.

  By dimension reasons we see that $\mathfrak{a}_{L_{\alpha_i},\R}^*$ is the annihilator of $\alpha_i^{\vee}$ in $\mathfrak{a}_{L,\R}^*$.
  By Lemma~\ref{lm:coroots}, the coroots $\alpha_1^{\vee},\dots,\alpha_\ell^{\vee}$ are linearly independent in $\mathfrak{a}_{L,\R}$, hence a basis of $\mathfrak{a}_{L,\R}^G$, so indeed $\bigcap_i \mathfrak{a}_{L_{\alpha_i},\R}^* \subset \mathfrak{a}_{G,\R}^*$.
\end{proof}

\begin{remark}\label{rk:generic-gln}
  Continue to suppose the hypotheses in Proposition~\ref{prop:rationality-supercusp}.
  If $\alggrp G$ is a subgroup of $\alggrp G' := (\prod_{i=1}^r \Res_{E_i/F}\GL_{n_i})/\alggrp{H}$ containing $(\alggrp G')^{\der}$,
  for some finite extensions $E_i/F$, integers $n_i \ge 1$ and a central subtorus $\alggrp{H}$ of $\prod_{i=1}^r \Res_{E_i/F}\GL_{n_i}$ that is moreover induced,
  then $\ang{\nu,\alpha^\vee}_{\abs} \in \Z$ for all $\alpha \in \Phi(\alggrp G, \alggrp A_{\alggrp Z})$.
  (First reduce to $\alggrp G'$; then reduce to $\alggrp H = 1$; finally check it in the case of $\Res_{E/F}\GL_{n}$ using \cite[Theorem 9.3]{MR584084}.)
  \end{remark}

\begin{prop}\label{prop:rationality-discrete-series}
    Suppose that $\alggrp{G}$ is quasisplit, $\alggrp{Q} = \alggrp{L}_{\alggrp Q}\alggrp{N}_{\alggrp Q}$ a standard parabolic subgroup, and $\tau$ a discrete series representation of $L_Q$ whose supercuspidal support is generic.
  If $\mu^G(\tau \delta_Q^s)$ has a pole at $s = s_0 \in \R$, then $s_0 \in \Q$.
\end{prop}

\begin{proof}
  By Proposition~\ref{prop:rationality-supercusp} we can write $\tau$ as quotient of $(\nInd_{P \cap L_Q}^{L_Q} \sigma \chi_\nu)^\sm$ for some standard parabolic $\alggrp{P} = \alggrp{L}\alggrp{N}$
  contained in $\alggrp{Q}$, some unitary supercuspidal representation $\sigma$ of $L$, and some $\nu \in (\mathfrak{a}_{L,\Q}^{L_Q})^*$.
  Write $\delta_Q = \chi_{\nu'}$ for $\nu' = 2\rho_Q \in (\mathfrak{a}^G_{L_Q,\Q})^*$.   Hence $(\nInd_{P \cap L_Q}^{L_Q} \sigma \chi_{\nu+s\nu'})^\sm$ surjects onto $\tau\delta_Q^s$.
  By \cite[Theorem~1]{MR570781} and Proposition~\ref{prop:product formula} we have
  \begin{equation*}
    \mu^G(\tau \delta_Q^s) = \frac{\mu^G(\sigma \chi_{\nu+s\nu'})}{\mu^{L_Q}(\sigma \chi_{\nu+s\nu'})}
    = \prod_{\beta \in \Phi_{\mathrm{red}}(\alggrp P,\alggrp A_{\alggrp L})\setminus\Phi_{\mathrm{red}}(\alggrp P \cap \alggrp L_{\alggrp Q},\alggrp A_{\alggrp L})} \mu^{L_\beta}(\sigma \chi_{\nu+s\nu'}).
  \end{equation*}
  If $\mu^G(\tau \delta_Q^s)$ has a pole at $s = s_0 \in \R$, then by Lemma~\ref{lm:rationality-supercusp} we deduce that
  there exists $\beta \in \Phi_{\mathrm{red}}(\alggrp{P},\alggrp{A}_{\alggrp{L}})\setminus\Phi_{\mathrm{red}}(\alggrp{P} \cap \alggrp{L_{\alggrp Q}},\alggrp{A}_{\alggrp{L}})$ such that
  $\nu+s_0\nu' \in (\mathfrak{a}^{L_\beta}_{L,\Q})^* \oplus \mathfrak{a}_{L_\beta,\R}^*$ or equivalently $s_0\nu' \in (\mathfrak{a}^{L_\beta}_{L,\Q})^* \oplus \mathfrak{a}_{L_\beta,\R}^*$
  (by rationality of $\nu$). 

  We claim that $\nu' \notin \mathfrak{a}_{L_\beta,\R}^*$, i.e.\ $\langle\nu',\beta^\vee\rangle_L \ne 0$, where $\langle\cdot,\cdot\rangle_{L}$ denotes the pairing $\mathfrak{a}_{L,\R}^*\times \mathfrak{a}_{L,\R}\to \R$.
  Note that $\beta_Q := \beta|_{\alggrp A_{\alggrp L_{\alggrp Q}}} \ne 1$, i.e.\ $\beta_Q \in \Phi(\alggrp G,\alggrp A_{\alggrp L_{\alggrp Q}})$.
  (If not, then $\alggrp A_{\alggrp L_{\alggrp Q}}$ is contained in $\ker(\beta)$ and so $\alggrp A_{\alggrp L_{\alggrp Q}}$ centralizes the unipotent subgroup
  $\alggrp U_{(\beta)} \subset \alggrp P$ whose Lie algebra consists of positive integer multiples of $\beta$, i.e.\
  $\alggrp U_{(\beta)} \subset \alggrp L_{\alggrp Q}$, which is a contradiction.)
  Hence $\langle \nu',\beta^\vee\rangle_L = \langle \nu', \beta_Q^\vee\rangle_{L_Q} \ne 0$ by Lemma~\ref{lm:coroots}(i), as $\nu' = 2\rho_Q$ in $(\mathfrak{a}^G_{L_Q,\Q})^*$,
  proving the claim.
  Therefore, the image of $\nu'$ under the projection $\mathfrak{a}_{L,\Q}^* \twoheadrightarrow (\mathfrak{a}^{L_\beta}_{L,\Q})^*$ has to be nonzero, and hence $s_0 \in \Q$.
\end{proof}

\begin{remark}\label{rk:pole-qsplit}
  More precisely, the above proof shows that (in the notation of the proof) there exists a root $\gamma \in \Phi(\alggrp G,\alggrp A_{\alggrp Z})$ that occurs in $\alggrp N_{\alggrp Q}$
  such that $\ang{\nu+2s\rho_Q,\gamma^\vee}_\abs \in \{\pm \frac 12,\pm1\}$.
  In particular, in case the Levi subgroup $\alggrp L_{\alggrp Q}$ satisfies the assumption of Remark~\ref{rk:generic-gln} we deduce that $s\ang{2\rho_Q,\gamma^\vee}_\abs \in \frac 12\Z$.
\end{remark}

\subsection{Rationality: inner forms}
\label{sec:rati-inner-forms}
In this subsection we prove the rationality of the reducibility points of $(\nInd_{P}^{G}\sigma\delta_{P}^{s})^\sm$ for some inner forms $\alggrp{G}$, which include all simply-connected almost simple groups of rank one (cf.\ subsection~\ref{sec:rank-one-groups} below).
For hermitian quaternionic groups of maximal Witt rank Mui\'{c}--Savin~\cite[\S2]{MR1749954} proved that the $\mu$-function of a discrete series representation on a Siegel Levi coincides with the $\mu$-function of its Jacquet--Langlands transfer by using a global argument.
This method was used in some further cases in \cite{Konno} and \cite{MR3194161}.
We will prove a similar result in greater generality, though we do not determine the $\mu$-function completely.
We remind the reader that in this subsection all representations are smooth and complex.

Our setting is the following: $\alggrp{G}$ is any connected reductive group, $\alggrp{P} = \alggrp{L}\alggrp{N}$ is a maximal parabolic subgroup such that $\alggrp{L}$ is an inner form of a group $\alggrp{L}'$ satisfying 
\[(\widetilde{\alggrp{L}}^{\lowprime})^{\der} \subset \alggrp{L}' \subset \widetilde{\alggrp{L}}^{\lowprime},\]
where 
\[\widetilde{\alggrp{L}}^{\lowprime} := \left(\prod_{i=1}^r \Res_{E_i/F}\GL_{n_i}\right)/\alggrp{H}\]
for some finite extensions $E_i/F$, integers $n_i \ge 1$ and a central subtorus $\alggrp{H}$ of $\prod_{i=1}^r \Res_{E_i/F}\GL_{n_i}$.
We moreover assume that $\alggrp{H}$ is an induced torus.
Then we prove the following.

\begin{thm}\label{thm:rationality, inner}
Let $\sigma$ be a discrete series representation of $L$.
If $\mu^G(\sigma\delta_{P}^{s})$ has a pole at $s = s_{0}\in \R$, then $s_{0}\in \Q$.
\end{thm}

We obtain the following corollary from our argument.
The inner form $\alggrp L$ of $\alggrp{L}'$ gives rise to an inner form $\widetilde{\alggrp{L}}$ of $\widetilde{\alggrp{L}}^{\lowprime}$ such that $\widetilde{\alggrp{L}}^{\lowprime[\der]} \subset \alggrp{L} \subset \widetilde{\alggrp{L}}$, as we now explain.
We have $\alggrp{Z}_{\alggrp{L}'} = \alggrp{L}'\cap \alggrp{Z}_{\widetilde{\alggrp{L}}^{\lowprime}}$ and hence a map $\alggrp{L}'/\alggrp{Z}_{\alggrp{L}'}\to \widetilde{\alggrp{L}}^{\lowprime}/\alggrp{Z}_{\widetilde{\alggrp{L}}^{\lowprime}}$ that induces $H^{1}(F,\alggrp{L}'/\alggrp{Z}_{\alggrp{L}'})\to H^{1}(F,\widetilde{\alggrp{L}}^{\lowprime}/\alggrp{Z}_{\widetilde{\alggrp{L}}^{\lowprime}})$.
The inner form $\alggrp{L}$ of $\alggrp{L}'$ corresponds to an element of $H^{1}(F,\alggrp{L}'/\alggrp{Z}_{\alggrp{L}'})$ and by the above map we get the desired 
inner form $\widetilde{\alggrp{L}}$ of $\widetilde{\alggrp{L}}^{\lowprime}$.

\begin{cor}\label{cor:rationality, inner}
  Suppose that $\sigma_1$, $\sigma_2$ are discrete series representations of $L$ that are conjugate under the action of $\widetilde{L}$.
  Then $\mu^G(\sigma_1\delta_{P}^{s}) = \mu^G(\sigma_2\delta_{P}^{s})$.
\end{cor}

This verifies \cite[Working Hypothesis 1.1]{MR3194161} (in our more general setting).
See also Remark~\ref{rk:choiy}.

We give an outline of the proof of Theorem~\ref{thm:rationality, inner}, which uses global methods.
Let $\alggrp{G}'$ be the quasisplit inner form of $\alggrp{G}$.
We will take a suitable globalization $\glob{G}$ of $\alggrp{G}$ (resp.\ $\glob{G'}$ of $\alggrp{G}'$) and a nice automorphic representation $\Pi$ (resp.\ $\Pi'$) of $\glob{G}$ (resp. $\glob{G'}$).
By using a result of Moeglin--Waldspurger~\cite{MR1361168}, one can compare the products of $\mu$-functions for $\Pi$ and $\Pi'$.
By Proposition~\ref{prop:rationality-discrete-series}, the $\mu$-functions for $\Pi'$ have rational poles.
(Note that the genericity assumption is automatic in this case, see Lemma~\ref{lem:levi-generic} below.)
We also know that the zeros of the $\mu$-function are rational by Proposition~\ref{prop:mu-max-parab}.
Therefore, the $\mu$-function for $\alggrp{G}(F)$ also has rational poles.

We start by giving some notation for real and global groups.
Bold letters will be used for global objects: for example, $\glob{F}$ will be a number field and $\glob{G}$ will be a connected reductive group over $\glob{F}$.
For each place $v$ of $\glob{F}$, $\glob{F}_{v}$ denotes the completion of $\glob{F}$ at $v$.
If $v$ is an infinite place of $\glob{F}$, a representation of $\glob{G}(\glob{F}_{v})$ means a $(\Lie(\glob{G}(\glob{F}_{v}))\otimes_{\R}\C,K_{v})$-module, where $K_{v}$ is a fixed maximal compact subgroup of $\glob{G}(\glob{F}_{v})$.
If $\glob{P} = \glob{L}\glob{N}$ is a parabolic subgroup of $\glob{G}$ and $\sigma$ a representation of $\glob{G}(\glob{F}_{v})$, let $(\nInd_{\glob{P}(\glob{F}_{v})}^{\glob{G}(\glob{F}_{v})}\sigma)^{\sm}$ be the normalized parabolic induction in the context of $(\Lie(\glob{G}(\glob{F}_{v}))\otimes_{\R}\C,K_{v})$-modules.
Let $\mathbb{A} = \mathbb{A}_{\glob{F}}$ be the adele ring of $\glob{F}$.
Let $(\nInd_{\glob{P}(\mathbb{A})}^{\glob{G}(\mathbb{A})}\sigma)^{\sm}$ be the normalized parabolic induction for a cuspidal automorphic representation $\sigma$ of $\glob{L}(\mathbb{A})$ (cf.\ \cite[II.1]{MR1361168}).

We now construct nice globalizations.

\begin{lem}\label{lem:global-fields}
There exists a number field $\glob{F}$, a place $v_{0}$ of $\glob{F}$, and finite extensions $\glob{E}/\glob{F},\glob{E}_{1}/\glob{F},\ldots,\glob{E}_{r}/\glob{F}$ such that $\glob{E}/\glob{F}$ is Galois, $\glob{F}_{v_{0}}\simeq F$, $v_{0}$ does not decompose in $\glob{E}$, $\alggrp{G}$ splits over $\glob{E}_{v_{0}}$, $\glob{E}\supset \glob{E}_{i}$, and $(\glob{E}_{i})_{v_{0}}\simeq E_{i}$.
In particular we have $\Gal(\glob{E}_{v_{0}}/F)\simeq\Gal(\glob{E}/\glob{F})$.
\end{lem}
\begin{proof}
Take a finite Galois extension $E/F$ such that $\alggrp{G}$ splits over $E$ and $E\supset E_{i}$ for each $i = 1,\ldots,r$.
Then by \cite[Lemma~3.1]{MR3861809} there exists a Galois extension $\glob{E}/\glob{F}$ of number fields and a place $v_{0}$ of $\glob{F}$ such that $v_{0}$ does not decompose in $\glob{E}$, $\glob{E}_{v_{0}}\simeq E$, $\glob{F}_{v_{0}}\simeq F$.
The subfield $E_{i}\subset E$ corresponds to a subgroup $H_{i}\subset \Gal(E/F)\simeq \Gal(\glob{E}/\glob{F})$ and let $\glob{E}_{i}$ be the subfield of $\glob{E}$ corresponding to $H_{i}$.
Then this gives the desired properties.
\end{proof}
As above, let $\alggrp{G}'$ be the quasisplit inner form of $\alggrp{G}$, and fix an inner twist $\xi \colon \alggrp G_{\o F} \congto \alggrp G'_{\o F}$ over $\o F$.
After conjugating by $\alggrp G'(\o F)$ we may assume that the parabolic subgroup $\xi(\alggrp P_{\o F})$ and its Levi $\xi(\alggrp L_{\o F})$ are defined over $F$ (\cite[Lemma 3]{MR4048514}), and we let $\alggrp{P}' = \alggrp{L}'\alggrp{N}'$ the parabolic subgroup of $\alggrp G'$ obtained in this way, i.e.\ $\xi(\alggrp P_{\o F}) = \alggrp{P}'_{\o F}$ and $\xi(\alggrp L_{\o F}) = \alggrp{L}'_{\o F}$.
Then $\xi$ induces an inner twist $\alggrp L_{\o F} \congto \alggrp L'_{\o F}$ (\cite[Lemma 3]{MR4048514}) and so $\alggrp L'$ is the unique quasisplit inner form, consistent with our notation above.
In particular, $\xi$ induces $F$-isomorphisms $\alggrp Z_{\alggrp L} \cong \alggrp Z_{\alggrp L'}$ and $\alggrp A_{\alggrp L} \cong \alggrp A_{\alggrp L'}$.
We note that $\alggrp{G}'$ splits over $\glob{E}_{v_{0}}$, because $\alggrp{L}'$ splits over $\glob{E}_{v_{0}}$.
(Note that $(\alggrp{L}')^\der = (\widetilde{\alggrp{L}}^{\lowprime})^{\der}$ splits over $\glob{E}_{v_{0}}$ and similarly for the radical $\alggrp{Z}_{\alggrp{L}'}^0$, noting that $\alggrp{Z}_{\alggrp{L}'} = \alggrp{L}'\cap \alggrp{Z}_{\widetilde{\alggrp{L}}^{\lowprime}}$.)

The quasisplit group $\alggrp{G}'$ is classified by the action of $\Gal(\glob{E}_{v_{0}}/F)$ on the root datum of $\alggrp{G}'$.
Since $\Gal(\glob{E}_{v_{0}}/F)\simeq \Gal(\glob{E}/\glob{F})$, it also determines a quasisplit group $\glob{G}'$ over $\glob{F}$.
The parabolic subgroup $\alggrp{P}'$ corresponds to a set of simple roots and let $\glob{P}' = \glob{L}'\glob{N}'\subset\glob{G}'$ be the corresponding parabolic subgroup. Then we have $(\glob{G}'_{v_{0}},\glob{L}'_{v_{0}},\glob{N}'_{v_{0}})\simeq (\alggrp{G}',\alggrp{L}',\alggrp{N}')$, where $\glob{G}'_{v_{0}} := \glob{G'}\times_{\glob{F}}\glob{F}_{v_{0}}$ (and this notation is used for other groups as well).

\begin{lem}
There exists an inner form $(\glob{G},\glob{P},\glob{L})$ of $(\glob{G}',\glob{P}',\glob{L}')$ such that $(\glob{G}_{v_{0}},\glob{P}_{v_{0}},\glob{L}_{v_{0}})\simeq (\alggrp{G},\alggrp{P},\alggrp{L})$ and $(\glob{G}_{v},\glob{P}_{v},\glob{L}_{v})\simeq (\glob{G}'_{v},\glob{P}'_{v},\glob{L}'_{v})$ for any infinite place $v$.
\end{lem}
\begin{proof}
This follows like in the proof of \cite[Lemma~3.2]{MR3861809}.Let $V_{\infty}$ be the set of infinite places.
Inner forms of $(\glob{G}',\glob{P}',\glob{L}')$ are classified by $H^{1}(\glob{F},I(\glob{G}',\glob{P}',\glob{L}'))$, where $I(\glob{G}',\glob{P}',\glob{L}')$ is the group of inner automorphisms of the triple $(\glob{G}',\glob{P}',\glob{L}')$.
Since the normalizer of $\glob{P}'$ in $\glob{G}'$ is $\glob{P}'$ and the normalizer of $\glob{L}'$ in $\glob{P}'$ is $\glob{L}'$, we have $I(\glob{G}',\glob{P}',\glob{L}')\simeq \glob{L}'/\glob{Z}_{\glob{G}'}$, where $\glob{Z}_{\glob{G}'}$ is the center of $\glob{G}'$.
Therefore it is sufficient to prove that the map
\begin{equation}\label{eq:Galois cohomology lemma, want to prove}
H^{1}(\glob{F},\glob{L}'/\glob{Z}_{\glob{G}'})\to \prod_{v\in V_{\infty}\cup\{v_{0}\}}H^{1}(\glob{F}_{v},\glob{L}'/\glob{Z}_{\glob{G}'})
\end{equation}
is surjective.
By replacing $\glob{G}'$ with $\glob{G}'/\glob{Z}_{\glob{G}'}$, we assume that $\glob{Z}_{\glob{G}'}$ is trivial.
Hence it is sufficient to prove that $H^{1}(\glob{F},\glob{L}')\to \prod_{v\in V_{\infty}\cup\{v_{0}\}}H^{1}(\glob{F}_{v},\glob{L}')$ is surjective.

First we observe that $\glob{Z}_{\glob{L}'}$ is an induced torus.
Indeed, let $\overline{\glob{F}}$ be an algebraic closure of $\glob{F}$.
Since $\glob{Z}_{\glob{G}'}$ is trivial, $\glob{Z}_{\glob{L}'}$ is a torus and the fundamental coweights form a basis of the cocharacter lattice $X_{*}(\glob{T}'\times_{\glob{F}}\overline{\glob{F}})$, where $\glob{T}'\subset \glob{L}'$ is a minimal Levi subgroup (a torus).
Then a Galois-stable subset of fundamental coweights is a basis of $X_{*}(\glob{Z}_{\glob{L}'}\times_{\glob{F}}\overline{\glob{F}})$,
i.e.\ $\glob{Z}_{\glob{L}'}$ is an induced torus.
Hence there exist finite extensions $\glob{E}_{1},\ldots,\glob{E}_{r}$ such that $\glob{Z}_{\glob{L}'}\simeq \prod_i \Res_{\glob{E}_{i}/\glob{F}}\mathbb{G}_{\mathrm{m}}$.
Let $v'\ne v_{0}$ be a finite place of $\glob{F}$ which does not decompose in $\glob{E}_{i}$ for $i = 1,\ldots,r$ and we regard $v'$ also as a place of $\glob{E}_{i}$.
Let $V(\glob{F})$ (resp.\ $V(\glob{E}_{i})$) be the set of places of $\glob{F}$ (resp.\ $\glob{E}_{i}$).
Then 
\[
\begin{tikzcd}
H^{2}(\glob{F},\glob{Z}_{\glob{L}'})\arrow[r]\arrow[d,"\wr"] & \bigoplus_{v\in V(\glob{F})\setminus\{v'\}}H^{2}(\glob{F}_{v},\glob{Z}_{\glob{L}'})\arrow[d,"\wr"]\\
\bigoplus_{i = 1}^{r}H^{2}(\glob{E}_{i},\mathbb{G}_{\mathrm{m}}) \arrow[r]\arrow[r] 
& \bigoplus_{i = 1}^{r}\bigoplus_{v\in V(\glob{E}_{i})\setminus\{v'\}}H^{2}((\glob{E}_{i})_{v},\mathbb{G}_{\mathrm{m}}).
\end{tikzcd}
\]
By the local-global principle for Brauer groups we see as in \cite[Lemma~3]{MR2271021} that the bottom horizontal map is an isomorphism.

Set $V' := V(\glob{F})\setminus\{v'\}$.
As $\glob{Z}_{\glob{L}',v}$ is an induced torus, for any $v\in V(\glob{F})$,
$H^{1}(\glob{F}_{v},\glob{Z}_{\glob{L}'})$ is trivial by Hilbert 90.
Hence $H^{1}(\glob{F}_{v},\glob{L}')\to H^{1}(\glob{F}_{v},\glob{L}'/\glob{Z}_{\glob{L}'})$ is injective~\cite[I.5.7~Proposition~42]{MR1867431}.
We have the commutative diagram of pointed sets with exact rows:
\[
\begin{tikzcd}
H^{1}(\glob{F},\glob{L}')\arrow[r]\arrow[d] &  H^{1}(\glob{F},\glob{L}'/\glob{Z}_{\glob{L}'})\arrow[d]\arrow[r] & H^{2}(\glob{F},\glob{Z}_{\glob{L}'})\arrow[d]\\
\bigoplus_{v\in V'}H^{1}(\glob{F}_{v},\glob{L}')\arrow[r,hook] & \bigoplus_{v\in V'} H^{1}(\glob{F}_{v},\glob{L}'/\glob{Z}_{\glob{L}'})\arrow[r] & \bigoplus_{v\in V'}H^{2}(\glob{F}_{v},\glob{Z}_{\glob{L}'}).
\end{tikzcd}
\]
By the previous paragraph, the right vertical map is an isomorphism.
We also have that the middle map is surjective~\cite[Proposition~1]{MR2271021}.
Therefore the left vertical map is surjective and hence \eqref{eq:Galois cohomology lemma, want to prove} is surjective.
\end{proof}

\begin{remark}
  We point out that the proof of \cite[Lemma~3.2]{MR3861809} contains a mistake, namely it only works if the torus $A = M'_{\ad}/M'_{\ad,\der}$ considered there is an induced torus (for example, if $M'_{\ad}$ is split).
  Our proof works with the center instead of the cocenter of the Levi, which has the advantage that it is always induced.
\end{remark}

Take $\glob{G}$ and $\glob{P} = \glob{L}\glob{N}$ as in the lemma.
We fix an inner twist $\glob{\xi} \colon \glob{G}_{\o{\glob F}} \congto \glob{G}'_{\o{\glob F}}$ sending $(\glob{P}_{\o{\glob F}},\glob{L}_{\o{\glob F}})$ to $(\glob{P}'_{\o{\glob F}},\glob{L}'_{\o{\glob F}})$ such that $\glob{\xi}_{v_0}$ is equivalent to $\xi$ and $\glob{\xi}_v$ is trivial for all infinite places $v$.
As above, $\glob\xi$ induces an isomorphism $\glob A_{\glob L} \cong \glob A_{\glob L'}$ which becomes the above isomorphism $\alggrp A_{\alggrp L} \cong \alggrp A_{\alggrp L'}$ after base change to $\glob F_{v_0} = F$.

Recall that we have $\alggrp{H}\subset \prod_{i = 1}^{r}\Res_{E_{i}/F}\GL_{1}$.
Such a subtorus is classified by a $\Gal(\glob{E}_{v_{0}}/F)$-stable saturated subgroup of $X^{*}(\prod_{i = 1}^{r}\Res_{E_{i}/F}\GL_{1} \times_{F} \glob{E}_{v_{0}})$ and therefore by Lemma \ref{lem:global-fields} we have a subtorus $\glob{H}\subset \prod_{i = 1}^{r}\Res_{\glob{E}_{i}/\glob{F}}\GL_{1}$ such that $\glob{H}_{v_{0}}\simeq \alggrp{H}$.
As $\alggrp{H}$ is an induced torus, so is $\glob{H}$.
We put $\widetilde{\glob{L}}^{\lowprime} := (\prod_{i = 1}^{r}\Res_{\glob{E}_{i}/\glob{F}}\GL_{n_{i}})/\glob{H}$.
The subgroups $(\widetilde{\glob{L}}^{\lowprime})^\der \subset \glob{L}'' \subset \widetilde{\glob{L}}^{\lowprime}$ are in bijection with subgroups $(\widetilde{\alggrp{L}}^{\lowprime})^\der \subset \alggrp{L}'' \subset \widetilde{\alggrp{L}}^{\lowprime}$ upon completion at $v_0$ (since $\widetilde{\glob{L}}^{\lowprime}/(\widetilde{\glob{L}}^{\lowprime})^\der$ is a torus that splits over $\glob{E}$), and any such $\glob{L}''$ is quasisplit and splits over $\glob{E}$.
In particular, we get $(\widetilde{\glob{L}}^{\lowprime})^\der \subset \glob{L}' \subset \widetilde{\glob{L}}^{\lowprime}$.

Let $\glob{Z}_{\glob{L}'}$ (resp.\ $\glob{Z}_{\widetilde{\glob{L}}^{\lowprime}}$) be the center of $\glob{L}'$ (resp.\ $\widetilde{\glob{L}}^{\lowprime}$).
As explained before Corollary~\ref{cor:rationality, inner}, the inner form $\glob{L}$ of $\glob{L}'$ gives rise to an inner form $\widetilde{\glob{L}}$ of $\widetilde{\glob{L}}^{\lowprime}$ such that $\widetilde{\glob{L}}^{\lowprime[\der]} \subset \glob{L} \subset \widetilde{\glob{L}}$.
We remark that $\widetilde{\glob{L}}_{v_0} \cong \widetilde{\alggrp L}$ and that $\glob{Z}_{\widetilde{\glob{L}}} = \glob{Z}_{\widetilde{\glob{L}}^{\lowprime}}$ (a torus).

Recall that we have a discrete series representation $\sigma$ of $L \simeq \glob{L}(\glob{F}_{v_{0}})$.
Take an irreducible discrete representation $\widetilde{\sigma}$ of $\widetilde{\glob{L}}(\glob{F}_{v_{0}})$ whose restriction to $L$ contains $\sigma$.
(See for example \cite[\S2]{MR1141803}.)

Recall that $\mathbb{A} = \mathbb{A}_{\glob{F}}$ is the adele ring of $\glob{F}$.

\begin{lem}\label{lem:globalise-char}
There exists a unitary character $\glob{Z}_{\widetilde{\glob{L}}^{\lowprime}}(\glob{F})\backslash \glob{Z}_{\widetilde{\glob{L}}^{\lowprime}}(\mathbb{A})\to \C^{\times}$ whose $v_{0}$-component is the central character of $\widetilde{\sigma}$.
\end{lem}
\begin{proof}
First we remark that the $\glob{F}$-rational rank of $\glob{Z}_{\widetilde{\glob{L}}^{\lowprime}}$ is equal to the $F$-rational rank of $\glob{Z}_{\widetilde{\glob{L}}^{\lowprime},v_{0}}$ by Lemma~\ref{lem:global-fields}. 
Let $\omega_{\widetilde{\sigma}}$ be the central character of $\widetilde{\sigma}$ and $V^{\infty}$ be the set of finite places of $\glob{F}$.
For each $v\in V^{\infty}$, let $\glob{Z}_{\widetilde{\glob{L}}^{\lowprime},v}^{0}$ be the maximal compact subgroup of $\glob{Z}_{\widetilde{\glob{L}}^{\lowprime},v}$ (which is the unique parahoric subgroup for almost all $v$).
Then $K := \prod_{v\in V^{\infty}}\glob{Z}_{\widetilde{\glob{L}}^{\lowprime},v}^{0}$ is a compact subgroup of $\glob{Z}_{\widetilde{\glob{L}}^{\lowprime}}(\mathbb{A})$. Note that $\glob{Z}_{\widetilde{\glob{L}}^{\lowprime}}(\glob{F})\cap K$ is trivial since there is an infinite place.
Hence $K\hookrightarrow \glob{Z}_{\widetilde{\glob{L}}^{\lowprime}}(\glob{F})\backslash \glob{Z}_{\widetilde{\glob{L}}^{\lowprime}}(\mathbb{A})$ is a closed subgroup.
Consider the character of $K$ defined by $(k_{v})_{v\in V^{\infty}}\mapsto \omega_{\widetilde{\sigma}}(k_{v_{0}})$ and extend it to a unitary character $\omega'$ of the locally compact abelian group $\glob{Z}_{\widetilde{\glob{L}}^{\lowprime}}(\glob{F})\backslash \glob{Z}_{\widetilde{\glob{L}}^{\lowprime}}(\mathbb{A})$.
Let $\omega'_{v_{0}}$ be the $v_{0}$-component of $\omega'$.
Then $\omega_{\widetilde{\sigma}}(\omega'_{v_{0}})^{-1}$ is a unitary unramified character.
Hence $\omega_{\widetilde{\sigma}}(\omega'_{v_{0}})^{-1} = \chi_{\nu}$ for some $\nu \in \sqrt{-1}\mathfrak{a}_{\glob{Z}_{\widetilde{\glob{L}}^{\lowprime},v_{0}},\R}^*$.
Since the $\glob{F}$-rational rank of $\glob{Z}_{\widetilde{\glob{L}}^{\lowprime}}$ is equal to the $F$-rational rank of $\glob{Z}_{\widetilde{\glob{L}}^{\lowprime},v_{0}}$, $\nu$ can be seen in $\sqrt{-1}\mathfrak{a}_{\glob{Z}_{\widetilde{\glob{L}}^{\lowprime}},\R}^*$, where $\mathfrak{a}^*_{\glob{Z}_{\widetilde{\glob{L}}^{\lowprime}},\R} := X^{*}(\glob{Z}_{\widetilde{\glob{L}}^{\lowprime}})\otimes_{\Z}\R$.
Then $\omega\colon \glob{Z}_{\widetilde{\glob{L}}^{\lowprime}}(\glob{F})\backslash \glob{Z}_{\widetilde{\glob{L}}^{\lowprime}}(\mathbb{A})\to \C^{\times}$ defined by $\omega(a) := \omega'(a)q^{H(\nu(a))}$ gives the desired character, where $H\colon \glob{Z}_{\widetilde{\glob{L}}^{\lowprime}}(\glob{F})\backslash \glob{Z}_{\widetilde{\glob{L}}^{\lowprime}}(\mathbb{A})\to \mathfrak{a}_{\glob{Z}_{\widetilde{\glob{L}}^{\lowprime}},\R}$ is the global Harish-Chandra homomorphism.
\end{proof}

As $\glob{H}$ is an induced torus, we have
\[\widetilde{\glob{L}}^{\lowprime}(\glob{F})\simeq \Big(\prod_{i = 1}^{r}(\Res_{\glob{E}_{i}/\glob{F}}\GL_{n_{i}})(\glob{F})\Big)/\glob{H}(\glob{F})\] and 
similarly for $\widetilde{\glob{L}}^{\lowprime}(\glob{F}_v)$ and $\widetilde{\glob{L}}^{\lowprime}(\mathbb{A})$. (In the adelic case, this follows for example from the argument in \cite[\S III.2.3]{MR762353}.)
The representation $\widetilde{\sigma}$ is a discrete series of an inner form of $\prod_{i = 1}^{r}(\Res_{\glob{E}_{i}/\glob{F}}\GL_{n_{i}})(\glob{F}_{v_{0}})$.
Let $\widetilde{\sigma}'$ be the Jacquet--Langlands transfer of $\widetilde{\sigma}$ to the quasisplit form.
Since the Jacquet--Langlands correspondence preserves central characters, $\widetilde{\sigma}'$ is a representation of $\widetilde{\glob{L}}^{\lowprime}(\glob{F}_{v_{0}})$.

Let $V_{0}$ be a finite set of finite places containing $v_0$ and all (finite) places $v$ such that $\glob{G}_{v}$ is not quasisplit.
Let $v_s \notin V_{0}$ be any finite place.

\begin{lem}
There exists a unitary cuspidal automorphic representation $\widetilde{\Pi}'$ of $\widetilde{\glob{L}}^{\lowprime}(\mathbb{A})$ such that $\widetilde{\Pi}'_{v_{0}}\simeq \widetilde{\sigma}'$, $\widetilde{\Pi}'_{v}$ is a discrete series for each $v\in V_{0}$, and $\widetilde{\Pi}'_{v_s}$ is supercuspidal.
\end{lem}
\begin{proof}
By Lemma~\ref{lem:globalise-char} there is a unitary character $\omega\colon \glob{Z}_{\widetilde{\glob{L}}^{\lowprime}}(\glob{F})\backslash\glob{Z}_{\widetilde{\glob{L}}^{\lowprime}}(\mathbb{A}) \to \C^\times$ whose $v_{0}$-component is the central character of $\widetilde{\sigma}'_{v_{0}}$.
Then the lemma follows by applying the trace formula with central character $\omega$.
See the argument in the proof of \cite[III.3~Proposition]{MR928510}.
\end{proof}

Let $\widetilde{\Pi}$ be the global Jacquet--Langlands transfer of $\widetilde{\Pi}'$ to $\widetilde{\glob{L}}(\mathbb{A})$~\cite[Theorem~5.1]{MR2390289}.
This is again unitary and exists because $\widetilde{\Pi}'_{v}$ is a discrete series for each $v\in V_{0}$.
(We use again that the Jacquet--Langlands correspondence preserves central characters.)
By construction, $\widetilde{\Pi}_{v_0} \cong \widetilde{\sigma}$, $\widetilde{\Pi}_{v}$ is a discrete series for each $v\in V_{0}$, $\widetilde{\Pi}_{v} \cong \widetilde{\Pi}'_{v}$ for all $v\notin V_0$, and $\widetilde{\Pi}$ is cuspidal, as it is supercuspidal at $v_s$.

From now on we fix compatible collections of hyperspecial subgroups $K_{v}\subset \glob{G}(\glob{F}_{v})$ and $K'_{v}\subset \glob{G}'(\glob{F}_{v})$ for all finite places $v \notin V_0$ (arising from some integral models of $\glob{G}$ and $\glob{G}'$).
Recall that we have fixed an inner twist $\glob{\xi} \colon (\glob{G}_{\o{\glob F}},\glob{P}_{\o{\glob F}},\glob{L}_{\o{\glob F}}) \congto (\glob{G}'_{\o{\glob F}},\glob{P}'_{\o{\glob F}},\glob{L}'_{\o{\glob F}})$.

\begin{lem}\label{lem:hyperspecial}
  We may increase $V_0$ such that for all finite places $v \notin V_0$ the isomorphism $\glob\xi_v$ is defined over $\glob{F}_v$, and moreover $\glob\xi_v(K_v) = K'_v$.
\end{lem}

\begin{proof}
  By spreading out and increasing $V_0$ we obtain connected reductive $\mathcal{O}_{\glob{F},V_0}$-group scheme models of $\glob{G}, \glob{P}, \glob{L}$ and $\glob{G}', \glob{P}', \glob{L}'$ (for some finite set of places $V_0$), which we denote by the same letters. 
  (Here, $\mathcal{O}_{\glob{F},V_0}$ denotes the localization of $\mathcal{O}_{\glob{F}}$ away from $V_0$.)
  By increasing $V_0$ we may assume that $\glob{G}(\mathcal{O}_{\glob{F},v}) = K_v$ and $\glob{G}'(\mathcal{O}_{\glob{F},v}) = K'_v$ for finite $v \notin V_0$.
  Increasing $V_0$ further, the class of $\glob\xi$ in $H^1(\glob{F},\glob{L}'/\glob{Z}_{\glob{G}'})$ spreads out to an \'etale $\glob{L}'/\glob{Z}_{\glob{G}'}$-torsor over $\mathcal{O}_{\glob{F},V_0}$, which has to become trivial over $\mathcal{O}_{\glob{F},v}$ for any finite $v \notin V_0$ by Lang's theorem and the smoothness of $\glob{L}'/\glob{Z}_{\glob{G}'}$.
  This proves that $\glob\xi_v$ is defined over $\glob{F}_v$. 
    On the other hand, pick a finite extension $\glob{E}'/\glob{F}$ such that $\glob\xi$ descends to $\glob{E}'$.
    We increase $V_0$ such that $\glob\xi$ spreads out to an isomorphism over $\mathcal{O}_{\glob{E}',V_0}$ and such that $\mathcal{O}_{\glob{F},V_0} \to \mathcal{O}_{\glob{E}',V_0}$ is \'etale.
  Completing at any finite place $v \notin V_0$ we obtain an isomorphism $(\glob{G},\glob{P}) \times_{\mathcal{O}_{\glob{F},V_0}} \mathcal{O}_{\glob{E}',v} \congto (\glob{G}',\glob{P}') \times_{\mathcal{O}_{\glob{F},V_0}} \mathcal{O}_{\glob{E}',v}$ that arises from an isomorphism over $\mathcal{O}_{\glob{F},v}$ by \'etale descent (since we already know the cocycle condition holds on the generic fiber).
  This completes the proof.
\end{proof}

In particular, for finite places $v \notin V_0$ we may identify $\glob{G}(\glob{F}_{v})$ with $\glob{G}'(\glob{F}_{v})$ (and likewise for $\glob{P}$, $\glob{L}$, $\glob{N}$)
and $K_v$ with $K'_v$.

\begin{lem}\label{lem:density argument}
  Suppose we are given irreducible constituents $\sigma_v$ of $\wt\Pi_v|_{\glob{L}(\glob{F}_{v})}$ for all $v \in V_0$.
  Then there exists a (unitary) cuspidal automorphic representation $\Pi$ of $\glob{L}(\mathbb{A})$ which is a quotient of $\widetilde{\Pi}|_{\glob{L}(\mathbb{A})}$, a (unitary) cuspidal automorphic representation $\Pi'$ of $\glob{L}'(\mathbb{A})$ which is a quotient of $\widetilde{\Pi}'|_{\glob{L}'(\mathbb{A})}$, and a finite set $V_{1}$ of places containing $V_{0}$ and the set of infinite places such that
\begin{itemize}
\item for $v \in V_0$, $\Pi_{v}\simeq \sigma_v$;
\item for $v\in V_{1}\setminus V_{0}$, $\Pi_{v}\simeq \Pi'_{v}$;
\item for $v\notin V_{1}$, $\widetilde{\Pi}_{v} \simeq \widetilde{\Pi}'_{v}$ is unramified.
  \end{itemize}
Moreover, $\Pi'$ and $V_1$ do not depend on the choice of the $\sigma_v$ ($v \in V_0$).
\end{lem}

\begin{proof}
We fix a maximal compact subgroup $\widetilde{K}_{v}\subset\widetilde{\glob{L}}(\glob{F}_{v})$ for infinite $v$ (the subgroup used implicitly in defining automorphic representations on $\widetilde{\glob{L}}$) and a compact open subgroup $\widetilde{K}_{v}\subset\widetilde{\glob{L}}(\glob{F}_{v})$ for finite $v$ such that $\prod_{v \nmid \infty} \widetilde{K}_{v}$ is compact open in $\widetilde{\glob{L}}(\mathbb{A}^\infty)$ where $\mathbb{A}^{\infty}$ is the finite part of $\mathbb{A}$.
Let $\widehat{\widetilde{\Pi}}$ be the cuspidal $L^{2}$-automorphic representation of $\widetilde{\glob{L}}(\mathbb{A})$ whose subspace of $\prod_{v}\widetilde{K}_{v}$-finite vectors is $\widetilde{\Pi}$.
By \cite[Theorem~1.1.1]{MR3996349} we have an $L^{2}$-automorphic subrepresentation $\widehat{\Pi}_{0}$ of $\widehat{\widetilde{\Pi}}|_{\glob{L}(\mathbb{A})}$.
For each place $v$, let $\widehat{\widetilde{\Pi}}_{v}$ (resp.\ $\widehat{\Pi}_{0,v}$) be the $v$-component of $\widehat{\widetilde{\Pi}}$ (resp.\ $\widehat{\Pi}_{0}$).
Then $\widehat{\widetilde{\Pi}}_v$ is the (unique) unitary completion of $\widetilde{\Pi}_v$ and $\widehat{\Pi}_{0,v}$ is a subrepresentation of $\widehat{\widetilde{\Pi}}_v|_{\glob{L}(\glob{F}_{v})}$.
Likewise we have a compact subgroup $\prod_v \widetilde{K}'_{v}\subset\widetilde{\glob{L}}^{\lowprime}(\mathbb{A})$ and an $L^{2}$-automorphic subrepresentation $\widehat{\Pi}'$ of $\widehat{\widetilde{\Pi}}^{\lowprime}|_{\glob{L}'(\mathbb{A})}$.

Let $V_{1}$ be a finite set of places containing $V_{0}$ and the set of infinite places such that for any $v\notin V_{1}$, $\glob{G}_{v}$ is unramified (in particular isomorphic to $\glob{G}'_{v}$) and $\widetilde{\Pi}_{v} \simeq \widetilde{\Pi}'_{v}$ is unramified.
For each $v\in V_{0}$, let $\widehat{\sigma}_{v}$ be the (unique) unitary completion of $\sigma_{v}$, so in particular $\widehat{\sigma}_{v} \subset \widehat{\widetilde{\Pi}}_{v}|_{\glob{L}(\glob{F}_{v})}$ (it is the closure of $\sigma_{v}$).
Let $\widehat{\sigma}_{v} := \widehat{\Pi}'_{v}$ for $v\in V_{1}\setminus V_{0}$, and note that $\widehat{\widetilde{\Pi}}_v \cong \widehat{\widetilde{\Pi}}^{\lowprime}_v$ for such $v$.
For each $v\in V_{1}$ define $X_{v} := \{g_v \in \widetilde{\glob{L}}(\glob{F}_{v})\mid \widehat{\Pi}_{0,v}\circ \Ad(g_v)\simeq \widehat{\sigma}_{v}\}$.
Note that $\widehat{\Pi}_{0,v}$ and $\widehat{\sigma}_{v}$ are extended to $\glob{L}(\glob{F}_{v})\glob{Z}_{\widetilde{\glob{L}}}(\glob{F}_{v})$, as subrepresentations of $\widehat{\widetilde{\Pi}}_{v}|_{\glob{L}(\glob{F}_{v})\glob{Z}_{\widetilde{\glob{L}}}(\glob{F}_{v})}$, and $\glob{L}(\glob{F}_{v})\glob{Z}_{\widetilde{\glob{L}}}(\glob{F}_{v})\subset \widetilde{\glob{L}}(\glob{F}_{v})$ is open and of finite index.
Hence by Clifford theory, $X_{v}$ is non-empty \cite[Proposition 4.1.3]{MR3996349} and also open as it is $\glob{L}(\glob{F}_{v})\glob{Z}_{\widetilde{\glob{L}}}(\glob{F}_{v})$-stable.
Since $\widetilde{\glob{L}}(\glob{F})\subset \prod_{v\in V_{1}}\widetilde{\glob{L}}(\glob{F}_{v})$ is dense, there exists $g\in \widetilde{\glob{L}}(\glob{F})\cap \prod_{v\in V_{1}}X_{v}$.
Then $\widehat{\Pi} := \widehat{\Pi}_{0}\circ \Ad(g)$ is a cuspidal $L^{2}$-automorphic representation, a subrepresentation of $\widehat{\widetilde{\Pi}}|_{\glob{L}(\mathbb{A})}$ and $\widehat{\Pi}_{v}\simeq \widehat{\sigma}_{v}$ for any $v\in V_{1}$.
As $\widehat{\widetilde{\Pi}}|_{\glob{L}(\mathbb{A})}$ is unitary, $\widehat{\Pi}$ is also a quotient of $\widehat{\widetilde{\Pi}}|_{\glob{L}(\mathbb{A})}$.
Let $\Pi$ be the subspace of $\prod_{v}(\widetilde{K}_{v}\cap \glob{L}(\glob{F}_{v}))$-finite vectors in $\widehat{\Pi}$. Then this is an irreducible automorphic representation of $\glob{L}(\mathbb{A})$.
The surjective map $\widehat{\widetilde{\Pi}}|_{\glob{L}(\mathbb{A})} \to \widehat{\Pi}$ is nonzero on $\widetilde{\Pi}$ since $\widetilde{\Pi}\subset \widehat{\widetilde{\Pi}}$ is dense and the image is contained in $\Pi$.
Therefore we get a nonzero homomorphism $\widetilde{\Pi}|_{\glob{L}(\mathbb{A})}\to \Pi$, and it is surjective by irreducibility.
Likewise we define $\Pi'$ and get a surjection $\widetilde{\Pi}'|_{\glob{L}'(\mathbb{A})}\onto \Pi'$.
\end{proof}

Take $\Pi$ and $V_{1}$ as in the lemma, taking $\sigma_{v_0} := \sigma$ and the other $\sigma_v$ arbitrary.
Let $V$ be a finite set of places containing $V_{1}$ such that for all $v\notin V$ the subgroup $K_{v}\cap \glob{L}(\glob{F}_{v})$ of $\glob{L}(\glob{F}_{v})$ is hyperspecial and $\Pi_{v}$ and $\Pi'_{v}$ have nonzero $K_{v}\cap \glob{L}(\glob{F}_{v})$-fixed vectors.
(Such a set exists by a spreading out argument.)
Then we have
\begin{itemize}
\item $\Pi_{v_{0}}\simeq \sigma$;
\item for $v\in V_{0}$, $\Pi_{v}$ and $\Pi'_{v}$ are discrete series;
\item for $v\in V_{1}\setminus V_{0}$, $\glob{G}_{v}\simeq \glob{G}'_{v}$ and $\Pi_{v}\simeq \Pi'_{v}$;
\item for $v\in V\setminus V_{1}$, $\glob{G}_{v}\simeq \glob{G}'_{v}$ and $\Pi_{v},\Pi'_{v}$ are quotients of $\widetilde{\Pi}_{v}|_{\glob{L}(\glob{F}_{v})}$ and $\widetilde{\Pi}_{v}$ is an unramified $\widetilde{\glob{L}}(\glob{F}_{v})$-representation;
\item for $v\notin V$, $\glob{G}_{v}\simeq \glob{G}'_{v}$ and $\Pi_{v}\simeq \Pi'_{v}$ have nonzero $K_{v}\cap \glob{L}(\glob{F}_{v})$-fixed vectors. \end{itemize}

By above, our inner twist $\glob{\xi}$ gives us natural identifications between $\Phi(\alggrp{P},\alggrp{A}_{\alggrp{L}})$, $\Phi(\alggrp{P}',\alggrp{A}_{\alggrp{L}'})$, $\Phi(\glob{P},\glob{A}_{\glob{L}})$, and $\Phi(\glob{P}',\glob{A}_{\glob{L}'})$.

Let $\overline{\glob{P}} = \glob{L}\overline{\glob{N}}$ be the parabolic subgroup of $\glob{G}$ opposite to $\glob{P}$.
For each $\alpha\in \Phi(\glob{P},\glob{A}_{\glob{L}})$ and each place $v$, we fix Haar measures on $\glob{N}(\glob{F}_{v}) \cap \glob{L}_{\alpha}(\glob{F}_{v})$ and $\overline{\glob{N}}(\glob{F}_{v}) \cap \glob{L}_{\alpha}(\glob{F}_{v})$ such that 
\begin{itemize}
\item if $v\notin V_{1}$, then $\glob{N}(\glob{F}_{v})\cap \glob{L}_{\alpha}(\glob{F}_{v})\cap K_{v}$ and $\overline{\glob{N}}(\glob{F}_{v})\cap \glob{L}_{\alpha}(\glob{F}_{v})\cap K_{v}$ have volume $1$;
\item for the product measure, $(\glob{N}(\glob{F}) \cap \glob{L}_{\alpha}(\glob{F}))\backslash (\glob{N}(\mathbb{A}) \cap \glob{L}_{\alpha}(\mathbb{A}))$ and $(\overline{\glob{N}}(\glob{F}) \cap \glob{L}_{\alpha}(\glob{F}))\backslash (\overline{\glob{N}}(\mathbb{A})\cap \glob{L}_{\alpha}(\mathbb{A}))$ have volume $1$.
\end{itemize}
We also choose measures on $\glob{N}'(\glob{F}_{v}) \cap \glob{L}'_{\alpha}(\glob{F}_{v})$ and $\overline{\glob{N}}'(\glob{F}_{v}) \cap \glob{L}'_{\alpha}(\glob{F}_{v})$ which satisfy the analogous conditions. 
Note that if $v\notin V_{1}$ then the measure on $\glob{N}(\glob{F}_{v})\cap \glob{L}_{\alpha}(\glob{F}_{v})$ (resp.\ $\overline{\glob{N}}(\glob{F}_{v})\cap \glob{L}_{\alpha}(\glob{F}_{v})$) is the same as that on $\glob{N}'(\glob{F}_{v}) \cap \glob{L}'_{\alpha}(\glob{F}_{v})$ (resp.\ $\overline{\glob{N}}'(\glob{F}_{v}) \cap \glob{L}'_{\alpha}(\glob{F}_{v})$) under our identification of $(\glob{G}(\glob{F}_{v}),\glob{P}(\glob{F}_{v}),\glob{L}(\glob{F}_{v}),\glob{N}(\glob{F}_{v}),K_{v})$ with $(\glob{G}'(\glob{F}_{v}),\glob{P}'(\glob{F}_{v}),\glob{L}'(\glob{F}_{v}),\glob{N}'(\glob{F}_{v}),K'_{v})$.

\begin{lem}\label{lem:comparison of mu for unramified}
For $v\in V\setminus V_{1}$, $j(\Pi_{v}\delta_{\glob{P}(\glob{F}_{v})}^{s}) = j(\Pi'_{v}\delta_{\glob{P}'(\glob{F}_{v})}^{s})$ for $s\in \C$.
\end{lem}

\begin{proof}
Let $\widetilde{\alggrp{T}}\subset \widetilde{\glob{L}}_{v}$ be a minimal Levi subgroup (a torus). Take a Borel subgroup $\alggrp{Q}_{\widetilde{\glob{L}}_{v}}$ of $\widetilde{\glob{L}}_{v}$ containing $\widetilde{\alggrp{T}}$ and a character $\widetilde{\chi}$ of $\widetilde{T}$ such that $(\nInd_{Q_{\widetilde{\glob{L}}_{v}}}^{\widetilde{\glob{L}}(\glob{F}_{v})}\widetilde{\chi})^{\sm}\twoheadrightarrow \widetilde{\Pi}_{v}$.
Set $\alggrp{T} := \widetilde{\alggrp{T}}\cap \glob{L}_{v}$ and let $\chi$ be the restriction of $\widetilde{\chi}$ to $T$.
Let $\alggrp{Q}$ (resp.\ $\overline{\alggrp{Q}}$) be the Borel subgroup of $\glob{G}_{v}$ containing $\glob{N}_{v}$ (resp.\ $\overline{\glob{N}}_{v}$) and such that $\alggrp{Q}\cap \glob{L}_{v} = \overline{\alggrp{Q}}\cap \glob{L}_{v} = \alggrp{Q}_{\widetilde{\glob{L}}_{v}}\cap \glob{L}_{v}$.
Then we have $(\nInd_{Q\cap \glob{L}(\glob{F}_{v})}^{\glob{L}(\glob{F}_{v})}\chi)^{\sm} = (\nInd_{Q_{\widetilde{\glob{L}}_{v}}}^{\widetilde{\glob{L}}(\glob{F}_{v})}\widetilde{\chi})^{\sm}|_{\glob{L}(\glob{F}_{v})}\twoheadrightarrow \widetilde{\Pi}_{v}|_{\glob{L}(\glob{F}_{v})} \twoheadrightarrow \Pi_v$ and a commutative diagram
\[
\begin{tikzcd}[column sep=3.3cm]
(\nInd_{Q}^{\glob{G}(\glob{F}_{v})}\chi\delta_{\glob{P}(\glob{F}_{v})}^{s})^{\sm}\arrow[r,"J_{\overline{Q}|Q}(\chi\delta_{\glob{P}(\glob{F}_{v})}^{s})"]\arrow[d,twoheadrightarrow] &
(\nInd_{\overline{Q}}^{\glob{G}(\glob{F}_{v})}\chi\delta_{\glob{P}(\glob{F}_{v})}^{s})^{\sm}\arrow[d,twoheadrightarrow]\\
(\nInd_{\glob{P}(\glob{F}_{v})}^{\glob{G}(\glob{F}_{v})}\Pi_{v}\delta_{\glob{P}(\glob{F}_{v})}^{s})^{\sm}\arrow[r,"J_{\overline{\glob{P}}(\glob{F}_{v})|\glob{P}(\glob{F}_{v})}(\Pi_{v}\delta_{\glob{P}(\glob{F}_{v})}^{s})"] &
(\nInd_{\overline{\glob{P}}(\glob{F}_{v})}^{\glob{G}(\glob{F}_{v})}\Pi_{v}\delta_{\glob{P}(\glob{F}_{v})}^{s})^{\sm}.
\end{tikzcd}
\]
There is a rational function $c(s)$ such that $J_{Q|\overline{Q}}(\chi\delta_{\glob{P}(\glob{F}_{v})}^{s})\circ J_{\overline{Q}|Q}(\chi\delta_{\glob{P}(\glob{F}_{v})}^{s}) = c(s)$ and from the above diagram we get $c(s) = j(\Pi_{v}\delta_{\glob{P}(\glob{F}_{v})}^{s})$.
The above diagram also holds when we replace $\Pi_{v}$ by $\Pi_{v}'$.
Hence $c(s) = j(\Pi'_{v}\delta_{\glob{P}'(\glob{F}_{v})}^{s})$ and we get the lemma.
\end{proof}

Let $s\in \C$ and for $v \notin V$ let $f_{v,\unram}\in (\nInd_{\glob{P}(\glob{F}_{v})}^{\glob{G}(\glob{F}_{v})}\Pi_{v}\delta_{\glob{P}(\glob{F}_{v})}^{s})^{\sm}$ (resp.\ $\overline{f}_{v,\unram}\in (\nInd_{\overline{\glob{P}}(\glob{F}_{v})}^{\glob{G}(\glob{F}_{v})}\Pi_{v}\delta_{\glob{P}(\glob{F}_{v})}^{s})^{\sm}$) be a $K_{v}$-fixed vector.
We assume $f_{v,\unram}(1) = \overline{f}_{v,\unram}(1)\ne 0$ and for almost all $v$, this is the vector in $\Pi_{v}$ used to define the restricted tensor product $\Pi = \bigotimes'_{v}\Pi_{v}$.
Then there exists a rational function $c_{v}(s,\glob{P},\Pi_{v})$ such that 
\[
J_{\overline{\glob{P}}(\glob{F}_{v})|\glob{P}(\glob{F}_{v})}(\Pi_{v}\delta_{\glob{P}(\glob{F}_{v})}^{s})f_{v,\unram} = c_{v}(s,\glob{P},\Pi_{v})\overline{f}_{v,\unram}.
\]
for each $s\in \C$.
An explicit formula of $c_{v}(s,\glob{P},\Pi_{v})$ can be found in \cite[3.1~Theorem]{MR571057}.
(Use the diagram in the proof of Lemma~\ref{lem:comparison of mu for unramified}; the product in the cited theorem is then taken over all roots of $\glob{N}_v$.)
From the formula, the infinite product $c^{V}(s,\glob{P},\Pi) = \prod_{v\notin V}c_{v}(s,\glob{P},\Pi_{v})$ is identified with a ratio of products of certain partial L-functions
that occur in the Langlands--Shahidi method, cf.\ \cite[(2.7)]{MR942520}.
For us, it suffices to know that $c^{V}(s,\glob{P},\Pi)$ converges if $\mathrm{Re}(s)$ is sufficiently large and gives a meromorphic function on $\C$.
This holds also for $\overline{\glob{P}}$. (In this case it converges if $-\mathrm{Re}(s)$ is sufficiently large.)

We have a global intertwining operator $J_{\overline{\glob{P}}|\glob{P}}(\Pi\delta_{\glob{P}(\mathbb{A})}^{s})\colon (\nInd_{\glob{P}(\mathbb{A})}^{\glob{G}(\mathbb{A})}\Pi\delta_{\glob{P}(\mathbb{A})}^{s})^{\sm}\to (\nInd_{\overline{\glob{P}}(\mathbb{A})}^{\glob{G}(\mathbb{A})}\Pi\delta_{\glob{P}(\mathbb{A})}^{s})^{\sm}$ that converges if $\mathrm{Re}(s)$ is sufficiently large.
Put
\begin{align*}
f^{V} &:= \bigotimes_{v\notin V}f_{v,\unram}\in \sideset{}{'}\bigotimes_{v\notin V}(\nInd_{\glob{P}(\glob{F}_{v})}^{\glob{G}(\glob{F}_{v})}\Pi_{v}\delta_{\glob{P}(\glob{F}_{v})}^{s})^{\sm}, \\
\overline{f}^{V} &:= \bigotimes_{v\notin V}\overline{f}_{v,\unram}\in \sideset{}{'}\bigotimes_{v\notin V}(\nInd_{\overline{\glob{P}}(\glob{F}_{v})}^{\glob{G}(\glob{F}_{v})}\Pi_{v}\delta_{\glob{P}(\glob{F}_{v})}^{s})^{\sm}.
\end{align*}
Take $0\ne f_{V} := \bigotimes_{v\in V}f_{v}\in \bigotimes_{v\in V}(\nInd_{\glob{P}(\glob{F}_{v})}^{\glob{G}(\glob{F}_{v})}\Pi_{v}\delta_{\glob{P}(\glob{F}_{v})}^{s})^{\sm}$.
Then if $\mathrm{Re}(s)$ is sufficiently large,
\begin{align*}
&J_{\overline{\glob{P}}|\glob{P}}(\Pi\delta_{\glob{P}(\mathbb{A})}^{s})(f^{V}\otimes f_{V})\\
& = \bigotimes_{v\notin V} J_{\overline{\glob{P}}(\glob{F}_{v})|\glob{P}(\glob{F}_{v})}(\Pi_{v}\delta_{\glob{P}(\glob{F}_{v})}^{s})f_{v,\unram}\otimes\bigotimes_{v\in V}J_{\overline{\glob{P}}(\glob{F}_{v})|\glob{P}(\glob{F}_{v})}(\Pi_{v}\delta_{\glob{P}(\glob{F}_{v})}^{s})f_{v}\\
& = \bigotimes_{v\notin V}c_{v}(s,\glob{P},\Pi_{v})\overline{f}_{v,\unram}\otimes \bigotimes_{v\in V}J_{\overline{\glob{P}}(\glob{F}_{v})|\glob{P}(\glob{F}_{v})}(\Pi_{v}\delta_{\glob{P}(\glob{F}_{v})}^{s})f_{v}\\
& = c^{V}(s,\glob{P},\Pi)\overline{f}^{V}\otimes\bigotimes_{v\in V}J_{\overline{\glob{P}}(\glob{F}_{v})|\glob{P}(\glob{F}_{v})}(\Pi_{v}\delta_{\glob{P}(\glob{F}_{v})}^{s})f_{v}.
\end{align*}
Both sides are meromorphic in $s$, hence this equality holds as meromorphic functions.
The same calculation gives
\begin{align*}
& J_{\glob{P}|\overline{\glob{P}}}(\Pi\delta_{\glob{P}(\mathbb{A})}^{s})J_{\overline{\glob{P}}|\glob{P}}(\Pi\delta_{\glob{P}(\mathbb{A})}^{s})(f^{V}\otimes f_{V})\\
& =
c^{V}(s,\glob{P},\Pi)c^{V}(s,\overline{\glob{P}},\Pi)f^{V}\otimes\bigotimes_{v\in V}J_{\glob{P}(\glob{F}_{v})|\overline{\glob{P}}(\glob{F}_{v})}(\Pi_{v}\delta_{\glob{P}(\glob{F}_{v})}^{s})J_{\overline{\glob{P}}(\glob{F}_{v})|\glob{P}(\glob{F}_{v})}(\Pi_{v}\delta_{\glob{P}(\glob{F}_{v})}^{s})f_{v}\\
& = \left(c^{V}(s,\glob{P},\Pi)c^{V}(s,\overline{\glob{P}},\Pi)\prod_{v\in V}j(\Pi_{v}\delta_{\glob{P}(\glob{F}_{v})}^{s})\right)f^{V}\otimes f_{V}.
\end{align*}
On the other hand, the composition of intertwining operators $J_{\glob{P}|\overline{\glob{P}}}(\Pi\delta_{\glob{P}(\mathbb{A})}^{s})J_{\overline{\glob{P}}|\glob{P}}(\Pi\delta_{\glob{P}(\mathbb{A})}^{s})$ is the identity~\cite[Theorem IV.1.10]{MR1361168}.
Hence we get
\[
c^{V}(s,\glob{P},\Pi)c^{V}(s,\overline{\glob{P}},\Pi)\prod_{v\in V}j(\Pi_{v}\delta_{\glob{P}(\glob{F}_{v})}^{s}) = 1
\]
and thus
\[
\prod_{v\in V}\mu^{\glob{G}(\glob{F}_{v})}(\Pi_{v}\delta_{\glob{P}(\glob{F}_{v})}^{s}) \sim c^{V}(s,\glob{P},\Pi)c^{V}(s,\overline{\glob{P}},\Pi),
\]
where $\sim$ denotes equality up to some factor in $\R_{>0}^\times$ that only depends on $(\glob{G},\glob{L})$ (or $(\glob{G}',\glob{L}')$ below).
Notice that if $v\notin V$ then $\glob{G}_{v}\simeq \glob{G}'_{v}$ and $\Pi_{v}\simeq \Pi'_{v}$.
Hence, by the same argument for $\Pi'$, we have
\[
\prod_{v\in V}\mu^{\glob{G}'(\glob{F}_{v})}(\Pi'_{v}\delta_{\glob{P}'(\glob{F}_{v})}^{s}) \sim c^{V}(s,\glob{P},\Pi)c^{V}(s,\overline{\glob{P}},\Pi).
\]
Hence we get
\[
\prod_{v\in V}\mu^{\glob{G}(\glob{F}_{v})}(\Pi_{v}\delta_{\glob{P}(\glob{F}_{v})}^{s})
\sim
\prod_{v\in V}\mu^{\glob{G}'(\glob{F}_{v})}(\Pi'_{v}\delta_{\glob{P}'(\glob{F}_{v})}^{s}).
\]
By Lemma~\ref{lem:comparison of mu for unramified} and $\Pi_{v}\simeq \Pi'_{v}$ for $v\in V_{1}\setminus V_{0}$, we have
\begin{equation}
  \label{eq:global-comparison-mu}
  \prod_{v\in V_{0}}\mu^{\glob{G}(\glob{F}_{v})}(\Pi_{v}\delta_{\glob{P}(\glob{F}_{v})}^{s})
  \sim
  \prod_{v\in V_{0}}\mu^{\glob{G}'(\glob{F}_{v})}(\Pi'_{v}\delta_{\glob{P}'(\glob{F}_{v})}^{s}).
\end{equation}
Since $\glob{G}'_{v}$ is quasisplit and $\Pi'_{v}$ is a discrete series for $v \in V_0$, if the right-hand side has a pole at $s \in \R$, then $s\in \Q$ by Proposition~\ref{prop:rationality-discrete-series}.
(Note that the supercuspidal support of $\Pi'_{v}$ is generic by Lemma~\ref{lem:levi-generic} below.) On the other hand, if $\mu^{\glob{G}(\glob{F}_{v})}(\Pi_{v}\delta_{\glob{P}(\glob{F}_{v})}^{s}) = 0$ for $v\in V_{0}\setminus\{v_{0}\}$ and $s \in \R$, then $s = 0$ by Proposition~\ref{prop:mu-max-parab}.
Hence if $s$ is a pole of $\mu^{\glob{G}(\glob{F}_{v_{0}})}(\Pi_{v_{0}}\delta_{\glob{P}(\glob{F}_{v_{0}})}^{s}) = \mu^{G}(\sigma\delta_{P}^{s})$, then $s\in \Q$.
This concludes the proof of Theorem~\ref{thm:rationality, inner}.

\begin{remark}
  With more work it might be possible to show that $\mu^G(\sigma \delta_P^s) = \mu^{G'}(\sigma' \delta_{P'}^s)$ for all constituents $\sigma'$ of $\widetilde{\sigma}'$, as in \cite{MR1749954}, \cite{MR3194161}.
\end{remark}

\begin{lem}\label{lem:levi-generic}
  Any supercuspidal representation of any Levi subgroup of $L'$ is generic.
\end{lem}

\begin{proof}
  Suppose that $\tau$ is a supercuspidal representation of a Levi subgroup $\alggrp{L}''$ of $\alggrp{L}'$.
  Then $\alggrp{L}'' = \widetilde{\alggrp{L}}^{\lowprime[\prime\prime]} \cap \alggrp{L}'$ for some Levi subgroup $\widetilde{\alggrp{L}}^{\lowprime[\prime\prime]}$ of $\widetilde{\alggrp{L}}^{\lowprime}$ and $\tau$ occurs in the restriction
  of a supercuspidal representation $\widetilde\tau$ of $\widetilde{\alggrp{L}}^{\lowprime[\prime\prime]}$ \cite[\S2]{MR1141803}.
  In this way we may assume that $\alggrp{L}' = \widetilde{\alggrp{L}}^{\lowprime}$ (perhaps by varying the nondegenerate character).
      By a similar (but easier) argument we consider the surjective homomorphism $\prod_{i=1}^r \Res_{E_i/F}\GL_{n_i} \onto \widetilde{\alggrp{L}}^{\lowprime}$ to reduce to the case when $\alggrp{H} = 1$.
  Finally the claim for the group $\prod_{i=1}^r \Res_{E_i/F}\GL_{n_i}$ follows from \cite[4.4 Theorem]{MR0579172}.
\end{proof}

\begin{proof}[Proof of Corollary~\ref{cor:rationality, inner}]
  As the discrete series $\sigma_{1},\sigma_{2}$ are $\widetilde L$-conjugate, there exists a discrete series $\widetilde{\sigma}$ of $\widetilde{L}$ such that $\sigma_{1},\sigma_{2}$ are irreducible constituents of $\widetilde{\sigma}|_L$.
  We then keep the setup and notation of the proof of Theorem~\ref{thm:rationality, inner} above.
  In particular, $\widetilde{\Pi}$ is a cuspidal automorphic representation of $\widetilde{L}$ such that $\widetilde{\Pi}_{v_{0}} \cong \widetilde{\sigma}$.
    Applying Lemma~\ref{lem:density argument} twice we have a finite set $V_1$ containing $V_0$ and the infinite places and cuspidal automorphic representations $\Pi_{1}, \Pi_2$ of $\glob{L}(\mathbb{A})$ (resp.\ $\Pi'$ of $\glob{L}'(\mathbb{A})$) which are quotients of $\widetilde{\Pi}|_{\glob{L}(\mathbb{A})}$ (resp.\ $\widetilde{\Pi}'|_{\glob{L}'(\mathbb{A})}$) such that
  \begin{itemize}
  \item $\Pi_{i,v_{0}}\simeq \sigma_{i}$ for $i \in \{1,2\}$;
  \item $\Pi_{1,v}\simeq \Pi_{2,v}$ for any $v\in V_0\setminus \{v_{0}\}$;
  \item $\Pi_{1,v}\simeq \Pi_{2,v} \cong \Pi'_v$ for any $v\in V_1\setminus V_0$;
              \end{itemize}
  Choose any finite set $V$ of places containing $V_{1}$ such that for all $v\notin V$ the subgroup $K_{v}\cap \glob{L}(\glob{F}_{v})$ of $\glob{L}(\glob{F}_{v})$ is hyperspecial and $\Pi_{1,v}$, $\Pi_{2,v}$, $\Pi'_{v}$ have nonzero $K_{v}\cap \glob{L}(\glob{F}_{v})$-fixed vectors.
      Following the argument above, we obtain from \eqref{eq:global-comparison-mu} that
  \begin{equation*}
    \prod_{v\in V_{0}}\mu^{\glob{G}(\glob{F}_{v})}(\Pi_{1,v}\delta_{\glob{P}(\glob{F}_{v})}^{s}) = 
    \prod_{v\in V_{0}}\mu^{\glob{G}'(\glob{F}_{v})}(\Pi_{2,v}\delta_{\glob{P}'(\glob{F}_{v})}^{s})
  \end{equation*}
  and hence by the first two bullets above that $\mu^{G}(\sigma_{1}\delta_{P}^{s}) = \mu^{G}(\sigma_{2}\delta_{P}^{s})$.
  \end{proof}

\begin{remark}\label{rk:choiy}
  We can extend Corollary~\ref{cor:rationality, inner}, showing that $J_{P|Q}(\sigma_{i}\delta_{P}^{s})\circ J_{Q|P}(\sigma_{i}\delta_{P}^{s})$ is independent of $i \in \{1,2\}$, where $\alggrp Q$ is an arbitrary parabolic subgroup with Levi subgroup $\alggrp L$, not just $\alggrp Q = \overline{\alggrp P}$, as in the statement of \cite[Working Hypothesis 1.1]{MR3194161}.
  (The above argument still works because of the natural identifications between $\Phi(\alggrp{P},\alggrp{A}_{\alggrp{L}})$, $\Phi(\alggrp{P}',\alggrp{A}_{\alggrp{L}'})$, $\Phi(\glob{P},\glob{A}_{\glob{L}})$, and $\Phi(\glob{P}',\glob{A}_{\glob{L}'})$. Alternatively, we can use formulas (12) and (14) in \cite[IV.1]{MR1989693} to reduce to the case where $\alggrp P$ is a maximal parabolic and $\alggrp Q$ its opposite.)
\end{remark}

\subsection{Rationality: rank one groups}
\label{sec:rank-one-groups}

Suppose that $\alggrp{G}$ is simply connected and almost simple of rank one over $F$. 
We first recall the classification of such groups, cf.\ \cite[Table II]{MR0224710} or \cite[\S4]{BT2}. 
First, $\alggrp{G} \cong \Res_{E/F} \alggrp{G}'$ for some finite extension $E/F$ and absolutely almost simple group $\alggrp{G}'$ of rank one, 
so we may assume that $\alggrp{G}$ is absolutely almost simple.
Let $D$ denote the nonsplit quaternion algebra over $F$, considered with its canonical involution (fixing precisely $F$).
All (skew-)hermitian forms we consider are non-degenerate.

\begin{prop}[Tits]\label{prop:rank-one-groups}
  If $\alggrp{G}$ is simply connected and absolutely almost simple of rank one over $F$, then $\alggrp{G}$ is isomorphic to one of the following:
  \begin{enumerate}
  \item $\SL_2(D')$, where $D'$ is a finite-dimensional central division algebra of dimension $d^2$ over $F$;
  \item $\SU(h)$, where $E/F$ is a quadratic extension, $h$ is a hermitian form over $E$ of rank 3 or 4 and Witt index 1
    (the groups are quasisplit $\SU_3$ and non-quasisplit $\SU_4$, respectively);
  \item $\SU(h)$, where $h$ is a hermitian form over $D$ of rank 1 or 2 and Witt index 1
    (the groups are inner forms of $\Sp_4$ and $\Sp_6$, respectively);
  \item $\tSU(h)$, where $h$ is a skew-hermitian form over $D$ of rank 4 or 5 and Witt index 1
    (the groups are an inner form of quasisplit $\Spin_8$ defined by a quadratic extension $E/F$ and an inner form of split $\Spin_{10}$, respectively).
  \end{enumerate}
\end{prop}

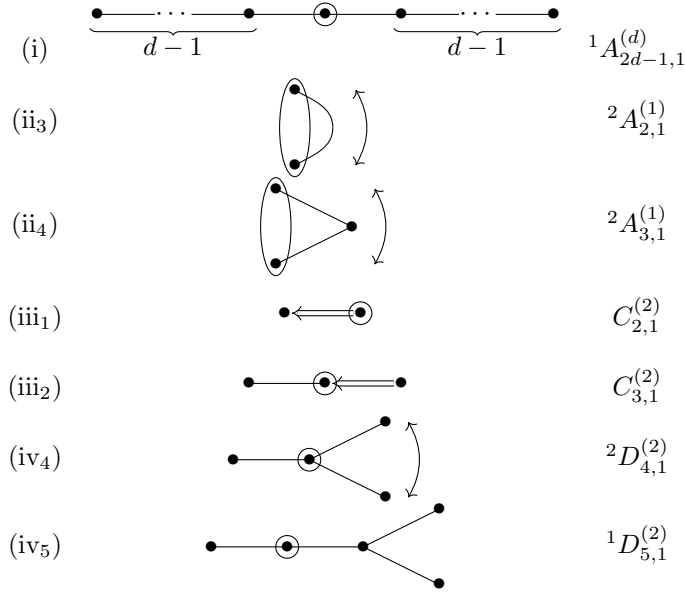
\begin{figure}[tb]
\centering
\[{\renewcommand*{\arraystretch}{2.2}
\begin{array}{ccc}
\text{(i)}&
\begin{tikzpicture}[every node/.style = {inner sep=0pt}]
\path[-] (0,0) node (ll)  {$\bullet$} -- ++ (1,0) node (lm) {$\cdots$}  -- ++ (1,0) node (lr) {$\bullet$}
  -- ++ (1,0) node (m) {$\bullet$}
  -- ++ (1,0) node (rl) {$\bullet$} -- ++ (1,0) node (rm) {$\cdots$}  -- ++ (1,0) node (rr) {$\bullet$};
\draw (ll.center) -- (lm);
\draw (lm) -- (lr.center) -- (m.center) -- (rl.center) -- (rm);
\draw (rm) -- (rr.center);
\draw [decorate,decoration={brace,raise=1mm}](lr.south east) -- (ll.south west) node[midway,below=2mm] {$d - 1$};
\draw [decorate,decoration={brace,raise=1mm}](rr.south east) -- (rl.south west) node[midway,below=2mm] {$d - 1$};
\draw (m) circle[radius=1.5mm];
\end{tikzpicture}&
{}^{1}A_{2d-1,1}^{(d)}
\\
\text{(ii$_3$)}&
\begin{tikzpicture}[every node/.style = {inner sep=0pt},baseline=(a.base)]
\path (0,0) node {$\bullet$}
(0,0.5) node (a) {}
(0,1) node {$\bullet$};
\draw (a) ellipse[x radius=2mm,y radius=6.5mm];
\draw [-] (0,0) to [out=30,in=-90] (0.5,0.5);
\draw [-] (0,1) to [out=-30,in=90] (0.5,0.5);
\draw [<->] (0.8,0) to [out=60,in=-60] (0.8,1);
\end{tikzpicture}&
{}^2 A_{2,1}^{(1)}
\\
\text{(ii$_4$)}&
\begin{tikzpicture}[every node/.style = {inner sep=0pt},baseline=(a2.base)]
\path (0,0) node (a1) {$\bullet$} -- ++ (0,0.5) node (c) {} -- ++ (0,0.5) node (a3) {$\bullet$} -- ++ (1,-0.5) node (a2) {$\bullet$};
\draw[-] (a1.center) -- (a2.center) -- (a3.center);
\draw (c) ellipse[x radius=2mm,y radius=6.5mm];
\draw [<->] (1.3,0) to [out=60,in=-60] (1.3,1);
\end{tikzpicture}&
{}^2 A_{3,1}^{(1)}
\\
\text{(iii$_1$)}&
\begin{tikzpicture}[every node/.style = {inner sep=0pt}]
\path (0,0) node (a1) {$\bullet$} -- ++ (-1,0) node (a2) {$\bullet$};
\draw[{-Implies[]},double equal sign distance] (a1.west) -- (a2);
\draw (a1) circle[radius=1.5mm];
\end{tikzpicture}&
C_{2,1}^{(2)}
\\
\text{(iii$_2$)}&
\begin{tikzpicture}[every node/.style = {inner sep=0pt}]
\path (0,0) node (a1) {$\bullet$} -- ++ (-1,0) node (a2) {$\bullet$} -- ++ (-1,0) node (a3) {$\bullet$};
\draw[{-Implies[]},double equal sign distance] (a1.west) -- (a2);
\draw[-] (a2.center) -- (a3.center);
\draw (a2) circle[radius=1.5mm];
\end{tikzpicture}&
C_{3,1}^{(2)}
\\
\text{(iv$_4$)}&
\begin{tikzpicture}[every node/.style = {inner sep=0pt},baseline=(a1.base)]
\path (0,0) node (a1) {$\bullet$} -- ++(1,0) node (a2) {$\bullet$} -- ++(1,0.5) node (a3) {$\bullet$} -- ++ (0,-1) node (a4) {$\bullet$};
\path[-,draw] (a1.center) -- (a2.center) -- (a3.center);
\path[-,draw] (a2.center) -- (a4.center);
\draw (a2) circle[radius=1.5mm];
\draw [<->] (2.3,-0.5) to [out=60,in=-60] (2.3,0.5);
\end{tikzpicture}&
{}^2 D_{4,1}^{(2)}
\\
\text{(iv$_5$)}&
\begin{tikzpicture}[every node/.style = {inner sep=0pt},baseline=(a1.base)]
\path (0,0) node (a1) {$\bullet$} -- ++(1,0) node (a2) {$\bullet$} -- ++(1,0) node (a3) {$\bullet$} -- ++(1,0.5) node (a4) {$\bullet$} -- ++ (0,-1) node (a5) {$\bullet$};
\path[-,draw] (a1.center) -- (a2.center) -- (a3.center) -- (a4.center);
\path[-,draw] (a3.center) -- (a5.center);
\draw (a2) circle[radius=1.5mm];
\end{tikzpicture}&
{}^1 D_{5,1}^{(2)}
\end{array}}\]
\caption{Tits indices for the groups in Proposition~\ref{prop:rank-one-groups}, with the case number on the left and the Tits label on the right.}
\label{fig:tits}
\end{figure}

Here, $\tSU(h)$ denotes the simply-connected cover of $\SU(h)$ (in case the Dynkin diagram has type D).

We will write case (ii$_r$), (iii$_r$), (iv$_r$) with $r \in \{1,2,3,4,5\}$ to refer to the subcase of Proposition~\ref{prop:rank-one-groups},
where the (skew-)hermitian form $h$ has rank $r$.
See Figure~\ref{fig:tits} for the Tits index in each case.
Note that case (ii$_4$) can alternatively be described by $\tSU(h)$ with $h$ skew-hermitian form over $D$ of rank 2 and Witt index 1, and
case (iii$_1$) can alternatively be described by $\Spin(h)$ with $h$ a quadratic form over $F$ of rank 5 and Witt index 1.
Note also that the isomorphism class of $\alggrp{G}$ is determined by the isomorphism class of $\{D',(D')^{\mathrm{op}}\}$ in case (i), 
by $E/F$ and $r$ in case (ii$_r$), by $r$ in case (iii$_r$), by $E/F$ in case (iv$_4$), and is unique in case (iv$_5$).

Let $\alggrp{G}'$ denote the quasisplit inner form of $\alggrp{G}$. Let $\alggrp{L}'$ be the Levi subgroup of $\alggrp{G}'$ that corresponds to the
minimal Levi $\alggrp{Z}$ of $\alggrp{G}$.

\begin{prop}\label{prop:rank-one-levis}
  Keep the above notation. According to the cases of Proposition~\ref{prop:rank-one-groups} we have:
  \begin{enumerate}
  \item[(i)] $\SL_d \times \SL_d \subset \alggrp{L}' \subset \GL_d \times \GL_d$.
  \item[(ii$_3$)] $\alggrp{L}' \cong \Res_{E/F} \GL_1$.
  \item[(ii$_4$)] $\SL_2 \subset \alggrp{L}' \subset \Res_{E/F} \GL_1 \times \U_2$.
  \item[(iii$_1$)] $\alggrp{L}' \cong \GL_2$.
  \item[(iii$_2$)] $\alggrp{L}' \cong \GL_2 \times \SL_2$.
  \item[(iv$_4$)] $\SL_2 \times \Res_{E/F} \SL_2 \subset \alggrp{L}' \subset \GL_2 \times \Res_{E/F} \GL_2$.
  \item[(iv$_5$)] $\SL_2 \times \SL_4 \subset \alggrp{L}' \subset \GL_1 \times \GL_2 \times \GL_4$.
  \end{enumerate}
\end{prop}

Here, $\U_2$ denotes the quasisplit unitary group defined by any (skew-)hermitian form of rank 2 over $E$.

\begin{proof}
  This is clear in all but the last two cases. 
  Case (iv$_5$): we consider the split group $\GSpin_{10}$ whose derived subgroup is $\alggrp{G}'$ and show that the
  corresponding Levi is contained between $\SL_2 \times \SL_4$ and $\GL_1 \times \GL_2 \times \GL_4$, which implies
  the claim by intersecting with the derived subgroup.
  The Levi subgroup in question is $\GL_2 \times \GSpin_6$ \cite[Theorem 2.7]{MR1913914} and $\GSpin_6 \cong \{ (g_1,g_2) \in \GL_1 \times \GL_4 : g_1^2 = \det(g_2)\}$
  by \cite[Proposition 2.1]{MR3719299}, so we are done.

  Case (iv$_4$): as in the previous case we consider instead the quasisplit group $\GSpin^*_8$ defined by the quadratic extension $E/F$.
  (We note that the group of automorphisms of the root datum of $\GSpin_{2n}$ is $\Z/2 \times \Z/2$, unlike some claims in the literature,
  but we use the unique non-trivial automorphism that is trivial on the center.
  Explicitly, it interchanges $e_0$ with $e_0+e_n$, $e_n$ with $-e_n$, and fixes $e_1,\dots,e_{n-1}$ in the description of \cite[Proposition 2.4]{MR1913914}.)
  It follows immediately from the description of the root datum that the Levi subgroup in question is $\GL_2 \times \GSpin^*_4$.
  We conclude by Lemma~\ref{lm:gspin4} below.
\end{proof}

\begin{lem}\label{lm:u2}
  We have
  \begin{equation*}
    \U_2 \cong \left(\frac{\GL_2 \times \Res_{E/F} \GL_1}{\GL_1}\right)^{N=1},
  \end{equation*}
  where $\GL_1$ is embedded antidiagonally and $N$ is the product of the determinant on $\GL_2$ and the norm map $\Res_{E/F} \GL_1 \to \GL_1$.
\end{lem}

\begin{proof}
  It suffices to show that
  \begin{equation}\label{eq:gu2}
    \GU_2 \cong \frac{\GL_2 \times \Res_{E/F} \GL_1}{\GL_1}
  \end{equation}
  such that the multiplier on the left corresponds to $N$ on the right.
  It is convenient to work with the skew-hermitian form defined by the matrix $J := (\begin{smallmatrix}&1 \\ -1&\end{smallmatrix})$.
  Then over $E$ the group $\GU_2$ becomes identified with $\GL_2 \times \GL_1$, with the nontrivial element $\tau \in \Gal(E/F)$ acting by the 
  automorphism $(g,\lambda) \mapsto (\lambda \cdot J^{-1} \cdot {}^t g^{-1} \cdot J,\lambda)$.
  Similarly the right-hand side of (\ref{eq:gu2}) becomes identified with $\frac{\GL_2 \times \GL_1 \times \GL_1}{\GL_1}$ with $\GL_1$ embedded
  via $x \mapsto \overline{(x^{-1},x,x)}$ and $\tau$ acting by $\overline{(g,x,y)} \mapsto \overline{(g,y,x)}$.
  It is then easily verified that the maps $(g,\lambda) \mapsto \overline{(g,1,\lambda(\det g)^{-1})}$ and $\overline{(g,x,y)} \mapsto (gx,(\det g)xy)$
  are inverse isomorphisms that commute with the action of $\tau$.
\end{proof}

\begin{lem}\label{lm:gspin4}
  The quasisplit group $\GSpin_4^*$ defined by the quadratic extension $E/F$ is isomorphic to
  $\{g \in \Res_{E/F} \GL_2 : \det(g) \in \GL_1 \subset \Res_{E/F} \GL_1 \}$.
\end{lem}

\begin{proof}
  We consider a variant of the argument in \cite[Proposition 2.1]{MR3719299}.
  We work over an algebraic closure of $F$ that contains $E$ and identify $\GSpin_4$ with $\alggrp{H} := \{(g_1,g_2) \in \GL_2^2 : \det(g_1) = \det(g_2)\}$
  by identifying the root data as follows. 
  Consider the maximal torus $\alggrp{T} := \{t = ((\begin{smallmatrix}a \\ &b\end{smallmatrix}),(\begin{smallmatrix}c \\ &d\end{smallmatrix})) \in \GL_2^2 : ab=cd\}$ of $\alggrp{H}$
  and define the basis $e_i \in X^*(\alggrp{T})$ ($0 \le i \le 2$) by $e_0(t) := d$, $e_1(t) := ad^{-1} = cb^{-1}$, $e_2(t) := bd^{-1} = ca^{-1}$.
  It is straightforward to verify that this identifies the root datum of $(\alggrp{H},\alggrp{T})$ with that of $\GSpin_4$ as given in \cite[Proposition 2.4]{MR1913914}
  and that it identifies the action of $\Gal(E/F)$ on the former root datum (interchanging $(a,b)$ with $(c,d)$) with the action on the latter
  (exchanging $e_0$ with $e_0+e_2$, $e_2$ with $-e_2$ and fixing $e_1$).
  This completes the proof.
\end{proof}

\begin{cor}\label{cor:rationality, rank one}
  Suppose that the adjoint group $\alggrp{G}^{\ad}$ is almost simple of rank one over $F$.
  Let $\sigma$ be a unitary supercuspidal smooth representation of $Z$ over $\C$.
  If $\mu^G(\sigma\delta_{B}^{s})$ has a pole at $s = s_{0}\in \R$, then $s_{0}\in \Q$.
\end{cor}

\begin{proof}
    Let $\alggrp G^{\sc}$ be the simply-connected cover of $\alggrp{G}^{\der}$.
  By Proposition~\ref{prop:solleveld} (and replacing $\sigma$ by an irreducible constituent of $\sigma|_{G^\sc}$) we may assume that $\alggrp G$ is semisimple and simply connected.
  As mentioned above, by replacing $F$ by a finite extension we may assume that $\alggrp{G}$ is moreover absolutely almost simple.
  Then the result follows by combining Theorem~\ref{thm:rationality, inner} with Propositions~\ref{prop:rank-one-groups} and \ref{prop:rank-one-levis}.
\end{proof}

\begin{remark}\label{rk:rank-1-known-cases}
  This result was previously known in case (i) by \cite{MR1040995} (cf.\ \S\ref{sec:GL(n,D)}), in case (ii$_3$) by \cite[\S7]{MR734788}, in case (ii$_4$) by \cite[\S3.1]{Konno}, and in case (iii$_1$) by \cite[\S3]{MR1749954}.
  As far as we know, it is new in cases (iii$_2$), (iv$_4$), and (iv$_5$).
\end{remark}

\begin{remark}\label{rk:rank-1-poles}
  Suppose that $\alggrp{G}^{\ad}$ is almost simple of rank one over $F$.
  In the setting of Corollary~\ref{cor:rationality, rank one} we can moreover bound the denominator of $s_0$, and hence narrow down the set of possible $s_0$ to an explicit finite set by Proposition~\ref{prop:mu-max-parab}(ii), as follows:
  \begin{itemize}
  \item in case (iii$_2$) and (iv$_4$), $s_0 \in \frac 1{20}\Z$;
  \item in case (iv$_5$), $s_0 \in \frac 1{28}\Z$.
  \end{itemize}
    (We recall that $s_0 \in \{ \pm \frac 1{2e} \mid e \mid d \}$ in case (i) \cite{MR1040995}, $s_0 \in \frac 14 \Z$ in case (ii$_3$) \cite{MR734788}, $s_0 \in \frac 16 \Z$ in case (ii$_4$) \cite{Konno}, and $s_0 \in \{ 0, \pm \frac 16, \pm \frac 12 \}$ in case (iii$_1$) \cite{MR1749954}.)
    To see this, we may again reduce to the case where $\alggrp G$ is simply connected and absolutely almost simple.
  By the proof of Theorem~\ref{thm:rationality, inner} it suffices to consider the quasisplit inner form $\alggrp G'$ with parabolic subgroup $\alggrp P' = \alggrp L'\alggrp N'$, where $\alggrp L'$ is described in Proposition \ref{prop:rank-one-levis}.
  By Remark~\ref{rk:pole-qsplit} it suffices to show that $\ang{2\rho_{P'},\gamma^\vee}_{\abs}$ is an integer dividing 10 (resp.\ 14) for $\gamma$ a root of $\alggrp A_{\alggrp Z'}$ in $\alggrp N'$.
  This is an explicit computation.
  (In fact, with more work we can show that $s_0 \in \frac 1{10}\Z$ in cases (iii$_2$) and (iv$_4$) and $s_0 \in \frac 1{14}\Z$ in case (iv$_5$) using Remark \ref{rk:generic-gln}.)
  \end{remark}

\subsection{Proof of Theorem~\ref{thm:unitary case}}
\label{subsec:proof of theorem for unitary}
Now we prove Theorem~\ref{thm:unitary case}.
We always assume Assumption~\ref{assump:on p} in this subsection.

Suppose $\alggrp{G}$ is an arbitrary connected reductive group, $\alggrp{B} = \alggrp{Z} \alggrp{U}$ a choice of minimal parabolic subgroup, $\alggrp{S}$ the maximal split subtorus of $\alggrp{Z}$.
We first reduce to the case where $F = \Q_p$, so we can apply the results of section~\ref{subsec:A criterion}.

\begin{lem}
  Theorem~\ref{thm:unitary case} holds over arbitrary $F$ if and only if it holds over $F = \Q_p$.
\end{lem}

\begin{proof}
  It suffices to show that Theorem~\ref{thm:unitary case} holds for $(G, B, \sigma)$ if and only if it holds for $(\Res G, \Res B, \sigma)$,
  where $\Res := \Res_{F/\Q_p}$ denotes Weil restriction of scalars.
  First note that $\Res \alggrp B$ is a minimal parabolic subgroup of $\Res \alggrp G$ with Levi subgroup $\Res \alggrp Z$.
  (This can be verified by extending scalars to an algebraic closure.)
  The split torus $\alggrp S$ is obtained by base change from a unique split torus $\alggrp S_0$ over $\Q_p$, hence the adjunction gives a homomorphism $\alggrp S_0 \to \Res \alggrp G$,
  which identifies $\alggrp S_0$ with a maximal split torus of $\Res \alggrp G$ (see, e.g.\ \cite[\S4.2]{MR4120166}) that is moreover contained in the minimal Levi subgroup $\Res \alggrp Z$.
  Noting that $\Lie(\Res \alggrp G)$ equals $\Lie \alggrp G$ considered as Lie algebra over $\Q_p$, we see that $\Phi(\alggrp G,\alggrp S)$ is identified with $\Phi(\Res \alggrp G, \alggrp{S}_{0})$   via the map sending $\mu : \alggrp S \to \mathbb G_m$ to the restriction of $\Res \mu$ to $\alggrp S_0$ 
  (noting that $\Res \mu$ factors through the maximal split torus $\mathbb G_m \into \Res \mathbb G_m$);
    equivalently, roots are identified in $X^*(\alggrp S_0) \cong X^*(\alggrp S)$ (under extension of scalars).
  The same is true for positive and simple roots.
  By comparing roots, we see that we have an equality of Levi subgroups $\Res \alggrp L_\alpha = \alggrp L_{\Res\alpha|_{\alggrp S_0}}$ for any $\alpha \in \Phi(\alggrp G, \alggrp S)$.
  Therefore, the condition that $\sigma|_{Z\cap L_\alpha'}$ is trivial for some simple root $\sigma$ is the same for both $G$ and $\Res G$.
    On the other hand, using that the coroot $\alpha^\vee$ is uniquely determined by having image in $\alggrp L_\alpha^\der$ and satisfying $\ang{\alpha,\alpha^\vee} = 2$,
  we deduce that coroots are identified in $X_*(\alggrp S_0) \cong X_*(\alggrp S)$.
  We claim that $e(\chi|_S) = [F:\Q_p]e(\chi)$ in $X^*(\alggrp S)_\R \cong X^*(\alggrp S_0)_\R$ for any $\chi : S \to C^\times$ (and likewise for characters of $A_L$).
  It suffices to check this when $\lvert\chi\rvert_C = \lvert\nu\rvert_F$ for some $\nu \in X^*(\alggrp S)$, in which case it follows from $\lvert\cdot\rvert_F = \lvert\cdot\rvert_{\Q_p}^{[F:\Q_p]}$ on $\Q_p^\times$.
  Therefore, $e(\chi|_S)$ is dominant if and only if $e(\chi|_{S_0})$ is dominant.
\end{proof}

From now on we assume that $F = \Q_p$.

Let $\rho := \frac 12\sum_{\alpha \in \Phi_{\mathrm{red}}^+} n_\alpha \alpha \in X^*(\alggrp{S})$, where $n_\alpha := \dim U_{(\alpha)} + \dim U_{(2\alpha)} \in \Z_{>0}$.
Note that $n_\alpha$ only depends on the Weyl group orbit of $\alpha$.

\begin{lem}\label{lm:rho}
  For any $\alpha \in \Phi_{\mathrm{red}}^+$ we have $\ang{\rho,\alpha^{\vee}} \ge n_\alpha$ with equality if and only if $\alpha$ is simple.
\end{lem}

\begin{proof}
  If $\alpha$ is simple, then $s_\alpha(\Phi_{\mathrm{red}}^+) = (\Phi_{\mathrm{red}}^+ \setminus \{\alpha\}) \cup \{-\alpha\}$ and the proof follows from
  \begin{equation*}
    -\ang{\rho,\alpha^{\vee}} = \ang{\rho,s_\alpha(\alpha^{\vee})} = \ang{s_\alpha(\rho),\alpha^{\vee}} = \frac 12\sum_{\beta \in s_\alpha(\Phi_{\mathrm{red}}^+)} n_\beta \ang{\beta,\alpha^{\vee}}
    = \ang{\rho,\alpha^{\vee}} - 2n_\alpha.
  \end{equation*}

  For general $\alpha$, we will work in the reduced root system $\Phi_{\mathrm{red}}$, so we can write $\alpha^{\vee} = \sum_{i=1}^r \alpha_i^{\vee}$ with $\alpha_i$ simple
  and therefore $\ang{\rho,\alpha^{\vee}} = \sum_{i=1}^r n_{\alpha_i}$ by the previous paragraph. 
  To show $\ang{\rho,\alpha^{\vee}} \ge n_\alpha$ it thus suffices to show that at least one $\alpha_i$ is in the same Weyl group orbit as $\alpha$, as this implies that $n_{\alpha_{i}} = n_{\alpha}$.
  Note also that once we have shown this, equality can only hold if $r = 1$, i.e.\ $\alpha$ simple.

  By dualizing it suffices to show the following. Suppose that $\Phi$ is a reduced root system.
  Then for any $\alpha \in \Phi^+$ with $\alpha = \sum_{\gamma \in \Delta} c_\gamma \gamma$ there is a $\gamma \in \Delta$ with $c_\gamma > 0$ and $\gamma$ is in the same Weyl group orbit as $\alpha$.
  Let $I := \{ \gamma \in \Delta : c_\gamma \ne 0 \}$ and let $\Phi_I := \Z I \cap \Phi$ denote the sub-root system generated by $I$.
  Since $\alpha\in \Phi_I$ lies in the Weyl group orbit of a simple root $\gamma$ of $\Phi_I^+$ (i.e.\ $\gamma \in I$), we are done.
\end{proof}

Recall that $\alggrp{T}'$ denotes a maximal split torus of (the split group) $\alggrp{G}_C$ containing $\alggrp{S}_C$
and recall that $\widetilde{\Phi}$ denotes the roots of $(\alggrp{G}_C,\alggrp{T}')$.
Let $\widetilde{\Delta} \subset \widetilde{\Phi}$ denote a choice of simple roots so that
$\Delta \subset \widetilde{\Delta}|_{\alggrp{S}} \subset \Delta \cup \{0\}$.
Let $\widetilde{\Delta}_0 := \{ \widetilde{\gamma} \in \widetilde{\Delta} : \widetilde{\gamma}|_{\alggrp{S}} = 1\}$. 
For $\alpha \in \Delta$ let $\mathcal{O}(\alpha) := \{ \widetilde{\gamma} \in \widetilde{\Delta} : \widetilde{\gamma}|_{\alggrp{S}} = \alpha \}$.

\begin{lem}\label{lm:rel-coroots}
  Suppose $\alpha \in \Delta$ and write $\alpha^{\vee} = \sum_{\widetilde{\gamma} \in \widetilde{\Delta}} c_{\widetilde{\gamma}} \widetilde{\gamma}^{\vee}$ with $c_{\widetilde{\gamma}} \in \Z$.
  Then $c_{\widetilde{\gamma}} \ge 0$ and $c_{\widetilde{\gamma}} \ne 0$ if and only if $\widetilde{\gamma}$ lies in the smallest union of connected components of $\widetilde{\Delta}_0 \sqcup \mathcal{O}(\alpha)$ that contain $\mathcal{O}(\alpha)$.
\end{lem}

Note that $\alpha^{\vee}$ can be expressed in the above form because we can pass to the simply-connected cover of the derived subgroup of $\alggrp{G}$ (in which case $X_*(\alggrp{T}') = \Z\widetilde{\Delta}^\vee$).

\begin{proof}
  By passing to the Levi subgroup $L_{\alpha}$ we may assume that $\alggrp{G}$ has semisimple rank 1 (i.e.\ $\widetilde{\Delta} = \widetilde{\Delta}_0 \sqcup \mathcal{O}(\alpha)$).
  By comparing with the simply-connected cover of the derived subgroup we may assume that $\alggrp{G}$ is simply connected.
  Factoring $\alggrp{G} = \alggrp{G}^\is \times \alggrp{G}^{\an}$ with $\alggrp{G}^\is$ almost simple and isotropic and $\alggrp{G}^{\an}$ anisotropic,
  we may assume that $\alggrp{G}$ is almost simple, and we need to show that $c_{\widetilde{\gamma}} \ne 0$ for all $\widetilde{\gamma} \in \widetilde{\Delta}$.

  Note that the $c_{\widetilde{\gamma}}$ are determined by the formulas
  \begin{equation*}
    \ang{\widetilde{\beta},\alpha^{\vee}} = \sum _{\widetilde{\gamma} \in \widetilde{\Delta}} c_{\widetilde{\gamma}} \ang{\widetilde{\beta},\widetilde{\gamma}^{\vee}}
    = 
    \begin{cases}
      2 & \text{if $\widetilde{\beta} \in \mathcal{O}(\alpha)$,} \\
      0 & \text{otherwise.}
    \end{cases}
  \end{equation*}
  In other words, $\alpha^{\vee} = 2\sum_{\widetilde{\beta} \in \mathcal{O}(\alpha)} \varpi^{\vee}_{\widetilde{\beta}}$, where $\varpi^{\vee}_{\widetilde{\beta}}$ denotes the fundamental coweight
  corresponding to $\widetilde{\beta}$.

  Dually, it suffices to show that if we express a fundamental weight $\varpi_{\widetilde{\beta}} = \sum_{\widetilde{\gamma} \in \widetilde{\Delta}} d_{\widetilde{\gamma}} \widetilde{\gamma}$ in terms of simple roots,
  then $d_{\widetilde{\gamma}} \ge 0$ for all $\widetilde{\gamma}$ and $d_{\widetilde{\gamma}} > 0$ if $\widetilde{\gamma}$ lies in the connected component of $\widetilde{\Delta}$ that contains $\widetilde{\beta}$.
  The first claim is true since fundamental weights form acute or right angles with each other. 
    If the second claim is false, we can pick (adjacent) $\widetilde{\gamma}_1, \widetilde{\gamma}_2 \in \widetilde{\Delta}$ such that $d_{\widetilde{\gamma}_1} = 0$, $d_{\widetilde{\gamma}_2} > 0$ and $\ang{\widetilde{\gamma}_2,\widetilde{\gamma}_1^{\vee}} < 0$.
  Then $\ang{\varpi_{\widetilde{\beta}},\widetilde{\gamma}_1^{\vee}} = \sum_{\widetilde{\gamma} \in \widetilde{\Delta}} d_{\widetilde{\gamma}} \ang{\widetilde{\gamma},\widetilde{\gamma}_1^{\vee}} < 0$, contradiction.
\end{proof}

Recall that $L'_{\alpha}$ is the group generated by $U \cap L_{\alpha}$ and $\overline{U} \cap L_{\alpha}$.

\begin{lem}\label{lm:trivial-alg-rep-parab}
  Suppose that $V$ is an irreducible algebraic representation of $G$ over $C$, so that $V^U$ is an irreducible algebraic representation of $Z$. 
  Let $\omega_{V^U} \colon S \to C^{\times}$ denote the central character of $V^U$ and let $\alggrp P = \alggrp L\alggrp N$ be a standard parabolic subgroup.
  Then the following are equivalent:
  \begin{enumerate}
  \item $V^N$ is trivial on $L'$;
  \item $V^U$ is trivial on $Z \cap L'$;
  \item $\omega_{V^U} \circ \alpha^{\vee} = 1$ for all $\alpha \in \Delta_L$.
  \end{enumerate}
\end{lem}

Here, $\Delta_L$ denotes the simple roots of $\alggrp L$.

\begin{proof}
  It is clear that (i)$\Rightarrow$(ii)$\Rightarrow$(iii).
  To show (iii)$\Rightarrow$(i), by taking $(G,V)$ to be $(L,V^{N})$, we may assume that $\alggrp L = \alggrp G$.
  As in \cite[II.4]{AHHV} let $\alggrp{G}^\sc$ denote the simply-connected cover of the derived subgroup of $\alggrp{G}$, which is a direct product of almost simple groups, precisely one of which is isotropic (of rank 1).
  Let $\alggrp{G}^\is$ denote the unique isotropic almost simple factor of $\alggrp{G}^\sc$.
  Then $G'$ is the image of $G^\is = (G^\is)'$ under the natural morphism $\iota\colon \alggrp{G}^\is\to \alggrp{G}$ (by Kneser--Tits, cf.\ \cite[II.4]{AHHV}).
  As $\iota$ is an isomorphism $\alggrp{U}^\is \to \alggrp{U}$ and $\iota^{-1}(\alggrp{Z}) = \alggrp{Z}^\is$ is a minimal Levi subgroup, we reduce to the case where $\alggrp{G}$ is a product of isotropic simply-connected groups.

  Let $\mu \in X^*(\alggrp{T}')$ denote the highest weight of $V$ (or $V^U$), so that $\omega_{V^U}$ is given by $\mu|_S$.
  In particular, $\ang{\mu,\widetilde{\gamma}^{\vee}} \ge 0$ for all $\widetilde{\gamma} \in \widetilde{\Delta}$.
  By our assumption on $\alggrp{G}$ no connected component of $\widetilde{\Delta}$ is contained in $\widetilde{\Delta}_0$, so for each $\widetilde\gamma \in \widetilde{\Delta}_0$ there exists $\widetilde\alpha \in \wt\Delta\setminus \wt\Delta_0$ such that $\wt\gamma$, $\wt\alpha$ lie in the same connected component of $\wt\Delta_0 \sqcup \mathcal O(\wt\alpha|_{\alggrp S})$.
  By Lemma~\ref{lm:rel-coroots} we deduce that $\omega_{V^U} \circ \alpha^{\vee} = 1$ for all $\alpha \in \Delta$ if and only if $\ang{\mu,\widetilde{\gamma}^{\vee}} = 0$ for all $\widetilde{\gamma} \in \widetilde{\Delta}$
  if and only if $\mu = 0$ if and only if $V$ is trivial on $G = G'$.
\end{proof}

\begin{remark}\label{rk:dominant-parab}
  Note that in the setting of Lemma~\ref{lm:trivial-alg-rep-parab} we have $\omega_{V^U} \circ \alpha^{\vee} = (\cdot)^{m_\alpha}$ for some $m_\alpha \in \Z_{\ge 0}$
  for any $\alpha \in \Delta$.
\end{remark}

\begin{proof}[Proof of Theorem~\ref{thm:unitary case}]
  We first prove that the conditions in (ii) and (iii) are necessary.
  Assume that $\underline{L}(\sigma')\in \mathcal{O}^{P_{1}}$ for a parabolic subgroup $\alggrp{P}_{1}\supsetneq \alggrp{P}$ with Levi subgroup $\alggrp{L}_{1} \supset \alggrp Z$.
  Then $\sigma_{1} := \underline{L}_{L_{1}}(\sigma')'$ is finite-dimensional by Corollary~\ref{cor:N-invts} and we have $(\sigma_{1})_{N\cap L_{1}}\simeq \sigma$.
  By Frobenius reciprocity we have $\sigma_{1}\hookrightarrow (\Ind_{P \cap L_{1}}^{L_{1}}\sigma)^{\cts}$, which is closed as $\sigma_{1}$ is finite-dimensional.
  Then by applying the functor $(\Ind_{P_{1}}^{G} \cdot )^{\cts}$, we get $(\Ind_{P_{1}}^{G}\sigma_{1})^{\cts}\hookrightarrow (\Ind_{P}^{G}\sigma)^{\cts}$, so $(\Ind_{P}^{G}\sigma)^{\cts}$ is reducible.
  The proof of the ``if'' part of (iii) is similar (see the introduction).

  In the remaining proof we may freely replace $C$ by a finite extension because this preserves the reducibility, and we will do so without further comment.

  Recall that we may assume that $F = \Q_p$.
  Recall that we have defined $e(\chi) = e_{\alggrp{S}}(\chi)\in \mathfrak{a}_{Z,\R}^{*}$ for a character $\chi\colon S\to C^{\times}$,
  and also that if $\chi\in X^{*}(\alggrp S)$ is algebraic, thought of as $\chi : S \to \Q_p^\times \subset C^\times$, then $e(\chi) = \chi$.
  It follows that $e(\lvert \chi\rvert_{\Q_p}) = -\chi$ for $\chi \in X^{*}(\alggrp S)$.
  In this proof we write $e(\cdot)$ for $e(\cdot|_{S})$.

  We may assume that $\alggrp{G}^{\der}$ is simply connected, by taking a $z$-extension.
  By Lemma~\ref{lem:sigma0-tau-decomp} we can write $\sigma \cong \sigma_0 \otimes \tau$ with $\sigma_0 \in \mathcal{O}^{L}$ and $\tau$ a smooth representation of $L$ such that $\alggrp{L}(\sigma_0') \in \mathcal{O}^{P}$ is equimaximal. 
  Let $\alggrp{Q} = \alggrp{L}_{\alggrp Q}\alggrp{N}_{\alggrp Q}$ denote the maximal parabolic subgroup for $\underline{L}(\sigma_0')$.
  By Corollary~\ref{cor:irreduciblity} we deduce that $(\Ind_{P \cap L_Q}^{L_Q} \tau)^\sm$ is reducible.
  We may relabel $L_Q$ as $G$ and assume without loss of generality that $Q = G$.
  In particular, by Lemma~\ref{lem:semisimple-lift} $\alggrp{L}(\sigma_0') \in \mathcal{O}^{G}$ is algebraic up to twist by a locally analytic character $\psi \colon G \to C^{\times}$.
  Twisting $\sigma$ by $\psi$ and using Lemma~\ref{lm:loc-an-char}, without loss of generality, $\alggrp{L}(\sigma_0') \in \mathcal{O}^{G}$ is algebraic.
  Then $(\sigma_0')^{U \cap L} \cong \alggrp{L}(\sigma_0')^U$ is an irreducible algebraic representation of $Z$.
  Let $\omega_0 \colon S \to C^{\times}$ denote its central character, which is algebraic (factoring through $\Q_p^\times \subset C^\times$) and dominant by Remark~\ref{rk:dominant-parab}.
  So if we write $\omega_0 = -\lambda \in X^*(\alggrp{S})$, then $\ang{\lambda,\alpha^{\vee}} \le 0$ for all $\alpha \in \Delta$.

  Note that $\tau$ is trivial on $L'$ by Lemma~\ref{lm:loc-an-char}.
  Let $\xi := \tau|_Z$, which is absolutely irreducible as $ZL' = L$.

  Taking coinvariants, we see that $\sigma_{U \cap L} \cong (\sigma_0)_{U \cap L} \otimes \xi$.   As $(\sigma_{0})_{U\cap L}$ is $\mathfrak{z}_{C}$-simple and $\xi$ is absolutely irreducible smooth $Z$-module, we see that $\sigma_{U\cap L}$ is absolutely irreducible by Proposition~\ref{prop:simple-O-P}.
  Let $\omega_\xi$, $\omega_{\sigma_{U\cap L}} \colon S \to C^{\times}$ denote the central characters of $\xi$, $\omega_{U\cap L}$, respectively.
  By applying $e(\cdot)$ to central characters we obtain that 
  \begin{equation}
    \lambda + e(\omega_\xi) = e(\omega_{\sigma_{U \cap L}}).\label{eq:cc}
  \end{equation}
  For $\beta \in \Delta_L$ we have $\ang{e(\omega_{\xi}),\beta^\vee} = e_{\mathbb{G}_{\mathrm{m}}}(\omega_{\xi}\circ\beta^{\vee}) = 0$ (as $\xi$ is trivial on $Z\cap L'$).
  
  We claim that $e(\omega_{\sigma_{U\cap L}}) - \lambda$ is dominant.
      Assume that $e(\omega_{\sigma}|_{A_{L}})$ is dominant (which is implied by $e(\omega_{\sigma_{U\cap L}})$ being dominant, as remarked above).
  For $\alpha\in\Delta_{L}$, as $\langle e(\omega_{\xi}),\alpha^{\vee}\rangle = 0$, we have $\langle e(\omega_{\sigma_{U\cap L}}) - \lambda,\alpha^{\vee}\rangle = 0$.
  On the other hand, let $\alpha\in\Delta\setminus\Delta_{L}$ and take $a\in \mathfrak{a}_{L,\R}$, $b\in \mathfrak{a}_{Z,\R}^{L} = \R\Delta_{L}^{\vee}$ such that $\alpha^{\vee} = a + b$.
  By Lemma~\ref{lem:dominant-proj}(i), we have $b\in\R_{\le 0}\Delta_{L}^{\vee}$.
  Therefore $\langle \lambda,a\rangle = \langle \lambda,\alpha^{\vee} - b\rangle \le 0$ since $\lambda$ is antidominant.
  Since $\langle e(\omega_{\sigma_{U\cap L}}) - \lambda,b\rangle = \langle e(\omega_{\xi}),b\rangle = 0$ and $\langle e(\omega_{\sigma_{U\cap L}}),a\rangle = \langle e(\omega_{\sigma}|_{A_{L}}),a\rangle = \langle e(\omega_{\sigma}|_{A_{L}}),\alpha^{\vee}\rangle\ge 0$, we have $\langle e(\omega_{\sigma_{U\cap L}}) - \lambda,\alpha^{\vee}\rangle = \langle e(\omega_{\sigma_{U\cap L}}),a\rangle - \langle \lambda,a\rangle\ge 0$.
  Therefore $e(\omega_{\sigma_{U\cap L}}) - \lambda$ is dominant.

  We embed $(\Ind_P^G \tau)^\sm$ as subrepresentation of $(\Ind_B^G \xi)^\sm = (\nInd_B^G \xi \delta_B^{-1/2})^\sm$.
  We can write $\delta_B(x) = \lvert(2\rho)(x)\rvert_{\Q_{p}}$ for $x \in S$ and so $\delta_{B\cap L_{\alpha}}(x) = \lvert\alpha(x)\rvert^{n_\alpha}_{\Q_{p}}$ for $\alpha \in \Phi_{\mathrm{red}}^+$.
  (Here $\rho \in X^*(\alggrp{S})$ and $n_\alpha \in \Z_{> 0}$ are as in Lemma~\ref{lm:rho}.)

    Then
  \begin{equation}\label{eq:7-parab}
    e(\omega_{\xi}\delta_B^{-1/2}) = (e(\omega_{\sigma_{U \cap L}}) - \lambda) + \rho,
  \end{equation}
  as $e(\delta_{B}) = -2\rho$.

  If $w\in N_{G}(Z)$ satisfies $\xi \delta_B^{-1/2} \cong w(\xi \delta_B^{-1/2})$, then by~\eqref{eq:7-parab} we deduce that $(e(\omega_{\sigma_{U\cap L}}) - \lambda) + \rho$   is fixed by $w$ (here we regard $w$ as an element of the Weyl group), which implies $w \in Z$, as $e(\omega_{\sigma_{U \cap L}}) - \lambda$ is dominant and $\rho$ is strictly dominant.
  Hence $\xi \delta_B^{-1/2}$ is $G$-regular, so $\mu^{G}$ has a pole at $\xi \delta_B^{-1/2}$ by Proposition~\ref{prop:G-regular}, and therefore $\mu^{L_{\alpha}}$ has a pole at $\xi \delta_B^{-1/2}$ for some $\alpha \in \Phi_{\mathrm{red}}^+$ by the product formula (Proposition~\ref{prop:product formula}).

  We will now work over $\overline{C} \cong \C$ (fixing a field isomorphism arbitrarily) and extend the absolute value $\lvert\cdot\rvert_C$ uniquely to $\overline{C}$. 
  Write $\xi \delta_B^{-1/2} \cong \xi_u \chi$ with $\xi_u$ unitary over $\C$ and $\chi$ a positive real unramified character of $Z$.
  We can write $\chi = \delta_{B \cap L_{\alpha}}^s \eta$ with $s \in \R$ and $\eta$ a positive real unramified character of $L_{\alpha}$ 
  (as $\mathfrak{a}_{Z,\R}^* = \mathfrak{a}_{L_{\alpha},\R}^* \oplus (\mathfrak{a}_{Z,\R}^{L_{\alpha}})^*$ and $\alpha$ spans $(\mathfrak{a}_{Z,\R}^{L_{\alpha}})^*$).
  By Proposition~\ref{prop:mu-max-parab} we have $\xi_u \cong s_\alpha(\xi_u)$ where $s_{\alpha}\in N_{L_\alpha}(Z)/Z$ is the non-trivial element.
  Thus $(\omega_{\xi_u} \circ \alpha^{\vee})^2 = 1$, where $\omega_{\xi_u} \colon S \to \C^{\times}$ is the central character of $\xi_u$.
  In particular, $\langle e(\omega_{\xi_u}),\alpha^{\vee}\rangle = e_{\mathbb{G}_{\mathrm{m}}}(\omega_{\xi_u}\circ\alpha^{\vee}) = 0$.
  By Proposition~\ref{prop:mu-max-parab} we know that $-1/2 \le s \le 1/2$.   By Corollary~\ref{cor:rationality, rank one} (applied to $L_\alpha$) we know that $s \in \Q$, so $\delta_{B \cap L_{\alpha}}^{s}(x) = \lvert \alpha(x)\rvert^{sn_{\alpha}}_{\Q_p}$ (taking values in $p^{\Q} \subset \R_{>0}^\times$).
  Hence $e(\delta_{B\cap L_{\alpha}}^{s}) = -sn_{\alpha}\alpha$.
  By applying $e(\cdot)$ to the central character of $\xi \delta_B^{-1/2} \cong \xi_u \chi = \xi_u \delta_{B \cap L_{\alpha}}^s \eta$, pairing with $-\alpha^\vee$, 
  and using (\ref{eq:7-parab}) we deduce
  \begin{equation*}
    \langle \lambda - \rho - e(\omega_{\sigma_{U \cap L}}),\alpha^{\vee}\rangle = 2sn_{\alpha}.
  \end{equation*}
  As $e(\omega_{\sigma_{U\cap L}}) - \lambda$ is dominant, and by Lemma~\ref{lm:rho} we get
  \begin{equation*}
    0\ge -\langle e(\omega_{\sigma_{U\cap L}}) - \lambda,\alpha^{\vee}\rangle = \langle \rho,\alpha^{\vee}\rangle + 2sn_{\alpha}\ge n_{\alpha}(2s + 1).
  \end{equation*}
  Since $s\ge -1/2$, we get $\langle e(\omega_{\sigma_{U\cap L}}) - \lambda,\alpha^{\vee}\rangle = 0$, $\langle\rho,\alpha^{\vee}\rangle = n_{\alpha}$ and $s = -1/2$.
  By Lemmas~\ref{lm:rho} and Proposition~\ref{prop:Waldspurger} we deduce that $\alpha$ is simple and $\xi_u$ is trivial on $Z \cap L_{\alpha}'$.
  Since $\delta_B = \delta_{B \cap L_{\alpha}} (\delta_{P_\alpha}|_Z)$ we get that $\xi \cong \xi_u \eta (\delta_{P_\alpha}|_Z)^{1/2}$ is trivial on $Z \cap L_{\alpha}'$.

  Case 1: $\alpha \notin \Delta_L$.
    Let $\alggrp L_{1}$ be the smallest Levi subgroup containing $\alggrp L$ and $\alggrp L_{\alpha}$, and let $\alggrp L_1 \alggrp N_1$ be the standard parabolic subgroup with Levi subgroup $\alggrp L_{1}$.
  Then since $\xi = \tau|_{Z}$ is trivial on $Z\cap L_{\alpha}'$, $\tau$ has an extension $\tau_{1}$ to $L_{1}$~\cite[II.7 Proposition]{AHHV} with $N \cap L_1$ acting trivially.
  Then $\underline L(\sigma') \cong \underline L(\sigma_0') \otimes \tau'$ in $\mathcal O^P$ by Lemma~\ref{lem:verma-tensor}.
  By Lemma~\ref{lem:cop} and as $\underline L(\sigma_0') \in \mathcal O^G$ we have $\underline L(\sigma') \cong \underline L(\sigma_0') \otimes \tau_1'$ lies in $\mathcal O^{P_1}$.
  
  Assume moreover $e(\omega_{\sigma_{U\cap L}})$ is dominant.
  Since $\lambda$ is antidominant, $\langle e(\omega_{\sigma_{U\cap L}}) - \lambda,\alpha^{\vee}\rangle = 0$ implies $\langle e(\omega_{\sigma_{U\cap L}}),\alpha^{\vee}\rangle = \langle \lambda,\alpha^{\vee}\rangle = 0$.
  By pairing~\eqref{eq:cc} with $\beta^\vee\in \Delta_{L}^{\vee}$ and using $\langle e(\omega_{\xi}),\beta^{\vee}\rangle = 0$, that $\lambda$ is antidominant and $e(\omega_{\sigma_{U \cap L}})$ is dominant we obtain that $\ang{\lambda,\beta^\vee} = 0$ for all $\beta \in \Delta_L$, and also $\ang{\lambda,\alpha^{\vee}} = 0$ by above.
  This implies that $\underline L(\sigma_0')^{N_1}$ is trivial on $L_1'$ by Lemma~\ref{lm:trivial-alg-rep-parab}.   Hence $\sigma_0' \cong \underline L(\sigma_0')^{N}$ (an $L$-stable subspace of $\underline L(\sigma_0')^{N_1}$) is trivial on $L \cap L_1'$.
  As moreover $\tau|_Z = \xi$ is trivial on $Z \cap L_\alpha'$ we deduce that $\sigma = \sigma_0 \otimes \tau$ is trivial on $Z \cap L_\alpha'$, as desired.

  Case 2: $\mu^{L_\beta}$ has a pole at $\xi \delta_B^{-1/2}$ only if $\beta \in \Delta_L$.
  From the embedding $(\nInd_P^G \tau\delta_P^{-1/2})^\sm\into(\Ind_B^G \xi\delta_B^{-1/2})^\sm$ we obtain
  (as in the proof of Lemma~\ref{lem:comparison of mu for unramified}) that $J_{P|\overline P}(\tau\delta_P^{-1/2}) J_{\overline P|P}(\tau\delta_P^{-1/2}) = J_{B|B'}(\xi\delta_B^{-1/2})J_{B'|B}(\xi\delta_B^{-1/2})$,
  where $\alggrp B'$ is the Borel subgroup of $\alggrp G$ such that $\alggrp B' \supset \overline {\alggrp N}$ and $\alggrp B' \cap \alggrp L = \alggrp B \cap \alggrp L$.
  By \cite[IV.3(4)]{MR1989693} we get $J_{B|B'}(\xi\delta_B^{-1/2})J_{B'|B}(\xi\delta_B^{-1/2})$ equals $\prod_\beta \mu^{L_\beta}(\xi\delta_B^{-1/2})^{-1}$ (up to nonzero constant), where
  the product runs through $\Phi_\red(\alggrp B,\alggrp A_{\alggrp Z}) \setminus \Phi_\red(\alggrp B \cap \alggrp L,\alggrp A_{\alggrp Z})$, which has no zeros by assumption.
  By \cite[IV.1(13)]{MR1989693} and the Weyl group regularity of $\omega_\xi\delta_B^{-1/2}$ established above we deduce that $J_{B|B'}$ and $J_{B'|B}$ are regular at $\xi\delta_B^{-1/2}$.
  Therefore, $J_{P|\overline P}$, $J_{\overline P|P}$ are regular at $\tau\delta_P^{-1/2}$ and induce an isomorphism $(\nInd_P^G \tau\delta_P^{-1/2})^\sm \cong (\nInd_{\overline P}^G \tau\delta_P^{-1/2})^\sm$.
  Analogously to Remark~\ref{rk:irred-smooth-princ-series} we deduce that $(\nInd_P^G \tau\delta_P^{-1/2})^\sm$ is irreducible.
  (The Jacquet module of $(\Ind_B^G \xi\delta_B^{-1/2})^\sm$ for $B$ is multiplicity-free, as $\xi\delta_B^{-1/2}$ is $G$-regular.
  Hence the Jacquet module of $(\nInd_P^G \tau\delta_P^{-1/2})^\sm$ for $P$ is multiplicity-free.
  It follows that $(\nInd_P^G \tau\delta_P^{-1/2})^\sm$ has an irreducible socle that is also its irreducible cosocle by the given isomorphism.)
\end{proof}

\begin{remark}\label{rk:unitary-open-cone}
We know that, by Remark~\ref{rk:rank-1-poles}, there exists an explicit $k_{\alpha} \in 2\Z_{>0}$ such that $(\nInd_{B \cap L_{\alpha}}^{L_{\alpha}}\xi_u\delta_{B\cap L_{\alpha}}^{s})^\sm$ reducible for $s \in \R$ implies that $s\in \frac{1}{k_{\alpha}}\Z$, where we use the notation of the proof.
Therefore we can strengthen Theorem~\ref{thm:unitary case}:
in part (ii) (resp.\ (iii)) it suffices to assume that $\langle e(\omega_{\sigma}|_{A_L}),\alpha^{\vee}\rangle > -2n_{\alpha}/k_{\alpha}$ for any $\alpha \in \Delta\setminus \Delta_L$ (resp.\ $\langle e(\omega_{\sigma_{U \cap L}}|_{S}),\alpha^{\vee}\rangle > -2n_{\alpha}/k_{\alpha}$ for any $\alpha \in \Delta$).
(Either of these assumptions implies that $\langle e(\omega_{\sigma_{U \cap L}}|_{S})-\lambda,\alpha^{\vee}\rangle > -2n_{\alpha}/k_{\alpha}$ for any $\alpha \in \Delta$ and hence that $e(\omega_{\sigma_{U \cap L}}|_{S})-\lambda+\rho$ is strictly dominant in the proof.
Here we assumed $F = \Q_{p}$ for the argument, but we remark that the condition $\langle e(\omega_{\sigma}|_{A_L}),\alpha^{\vee}\rangle > -2n_{\alpha}/k_{\alpha}$ is equivalent for $\alggrp{G}$ and $\Res_{F/\Q_{p}}\alggrp{G}$.)
\end{remark}

\subsection{Groups of semisimple rank one}
\label{sec:rank-1-case}

In this subsection we establish a more precise irreducibility result when $\alggrp G$ is semisimple of rank 1.
Without loss of generality (replacing $\alggrp G$ by $\Res_{F/\Q_p} \alggrp G$) we may and will assume that $F = \Q_p$.

\begin{thm}\label{thm:rank-1}
  Assume Assumption~\ref{assump:on p}.
  Suppose that $\alggrp G$ is of semisimple of rank 1 over $F = \Q_p$.
  Let $\sigma$ be a finite-dimensional absolutely irreducible continuous representation of $Z$.
  Suppose that either $\sigma$ is $\mathfrak z_C$-simple or that $U$ is abelian.
  Then $(\Ind_{B}^{G}\sigma)^{\cts}$ is absolutely reducible if and only if, after perhaps replacing $C$ with a finite extension, 
  $\underline L(\sigma') \in \mathcal O^G$.
\end{thm}

We remark that $\sigma \in \mathcal O^Z$, after perhaps replacing $C$ by a finite extension, by Lemma~\ref{lem:cts-rep}.
Also note that the condition $\underline L(\sigma') \in \mathcal O^G$ is made more explicit in Lemma \ref{lem:simple-in-O-Q}.

\begin{proof}
  Let $\omega_{\sigma}$ denote the central character of $\sigma$.
  Suppose that $(\Ind_{B}^{G}\sigma)^{\cts}$ is reducible.
  Assume first that $\alggrp G^\der$ is simply connected.
  Then, after perhaps replacing $C$ by a finite extension, we can write $\sigma = \sigma_0 \otimes \tau$ as in \S\ref{subsec:A criterion}.
  In particular, $\sigma_0$ is $\mathfrak z_C$-simple and $\underline L(\sigma_0')$ is equimaximal with maximal parabolic $Q$.
  Then $Q \ne B$ by Corollary~\ref{cor:criterion-easy-cases}, i.e.\ $Q = G$ and $\underline L(\sigma_0') \in \mathcal O^G$.
  By Lemma~\ref{lem:reduction to smooth} there exists an irreducible subrepresentation $\pi$ of $(\Ind_{B}^{G}\tau)^{\sm}$ that is not dense in $(\Ind_{B}^{G}\tau)^{\cts}$.
  Pick any nonzero element $f \in \pi$.
  If $f(1) = 0$ and $\sigma$ is $\mathfrak z_C$-simple, then by smoothness of $f$ and the Bruhat decomposition we get a contradiction from Lemma~\ref{lem:supported on big cell} (as $\dim_C \tau = 1$).
  If $f(1) = 0$ and $U$ is abelian, then pick a compact open subgroup $U_0$ of $U$ such that $\supp(f) \subset B \backslash Bw_0 U_0$ and any $U_0$-eigenvector $f' \in \ang{U_0 \cdot f}$ (which exists as $U_0$ is abelian, after perhaps replacing $C$ by a finite extension).
  Then $f'$ takes values in a 1-dimensional subspace of $\tau$ and we get a contradiction from Lemma~\ref{lem:supported on big cell}.
  So $f(1) \ne 0$ and then for any $z \in S$ we see that $zf - \omega_\sigma(z) f \in \pi$ vanishes at 1, so again by Lemma~\ref{lem:supported on big cell} we deduce that $zf = \omega_\sigma(z) f$ for all $z \in S$.
  By smoothness of $f$ we deduce that $f$ is fixed by $U$ and $\overline U$, hence by $G'$.
  Therefore, for $z \in Z \cap G'$ we have $zf(1) = f(z) = f(1)$, i.e.\ $f(1) \in \tau^{Z \cap G'}$.
  As $Z \cap G'$ is normal in $Z$ and $f(1) \ne 0$ we get that $Z\cap G'$ acts trivially on $\tau$.
  Therefore, $\tau$ extends to a smooth representation of $G$, so $\underline L(\sigma') \cong \underline L(\sigma_0') \otimes \tau' \in \mathcal O^G$.

  For general $\alggrp G$, as in the proof of Corollary~\ref{cor:socle-orlik-strauch}
  we take a $z$-extension $1 \to \alggrp T \to \wt {\alggrp G} \to \alggrp G \to 1$, where ${\wt{\alggrp{G}}}^{\lowprime[\der]}$ is simply connected and $1 \to T \to \wt G \to G \to 1$ on $F$-points.
  By pullback to $\wt G$ we obtain $\wt B = \wt Z \wt U$, and by inflation we obtain $\wt \sigma \in \mathcal O^{\wt Z}$.
  As $(\Ind_{\wt B}^{\wt G}\wt\sigma)^{\cts}$ is reducible, $\underline L(\wt\sigma') \in \mathcal O^{\wt G}$ by above.
  As $\underline L(\wt\sigma') \in \mathcal O^{\wt B}$ arises by inflation from $\underline L(\sigma') \in \mathcal O^{B}$ (in particular $T$ acts trivially on $\underline L(\wt\sigma')$), we deduce that $\underline L(\sigma') \in \mathcal O^G$.

  Conversely, suppose that $\underline L(\sigma') \in \mathcal O^G$ (so finite-dimensional).
  Then we obtain continuous maps as follows:
  \begin{align*}
    \underline L(\sigma')' = \mathcal F_G^G(\underline L(\sigma'),1) &\into \mathcal F_G^G(\underline L(\sigma'),(\Ind_B^G 1)^\sm) = \mathcal F_B^G(\underline L(\sigma'),1)\\
    & \into \mathcal F_B^G(\underline M(\sigma'),1) \cong (\Ind_{B}^{G}\sigma)^{\an} \into (\Ind_{B}^{G}\sigma)^{\cts},
  \end{align*}
  which proves the reducibility.
\end{proof}

\begin{remark}
  If $\alggrp G$ is one of the groups $\SL_2(D)$, quasisplit $\SU_3$, or the rank 1 inner form of $\Sp_4$ (see \S\ref{sec:rank-one-groups}), then the assumption that $\sigma$ is $\mathfrak z_C$-simple or that $U$ is abelian is satisfied: in the first and third cases $U$ is abelian, and in the second case $Z$ is abelian (so $\dim_C \tau = 1$, i.e.\ $\sigma$ is $\mathfrak z_C$-simple).
\end{remark}

In the following corollary we allow arbitrary $F/\Q_p$.

\begin{cor}\label{cor:rank-1-split}
  Assume Assumption~\ref{assump:on p}.
  Suppose that $\alggrp G$ is split of semisimple rank 1 over $F$, with unique simple root $\alpha$.
  Let $\chi\colon Z \to C^{\times}$ be a continuous character.
  Then $(\Ind_{B}^{G}\chi)^{\cts}$ is absolutely reducible if and only if $\chi \circ \alpha^\vee$ is non-positive algebraic.
\end{cor}

\begin{proof}
    We first use a $z$-extension to reduce to the case where $\alggrp G^\der$ is simply connected.
  Moreover, by Proposition~\ref{prop:isogenies}, we may assume $\alggrp{G} = \alggrp{G}^{\der} (\cong \SL_2)$.
  Then $\alpha^{\vee} \colon \mathbb{G}_{\mathrm{m}}\to \alggrp Z$ is an isomorphism.
  Note that either condition in the corollary is unchanged if we replace $C$ by a finite extension, and we will assume for the rest of the proof that $C$ is sufficiently large without further comment.
  We set $\widetilde{\alggrp{G}} := \Res_{F/\Q_{p}}\alggrp{G}$.
  If $\chi\circ\alpha^{\vee}$ is non-positive algebraic, then $\chi$ is algebraic and $\underline{L}(\chi^{-1})\in \mathcal{O}^{\widetilde{G}}$ by Lemma~\ref{lem:simple-in-O-Q}.
  Hence $(\Ind_{B}^{G}\chi)^{\cts}$ is reducible by Theorem~\ref{thm:rank-1}.
  Conversely, if $(\Ind_{B}^{G}\chi)^{\cts}$ is reducible, then $\underline{L}(\chi^{-1})\in \mathcal{O}^{\widetilde{G}}$ by Theorem~\ref{thm:rank-1}.
  Hence by Lemma~\ref{lem:semisimple-lift}\ref{lem:semisimple-lift-1} we know that $\underline{L}(\chi^{-1})$ is algebraic.
  Therefore $\chi^{-1} \cong \underline{L}(\chi^{-1})^{\wt U}$ is algebraic, and from $L(-d\chi)\in \mathcal{O}^{\widetilde{\mathfrak g}}$ we deduce that the algebraic character $\chi\circ\alpha^{\vee}$ is non-positive.
  \end{proof}

\appendix
\newcommand{\Phir}{\Phi_{\red}}
\newcommand{\vp}{\varphi}
\newcommand{\cB}{\mathcal{B}}
\newcommand{\p}{\mathfrak p}
\newcommand{\cO}{\mathcal{O}}
\newcommand{\g}{\mathfrak g}
\newcommand{\copp}{\cO^{\p}}
\newcommand{\cop}{\cO^{P}}
\newcommand{\col}{\cO^{L_P}}
\newcommand{\coll}{\cO^{\mathfrak l_P}}
\newcommand{\z}{\mathfrak z}
\newcommand{\h}{\mathfrak h}
\newcommand{\m}{\mathfrak m}
\newcommand{\qp}{{\Q_p}}
\newcommand{\cH}{\mathcal{H}}
\newcommand{\zp}{\Z_p}
\newcommand{\del}{\mathfrak d}
\newcommand{\q}{\mathfrak q}
\newcommand{\cF}{\mathcal{F}}

\renewcommand{\b}{\mathfrak b}
\renewcommand{\u}{\mathfrak u}
\renewcommand{\t}{\mathfrak t}
\renewcommand{\l}{\mathfrak l}

\section{Orlik--Strauch: the general case}
\label{app:orlik-strauch}

The goal of this appendix is to generalize the main results of \cite{OS2}, \cite{OS3} to a general connected reductive group.
As much as possible we keep the notation of \cite{OS2}, \cite{OS3}. In particular, $L$ (not $F$) denotes the ground field and $K$ (not $C$) denotes the coefficient field.
Let $\alggrp G$ be a connected reductive group over $L$.
Fix a minimal parabolic subgroup $\alggrp B$ and let $\alggrp P = \alggrp L_{\alggrp P} \alggrp U_{\alggrp P}$ denote a standard parabolic subgroup.
Recall the abelian categories $\copp$ and $\cop$ defined in \S\ref{sec:funct-orlik-stra}.

Just as in \cite{OS3}, for $M \in \cop$ and $V$ an admissible smooth representation of $L_P$ we can then define $\cF_P^G(M,V)$ as follows.
Pick any finite-dimensional (locally analytic) $P$-subrepresentation $W \subset M$ that generates $M$ as $U(\g)$-module.
(We use the convention of \cite{OS2}, \cite{OS3} to write $U(\g)$ for $U(\g \otimes_L K)$.)
Then $\cF_P^G(M,V)$ is the closed subrepresentation of $(\Ind_P^G W' \otimes V)^\sm$ that is annihilated by $\ker(U(\g)\otimes_{U(\p)} W \onto M)$, cf.\ \cite[\S3.8]{OS3}.

\begin{thm}\label{thm:OS-main}
  Suppose that $K$ is sufficiently large (depending only on $\alggrp G$), and keep the notation above.
  The main results of \cite{OS3} hold for $\alggrp G$. This means:
  \begin{enumerate}
  \item $\cF_P^G$ is functorial and exact in both arguments.
  \item If $Q = L_Q U_Q \supset P$, $M \in \mathcal{O}^{Q}$, $V$ an admissible smooth representation of $L_P$, then
    $\mathcal{F}_{P}^{G}(M,V)\simeq \mathcal{F}_{Q}^{G}(M,(\Ind_{P \cap L_{Q}}^{L_{Q}}V)^{\sm})$.
  \item Suppose that $M \in \cop$ such that 
    \begin{enumerate}
    \item $\alggrp P$ is maximal among parabolic subgroups of $\alggrp G$ such that $M \in \copp$,
    \item $M$ is simple as $U(\g)$-module,
    \end{enumerate}
    and suppose that $V$ is an irreducible (admissible) smooth representation of $L_P$. Then $\cF_P^G(M,V)$ is topologically irreducible.
      \end{enumerate}
\end{thm}

\begin{remark}
  Assumption (a) is weaker than saying $\p_K := \p \otimes_L K$ is maximal for $M$ because $\g_K$ can have more parabolic subalgebras than those
  coming from $G$ when $\alggrp G$ is nonsplit.
\end{remark}

Parts (i) and (ii) follow exactly as in \cite{OS2}, \cite{OS3} and the remainder of this appendix will focus on proving part (iii).
The basic idea is to deduce (iii) by comparison with the split case, by considering $\alggrp G \times_L L'$ for a carefully chosen finite extension $L'/L$, like in \cite[Appendice]{socle1} in the case of restrictions of scalars of split groups.
(We will need that $L'$ embeds in $K$, which is why we demand that $K$ be sufficiently large.)

Fix a maximal split torus $\alggrp S$ of $\alggrp G$ over $L$. Let $\Phi$ denote the possibly non-reduced root system of $(\alggrp G, \alggrp S)$ and $W$ its Weyl group.
Choose any special point $x$ of $\alggrp G$ in the apartment of $\alggrp S$.

Let $\Phi^+ \subset \Phi$ denote the set of positive roots corresponding to $\alggrp B$, with simple roots $\Delta$.
Choose a ``special'' subtorus $\alggrp T_s$ of $\alggrp G$ over $L$ containing $\alggrp S$, i.e.\ $\alggrp T_s$ becomes a maximal split torus after base change to the
maximal unramified extension of $L$. Let $\alggrp T$ denote the centraliser of $\alggrp T_s$ in $\alggrp G$, which is a maximal torus of $\alggrp G$ (as $\alggrp G$ becomes quasisplit
over the maximal unramified extension of $L$).

\begin{prop}\label{prop:special-point-base-change}
  There exists a finite Galois extension $L'/L$ splitting $\alggrp G$ such that $x$ remains special in the building of
  the split group $\alggrp G' : = \alggrp G \times_{L} L'$.
\end{prop}

\begin{proof}
  Assume that the valuation $\omega$ of $L$ satisfies $\omega(L^\times) \subset \Q$ and extend $\omega$ uniquely to an algebraic closure $\o L$. 
  All extensions of $L$ below will be taken inside $\o L$. Let $\Phir \subset \Phi$ be the subset of reduced roots.

  We first suppose $\alggrp G$ quasisplit, with splitting field $\wt L$. Let $\wt\Phi$ denote the absolute roots.
  Let $\vp$ be the valuation defined by a Chevalley--Steinberg system as in \cite[\S4]{BT2},
  so the hyperplanes in the apartment of $\alggrp S$ are given by $\{ a(x-\vp)+r = 0\ (r \in \Gamma_a, a \in \Phir) \}$.
  From \cite[4.2.21]{BT2} we get that $\Gamma_a$ is an infinite cyclic subgroup of $\Q$ for all $a \in \Phir$, and from \cite[4.2.1]{BT2} that 
  $\Gamma_{\wt a}$ is an infinite cyclic subgroup of $\Q$ for all $\wt a \in \wt\Phi$.
  Hence there exists $e \in \Z_{> 0}$ such that $\Gamma_a \subset \frac 1e \Gamma_{\wt a}$ whenever $\wt a|_{\alggrp S} = a \in \Phir$
  and $2 \Gamma_a \subset \frac 1e \Gamma_{\wt b}$ whenever $\wt b|_{\alggrp S} = 2a \in \Phi$ and $a \in \Phir$.
  Letting $L'/\wt L$ denote any finite extension of ramification index a multiple of $e$, we get that $x$ is still special over $L'$ (cf.\ \cite[4.2.24]{BT2}).
  (In fact, $e = 2$ works always and $e = 1$ works if $\Phi$ is reduced.)
  The same argument shows that if $x'$ lies in the apartment of $\alggrp S$ such that $a(x'-\vp) \in \Q$ for all $a \in \Phir$,
  then $x'$ becomes special after base change to a suitably ramified extension of $\wt L$.

  For $\alggrp G$ general, 
  let $L_0/L$ be a finite unramified extension with Galois group $\Gamma$ such that $\alggrp T_s$ becomes split over $L_0$; in particular,
  $\alggrp G \times_L L_0$ becomes quasisplit.
  We consider the embedding of buildings $\cB_L(\alggrp G) \into \cB_{L_0}(\alggrp G \times_L L_0)$ (unramified base change).
  Then the vertex $x$ lies in a unique $\Gamma$-invariant facet $F$ in the apartment of $\alggrp T_s$ (inside $\cB_{L_0}(\alggrp G \times_L L_0)$), so $x$ equals the average of the vertices of $F$.
  By above it then suffices to show that any vertex $x'$ of the apartment of $\alggrp T_s$ satisfies $a(x'-\vp) \in \Q$ for all $a \in \Phir$
  (where $\Phir$ denotes the reduced roots of $\alggrp G \times_L L_0$). 
  This is clear: as $x'$ is a vertex we have $a(x'-\vp) \in \Gamma_a \subset \Q$ for $a \in X \subset \Phir$ for some maximal linearly independent subset $X$, 
  which implies $a(x'-\vp) \in \Q$ for all $a \in \Phir$ (as $\Z X \subset \Z \Phir$ has finite index).
\end{proof}

Let $L'/L$ be a finite Galois extension splitting $\alggrp G$ such that the image $x'$ of the special point $x$ 
in the building of $\alggrp G' : = \alggrp G \times_{L} L'$ is still special.
We assume that $K$ is large enough to contain an embedding $L' \into K$ and we fix such an embedding, extending the given embedding $L \into K$.
Let $k_L$ (resp.\ $k_{L'}$) denote the residue field of $L$ (resp.\ $L'$).

Let $\alggrp P$ be a (standard) parabolic subgroup of $\alggrp G$ containing $\alggrp S$ with Lie algebra $\p$.
Let $\alggrp U_{\alggrp P}^-$ be the unipotent radical of the opposite parabolic $\alggrp P^{-}$ (with respect to $\alggrp S$) with Lie algebra $\u_P^-$.

Let $\alggrp S' := \alggrp S \times_{L} {L'}$ and $\alggrp T' := \alggrp T \times_{L} {L'}$, so that by construction $x'$ is contained in the apartment of $\alggrp T'$.
Let $\alggrp G'_{x'}$ denote the connected reductive model of $\alggrp G'$ over $\cO_{L'}$ defined by $x'$, and
let $\alggrp S'_{x'}$ denote the scheme-theoretic closure of $\alggrp S'$ in $\alggrp G'_{x'}$ (a split torus extending $\alggrp S'$)
and similarly define the split torus $\alggrp T'_{x'}$.
We define the parabolic subgroup $\alggrp P' := \alggrp P \times_L L'$ of $\alggrp{G}'$ and its unipotent radical $\alggrp {U}^-_{\alggrp P'} := \alggrp U_{\alggrp P}^- \times_L L'$.
Let also $\alggrp G_x$ denote the connected reductive model of $\alggrp G$ over $\cO_{L}$ defined by $x$.

Let $G := \alggrp G(L)$, $G' := \alggrp G'(L')$, etc., so $G$ is a closed $L$-analytic subgroup of $G'$.
Let $G'_0 := \alggrp G'_{x'}(\cO_{L'})$ and $G_0 := G \cap G_0'$; these are compact open in $G'$ and $G$, respectively.
Let $P'_0 := P' \cap G'_0$, $U^-_{P',0} := U^-_{P'} \cap G'_0$, $P_0 := P \cap G_0$, $U_{P,0}^- := U_P^- \cap G_0$.

By construction $G_0$ contains $\alggrp G_x(\cO_L)$, so that $G = G_0 P$. (Any compact subgroup of $G'$ that fixes
$x'$ has to be contained in $G'_0$, as $\alggrp G'$ is split.)

Let $\g'_0 := \Lie \alggrp G'_{x'}$, which is an $\cO_{L'}$-Lie lattice inside $\g' := \Lie \alggrp G' = \g \otimes_L L'$ that is stable by the adjoint action of $G'_0$.
Let $\g_0 := \g \cap \g'_0$, which is an $\cO_{L}$-Lie lattice inside $\g$ that is stable by the action of $G_0$,
and moreover $\g_0 \otimes_{\cO_L} \cO_{L'}$ is of finite index in $\g'_0$.
By the algebraic action of $\alggrp S'_{x'}$ on $\g'_0=\Lie \alggrp G'_{x'}$ we see that
\begin{equation}\label{eq:g0'-decomp}
  \g'_0 = (\g'_0 \cap \u_{P'}^-) \oplus (\g'_0 \cap \p')
\end{equation}
as $\cO_{L'}$-modules.

By Lemma~\ref{lm:splitting} (applied with $\alggrp S'_0 = \alggrp S'_{x'}$, $V = \g$, $V_1 = \u_{P}^-$, $V_2 = \p$, $M' = \g'_0$) we also have
\begin{equation}\label{eq:g0-decomp}
  \g_0 = (\g_0 \cap \u_{P}^-) \oplus (\g_0 \cap \p)
\end{equation}
as $\cO_{L}$-modules;
we even see that $\g_0 \cap \u_{P}^- = \bigoplus_{\alpha \in \Phi^-\setminus \Phi_P^-} (\g_0 \cap \u_\alpha)$.

\begin{lem}\label{lm:splitting}
  Suppose $A \subset B$ are integral domains with fields of fractions $E \subset F$.

  \begin{enumerate}
  \item Suppose $\alggrp S$ is a split torus over $E$, $\alggrp S_0'$ a split torus over $B$, and we are given an isomorphism
    $\alggrp S \times_E F \cong \alggrp S_0' \times_{B} F$. Then there exists a unique split torus $\alggrp S_0$ over $A$
    together with isomorphisms $\alggrp S_0 \times_{A} E \cong \alggrp S$, $\alggrp S_0 \times_A B \cong \alggrp S_0'$ compatible with the 
    isomorphism above after base change to $F$.
  \item Keep the notation as in (i). 
    Suppose $V = V_1 \oplus V_2$ is an isomorphism of finite-dimensional $\alggrp S$-modules and suppose that
    $M'$ is an $\alggrp S_0'$-module together with an isomorphism $M' \otimes_B F \cong V \otimes_E F$ that is compatible with 
    the actions of $\alggrp S_0' \times_{B} F \cong \alggrp S \times_E F$.
    Let $M := M' \cap V$. If for all $\chi \in X^*(\alggrp S)$ we have $(V_1)_\chi = 0$ or $(V_2)_\chi = 0$, then 
    $M = (M \cap V_1) \oplus (M \cap V_2)$.
      \end{enumerate}
\end{lem}

\begin{proof}
  (i) We recall from \cite[\S I.2.5]{Jantzen} that over an integral domain $A$ we have an (anti-)equivalence of categories
  between abelian groups and diagonalizable group schemes over $A$, and this equivalence is 
  moreover compatible with base change $A \to B$. Thus the claim becomes obvious.
    
  (ii) Let $X := X^*(\alggrp S_0)$, so that we may identify $\alggrp S_0 \times_A C$ with the spectrum of the Hopf algebra $C[X]$ for any map of
  integral domains $A \to C$. Let $\Delta_V \colon V \to V \otimes_E E[X]$ denote the comodule corresponding to $V$ and
  $\Delta_{M'} \colon M' \to M' \otimes_B B[X]$ the comodule corresponding to $M'$, so that $\Delta_V$ and $\Delta_{M'}$ become
  identified after base change to $F$. It follows that $\Delta_V$ sends $M$ to $(V \otimes_E E[X]) \cap (M' \otimes_B B[X])$
  inside $V \otimes_E F[X]$. Hence $\Delta_V(M) \subset M \otimes_A A[X]$, i.e.\ $M$ becomes an $\alggrp S$-module whose
  base change to $B$ (resp.\ $E$) is identified with $M'$ (resp.\ $V$). Therefore, $M = \bigoplus_{\chi \in X} M_\chi$ \cite[\S I.2.11]{Jantzen}.
  For any $\chi \in X$ we have $M_\chi \subset V_\chi = (V_1)_\chi \oplus (V_2)_\chi = (V_i)_\chi$ for some $i$, by assumption,
  so $M_\chi \subset M \cap V_i$ and we are done.
\end{proof}

Recall the equivalence between uniform pro-$p$ groups $H$ and powerful $\Z_p$-Lie algebras $\h$ \cite[Theorem 9.10]{DDMS}.
If $\h$ is a powerful $\Z_p$-Lie algebra, then the corresponding uniform pro-$p$ group is defined by the convergence of the Baker--Campbell--Hausdorff formula \cite[Theorem 9.8]{DDMS}, and we denote it by $\BCH(\h)$.
It is locally $\Q_p$-analytic group with Lie algebra $\h \otimes_{\Z_p} \Q_p$.
If $\h' \subset \h$ is a saturated sub-$\Z_p$-Lie algebra of $\h$, then $\h'$ is also powerful and $\BCH(\h') \subset \BCH(\h)$ is a closed subgroup \cite[Scholium to Theorem 9.10]{DDMS}.
Conversely, if $H$ is a uniform pro-$p$ group, then it is (uniquely) a locally $\Q_p$-analytic group and the $\Z_p$-lattice $\log(H) \subset \Lie(H)$ is the corresponding powerful $\Z_p$-Lie algebra.
(Here, $\log$ denotes the functorial logarithm map of a locally analytic group, which is defined on the union of all compact subgroups \cite[III.7.6]{MR0573068}.)
An \emph{$L$-uniform group} is a uniform pro-$p$ group $H$ together with an $\mathcal O_L$-Lie algebra structure on $\log(H)$ that extends the given $\Z_p$-Lie algebra structure.

\begin{lem}\label{lem:L-uniform}
  Any $L$-uniform group $H$ is (uniquely) a locally $L$-analytic group.
\end{lem}

\begin{proof}
  The Baker--Campbell--Hausdorff formula does not depend on the base field, so it converges on the $\mathcal O_L$-Lie algebra $\log(H)$, making it into a locally $L$-analytic group with Lie algebra $\Lie(H)$ \cite[\S17]{MR2810332}.
  The uniqueness follows exactly as in the proof of \cite[Theorem 29.8]{MR2810332}.
\end{proof}

This lemma implies that our definition agrees with the one in \cite[Remark 2.2.5]{OS}.
(If $H$ is an open uniform subgroup of a locally $L$-analytic group $G$ such that $\log(H) \subset \Lie(G)$ is $\mathcal O_L$-stable, then we give $\log(H)$ the induced $\mathcal O_L$-Lie algebra structure.
Conversely, if $H$ is $L$-uniform in our sense, then we take $G := H$ by Lemma~\ref{lem:L-uniform}.)
We get an equivalence between $L$-uniform groups (with locally $L$-analytic morphisms) and powerful $\mathcal O_L$-Lie algebras (meaning that the underlying $\Z_p$-Lie algebra is powerful).

If $\h$ is a powerful $\Z_p$-Lie lattice in $\g$, then so is $p^n\h$ for any $n \ge 0$.
For $n$ sufficiently large there exists a (non-canonical!)\ open embedding $i \colon \BCH(p^n\h) \into G_0$ of locally analytic groups whose associated map on Lie algebras is the identity \cite[Proposition 18.17]{MR2810332}.

\begin{lem}\label{lm:congruence-subgp}
  If $\ell \ge \kappa$, then $G'_{(\ell)} := \ker\big(\alggrp G'_{x'}(\cO_{L'}) \onto \alggrp G'_{x'}(\cO_{L'}/p^\ell)\big)$ is a uniform pro-$p$ group,
  and we have $\log(G'_{(\ell)}) = p^\ell \g'_0$ inside $\g'$.
\end{lem}

\begin{proof}
  For any smooth group scheme $\alggrp H$ over $\cO_{L'}$ let $H_{(\ell)} := \ker\big(\alggrp H(\cO_{L'}) \onto \alggrp H(\cO_{L'}/p^\ell)\big)$.
  Let $\Phi'$ denote the roots of $(\alggrp G', \alggrp T')$ and fix some subset of positive roots $\Phi'^+$.
  Then $G'_{(\ell)} = G'_{x',(\ell)}$ is contained in the Iwahori subgroup defined by the Borel subgroup of $\alggrp G'_{x'} \times_{\cO_{L'}} k_{L'}$
  that contains $\alggrp T'_{x'} \times_{\cO_{L'}} k_{L'}$ and corresponds to $\Phi'^+$, hence
  \begin{equation*}
    G'_{x',(\ell)} \subset \prod_{a \in \Phi'^-} \alggrp U'_{a,x'}(\cO_{L'}) \cdot \alggrp T'_{x'}(\cO_{L'}) \cdot \prod_{a \in \Phi'^+} \alggrp U'_{a,x'}(\cO_{L'}),
  \end{equation*}
  where $\alggrp U'_{a,x'}$ is the scheme-theoretic closure of the root subgroup $\alggrp U'_{a}$ inside $\alggrp G'_{x'}$   and we fixed some ordering of the roots. As $\prod_{a \in \Phi'^-} \alggrp U'_{a,x'} \times \alggrp T'_{x'} \times \prod_{a \in \Phi'^+} \alggrp U'_{a,x'} \to \alggrp G'_{x'}$
  is an open immersion,    it follows that
  \begin{equation*}
    G'_{x',(\ell)} = \prod_{a \in \Phi'^-} U'_{a,x',(\ell)} \cdot T'_{x',(\ell)}\cdot \prod_{a \in \Phi'^+} U'_{a,x',(\ell)}.
  \end{equation*}
  By explicit calculation for $\alggrp U'_{a,x'} \cong \mathbb G_a$ and $\alggrp T'_{x'} \cong \mathbb G_m^r$
  we see that $\log(U'_{a,x',(\ell)}) = p^\ell \Lie \alggrp U'_{a,x'}$ and $\log(T'_{a,x',(\ell)}) = p^\ell \Lie \alggrp T'_{a,x'}$.

  To justify that $G'_{x',(\ell)}$ is uniform, it suffices to show it is a standard $\qp$-analytic group \cite[Theorem 8.31]{DDMS}.
  Fix any $\Z$-basis $\lambda_1,\dots,\lambda_r$ of $X_*(\alggrp T'_{x'})$.
  By \cite[\S3.2.13, \S4.6.15]{BT2} there exist $\cO_{L'}$-isomorphisms $x_a \colon \mathbb G_a \congto \alggrp U_{a,x'}$
  ($a \in \Phi'$) defining a Chevalley system as in \cite[\S3.2.2]{BT2}.
  Then we get an topological isomorphism
  \begin{align*}
    \theta \colon (p^\ell \cO_{L'})^{\Phi'} \times (p^\ell \cO_{L'})^r &\congto G'_{x',(\ell)} \\
    (u_a, t_i) &\mapsto \prod_{a \in \Phi'^-} x_a(u_a) \prod_{i=1}^r \lambda_i(1+t_i) \prod_{a \in \Phi'^+} x_a(u_a).
  \end{align*}
  By the commutation relations in \cite[\S3.2.3]{BT2}, the relation $x_a(u)x_{-a}(v) = x_{-a}(-v(1-uv)^{-1}) a^\vee(1-uv) x_a(u(1-uv)^{-1})$
  (which follows from \cite[\S3.2.1]{BT2}), as well as $\lambda_i(1+t)x_a(u) = x_a((1+t)^{\ang{a,\lambda_i}}u)\lambda_i(1+t)$ 
  it follows that $\theta$ is a global chart making $G'_{x',(\ell)}$ into a standard group and hence $G'_{x',(\ell)}$ is uniform.

    By using this as well as \cite[Theorem 4.17]{DDMS}
  with a minimal topological generating set contained in $\bigcup_{a \in \Phi'} U'_{a,x',(\ell)} \cup T'_{x',(\ell)}$ we deduce 
  $\log(G'_{x',(\ell)}) = p^\ell \g'_0$, as required.
  (Note that \cite[Theorem 4.17]{DDMS} shows that if $g_1,\dots,g_r$ is a minimal topological generating set of a uniform group
  $H$, then $\log(H) = \bigoplus_i \Z_p \log(g_i)$, and note also that $\log$ is functorial.)
\end{proof}

Let $\kappa := 1$ if $p > 2$ and $\kappa := 2$ if $p = 2$.
For $m_0 \ge \kappa$ the $\mathcal O_{L'}$-Lie lattice $p^{m_0} \g_0'$ in $\g'$ is powerful.
Define $L'$-uniform groups $H' := \BCH(p^{m_0} \g_0')$, $H'^- := \BCH(p^{m_0} \g_0' \cap \u_{P'}^-)$, $H'^+ := \BCH(p^{m_0} \g_0' \cap \p')$.
By Lemma~\ref{lm:congruence-subgp} and the equivalence between uniform groups and powerful $\Z_p$-Lie algebras we get a \emph{unique} open embedding $H' := \BCH(p^{m_0} \g_0') \into G'_0$ whose image is $G'_{(m_0)}$ and whose derivative is the identity.
For any $\gamma \in G'_0$, the conjugation action of $\gamma$ on $G'_0$ induces the adjoint action of $\gamma$ on $p^{m_0} \g_0'$.
Note that $H'^-$, $H'^+$ are closed subgroups of $H'$.
For $m_0$ sufficiently large we have $H'^- \subset U_{P',0}^-$ and $H'^+ \subset P'_0$ because this true on the level of Lie algebras.
(Actually, $m_0 \ge \kappa$ suffices, by a variant of Lemma~\ref{lm:congruence-subgp}.)
By~\eqref{eq:g0'-decomp} and \cite[Theorem 9.8]{DDMS} we see that the multiplication map  $H'^- \times H'^+ \to H'$ is a topological isomorphism and hence that $H'^- = H' \cap U_{P'}^-$, $H'^+ = H' \cap P'$.
From $H' \lhd G'_0$ it then follows that $H'^- \lhd U^-_{P',0}$ and $H'^+ \lhd P'_0$.

The closed subgroup $H := \BCH(p^{m_0} \g_0)$ of $H'$ is contained in $G$
for $m_0$ sufficiently large, because this is true on the level of Lie algebras.
As $G_0 = G \cap G_0'$ normalizes $\g_0 = \g \cap \g_0'$, it follows that $H\lhd G_0$.
Letting $H^- := \BCH(p^{m_0} \g_0 \cap \u_P^-)$, $H^+ := \BCH(p^{m_0} \g_0 \cap \p)$ we see that $H^- \lhd U_{P,0}^-$, $H^+ \lhd P_0$.
By~\eqref{eq:g0-decomp}, the multiplication map $H^- \times H^+ \to H$ is a topological isomorphism and hence $H^- = H \cap U_{P}^-$, $H^+ = H \cap P$.
By construction, $H$, $H^-$, $H^+$ are $L$-uniform groups.

For $m \ge 0$ let $H^m := P_{m+1}(H) = \BCH(p^{m+m_0} \g_0) \lhd G_0$ (where $(P_m(H))_{m \ge 1}$ is the lower $p$-series, cf.\ \cite[Definition 1.15]{DDMS}) and likewise $H^{-,m}\lhd U_{P,0}^-$, $H^{+,m}\lhd P_0$,
$H'^m\lhd G'_0$, $H'^{-,m}\lhd U^-_{P',0}$. These are all $L$- (resp.\ $L'$-)uniform groups. Note that by construction
$\g_0$ is an $\cO_L$-direct summand of $\g_0'$.
In particular, $H^m$ is a closed subgroup of $H'^m$ which is topologically a direct factor (an ordered basis of $H^m$ as a uniform
group can be extended to an ordered basis of $H'^m$).

For any $r \in p^{\Q} \cap (p^{-1},1)$ we recall that we have a continuous algebra (semi)norm $\lvert\cdot\rvert_r$ on the locally analytic distribution algebra $D(H)$, defined by the uniform structure of $H$ \cite[2.2.6]{OS}.
Its completion $D_r(H)$ (or more precisely $D^{(L)}_r(H)$) is a noetherian Banach algebra.
Let $U_r(\g) = U_{r}(\mathfrak{g},H)$ denote the closure of $U(\g)$ in $D^{(L)}_r(H)$ and $U_r(\g') = U_r(\g',H')$ the closure of $U(\g')$ in $D^{(L')}_r(H')$.
Then $D_r(H)$ is free as left and right $U_r(\g)$-module, admitting a basis in $K[H]$, and $U_r(\g)$ is noetherian \cite[Theorem 1.4.2]{MR2309143}.

\begin{cor}\label{cor:distrib}
    Suppose $r^{p^m} \in p^{\Q} \cap (p^{-1},p^{-1/\kappa(p-1)})$ for some $m \ge 0$.   Then we have a commutative diagram
  \begin{equation*}
    \xymatrix{
    U(\g) \ar@{=}[r]\ar[d] & U(\g')\ar[d] \\
    U_r(\g)\ar@{-->}[r] & U_r(\g')
    }
  \end{equation*}
  where the bottom map is a morphism of Banach algebras.
  Moreover, all maps are equivariant for the adjoint action of $G_0 \subset G'_0$.
\end{cor}

\begin{proof}
  Let $s := r^{p^m}$. By density the bottom map is unique, if it exists, and it is automatically an algebra homomorphism and $G_0$-equivariant.
  It thus suffices to consider the underlying Banach spaces of the bottom row.

  Note that $p^{m_0}\g_0 \otimes_{\cO_L} \cO_{L'}$ is powerful and we let $H'' := \BCH(p^{m_0}\g_0 \otimes_{\cO_L} \cO_{L'}) \subset H'$ (another $L'$-uniform group, open inside $H'$).
  Then $U_r(\g,H) = U_r(\g',H'')$ by \cite[\S5, \S6]{schmidt} (taking the closure of $U(\g)$ in $D^{(L)}_s(H^m) \subset D^{(L)}_r(H)$, resp.\ of $U(\g')$ in $D^{(L')}_s(H''^{p^m}) \subset D^{(L')}_r(H'')$).
    We may thus work over $L'$ and it suffices to show that if $H^{(1)} \subset H^{(2)}$ are $L'$-uniform open subgroups of $G'$
  with corresponding powerful $\cO_{L'}$-Lie lattices $\h_0^{(1)} \subset \h_0^{(2)}$ in $\g'$,   we get a morphism of Banach spaces $U_r(\g',H^{(1)}) \to U_r(\g',H^{(2)})$
    (compatible with maps from $U(\g')$).
  
  Pick $(X_i)_{i=1}^d$ an ordered $\cO_{L'}$-basis of $p^m \h_0^{(2)}$ such that $(Y_i := \varpi_{L'}^{\ell_i} X_i)_{i=1}^d$ is
  an ordered $\cO_{L'}$-basis of $p^m \h_0^{(1)}$ (with $\ell_i \in \Z_{\ge 0}$). By \cite[\S5, \S6]{schmidt} we have
  \begin{align*}
    U_r(\g',H^{(2)}) &= \bigg\{ \sum_{\beta\in\Z_{\ge 0}^{d}} d_\beta \mathbf{X}^\beta : d_\beta \in K,\ |d_\beta| s^{\kappa|\beta|} \to 0 \bigg\},\\
    U_r(\g',H^{(1)}) &= \bigg\{ \sum_{\beta\in\Z_{\ge 0}^{d}} e_\beta \mathbf{Y}^\beta : e_\beta \in K,\ |e_\beta| s^{\kappa|\beta|} \to 0 \bigg\}
  \end{align*}
  and the topology can be defined by norms
  \begin{equation*}
    \bigg\lVert \sum_{\beta \in \Z_{\ge 0}^d} d_\beta \mathbf{X}^\beta \bigg\rVert_s = \sup_\beta |d_\beta| s^{\kappa|\beta|},\qquad
    \bigg\lVert \sum_{\beta \in \Z_{\ge 0}^d} e_\beta \mathbf{Y}^\beta \bigg\rVert_s = \sup_\beta |e_\beta| s^{\kappa|\beta|}.
  \end{equation*}
  Taking the inclusion, we get a norm-decreasing map of Banach spaces $U_r(\g',H^{(1)}) \to U_r(\g',H^{(2)})$ that is
  compatible with maps from $U(\g')$, as required.
\end{proof}

Let $D(\g,P_0)$ be the subring of $D(G_0)$ generated by $U(\g)$ and $D(P_0)$. 
Then any object of $\mathcal O^P$ becomes a $D(\g,P_0)$-module.
We define a continuous algebra norm $\lvert\cdot\rvert_r$ on $D(\wt H)$ for any subgroup $H \le \wt H \le G_0$ by using $\lvert\cdot\rvert_r$ on $D(H)$ \cite[(5.5.5)]{OS2}.
Let $D_r(\wt H)$ denote its completion.
Likewise we define $\lvert\cdot\rvert_r$ on $D(P_0)$ by using $\lvert\cdot\rvert_r$ on $D(H^+)$, and let $D_r(P_0)$ denote its completion, which is a closed subring of $D_r(G_0)$.
Let $D_r(\g,P_0)$ be the subring of $D_r(G_0)$ generated by $U_r(\g)$ and $D_r(P_0)$.
Then $D_r(\g,P_0)$ is a finitely generated $U_r(\g)$-module, hence a closed subring and noetherian.
Let $U_r(\u_P^-) := U_r(\u_P^-,H^-)$ and $U_r(\p) := U_r(\p,H^+)$.

\begin{lem}\label{lm:iwahori-decomp}
  We have canonical isomorphisms of Banach spaces:  \begin{align}
    D_r(H) &\cong D_r(H^-) \wh\otimes D_r(H^+), \label{eq:iw1} \\
    U_r(\g) &\cong U_r(\u_P^-) \wh\otimes U_r(\p), \label{eq:iw2} \\
    D_r(HP_0) &\cong D_r(H^-) \wh\otimes D_r(P_0), \label{eq:iw3} \\
    D_r(\g,P_0) &\cong U_r(\u_P^-) \wh\otimes D_r(P_0). \label{eq:iw4}
  \end{align}
\end{lem}

\begin{proof}
  As $H \lhd G_0$ is open we have $HP_0 = H^- P_0$ with open subgroup $H = H^-H^+$ (i.e.\ both having Iwahori decomposition with respect to $U_P^- \times P \into G$).
  The proof of \eqref{eq:iw1} and \eqref{eq:iw3} is exactly as in \cite[Proposition 3.3.4]{OS}, with $H$ playing the role of $I_0$ and $HP_0$ the role of $I$.
  We note that these isomorphisms are induced by convolution, as follows by comparison with \cite[Remark A.4]{ST-duality}.
  Then \eqref{eq:iw2} follows from \eqref{eq:iw1}.
  For \eqref{eq:iw4}, note using~\eqref{eq:iw3} that the right-hand side is contained in the left-hand side, as $D_r(\g,P_0)$ is a closed subring of $D_r(HP_0)$ (in turn a closed subring of $D_r(G_0)$).
  As $D_r(P_0) = U_r(\p) K[P_0]$ we have $D_r(\g,P_0) = U_r(\g) K[P_0]$, which is contained in the right-hand side by \eqref{eq:iw2}.
\end{proof}

\begin{lem}\label{lm:U_r-p}
  Suppose that $W$ is a finite-dimensional $\p_K$-module.
  If $r^{p^m} \in p^{\Q} \cap (p^{-1},p^{-1/\kappa(p-1)})$ for some $m \ge 0$ and $r$ is sufficiently close to 1, then we have $U_r(\p) \wh\otimes_{U(\p)} W \cong W$.
\end{lem}

\begin{proof}
  By \cite[Proposition 18.17]{MR2810332} there exists a compact open subgroup $P_{00}$ of $P$ such that $W$ lifts to a locally analytic representation of $P_{00}$.
  Take $r$ sufficiently close to 1 so that $H^{+,m} \subset P_{00}$.
  Note that $U_r(\p) = D_s(H^{+,m}) \subset D_r(H^+)$ and that $W$ is a coadmissible $D(H^{+,m})$-module, where $s := r^{p^m}$.
  We have
  \begin{equation*}
    W \to U_r(\p) \wh\otimes_{U(\p)} W \to D_s(H^{+,m}) \wh\otimes_{D(H^{+,m})} W
  \end{equation*}
  whose composition is an isomorphism for all $r$ sufficiently close to 1, by coadmissibility of $W$. But the first map has
  dense image, hence is surjective (as the image is finite-dimensional, hence complete), 
  so both maps are bijective.
\end{proof}

Fix again $r^{p^m} \in p^{\Q} \cap (p^{-1},p^{-1/\kappa(p-1)})$ for some $m \ge 0$ and $r$ sufficiently close to 1.
Recall that $\m_r$ is defined to be the $U_r(\g)$-submodule of $D_r(HP_0) \otimes_{D(\g,P_0)} M$ generated by $M$.
It is finitely generated and hence carries a canonical topology such that the image of $M$ in $\m_r$ is dense.
As in \cite[Sublemma 5.6]{OS2}
it follows that $\m_r \cong D_r(\g,P_0) \otimes_{D(\g,P_0)} M$.

\begin{lem}\label{lm:m_r}
  Assume $r^{p^m} \in p^{\Q} \cap (p^{-1},p^{-1/\kappa(p-1)})$ for some $m \ge 0$ and $r$ sufficiently close to 1.
  Then $\m_r \cong U_r(\g) \otimes_{U(\g)} M \cong U_r(\u_P^-) \otimes_{U(\u_P^-)} M$ (with canonical topologies).
\end{lem}

\begin{proof}
  We prove this more generally for any $M \in \mathcal O^P$, where we define $\m_r$ to be $D_r(\g,P_0) \otimes_{D(\g,P_0)} M$.
  Then $\m_r$ is a finitely generated $D_r(\g,P_0)$-module (hence also a finitely generated $U_r(\g)$-module) and we endow it with its canonical topology.
  To show that the natural maps $U_r(\u_P^-) \otimes_{U(\u_P^-)} M \to U_r(\g) \otimes_{U(\g)} M \onto \m_r$ are isomorphisms,
  we use a finite presentation $\bigoplus_{i=1}^n U(\g) \otimes_{U(\p)} W_i \to U(\g) \otimes_{U(\p)} W \to M \to 0$ in $\cop$, where $W$ and $W_i$ are   finite-dimensional locally analytic representations of $P$ that are direct sums of absolutely simple $\l_{P,K}$-modules, 
  to reduce to the case where $M = U(\g) \otimes_{U(\p)} W$. 
  Then, using \eqref{eq:iw2} of Lemma~\ref{lm:iwahori-decomp} (and canonical $U_r(\g)$-topologies) as well as Lemma~\ref{lm:U_r-p}, we get
  \begin{align*}
    U_r(\g) \otimes_{U(\g)} M & \cong U_r(\g) \otimes_{U(\p)} W \cong U_r(\g) \wh\otimes_{U(\p)} W \\
    & \cong U_r(\u_P^-) \wh\otimes (U_r(\p) \wh\otimes_{U(\p)} W). \\
    & \cong U_r(\u_P^-) \otimes W.
  \end{align*}
  By comparing this with
  \begin{align*}
    \m_r &= D_r(\g,P_0) \otimes_{D(\g,P_0)} U(\g) \otimes_{U(\p)} W = D_r(\g,P_0) \otimes_{D(P_0)} W \\
           & = D_r(\g,P_0) \wh\otimes_{D(P_0)} W \\
    &= (U_r(\u_P^-) \wh\otimes D_r(P_0)) \wh\otimes_{D(P_0)} W \\
    &= U_r(\u_P^-) \wh\otimes (D_r(P_0) \wh\otimes_{D(P_0)} W) \\
    &= U_r(\u_P^-) \otimes W
  \end{align*}
  (using \eqref{eq:iw4} of Lemma~\ref{lm:iwahori-decomp}) we deduce the claim.
\end{proof}

\begin{lem}\label{lm:feaux}
  Suppose $V$ is a Banach space over $K$, $\cH$ a finitely generated commutative subalgebra of $\End_{\mathrm{cts}}(V)$,
    such that every $v \in V$ can be written as a convergent series $v = \sum_{\lambda} v_\lambda$ of $\lambda$-eigenvectors
  (with respect to the filter of cofinal subsets of $\Hom_{\mathrm{alg}}(\cH,K)$).
  If $M \subset V$ is a dense $\cH$-stable subspace such that $M = \bigoplus_\lambda M_\lambda$ and $M_\lambda$ is finite-dimensional for all $\lambda$,
  then we have an order-preserving bijection
  \begin{equation*}
    \{ \text{$\cH$-invariant closed subspaces of $V$} \} \longleftrightarrow \{ \text{$\cH$-invariant subspaces of $M$} \}
  \end{equation*}
  given by intersecting with $M$, respectively taking closure in $V$.
\end{lem}

\begin{remark}\label{rk:feaux}
  In particular, $M_\lambda = V_\lambda$ for all $\lambda$. Also, if $W$ is an $\cH$-invariant closed subspace of
  $V$, then $M \cap W \subset W$ and $M/(M \cap W) \subset V/W$ also satisfy the conditions in the lemma.
  (Note that if $\sum_\lambda v_\lambda \in W$, then $v_\lambda \in W$ for any $\lambda$, by \cite[Kor.\ 1.3.8]{feaux-diss}.)
\end{remark}

\begin{proof}
  By the assumptions on $V$, \cite[Kor.\ 1.3.12]{feaux-diss} gives uniqueness of the representation $v = \sum_\lambda v_\lambda$ and
  continuous projection maps $\pi_\lambda : V \onto V_\lambda$, $v \mapsto v_\lambda$. As $M = \bigoplus_\lambda M_\lambda$ is dense in $V$ and
  $M_\lambda$ is finite-dimensional, $\pi_\lambda(M) = \pi_\lambda(M_\lambda)$ has to equal $V_\lambda$, i.e. $M_\lambda = V_\lambda$.
  Therefore, the last claim follows from the last item of \cite[Kor.\ 1.3.12]{feaux-diss}.
\end{proof}

\begin{thm}\label{thm:OS-4.7}
  Assume \cite[Assumption 4.1]{OS3} for the absolute root system. Let $M \in \cop$ be such that 
  \begin{enumerate}
  \item $P$ is maximal among parabolic subgroups of $G$ such that $M \in \copp$, and
  \item $M$ is simple as $U(\g)$-module.
  \end{enumerate}
  Assume $r^{p^m} \in p^{\Q} \cap (p^{-1},p^{-1/\kappa(p-1)})$ for some $m \ge 0$ and $r$ sufficiently close to 1, so in particular $\m_r \ne 0$.
  Then for every $g \in G_0$ the $U_r(\g)$-module $\delta_g \star \m_r$ is simple, and if $\delta_{g_1} \star \m_r \cong \delta_{g_2} \star \m_r$
  as $U_r(\g)$-modules for some $g_i \in G_0$, then $g_1 H^{-,m} P_0 = g_2 H^{-,m} P_0$.
\end{thm}

This generalizes \cite[Theorem 4.7]{OS3}.

\begin{proof}
  \emph{We first show that $\m_r$ is simple (which implies that $\delta_g \star \m_r$ is simple), following the proof in \cite[Theorem 4.5]{OS3}.}
  For this it will suffice to show that the map $M \to \m_r = U_r(\g)\otimes_{U(\g)} M$ is injective and satisfies the hypotheses of Lemma~\ref{lm:feaux} with $V := \m_r$
  and $\cH := U(\mathfrak a_P)$, where $\alggrp A_{\alggrp P}$ is the maximal split subtorus of the center of $\alggrp L_{\alggrp P}$ (because then every (closed) $U_r(\g)$-submodule
  of $\m_r$ corresponds to a $U(\g)$-submodule of $M$). 
  We will show this claim is true more generally for any $M \in \copp$.

  Suppose first that $M = U(\g) \otimes_{U(\p)} W \in \copp$ for some $W$, a finite-dimensional $\p_K$-module that is a direct sum of absolutely simple $\l_{P,K}$-modules. 
  As in the proof of Lemma~\ref{lm:m_r}, $\m_r \cong U_r(\u_P^-) \otimes W$ as $U_r(\u_P^-)$-module (with canonical Banach topology), so 
  $M \to \m_r$ is identified with the injection $U(\u_P^-) \otimes W \into U_r(\u_P^-) \otimes W$.
  Note that $U(\mathfrak a_P)$ acts by continuous endomorphisms on $\m_r$ and that the $U(\mathfrak a_P)$-eigenspaces of $M$ are finite-dimensional 
  (as the roots $\Phi^-\setminus \Phi_P^-$ of $\u_P^-$ remain non-trivial on $\mathfrak a_P$).
  On the other hand, $U_r(\u_P^-) = D_s(H^{-,m})$ ($s = r^{p^m}$) has the description
  \begin{equation*}
    U_r(\u_P^-) = \bigg\{ \sum_{\beta \in \Z_{\ge 0}^d} d_\beta \mathbf{X}^\beta : d_\beta \in K,\ |d_\beta| s^{\kappa|\beta|} \to 0 \bigg\}
  \end{equation*}
  by \cite[\S5]{schmidt}, where $X_1,\dots,X_d$ is an ordered $\cO_L$-basis of
  $\log(H^{-,m}) = p^{m+m_0} \g_0 \cap \u_P^- = \bigoplus_{\alpha \in \Phi^-\setminus \Phi_P^-} (p^{m+m_0} \g_0 \cap
  \u_\alpha)$ consisting of $\mathfrak a_P$-eigenvectors. Moreover the norm is given by
  \begin{equation*}
    \bigg\lVert \sum_{\beta \in \Z_{\ge 0}^d} d_\beta \mathbf{X}^\beta \bigg\rVert_s = \sup_\beta |d_\beta| s^{\kappa|\beta|}.
  \end{equation*}
  By Lemma~\ref{lm:cat-O-p-equivalence}, $\mathfrak a_P$ acts diagonalizably on $W$, hence every vector of $\m_r$
  is a convergent sum of $\mathfrak a_P$-eigenvectors.

  For general $M \in \copp$ take a finite-dimensional $\p_K$-submodule $W \subset M$ giving an exact sequence
  $0 \to \del \to U(\g) \otimes_{U(\p)} W \to M \to 0$ in $\copp$.   Tensoring with $U_r(\g)$ gives a commutative diagram
  \begin{equation*}
    \xymatrix{
    0\ar[r] & {\mathstrut\del} \ar[r]\ar@{^{(}->}[d] & U(\g) \otimes_{U(\p)} W \ar[r]\ar@{^{(}->}[d] & M\ar[r]\ar[d] & 0 \\
    0\ar[r] & U_r(\g)\del \ar[r] & U_r(\g) \otimes_{U(\p)} W \ar[r] & \m_r \ar[r] & 0
    }
  \end{equation*}
  where the middle vertical arrow is injective by above, and the left square is Cartesian by Lemma~\ref{lm:feaux} (noting that $U_r(\g)\del$ is the closure of $\del$), hence the right vertical arrow is injective.
  By Remark~\ref{rk:feaux} we see that Lemma~\ref{lm:feaux} applies to $M \into \m_r$.

  \emph{Now assume that $\delta_{g_1} \star \m_r \cong \delta_{g_2} \star \m_r$ as $U_r(\g)$-modules for some $g_i \in G_0$.}
  Following the beginning of the proof of \cite[Theorem 4.7]{OS3} we will first reduce to the case where $g_1 = 1$, $g_2 \in U^-_{P,0}$.
    Let $I_0$ be the Iwahori that is the inverse image under $\alggrp G_x(\cO_L) \to \alggrp G_x^{\red}(k_L)$ of $\alggrp{\o B}_x$, where
  $\alggrp{\o B}_x$ is the minimal parabolic defined by $\Phi^+$ in the reductive quotient $\alggrp G_x^{\red}$ of $\alggrp G_x \times_{\cO_L} k_L$.
  Then we get an Iwahori--Bruhat decomposition $\alggrp G_x(\cO_L) = \coprod_{w \in W/W_I} I_0 \dot w (P \cap \alggrp G_x(\cO_L))$ using the argument of \cite[\S IV.5]{AHHV},
  where $I \subset \Delta$ is the subset corresponding to $\alggrp P$, $W_{I}$ is the subgroup of $W$ generated by reflections $s_\alpha$ for $\alpha \in I$,
  and where $\dot w$ denotes a choice of representative in $\alggrp G_x(\cO_L) \cap N_{G}(S)$.
    From $\alggrp{G}_x(\cO_L) \subset G_0$ and $G = \alggrp{G}_x(\cO_L)P$ we have $G_0 = \alggrp{G}_x(\cO_L) P_{0}$, so $G_0 = \coprod_{w \in W/W_I} I_0 \dot w P_0$.
  As in the proof of \cite[Theorem 4.7]{OS3} we reduce to $g_1 = \dot w$ and $g_2 = u$ for some $w \in W$, $u \in U^-_P \cap I_0 \subset U^-_{P,0}$.

  Let now $\alggrp Q'$ (containing $\alggrp P'$) denote the parabolic subgroup of $\alggrp G'$ that is maximal subject to $M \in \cO^{\q'}$.
  Choose a set of simple roots $\Delta'$ of $\Phi'$ such that $\Delta \subset \Delta'|_{\alggrp S} \subset \Delta \cup \{0\}$
  and let $\Delta'_0 \subset \Delta'$ be the subset of simple roots restricting trivially to $\alggrp S$.
  Recall that $\Delta'$ is endowed with the $*$-action of $\Gamma := \Gal(L'/L)$. 
  Let $J'$ (resp.\ $I'$) denote the subset of $\Delta'$ corresponding to $\alggrp Q'$ (resp.\ $\alggrp P'$).
  The maximality conditions on $P$ (resp.\ $Q'$) imply that $I' = \bigcap_{\gamma \in \Gamma} \gamma(J')$.
  Then $I = (I'|_{\alggrp S}) \setminus \{0\}$ as subset of $\Delta$.

  Assuming $w \notin W_I$, by Step 1 of \cite[Theorem 4.7]{OS3} there exists a reduced relative root $\beta \in \Phi^+ \setminus \Phi^+_I$ such
  that $y':=\Ad(u^{-1})(y) \in \u^-_{P,K}$ fails to be injective on $\m_r$ for any $y \in \g_{(-\beta),K}\setminus \{0\}$ (where $\g_{(-\beta)} := \g_{-\beta} \oplus \g_{-2\beta}$).
  Take any lift $\beta' \in \Phi'^+ \setminus \Phi'^+_{I'}$ of $\beta$ and $y \in \g'_{-\beta',K}\setminus \{0\}$.
  Writing $y' = \sum_{\gamma'\in \Phi'^+\setminus \Phi'^+_{I'}} y_{\gamma'}$ with $y_{\gamma'} \in \g'_{-\gamma'}$ the argument
  in \cite[Theorem 4.7]{OS3} shows that $\gamma'^+ \in \Phi'^+_{J'}$ for the minimal element $\gamma'^+$ 
  of $\{\gamma' : y_{\gamma'} \ne 0\}$ for a certain lexicographic order on $\Z_{\ge 0}\Delta'$. 
  (By \cite[Corollary 5.5]{OS3} any nonzero element of $\u_{Q',K}^-$ acts injectively on $M$.)
  As $y'=\Ad(u^{-1})(y)$ with $u \in U^-_P$ we have $y_{\beta'} = y \ne 0$, and $y_{\gamma'} \ne 0$ implies $\gamma' \in \beta'+\Z_{\ge 0}(\Delta'\setminus I')$.
  By minimality of $\gamma'^+$ we deduce $\gamma'^+ = \beta'$, so $\beta' \in \Phi'^+_{J'}$. 
  Applying the same argument to the lift $\gamma(\beta')$ of $\beta$ for $\gamma \in \Gamma$, we deduce that
  $\beta' \in \bigcap_{\gamma \in \Gamma} \gamma^{-1}(\Phi'^+_{J'}) = \Phi'^+_{I'}$, contradiction. 
  Hence $w \in W_I$ and so we may assume that $g_1 = 1$.

  By Lemma~\ref{lm:m_r} we have $\m_r = U_r(\g) \otimes_{U(\g)} M$. By Corollary~\ref{cor:distrib} we have
  \begin{equation*}
    \m'_r := U_r(\g') \otimes_{U(\g')} M = U_r(\g') \otimes_{U_r(\g)} \m_r,
  \end{equation*}
  and moreover $\m_r \cong \delta_{u} \star \m_r$ as $U_r(\g)$-modules implies 
  $\m'_r \cong \delta_{u} \star \m'_r$ as $U_r(\g')$-modules (as $u \in U^-_{P,0} \subset G_0$).   Note that $\m'_r \ne 0$     because the map $M \to \m'_r$ is injective by the first part of the proof (applied to $\alggrp G'$).

  Recall that $u \in U^-_{P,0} \subset U^-_{P',0}$.
  Write $u = u_1 u_2$ with $u_1 \in U^-_{Q',0}$, $u_2 \in U_{P',0}^- \cap L_{Q'}$. We now show that $\delta_{u_2} \star \m_r' \cong \m_r'$
  as $U_r(\g')$-modules, which implies that $\m_r' \cong \delta_{u_1} \star \m_r'$ as $U_r(\g')$-modules.
  Exactly as in Step 2 of \cite[Theorem 4.7]{OS3} (as $M \in \mathcal O^{\q'}$) we can integrate the locally finite action of $\u_{P'}^- \cap \l_{Q'}$ on $M$ to a locally finite locally analytic 
  action of the unipotent group $U_{P'}^- \cap L_{Q'}$. 
  Moreover, $u \circ X \circ u^{-1} = \Ad(u)(X)$ on $M$ for all $u \in U_{P'}^- \cap L_{Q'}$, $X \in \g'$. 
  (Use, for example, that $\log(\Ad(u)) = \ad(\log(u)) \in \GL(\g')$ for all $u \in U_{P'}^- \cap L_{Q'}$.)
    Then we have an isomorphism $\delta_{u_2} \star M \congto M$, $x \mapsto u_2 \cdot x$ as $U(\g')$-modules, which extends
  to an isomorphism $\delta_{u_2} \star \m'_r \congto \m'_r$ as the map $U(\g') \to U_r(\g')$ is $U_{P',0}^- \cap L_{Q'}$-equivariant.

  We now have $\m_r' \cong \delta_{u_1} \star \m_r'$ as $U_r(\g')$-modules.
  Applying \cite[Theorem 4.7]{OS3} for the split group $\alggrp G'$ (note that since we are now reduced to $g_1 = 1$ and $g_2 = u_1 \in U^-_{Q',0}$ 
  the proof there only uses that $Q'$ is maximal for $M \in \cO^{\q'}$ and that $\m'_r = U_r(\g') \otimes_{U(\g')} M \ne 0$, i.e.\ the $Q'$-action on $M$ is not used!)\ 
  we deduce that $u_1 \in H'^m Q'_0$ and hence $g_1^{-1} g_2 = u_1 \in H'^{m} Q'_0 \cap G$. 
  
  We next show that $g_1^{-1} g_2 \in H'^{m} Q'_0 \cap G = H'^{m} P'_0 \cap G$. Consider the reduction map $\pi \colon \alggrp G'_{x'}(\cO_{L'}) \onto \alggrp G'_{x'}(\cO_{L'}/p^{m+m_0})$,
  which has kernel $H'^m$ by Lemma~\ref{lm:congruence-subgp}. 
  For each subset $K' \subset \Delta'$ we get a parabolic subgroup $\alggrp P'_{K'}$ of $\alggrp G'$ that contains $\alggrp T'$,
  and its schematic closure of $\alggrp P'_{K'}$ in $\alggrp G'_{x'}$ is a smooth $\cO_{L'}$-group scheme $\alggrp P'_{K',x'}$ with generic fiber $\alggrp P'_{K'}$
  \cite[\S 2.9]{KP}. We note that $\alggrp P'_{K_1',x'} \times_{\alggrp G'_{x'}} \alggrp P'_{K_2',x'} = \alggrp P'_{K_1' \cap K_2',x'}$ for $K_1', K_2' \subset \Delta'$.
  (By \cite[Proposition 2.9.1(4), Proposition 2.9.2]{KP} the left-hand side is $\cO_{L'}$-smooth, hence $\cO_{L'}$-flat, and we can verify
  the equality on the geometric generic fiber, where it is clear.)
  For any $\gamma \in \Gamma$ we get a morphism of buildings $\gamma \colon \cB_{L'}(\alggrp G') \to \cB_{L'}(\alggrp G')$, commuting 
  with the embedding $\cB_L(\alggrp G) \into \cB_{L'}(\alggrp G')$, hence stabilizing $x'$. Therefore the $\gamma$-linear isomorphism
  $\alggrp G' \to \alggrp G'$ over $L'$ extends to a $\gamma$-linear isomorphism
  $\alggrp G'_{x'} \to \alggrp G'_{x'}$ over $\cO_{L'}$ \cite[Corollary 2.8.10]{KP}, and it identifies $\alggrp Q' = \alggrp P'_{J'}$ with $\alggrp P'_{\gamma(J')}$.
      It follows that $\pi(g_1^{-1} g_2) \in \bigcap_{\gamma \in \Gamma} \alggrp P'_{\gamma(J'),x'}(\cO_{L'}/p^{m+m_0}) = \alggrp P'_{x'}(\cO_{L'}/p^{m+m_0})$ 
  (where the equality holds by the fiber product computation above, and recall $\alggrp P' = \alggrp P'_{I'}$),
  so $g_1^{-1} g_2 \in H'^{m} P'_0 \cap G$.

  We now know that $g_1^{-1} g_2 \in H'^{m} P'_0 \cap G$.
  As $\alggrp U_{\alggrp P}^- \times \alggrp P \to \alggrp G$ is an open immersion of schemes over $L$, we finally deduce that $g_1^{-1} g_2 \in (H'^{-,m} \cap U_P^-)(P'_0 \cap P) = H^{-,m} P_0$.
    For the equality $H'^{-,m} \cap U_P^- = H^{-,m}$ we used that the logarithm map is bijective (even a locally analytic isomorphism) between $U_{P'}^-$ and $\u_{P'}^-$
  (as $\alggrp U_{\alggrp P'}^-$ is unipotent) and that $(p^{m+m_0} \g'_0 \cap \u_{P'}^-) \cap \u_P^- = p^{m+m_0} \g_0 \cap \u_P^-$.
\end{proof}

\begin{remark}\label{rk:OS-4.7}  The result is true even if we only assume that $M \in \copp$ (satisfying (i) and (ii)) and if we \emph{define} $\m_r := U_r(\g)\otimes_{U(\g)} M$. 
  The simplicity proof goes through as before.
  Now suppose that $\delta_{g_1} \star \m_r \cong \delta_{g_2} \star \m_r$ as $U_r(\g)$-modules for some $g_i \in G_0$.
  By Corollary~\ref{cor:distrib} we have $\delta_{g_1} \star \m_r' \cong \delta_{g_2} \star \m_r'$ as $U_r(\g')$-modules, where $\m_r' := U_r(\g') \otimes_{U(\g')} M = U_r(\g') \otimes_{U_r(\g)} \m_r$.
  If $\alggrp G^{\der}$ is simply connected, then $M \in \mathcal O^{\q'}$ lifts to an object of $\mathcal O^{Q'}$, after perhaps replacing $C$ by a finite extension, by Lemma~\ref{lem:semisimple-lift}.
  By \cite[Theorem 4.7]{OS3} for the split group $\alggrp G'$ we deduce $g_1^{-1} g_2 \in H'^{m} Q'_0 \cap G$, and the final part of the above proof shows that $H'^{m} Q'_0 \cap G = H^{-,m} P_0$.
  For general $\alggrp G$, let $\pi : \wt{\alggrp G}^{\lowprime} \onto \alggrp G'$ be a $z$-extension, so $(\wt {\alggrp G}^{\lowprime})^\der$ is simply connected and $\pi : \wt G' \onto G'$ surjects on $L'$-points.
  We may and will suppose that $\wt{\alggrp G}^{\lowprime}$ and the torus $\ker(\pi)$ are split as well.   By \cite[Proposition 1.1.4]{MR3905466}, $\pi$ extends to a surjective map $\wt{\alggrp G}^{\lowprime}_{x'} \onto \alggrp G'_{x'}$ of (connected reductive) parahoric group schemes over $\cO_{L'}$ with kernel a split torus.
  In particular, by checking on special fibers, the map on $\cO_{L'}$-Lie algebras $\wt{\g}'_0 \to \g'_0$ is surjective, where $\wt{\g}'_0 := \Lie \wt{\alggrp G}^{\lowprime}_{x'}$.
  (Lie algebras of smooth group schemes commute with base change.)   Also, by taking fppf-cohomology we obtain $\pi(\wt G'_0) = G'_0$.
  Following our construction above, we have $\wt H' := \BCH(p^{m_0} \wt \g'_0) \lhd \wt G'_0$ and $\pi(\wt H') = H'$.
  Letting $\wt{\alggrp Q}^{\lowprime}$ be the pre-image of $\alggrp Q'$ under $\pi$ and $\wt Q'_0 := \wt Q' \cap \wt G'_0$ we have $\pi(\wt Q'_0) = Q'_0$.
  As the map $\wt{\g}'_0 \onto \g'_0$ splits as $\Z_p$-modules, there exists a minimal set $\wt x_i'$ ($1 \le i \le d'$) of topological generators of $\wt H'$ such that $x'_i := \pi(\wt x'_i)$ ($1 \le i \le d$) is a minimal set of topological generators of $H'$ and $\pi(\wt x'_i) = 1$ for $d < i \le d'$ (for some $d' \ge d \ge 0$).
  The construction in \cite[\S2.2.6]{OS} then shows that the natural map $D^{(L)}(\wt H') \to D^{(L)}(H')$ is continuous with respect to $\lvert\cdot\rvert_r$-norms, so induces a map $D^{(L)}_r(\wt H') \to D^{(L)}_r(H')$ (in fact, both maps are strict surjections).
  By taking the closures of enveloping algebras we obtain a commutative square
  \begin{equation*}
    \xymatrix{
    U(\wt\g') \ar@{->>}[r]\ar[d] & U(\g')\ar[d] \\
    U_r(\wt\g',\wt H')\ar@{-->}[r] & U_r(\g',H')
    }
  \end{equation*}
  where the bottom map is a morphism of Banach algebras and $\wt\g' := \Lie \wt G'$.
  By density of the enveloping algebras, all maps are equivariant for the adjoint action of $\wt G'_0 \onto G'_0$.
  Pick $\wt g'_i \in \wt G'_0$ for $i = 1,2$ such that $\pi(\wt g'_i) = g_i$ in $G'_0$.
  From $\delta_{g_1} \star \m_r' \cong \delta_{g_2} \star \m_r'$ as $U_r(\g')$-modules we deduce that $\delta_{\wt g'_1} \star \m_r' \cong \delta_{\wt g'_2} \star \m_r'$ as $U_r(\wt\g')$-modules.
  Hence, as $(\wt {\alggrp G}')^\der$ is simply connected, we deduce $(\wt g'_1)^{-1}\wt g'_2 \in \wt H'^m \wt Q'_0$.
  Applying $\pi$ we obtain $g_1^{-1} g_2 \in H'^m Q'_0 \cap G$ and we conclude as before.
\end{remark}

We can now complete the proof of Theorem~\ref{thm:OS-main}.

\begin{proof}[Proof of Theorem~\ref{thm:OS-main}(iii)]
  In the proof of \cite[Proposition 3.7]{OS3} an Iwahori--Bruhat decomposition is used.  Here we temporarily forget our choice of $G_0$ and $P_0$ above
  and can use any special point $x$ (e.g.\ the one above). Let $\alggrp G_x$ over $\cO_L$ denote the (connected smooth) parahoric group scheme of \cite{BT2}, with generic
  fiber $\alggrp G$.  Let $I_0$ be the Iwahori that is the inverse image under $G_0 := \alggrp G_x(\cO_L) \to \alggrp G_x^{\red}(k_L)$ of $\alggrp{\o B}_x$, where
  $\alggrp{\o B}_x$ is the minimal parabolic defined by $\Phi^+$ in the reductive quotient $\alggrp G_x^{\red}$ of $\alggrp G_x \times_{\cO_L} k_L$. 
  (By \cite{BT2}, $\alggrp S$ extends to $\cO_L$, $\alggrp S_x \times k_L$ is a maximal split torus of $\alggrp G_x^{\red}$ and the
  root system of $(\alggrp G_x^{\red}, \alggrp S_x \times k_L)$ is naturally a subset of that of $(\alggrp G, \alggrp S)$.)  Then we get
  an Iwahori--Bruhat decomposition $G_0 = \coprod_{w \in W/W_I} I_0 \dot w (P \cap G_0)$ using the argument of \cite[\S IV.5]{AHHV}.

  For the proof of \cite[Theorem 4.5]{OS3} we define $G_0$, $P_0$, $H = H^- H^+$, $H^m = H^{-,m} H^{+,m}$ as in the proof of Theorem \ref{thm:OS-4.7}.
  The proof of \cite[Theorem 4.5]{OS3} (and \cite[Theorem 5.8]{OS2}) then proceeds just as before, except that we apply Theorem \ref{thm:OS-4.7} instead of \cite[Theorem 4.7]{OS3}.
\end{proof}

\newcommand{\etalchar}[1]{$^{#1}$}
\providecommand{\bysame}{\leavevmode\hbox to3em{\hrulefill}\thinspace}
\providecommand{\MR}{\relax\ifhmode\unskip\space\fi MR }
% \MRhref is called by the amsart/book/proc definition of \MR.
\providecommand{\MRhref}[2]{%
  \href{http://www.ams.org/mathscinet-getitem?mr=#1}{#2}
}
\providecommand{\href}[2]{#2}

\end{document}